\documentclass{article}
\UseRawInputEncoding

\title{\huge On Cohen-Jones Isomorphism  \\in String Topology}

\date{}
\author{Syunji Moriya\footnote{Corresponding address: Department of Mathematics and Information Sciences, Osaka Prefecture University, Sakai, 599-8531, Japan   \   
E-mail adress: \texttt{moriyasy@gmail.com} }}

\usepackage{amsthm}
\usepackage{amscd}
\usepackage{amsmath,amsthm,amssymb,amscd,mathrsfs,picture}
\usepackage{setspace}
\usepackage[arrow,matrix]{xy}
\input xy
\xyoption{all}

\usepackage{here}

\usepackage{enumerate}
\theoremstyle{definition}
\newtheorem{defi}{Definition}[subsection]

\newtheorem{prop}[defi]{Proposition}

\newtheorem{rem}[defi]{Remark}

\newtheorem{lem}[defi]{Lemma}

\newtheorem{thm}[defi]{Theorem}

\newtheorem{cor}[defi]{Corollary}
\newtheorem{exa}[defi]{Example}

\newenvironment{itemize2}{
 \begin{list}{$\bullet$\ \ }%
 {\setlength{\itemindent}{0pt}
  \setlength{\leftmargin}{1.5em}    
  \setlength{\rightmargin}{0em}   
  \setlength{\labelsep}{0.5em}      
  \setlength{\labelwidth}{3em}    
  \setlength{\itemsep}{0em}       
  \setlength{\parsep}{0em}        
  \setlength{\listparindent}{0em} 
  \setlength{\topsep}{0pt} 
}
}{
 \end{list}
}

\newcommand{\eps}{\epsilon}

\newcommand{\ul}{\underline}
\newcommand{\Z}{\mathbb{Z}}
\newcommand{\R}{\mathbb{R}}

\newcommand{\LMT}{LM^{-\tau}}
\newcommand{\LMTM}{LM^{-TM}}
\newcommand{\MT}{M^{-\tau}}
\newcommand{\MTM}{M^{-TM}}
\newcommand{\TV}{\widetilde{V}}
\newcommand{\EVN}{ev^*\M}
\newcommand{\M}{\mathcal{M}}
\newcommand{\D}{\mathcal{D}}

\newcommand{\tF}{\widetilde{\mathcal{F}}}

\newcommand{\tPhi}{\widetilde{\Phi}}
\newcommand{\tc}{\widetilde{c}}

\newcommand{\Sph}{\mathbb{S}}
\newcommand{\CL}{\mathbb{L}}
\newcommand{\athc}{\mathbb{T}}
\newcommand{\bthc}{\widehat{\mathbb{T}}}
\newcommand{\IM}{\mathbb{IM}}
\newcommand{\TIM}{\overline{\mathbb{IM}}}
\newcommand{\BIM}{\widetilde{\mathbb{IM}}}
\newcommand{\tpsi}{\widetilde{\psi}}
\newcommand{\bpsi}{\overline{\psi}}
\newcommand{\teps}{\widetilde{\epsilon}}
\newcommand{\id}{\operatorname{id}}
\newcommand{\trho}{\widetilde{\rho}}
\newcommand{\homega}{\widehat{\omega}}
\newcommand{\hz}{\widehat{z}}
\newcommand{\bnu}{\overline{\nu}}
\newcommand{\tUpsilon}{\widetilde{\Upsilon}}
\newcommand{\hUpsilon}{\widehat{\Upsilon}}

\newcommand{\B}{\mathcal{B}}
\newcommand{\tB}{\widetilde{\mathcal{B}}}

\newcommand{\T}{\mathscr{T}}
\newcommand{\CT}{\mathscr{T}^c}
\newcommand{\K}{\mathcal{K}}
\newcommand{\CK}{\mathcal{CK}}
\newcommand{\CB}{\mathcal{CB}}
\newcommand{\oper}{\mathcal{O}}
\newcommand{\aoper}{\mathcal{P}}
\newcommand{\toper}{\widetilde{\mathcal{O}}}
\newcommand{\boper}{\overline{\mathcal{O}}}
\newcommand{\len}{\mathcal{L}}

\newcommand{\Doper}{\mathcal{D}}
\newcommand{\tDoper}{\widetilde{\mathcal{D}}}
\newcommand{\OB}{\mathrm{Ob}}
\newcommand{\C}{\mathcal{C}}
\newcommand{\CC}{\mathcal{C}^\Delta}
\newcommand{\TOP}{\mathcal{TOP}}

\newcommand{\CG}{\mathcal{CG}} 
\newcommand{\CCG}{\mathcal{CG}^\Delta}
\newcommand{\SP}{\mathcal{SP}}
\newcommand{\Alg}{\mathcal{ALG}}
\newcommand{\sqcat}{\tilde\square}
\newcommand{\thc}{\operatorname{THC}}
\newcommand{\tot}{\operatorname{Tot}}
\newcommand{\ttot}{\widetilde{\operatorname{Tot}}}
\newcommand{\MK}{\widetilde{\K}}
\newcommand{\MCK}{\overline{\CK}}

\newcommand{\Inhom}{\operatorname{Map}}
\newcommand{\Map}{\operatorname{Map}}
\newcommand{\hotimes}{{\,\hat\otimes\,}}
\newcommand{\FO}{F_{\oper}}        
\newcommand{\GL}{G_{l}} 
\newcommand{\Ho}{\mathrm{Ho}}
\newcommand{\epsS}{($\eps$-1)}
\newcommand{\epsSigma}{($\eps$-2)}
\newcommand{\epsev}{($\eps$-4)}
\newcommand{\epscone}{($\eps$-5)}
\newcommand{\epscomp}{($\eps$-3)}
\newcommand{\psiS}{($\psi$-1)}
\newcommand{\psiSigma}{($\psi$-2)}
\newcommand{\psiev}{($\psi$-5)}

\newcommand{\psicomp}{($\psi$-3)}
\newcommand{\psibp}{($\psi$-4)}
\newcommand{\tpsiuni}{($\tpsi$-1)}
\newcommand{\tpsiS}{($\tpsi$-2)}
\newcommand{\tpsiSigma}{($\tpsi$-3)}
\newcommand{\tpsiev}{($\tpsi$-6)}

\newcommand{\tpsicomp}{($\tpsi$-5)}
\newcommand{\tpsiend}{($\tpsi$-4)}
\newcommand{\tpsiconti}{($\tpsi$-7)}

\newcommand{\zS}{($z$-1)}
\newcommand{\zSigma}{($z$-2)}

\newcommand{\zcomp}{($z$-3)}
\newcommand{\zcos}{($z$-4)}

\newcommand{\zev}{($z$-5)}
\newcommand{\zconti}{($z$-6)}

\newcommand{\omegaS}{($\omega$-1)}
\newcommand{\omegaSigma}{($\omega$-2)}

\newcommand{\omegacomp}{($\omega$-3)}
\newcommand{\omegaend}{($\omega$-5)}
\newcommand{\omegacos}{($\omega$-4)}
\newcommand{\omegaconti}{($\omega$-6)}

\newcommand{\homegaS}{($\homega$-1)}
\newcommand{\homegaSigma}{($\homega$-2)}
\newcommand{\homegaev}{($\homega$-6)}

\newcommand{\homegacomp}{($\homega$-3)}
\newcommand{\homegaend}{($\homega$-5)}
\newcommand{\homegacos}{($\homega$-4)}
\newcommand{\homegaconti}{($\homega$-7)}

\newcommand{\ordI}{(o-1)}
\newcommand{\ordII}{(o-2)}
\newcommand{\ordIII}{(o-3)}
\newcommand{\lev}{\operatorname{lev}}
\newcommand{\ari}{\operatorname{ar}}
\newcommand{\ci}{\mathbf{c}}
\newcommand{\ve}{\mathbf{v}}

\newcommand{\df}{\overline{d}}
\newcommand{\dl}{\underline{d}}
\newcommand{\colim}{\operatorname{colim}}
\newcommand{\col}{\mathbf{c}}
\newcommand{\ocol}{\bar{\mathbf{c}}}
\newcommand{\ocola}{\bar{\mathbf{c}}_a}
\newcommand{\ocolb}{\bar{\mathbf{c}}_b}
\newcommand{\ucol}{\underline{\mathbf{c}}}
\newcommand{\vs}{\vspace}
\newcommand{\branchII}{\put(0,-13){\line(0,1){24} }}
\newcommand{\branchIII}{\put(0,-13){\line(1,1){24}}}
\newcommand{\branchI}{\put(0,-13){\line(-1,1){24} }}
\newcommand{\branchIV}{\put(-12,-1){\line(1,1){12}}}
\newcommand{\branchV}{\put(12,-1){\line(-1,1){12}}}

\newcommand{\socirc}{\put(0,-13){\circle{3}}}
\newcommand{\somark}{\put(1,-24)}

\newcommand{\sacirc}{\put(-6,-7){\circle{3}}}
\newcommand{\samark}{\put(-17,-20)}

\newcommand{\saacirc}{\put(-12,-1){\circle{3}}}
\newcommand{\saamark}{\put(-23,-14)}

\newcommand{\saaacirc}{\put(-18,5){\circle{3}}}
\newcommand{\saaamark}{\put(-16,8)}

\newcommand{\saaaacirc}{\put(-24,11){\circle{3}}}
\newcommand{\saaaamark}{\put(-24,19)}

\newcommand{\saabbcirc}{\put(0,11){\circle{3}}}
\newcommand{\saabbmark}{\put(0,19)}

\newcommand{\sbbcirc}{\put(12,-1){\circle{3}}}
\newcommand{\sbbmark}{\put(16,-9)}

\newcommand{\sbbacirc}{\put(6,5){\circle{3}}}
\newcommand{\sbbamark}{\put(10,7)}

\newcommand{\sbbbbcirc}{\put(24,11){\circle{3}}}
\newcommand{\sbbbbmark}{\put(24,19)}
\begin{document}
\maketitle
\begin{abstract}
The loop product is an operation in string topology. Cohen and Jones \cite{cj,cohen} gave a homotopy theoretic realization of the loop product as a classical ring spectrum $\LMTM$ for a manifold $M$. Using this, they presented a proof of the statement that the loop product is isomorphic to the Gerstenhaber cup product on the Hochschild cohomology   $HH^*(C^*(M)\,;C^*(M))$ for simply connected $M$. However, some parts of their proof are technically difficult to justify. The main aim of the present paper is to give detailed modification to a geometric part of their proof. To do so, we set up an `up to higher homotopy' version of McClure-Smith's cosimplicial product. We prove a structured version of the Cohen-Jones isomorphism in the category of symmetric spectra.
\end{abstract}
\begin{center}
2010MSC\  primary: 55P50, \ secondary: 55P48\\
Keywords: symmetric spectra, Atiyah duality, associahedra, \\cosimplicial space,
colored operad
\end{center}
\tableofcontents

\section{Introduction}

\indent Throughout this paper, $M$ denotes a smooth closed manifold of a finite dimension $d$. Let $LM$ denote the free loop space of $M$, i.e., $LM$ is the space of all continuous maps from the circle $S^1$ to the manifold $M$ with the compact-open topology.\\
\indent String topology was initiated by Chas and Sullivan \cite{cs}. It is a study of certain algebraic structures on the homology of  $LM$, which can be seen as a generalization of Goldman's Lie algebra for a Riemann surface. These structures are studied in many ways e.g., relation to counting problem of closed geodesics (see Goresky-Hingston \cite{gh}), generalization to Gorenstein spaces (see F\'elix-Thomas \cite{ft}, Kuribayashi-Menichi-Naito\cite{kmn12, kmn13},  Naito \cite{naito}, and Kuribayashi-Naito-Wakatsuki-Yamaguchi \cite{KNWY}). Another interesting subject is the relationship between string topology operations   and  intrinsic operations on Hochschild cohomology of the cochain  due to Gerstenhaber \cite{gerstenhaber} and Jones \cite{jones}.  Cohen and Jones \cite{cj,cohen} gave  a proof of the claim that there exists an isomorphism between the loop product, a most basic string topology operation,
and the Gerstenhaber cup product on the Hochschild cohomology of the singular cochain algebra $C^*(M)$  over a  field $k$
\[
\bigl( H_{*-d}(LM; k),\ \text{the loop product}\bigr) \cong \bigl( HH^*(C^*(M);C^*(M)),\ \text{the cup product}\bigr),
\]
which we call the \textit{Cohen-Jones isomorphism}. In the case of  characteristic zero, 
F\'elix and Thomas \cite{ft2} showed an isomorphism of BV-algebras. (In contrast, Menichi \cite{menichi} showed non-existence of such an isomorphism  in characteristic 2.) \\
\indent Though the idea of the proof of the Cohen-Jones isomorphism in \cite{cj,cohen} is very interesting,  the author of the present paper encountered  technical difficulties with their argument. The difficulties can be divided into two parts. One is the comparison between the loop product and a natural product on the topological Hochschild cohomology. The other is the comparison between topological and chain level Hochschild cohomologies.  The aim of the present paper is to resolve the former part, which we consider the bigger difficulty.  We shall state the main theorem. In \cite{cj} 
the loop product was realized as a product on a Thom spectrum $LM^{-TM}$ in the stable homotopy category. By this product, we regard $LM^{-TM}$ as a classical associative ring spectrum, i.e., a monoid in the stable homotopy category, and call it the \textit{Cohen-Jones ring spectrum}.
 The main result of the present paper is the following theorem, which is regarded as a symmetric spectrum version of the Cohen-Jones isomorphism.
\begin{thm}\label{mainthm1}
There exists a non-unital $A_\infty$-symmetric ring spectrum $\LMT$ satisfying the following two conditions.
\begin{itemize}
\item[(1)] (Theorem \ref{TLMTLMTM})   $\LMT$ is isomorphic to the Cohen-Jones ring spectrum $LM^{-TM}$ as a classical non-unital associative ring spectrum. 
\item[(2)] Suppose $M$ is simply connected. Let $Q$ be a fibrant cofibrant replacement of the function  spectrum  $F(M)$ of $M$ in a suitable model category  (see Proposition \ref{propnonunitalmodel}). Then, $LM^{-\tau}$  and the topological Hochschild cohomology $\thc (Q\, ;Q)$ is weak equivalent as  non-unital $A_\infty$-symmetric ring spectra. 
\end{itemize}
\end{thm}
For terminologies, see subsection \ref{SSNT} and for topological Hochschild cohomology, see \ref{SStophochschild}. The function spectrum $F(M)$ is the spectrum of maps from the  manifold  with a disjoint basepoint $M_+$ to the sphere spectrum. In other words, it is  (a structured version of) the Spanier-Whitehead dual of $M_+$. $F(M)$ is regarded as  an analogue of the cochain algebra $C^*(M)$. The equivalence in  (2) or the isomorphism of Cohen-Jones is more or less a  folklore,  but the author could not find a  rigorous proof of it  anywhere except for the case of rational coefficient.  It is important  to have a detailed proof for further study of the relationship between string topology and Hochschild cohomology.  For works related to (1), Poirier and Rounds  \cite{pr} constructed a chain map which encodes string topology TQFT operations, and Irie \cite{irie} constructed a chain-level  BV-structure including the loop product over reals and give  applications to symplectic topology, based on  Fukaya's idea \cite{fukaya} \\
\indent We shall look at the outline of the proof of the Cohen-Jones isomorphism in \cite{cj,cohen} to explain the main task in the present paper. 
The authors of \cite{cj} use two cosimplicial objects 
\[
(\CL_M)_* ,  \quad  CH^\bullet_M.
\]
$(\CL_M)_*$  is the cosimplicial spectrum given by $(\CL_M)_n=(M^{\times n}_+)\wedge M^{-TM}$ for each cosimplicial degree $n$, where  $M^{-TM}=\Sigma^{-N}Th(\nu )$ is the Thom spectrum for the normal bundle $\nu$ of a fixed embedding $M\to \R^N$. There is a weak equivalence $LM^{-TM}\simeq \tot (\CL_M)_*$, where $\tot$ denotes the totalization, see subsection \ref{SSNT}. (In the present paper, we denote the cosimplicial degree by the superscript $\bullet$ but here, for $(\CL_M)_*$ we use the notation of \cite{cj}.) $CH^\bullet_M$ is the cosimplicial cochain complex given by $CH^n_M=Hom(C^*(M)^{\otimes n},C^*(M))$. The (normalized) total complex of $CH^\bullet_M$ is the Hochschild complex  $CH^*(C^*(M)\,;C^*(M))$.  Though $CH^\bullet_M$ does not appear in \cite{cj}, it is used implicitly (see \cite{cohen}).  They use products on these cosimplicial objects
\[
(\CL_M)_p\wedge (\CL_M)_q\to (\CL_M)_{p+q},\quad {CH}^p_M\otimes {CH}^q_M\to {CH}^{p+q}_M.
\]
The  first one is defined by using a natural product on $M^{-TM}$  and is claimed to induce the loop product on $\tot (\CL_M)_*\simeq LM^{-TM}$, and the second one is the product which induces the Gerstenhaber cup product on the Hochschild complex.  
They define a degreewise quasi-isomorphism $C_*((\CL_M)_*)\simeq CH^\bullet_M$ of cosimplicial chain complexes, which preserves products,  using the Atiyah duality $M^{-TM}\simeq F(M)$ and a quasi-isomorphism $C_*(F(M))\simeq C^*(M)$. (Here, $C_*$ is an unclear  chain functor   for symmetric spectra.) Using this quasi-isomorphism, they essentially state that there is a zig-zag of quasi-isomorphisms
\[
C_*(LM^{-TM})\simeq C_*(\tot (\CL_M)_* )\simeq \tot C_*((\CL_M)_* )\simeq \tot CH^\bullet_M = CH^*(C^*(M);C^*(M))  
\]
which sends the loop product to the cup product. (Here $\tot$ denotes different two notions,  spectrum-level totalization and total complex.)\\
\indent   In the construction of the  quasi-isomorphism $C_*((\CL_M)_*)\simeq CH^\bullet_M$, the topological Hochschild cohomology for $F(M)$ is  used. The topological Hochschild is not well-defined in the classical stable homotopy category since it is defined by using  the symmetry of the monoidal structure and the totalization of a cosimplicial object. It is well-defined in the category of symmetric spectra \cite{hss, mms} since that category has a structure of a monoidal model category.  In \cite{cohen}, Cohen gives  details of construction of \cite{cj} in the category of  symmetric spectra, especially a very effective realization of the Thom spectrum $M^{-TM}$ which is also useful for other application. Actually, the author of the present paper crucially applied it to a model for a knot space to construct a new cohomology spectral sequence in \cite{moriya1}.\\
\indent To explain our tasks, it is necessary to specify  what the  difficulty for the author is. Here, we give a quick view. More detailed explanation is given in subsection \ref{SSproblem}.  The difficulty  is that the first and last coface operators $d^0$, $d^{n+1}$ on $(\CL_M)_n$ are not well-defined. In \cite{cohen}, the coface operators are defined by using morphisms \[
\Delta_r:M^{-\tau}(e)\to M^{-\tau}(e)\wedge (\nu_{2\epsilon}(e)_+),\quad \Delta_l:M^{-\tau}(e)\to  (\nu_{2\epsilon}(e)_+)\wedge M^{-\tau}(e)
\]
(see the paragraph under Theorem 6 of \cite{cohen}), where $M^{-\tau}(e)$ is a model of $M^{-TM}$ which has a structure of unital commutative symmetric ring spectrum. These morphisms are not well-defined morphisms of symmetric spectra since they do not commute with the action of the sphere spectrum. The author  tried to modify this using another model $M^{-\tau}$ of $M^{-TM}$ introduced by Cohen in \cite{cohen}. While $M^{-\tau}$ is {\em non}-unital, the  morphisms analogous to $\Delta_r$ and $\Delta_l$ are well-defined morphisms of symmetric spectra for this  model. Unfortunately,  this modification  caused another problem. In order for a product on a cosimplicial object to induce a product on the totalization, it must satisfy some condition. The only sufficient condition which the author knows is the condition of  McClure-Smith \cite{ms} (for $A_\infty$-structure), which states some compatibility between coface and degeneracy operators and the product. For example, $CH_M^\bullet$ satisfies this condition.  The problem is that the non-unital version of Cohen's product on $(\CL_M)_*$ does not satisfy the condition of McClure-Smith.  \\
\indent In the present paper, for a transparent proof,  we give all involved construction in  the category of symmetric spectra. Our construction are based on the non-unital model of Cohen. We define a  product on a model of $LM^{-TM}$ which encodes the loop product in the category of symmetric spectra  as in (1) of Theorem \ref{mainthm1}, and a product on  a modified version of $(\CL_M)_*$ ($\CL^\bullet$ in the notion of the present paper) which satisfies a slight generalization of the McClure-Smith's condition. The generalization  is given as an  action of a monad $\widetilde{\K}$ on a cosimplicial object. This product is different from the above non-unital version of Cohen's product. We prove the product induced by the $\widetilde{\K}$-action on $\CL^\bullet$ is isomorphic to the product on $\LMT$. Then we construct a zig-zag of weak equivalences
\[
\CL^\bullet \leftarrow \IM^\bullet \rightarrow \thc^\bullet (A',B)
\]
where $A'$ and $B$ are suitable (fibrant or cofibrant) models of $F(M)$ and $\thc^\bullet$ is the cosimplicial object the totalization of which is the topological Hochschild cohomology. As the action of $\widetilde{\K}$ is a generalization of the  McClure and Smith's product, $\thc^\bullet$ has a natural action of $\widetilde{\K}$. The author tried  to connect $\CL^\bullet$ and $\thc^\bullet (A',B)$  by a zig-zag of weak equivalences which preserve the action of $\widetilde{\K}$ but it turned out to be difficult. To avoid this difficulty, we need a generalization of the  McClure-Smith's product, laxer than an action of $\widetilde{\K}$. We give such a generalization as  an action of another monad $\overline{\CK}$. The construction of the monad $\overline{\CK}$ and the zig-zag of weak equivalences preserving the action of $\overline{\CK}$ are  main tasks in the present paper. \\
\indent  The other main task is to establish an equivalence of two multiplicative objects defined by McClure-Smith \cite{ms}. Recall from \cite{ms}  the notion of homotopy totalization $\ttot$, which is the homotopy limit over the category of standard simplices $\Delta$,  see subsection \ref{SSNT} for the definition. $\ttot$ has a better homotopy invariance than $\tot$ and we need both of them. We construct an explicit isomorphism $\LMT\cong \tot (\CL^\bullet)$ while a $\overline{\CK}$-action induces an ($A_\infty$-) product on $\ttot$ (not on $\tot$).  
In \cite{ms}, a notion of  $\Xi_n$-algebra structure on a cosimplicial object was introduced and it was proved that a $\Xi_n$-structure induces  $E_n$-operad  actions both on $\tot$ and on  $\ttot$. Here, an $E_n$-operad is an operad weakly equivalent to the little $n$-cubes operad. The $\Xi_1$-structure is the  McClure-Smith product mentioned above.
We prove the following theorem.
\clearpage
\begin{thm}[Theorem \ref{Ttot}]\label{Tintro2}
Let $X^\bullet$ be a cosimplicial object over the category of topological spaces or symmetric spectra, and suppose $X^\bullet$ is equipped with a $\Xi_n$-structure and suppose  a canonical morphism $f^*:\tot(X^\bullet)\to \ttot(X^\bullet)$ (see subsection \ref{SSNT}) is a weak equivalence. Then, the $E_n$-actions on $\tot(X^\bullet)$ and $\ttot(X^\bullet)$ which are induced from the $\Xi_n$-structure are equivalent.
\end{thm}
We use the case of $n=1$ of this theorem to prove Theorem \ref{mainthm1} (2). As the both $E_n$-actions have many applications \cite{ms2, sinha2, ms, moriya, tsopmene}, this theorem will be useful in other context.
Theorem \ref{Tintro2}  is not so trivial as it looks since  the $E_n$-operad actions on $\tot$ and $\ttot$ are realized by  different operads  and there is no obvious morphism between them. The key observation is that  the two involved  operads  are naturally regarded as endomorphism operads on two different objects contained in a colored operad.\\
\indent There still exists a gap between Theorem \ref{mainthm1} and the Cohen-Jones isomorphism. Let $H_*(\thc (Q,Q))=\pi_*(\thc (Q,Q)\wedge Hk)$ with $Hk$ is the Eilenberg-MacLane spectrum for a base field $k$. If there exists an isomorphism of algebras
\[
H_*(\thc (Q, Q))\cong HH^*(C^*(M);C^*(M)) \cdots\cdots (*),
\]
where $Q$ is the object  in Theorem \ref{mainthm1}, we can complete the proof of Cohen-Jones isomorphism, combining the isomorphism $(*)$ with Theorem \ref{mainthm1}. The isomorphism $(*)$ is plausible since the function spectrum is an analogy of the singular cochain (the fibrant cofibrant replacement is necessary for the topological Hochschild to have the right homotopy type). Nevertheless, the construction of $(*)$ is non-trivial problem since it is related to comparison of symmetry of monoidal structures of symmetric spectra and chain complexes.  A similar  problem appeared in an earlier paper of Jones \cite{jones}, which is modified by Unghretti \cite{ungheretti}. Since the construction of the isomorphism $(*)$   is a quite general homotopical algebraic problem, we will resolve it in  another paper. \\ 

\indent An outline of this paper is as follows. In section \ref{Spreliminary}, we recall basic definitions  and known results we use later. Nothing is essentially new. We review the Cohen-Jones ring spectrum and (a slightly different version of) the realization of the Atiyah duality in symmetric spectra due to \cite{cohen}. We give a detailed explanation of errors in \cite{cohen}. We also recall a description of the Stasheff's associahedral operad by trees and the McClure-Smith's product for $A_\infty$-structures and  introduce a slight generalization of it,  which is applicable to the  $A_\infty$-structure on $\LMT$. We also recall the definition of topological Hochschild cohomology.\\
\indent In section \ref{Sainfty} we define the non-unital $A_\infty$-symmetric ring spectra $\LMT$ and prove the part (1) of Theorem \ref{mainthm1}.\\
\indent In section \ref{Sequivalence}, we prove the part (2) of Theorem \ref{mainthm1}. We construct a chain of equivalences of non-unital $A_\infty$-symmetric ring spectra:
\[
\begin{split}
\LMT\stackrel{\text{(A)}}{\cong}\tot(\CL^\bullet)&
\stackrel{\text{(B)}}{\simeq}
\ttot(\CL^\bullet)\stackrel{\text{(C)}}{\simeq}\ttot(\IM^\bullet)\\
&\stackrel{\text{(D)}}{\simeq}
\ttot(\thc^\bullet(A',B))\stackrel{\text{(E)}}{\simeq}
\tot(\thc^\bullet(A',B))=
\thc(A',B)\stackrel{\text{(F)}}{\simeq}
\thc(Q,Q)
\end{split}
\]
For the definitions of $A'$, $B$ and $Q$, see Proposition \ref{propinvarianceofthc}, and the equivalence (F) is proved in the same proposition.  $\thc$ denotes the topological Hochschild cohomology. In subsection \ref{SScosimplicialLMT}, we define a cosimplical symmetric spectrum $\CL^\bullet$ whose totalization is isomorphic to $\LMT$ and prove the $A_\infty$-structure on $\LMT$ comes from a (slight generalization of) McClure-Smith product on $\CL^\bullet$, which implies the isomorphism (A). Equivalences (C) and (D) are induced from morphisms between cosimplical symmetric spectra $\CL^\bullet\xleftarrow{p_0
} \IM^\bullet\xrightarrow{\bar q_2} \thc^\bullet(A',B)$. In subsection \ref{SSintermediate}, we define $\IM^\bullet$ and the two morphisms $p_0$, $\bar q_2$. In subsection \ref{SSgeneralizationMS}, we introduce the monad $\MCK$, and in subsection \ref{SSMCKalgebra}, we define an action of $\MCK$ on $\IM^\bullet$. We also need to establish the equivalences of $A_\infty$-structures on $\tot$ and $\ttot$ (B) and (E). We deal with the main part of this problem in subsection \ref{SScomparisontot}. In the final subsection \ref{SSproofTmain2}, we put  the results of previous subsections together, and establish the whole equivalences.\\

\textbf{Acknowledgements}: The author is most grateful to Masana Harada for reading the first draft of this paper and giving valuable comments. He is also deeply indebted to an anonymous referee for suggestions to improve readability  and pointing out errors. He is partially supported by JSPS KAKENHI Grant Number 26800037 and 17K14192.

\subsection{Notation and terminology}\label{SSNT}
\begin{itemize}
\item For a sequence of numbers $k_1,\dots, k_n$, we set $k_{\leq n}=k_1+\cdots +k_n$. If the numbering of a sequence begins with $0$, the sum also begins with the term numbered  $0$. 
 We someteimes use the word `map' instead of `morphism'
\item We will consider several subspaces of product spaces. For most of them, we use a particular tuple of  letters to denote elements of them. we sometimes prefix the letter on `part' to indicate a component.  For example, an element of $\M_k$ (see subsection \ref{SSatiyahduality}) is typically denoted by $(\phi, \eps, v)$ (possibly with super- or subscirpts). The $\phi$-part of the element $(\phi_1,\eps_1,v_1)\in \M_k$ is $\phi_1$, and its $\eps$-part is $\eps_1$. 
\item For two maps $\phi_1:A\to \R^{k_1}$ and $\phi_2:A\to \R^{k_2}$,  $\phi_1\times \phi_2 : A\to \R^{k_1+k_2}$ denotes the map given by $\phi_1\times \phi_2(x)=(\phi_1(x), \phi_2(x))$.
\item $\TOP$ denotes the category of all (unpointed) topological spaces and continuous maps. $\CG$ denotes the full subcategory of $\TOP$ consisting of all compactly generated spaces in the sense of \cite[Definition 2.4.21]{hovey} 
\item Our notion of a symmetric spectrum is that of Mandell-May-Schwede-Shipley\cite{mms} (or of Hovey-Shipley-Smith\cite{hss} when the category of simplicial set is replaced with $\CG$) and the category of symmetric spectra is denoted by $\SP$. For a symmetric spectrum  we refer to  the numbering of the underlying sequence as the  level. We recall definitions for the reader's convenience. For details, see \cite{mms,hss}. A symmetric spectrum $X$ consists of 
\begin{itemize}
\item a sequence $X_0,X_1,\dots X_k,\dots$ of pointed compactly generated spaces,
\item a pointed action of the symmetric group $\Sigma_k$ on $X_k$ for each $k$, and
\item a pointed map $\iota:S^1\wedge X_k\to X_{k+1}$ for each $k$
\end{itemize}
such that the $l$ times iteration of  $\iota$, $S^l\wedge X_k\to X_{l+k}$ is $\Sigma_l\times \Sigma_k$-equivariant. Here $S^l=\R^l\cup\{\infty\}$ and $\Sigma_l$ acts on it by permutation of components, and $\Sigma_l$ (resp. $\Sigma_k$) is regarded as the subgroup of $\Sigma_{l+k}$ of  permutations of the first $l$ (resp. last $k$) letters. We also call the map $\iota$ the action of $\Sph$ (or $S^1$). A morphism $X\to Y$ of symmetric spectra is a sequence of pointed maps $X_k\to Y_k$ compatible with the structures. The main feature of $\SP$ is that it admits a symmetric monoidal product $\wedge_S$ (in the notation of \cite{mms}) given by
\[
(X\wedge_SY)_n=\underset{k+l=n}{\bigvee} (\Sigma_n)_+\wedge_{\Sigma_k\times \Sigma_l}(X_k\wedge Y_l)/\sim.
\]
Here, $(\Sigma_n)_+$ denotes the discrete space with extra basepoint, and $\sim$ is the equivarent relation generated by 
$(\sigma,\, t\cdot x,\, y)\sim(\sigma\rho_{1,k}^{-1},\,x,\, t\cdot y)$ for $x\in X_k$ and $y\in Y_l$ and $t\in S^1$, where $\rho_{1,k}\in \Sigma_{k+1}$ is the $(1,k)$-shuffle which is the transposition of the first letter and the last $k$-letters (so this is the cyclic permutation $(1,k+1,k,\dots, 2)$), and $t\cdot -$ denotes the structure map. The symmetry isomorphism $\tau:X\wedge_SY\to Y\wedge_SX$ is given by $\tau(\sigma, x,y)=(\sigma\rho_{k,l}^{-1},y,x)$, where $\rho_{k,l}$ is the $(k,l)$-shuffle. The unit $\Sph$ of the monoidal product $\wedge_S$ is called the { sphere spectrum} and given by $\Sph_k=S^k=\R^k\cup\{\infty\}$. We often write $\wedge_S$ as $\otimes$ as below. For a spectrum $X$ with a product $\mu:X\otimes X\to X$, we denote $\mu(1,x,y)$ by $x\cdot y$, where $1\in \Sigma_{n}$ is the identity permutation.
\item We use two different expressions of $S^1$. One is $S^1=\R\cup\{\infty\}$, which is used in the structure map $\iota$ of a spectrum. The other is $S^1=[0,1]/0\sim 1$, which is used as the domain of a loop $c:S^1\to M$. We use the same notation for both expressions since it is easily understood by the context.
\item In this paper, $\C$ denotes either of $\CG$ or $\SP$. We regard $\C$ as a closed symmetric monoidal category tensored and cotensored over $\CG$. The monoidal product (resp. external tensor) is denoted by $\otimes$ (resp. $\hotimes$). If $\C=\CG$,  both $\otimes$ and $\hotimes$  are given by the cartesian product. If $\C=\SP$, $\otimes$ is  the product $\wedge_S$, and  $\hotimes$  is  given by $(X\hotimes K)_l=X_l\wedge (K_+)$ for $X\in \SP$, $K\in\CG$, where $K_+$ is the based space made by adding a disjoint basepoint to $K$, and $\wedge$ is the usual smash products of pointed spaces.  We denote by $\Map_{\C}$ the internal hom, which is adjoint to $\otimes$, and by $(-)^K$ the cotensor, which is adjoint to $\hotimes K$ for each $K\in\CG$. The function  spectrum $F(M)$ is defined as $\Sph^M$. (If $\C=\CG$,  $(-)^K=\Map_{\C}(K,-)$ but we use the both notations to ease  notations.) We often omit the subscript $\C$ of $\Map_{\C}$. Even when the category $\C$ is specified, we sometimes use the notation $\otimes$ instead of $\times$ or $\wedge_S$ for simplicity.
\item Our notion of a model category is that of \cite{hovey}. In this paper, we mainly deal with the following model categories.
\begin{itemize}
\item We endow $\CG$ the standard model structure, see \cite[Theorem 2.4.25]{hovey}
\item $\Delta$ denotes the category of standard simplices. Its objects  are  the  finite ordered sets $[n]=\{0,\dots,n\}$ and morphisms are weakly order preserving maps.  For a category $\mathcal{M}$, a { cosimplicial object $X^\bullet$ over $\mathcal{M}$ } is a functor $\Delta\to \mathcal{M}$. A morphism between cosimplicial objects is a natural transformation. $X^n$ denotes the object of $\mathcal{M}$ at $[n]$. We define maps
\[
d^i:[n]\to [n+1]\ (0\leq i\leq n+1),\ \qquad s^i: [n]\to [n-1]\ (0\leq i\leq n-1)
\]
by 
\[
d^i(k)=\left\{
\begin{array}{cc}
k & (k<i) \\
k+1 & (k\geq i)
\end{array}\right.,\qquad
s^i(k)=\left\{
\begin{array}{cc}
k & (k\leq i) \\
k-1 & (k>i)
\end{array}\right.
\]
 $d^i,\ s^i:X^n\to X^{n\pm 1}$ denote the maps corresponding to the same symbol. As is well-known, a cosimplicial object $X^\bullet$ is identified with a sequence of objects $X_0,X_1,\dots, X_n,\dots$ equipped with a family of maps $\{ d^i,\ s^i\}$ satisfying the cosimplicial identity, see \cite{gj}.  We denote the category of cosimplicial objects over $\mathcal{M}$ by $\mathcal{M}^\Delta$. 
\item For a cofibrantly generated model category $\mathcal{M}$, we endow two model structures on $\mathcal{M}^\Delta$. One is the termwise model structure whose fibrations and weak equivalences are  termwise ones. We call cofibrations in this model structure projective cofibrations. The other  is the Reedy model structure, see \cite[Theorem 5.2.5]{hovey} for the definition. We abbreviate cosimplicial symmetric spectrum as \textit{cs-spectrum}.
\item We use two model structures on $\SP$. One is the level model structure and the other is  the stable model structure, see \cite{mms} for the definition. Recall that a morphism $f:X_*\to Y_*$ is called a $\pi_*$-isomorphism if it induces an isomorphism between stable homotopy groups.  If a morphism is a  $\pi_*$-isomorphism, it is a weak equivalence in the stable model structure but the converse is false in general.
\end{itemize}
To emphasize the model structure we consider, we sometimes prefix names of morphisms with the name of model structure. For example, we say  a morphism $f:X^\bullet_*\to Y^\bullet_*$ between cs-spectra is a termwise stable fibration if each morphism  $f^n:X^n_*\to Y^n_*$ is a fibration in the stable model structure. For the compatibility of model structures with symmetric monoidal and tensor structures, see \cite{hovey,mms}.
\item The homotopy category of $\SP$ with respect to the stable model structure  is equivalent to the classical stable homotopy category as symmetric monoidal categories. We mean by a \textit{classical (associative) ring spectrum} an associative monoid in this category. 
\item We sometimes omit the superscript of a cosimplicial object and the subscript of a spectrum if it is clear from the context, or unnecessary.   
\item Operad means \textit{non-symmetric} (or \textit{non-$\Sigma$}) operad (see \cite{km, muro})  and the object at arity $n$ of  an operad $\oper$ is denoted by $\oper(n)$. We mainly consider operads in $\CG$ or in the category of posets (with the cartesian monoidal structure) and  operads in $\CG$ is called topological operads. Our notion of  a $A_\infty$-operad is the \textit{non-unital} version i.e., a topological operad $\oper$ is an $A_\infty$-operad if $\oper(0)$ is empty  and for each $n\geq 1$ and $i\geq 0$ $\pi_i(\oper(n))\cong *$. 
\item  Recall the notion of action of $\oper$ or  $\oper$-algebras (or algebras over $\oper$) for an operad $\oper$ and  the notions of monads and algebras over a monad (see \cite{km} where the case of symmetric operad are considered but the case of non-symmetric operad are completely analogous except for just not taking coinvariants for the symmetric groups). In this paper, a $A_\infty$-sturucture means an action of an $A_\infty$-operad. For a monad or operad $M$ in the category $\C$, we denote by $\Alg_M(\C)$ the category of algebras over $M$.
There is a canonical way to construct a monad from an operad (see \cite{km}). With this construction,  algebras over an operad are the same as  algebras over the corresponding monad. Let $X$, $Y$ be two objects on which two monads (or operads) $M$, $N$ act  respectively (in some category). A morphism $f:X\to Y$ is said to be compatible with a morphism $F:M\to N$ of monads (or operads) if the following square is commutative:
\[
\xymatrix{M(X)\ar[r]^{F(f)}\ar[d] & N(Y)\ar[d]\\
X\ar[r]^{f} & Y}
\]
where the vertical morphisms are the actions of monads (or operads). 
\item Let $\mathsf{F}:\CG \to \SP$ be the functor given by $K\mapsto \Sph \hotimes K$. For a topological operad $\oper$ we define an operad $\mathsf{F}(\oper)$ in $\SP$ by $\mathsf{F}(\oper)(n)=\mathsf{F}(\oper(n))$ with the naturally induced composition. By abusing notation we mean by an $\oper$-algebra in $\SP$ an $\mathsf{F}(\oper)$-algebra.    For an $A_\infty$-operad $\oper$, we  call an $\oper$-algebra in $\SP$ a \textit{non-unital $A_\infty$-symmetric ring spectrum}
 (or \textit{nu-$A_\infty$-ring spectrum} in short).
 Let $X$ and $Y$ be two symmetric spectra on which two $A_\infty$-operads $\oper$, $\aoper$ act on respectively. A stable equivalence $f:X\to Y$ of symmetric spectra is called a (weak) equivalence of nu-$A_\infty$-ring spectra if there exists a morphism of operads $g:\oper\to \aoper$ which is compatible with $f$. Two nu-$A_\infty$-ring spectra are said to be (weakly) equivalent if they can be connected by a zig-zag of equivalences of nu-$A_\infty$-ring spectra.

\item Recall the notion of totalization $\tot(X^\bullet)$ and its variant $\ttot(X^\bullet)$ for a cosimplicial object $X^\bullet$ over $\C$ from \cite{ms}. For a cosimplicial space $K^\bullet$, an object $\Map_{\C^{\Delta}}(K^\bullet, X^\bullet)\in \C$ is defined as the subobject of $\prod_{n\geq 0}(X^n)^{K^n}$ consisting of elements consistent with cosimplicial operators (see Definition 18.3.2 of \cite{hirschhorn} for details, in which the corresponding notation is $\mathrm{hom}^{\Delta}(K,X)$). (So $\Map_{\C^{\Delta}}$ denotes an object of $\C$, not of $\C^\Delta$.) Let $\Delta^\bullet$ denote the cosimplicial space of standard simplices. We fix a projective cofibrant replacement of $\Delta^\bullet$  denoted by $\tilde \Delta^\bullet$ and a (termwise) weak equivalence $f: \tilde\Delta^\bullet\to \Delta^\bullet$.   We set $\tot(X^\bullet)=\Map_{\C^\Delta}(\Delta^\bullet, X^\bullet)$ as usual, and set  $\ttot(X^\bullet)=\Map_{\C^\Delta}(\tilde\Delta^\bullet, X^\bullet)$. A morphism $f^*:\tot(X^\bullet)\to \ttot(X^\bullet)$ is naturally induced by $f$.  $\ttot$ has  better homotopy invariance than $\tot$. In practice, $\ttot$ is invariant under termwise equivalences between termwise fibrant objects while $\tot$ is invariant only under those between Reedy fibrant objects (see Corollary 18.4.4 and  Theorem 18.6.6 of \cite{hirschhorn}).

\end{itemize}
\section{Preliminary}\label{Spreliminary}

\subsection{Cohen-Jones ring spectrum}\label{SScohenjones}
In this subsection, we recall the definition of the Cohen-Jones ring spectrum, so we deal with the classical spectra. \\
\indent Let $e:M\to \R^{k}$ be a smooth embedding. For $\epsilon >0$, we denote by $\nu_\epsilon(e)$ the open subset of $\R^k$ consisting of points whose Euclidean distance from $e(M)$ are smaller than $\epsilon$. Let $L_e$ denote the minimum of $1$ and the least upper bound of $\eps>0$ such that
there exists a retraction $\pi_e:\nu_\eps (e)\to e(M)$ satisfying the following conditions.
\begin{enumerate}
\item For any $x\in\nu_\eps (e)$ and any $y\in M$, $|\pi_e(x)-x|\leq |e(y)-x|$ and the equation holds if and only if $\pi_e(x)=e(y)$. Here, $|-|$ denotes the usual Euclid norm in $\R^k$.

\item For any $y\in M$, $\pi_e^{-1}(\{e(y)\})=B_\eps(e(y))\cap(e(y)+(T_yM)^{\perp })$. Here $B_\eps(e(y))$ is the open ball with center $e(y)$ and radius $\eps$, and  $(T_yM)^{\perp}$ is the orthogonal complement of the image of $T_yM$ by (the differential of) $\phi$.
\item The closure $\bar{\nu}_\eps(e)$ of $\nu_\eps(e)$ is a smooth submanifold of $\R^{k}$ with boundary.
\end{enumerate} 
(Such a retraction exists for a sufficiently small $\eps>0$ by a version of tubular neighborhood theorem.) Note that the retraction $\pi_e$ satisfying the above three conditions is unique.  We also consider $\pi_e$ as a map to $M$ by identifying $M$ and $e(M)$ so we have a disk bundle $\pi_e:\bar{\nu}_\eps(e)\to M$ over $M$.\\
\indent In the rest of this paper, we fix an embedding $e_0:M\to \R^{k_0}$ and for a linear injective map $\phi:\R^{k_0}\to \R^{k}$, we abbreviate $\nu_\eps(\phi\circ e_0)$ and $\pi_{\phi\circ e_0}$ as $\nu_\eps(\phi)$ and $\pi_\phi$, respectively. We always regard $M$ as a subspace of $\R^{k_0}$ via $e_0$ so for example, we write $\phi(M)$ instead of $\phi\circ e_0(M)$.\\
\indent In this subsection, we fix a  linear injective map $\phi_0:\R^{k_0}\to \R^N$ with $N>>d$. Let $\eps_0>0$ be a sufficiently small number.
Let $ev^*\bnu_{\eps_0}(\phi_0)$ be the pullback of the disk bundle $\bnu_{\eps_0}(\phi_0)$ by the evaluation $ev:LM\to M$ at the  basepoint  $0\in S^1=[0,1]/0\sim 1$. We define a (classical)  spectrum $\LMTM$ (an object of the homotopy category $\Ho(\SP)$) by
\[
\LMTM=\Sigma^{-N}Th(ev^*\bnu_{\eps_0}(\phi_0 )).
\]
Here,  $Th(-)$ denote the Thom space. Let
$
E:ev^*\bnu_{\eps_0}(\phi_0)\times ev^*\bnu_{\eps_0}(\phi_0)\longrightarrow \bnu_{\eps_0}(\phi_0)\times\bnu_{\eps_0}(\phi_0)
$ 
be the bundle map induced from the evaluation, and $ev_\infty:LM\times_MLM\longrightarrow M$ be the evaluation at the basepoint, considering $LM\times_MLM$ as the space of two loops which have a common basepoint.\\
\indent   $\bar \nu_{\eps_0}(\phi_0\times\phi_0)$ is a subspace of $\bar\nu_{\eps_0}(\phi_0)\times\bar\nu_{\eps_0}(\phi_0)$. To define a product on $\LMTM$, we need the following lemma whose  proof  is an easy excercise of differential topology.
\begin{lem}\label{Llooptube}
Under the above notations,   there exists a homeomorphism \\
$\alpha: E^{-1}(\bnu_{\eps_0}(\phi_0\times\phi_0))\cong ev_\infty^*\bnu_{\eps_0}(\phi_0\times\phi_0)$
which makes the following diagram commute:
\[
\xymatrix{
LM\times_M LM\ar[rd]\ar[d]& \\
E^{-1}(\bnu_{\eps_0}(\phi_0\times\phi_0))\ar[r]^{\alpha}\ar[rd]&ev_\infty^*\bnu_{\eps_0}(\phi_0\times\phi_0)\ar[d]\\
&\bnu_{\eps_0}(\phi_0\times\phi_0),
}
\]
where the top vertical and slanting arrows are induced by the embedding $\phi_0\times \phi_0:M\to \bar\nu_{\eps_0}(\phi_0\times\phi_0)$ and the bottom arrows are natural bundle maps.
Furthermore, a  homeomorphism which makes  the above diagram commutative  is unique up to isotopies.\qed
\end{lem}

 We define a product on $LM^{-TM}$ as follows.
\[
\begin{split}
\LMTM\wedge\LMTM&\cong\Sigma^{-2N}Th(ev^*\bnu_{\eps_0}(\phi_0)\times ev^*\bnu_{\eps_0}(\phi_0))\\
&\to \Sigma^{-2N}(E^{-1}(\bnu_{\eps_0}(\phi_0\times\phi_0)/E^{-1}\partial\nu_{\eps_0}(\phi_0\times\phi_0))\\
&\qquad (\text{collapse the outside of } \bnu_{\eps_0}(\phi_0\times\phi_0)\subset \bnu_{\eps_0}(\phi_0)\times \bnu_{\eps_0}(\phi_0))\\
&\cong \Sigma^{-2N}(Th(ev_\infty^*\bnu_{\eps_0}(\phi_0\times\phi_0)) \quad (\text{by Lemma \ref{Llooptube}})\\
&\to \Sigma^{-2N}(Th(ev^*\bnu_{\eps_0}(\phi_0\times\phi_0))\quad (\text{concatenation})\\
&\cong \Sigma^{-2N}(\Sigma^{N}Th(ev^*\bnu_{\eps_0}(\phi_0) )\cong \Sigma^{-N}(Th(ev^*\bnu_{\eps_0}(\phi_0))\cong \LMTM.
\end{split}
\]
\begin{defi}\label{DCJspectra}
We call the classical ring spectrum $\LMTM$ equipped with the above product the \textit{Cohen-Jones ring spectrum}.
\end{defi}


\subsection{Atiyah duality for symmetric spectra}\label{SSatiyahduality}
In this subsection, we exhibit a realization of Atiyah duality in the category of symmetric spectra essentially due to Cohen \cite{cohen}. We warn the reader that  a symmetric spectrum $\MT$ defined here is  slightly different from the object of the same notation in \cite{cohen}.\\ 
\indent For each $k\geq 0$, we put
\[
V_k=\left\{\phi:\R^{k_0}\to\R^{k}\left|
\begin{split}
\phi \text{ is }&\text{a linear map such that} \\
&{}^\exists c\geq 1 \ {}^\forall v\in \R^{k_0}\  |\phi(v) |=c|v|
\end{split}\right.
\right\}. 
\]
Here, $|\phi(v)|$ (resp. $|v|$) denotes the standard Euclidean norm in $\R^{k}$ (resp. $\R^{k_0}$). Of course, if $k<k_0$, $V_k$ is empty. $V_k$ is topologized as a subspace of the Euclidean  space of linear maps (see subsection \ref{SScohenjones} for notations).  We define $|\phi|$ as  the $c\ (=\frac{|\phi(v)|}{|v|})$ for $\phi \in V_k$. Since $\phi$ expands distance constantly, we have $L_{\phi\circ e_0}=|\phi|L_{e_0}$. \\
\indent We define a sequence of unpointed spaces $\{\M_k\}_{k\geq 0}$ by
\[
\M_k=\{(\phi,\eps, v)\mid \phi \in V_k,\  0<\eps <L_{e_0}/16,\ v \in\bar\nu_\eps(\phi)\}
\]
and put $\TV_k =V_k\times (0,L_{e_0}/16)$. $\M_k$ is considered as a disk bundle over $\TV_k\times M$ with the projection  $(\phi,\eps,v)\mapsto (\phi,\eps,\pi_\phi(v))$. We define  a pointed space $\MT_k$ as the Thom space associated to $\M_k$. 
The sequence $\MT=\{\MT_k\}_{k\geq 0}$ is equipped with a structure of symmetric spectrum as follows.  The action of $\Sigma_k$ is induced by the permutation of the components 
of $\R^k$. The action of $\Sph$ (or structure map) is given by $(\phi,\eps,v)\mapsto 
(0\times\phi,\eps,(t,v))$ for $t\in S^1=\R^1\cup\{\infty\}$. As $\TV_k$ is $k/2-1$-connected, $\MT$ is 
isomorphic to the Thom spectra $M^{-TM}$ in \cite{cj} as objects of the homotopy 
category of $\SP$.
\\
\indent To connect $\MT$ and the function spectrum $F(M)$ by $\pi_*$-isomorphisms, we define two symmetric spectra $\Gamma(M)$, $\Gamma'(M)$ and three $\pi_*$-isomorphisms $\kappa_1$, $\kappa_2$, and $\rho$ fitting into the following diagram:
\[
\MT\xrightarrow{\rho}\Gamma(M)\xleftarrow{\kappa_2}\Gamma'(M)\xrightarrow{\kappa_1}F(M).
\]
We first define $\Gamma(M)$.  For $(\phi,\eps)\in\TV_k$, 
let  $B_{\phi,\eps}$ be the  trivial $k$-sphere bundle over $M$ whose fiber at $x\in M$ is $\bar B_\eps(\phi(x))/\partial \bar B_\eps(\phi(x))$, where $\bar B_\eps(\phi(x))\subset \R^k$ is the closed ball with radius $\eps$ and center $\phi(x)$. 
Let $\Gamma_{\phi,\eps}$ be the space of sections $M\to B_{\phi,\eps}$. We put
\[
\tilde\Gamma _k (M)=\{(\phi,\eps,\theta )\mid (\phi,\eps)\in \TV_k, \theta\in \Gamma_{\phi,\eps}\}.
\]
We give $\Gamma_{\phi,\eps}$ the compact-open topology and $\tilde\Gamma_k (M)$ the topology as a fiber bundle over $\TV_k$.  The  space $\Gamma_k(M)$ at level $k$ is obtained from $\tilde\Gamma_k(M)$ by collapsing  the subspace  $\{(\phi,\eps,\sigma_\infty)\mid(\phi,\eps)\in\TV_k  \}$ to one point, where $\sigma_\infty$ is the section consisting of the points represented by the boundaries.  The actions of $\Sph$ and $\Sigma_k$  on the sequence $\Gamma(M)=\{\Gamma_k(M)\}_k$ is defined  exactly analogously to those on $\MT$.\\
\indent  The definition of $\Gamma'(M)$ is as follows. We put
$\Gamma_k'(M)=\{(\phi,\eps, f)\mid (\phi,\eps)\in\TV_k,\ f\in F(M)_k \}/\{(\phi,\eps,*)\}$, where $*$ is the basepoint of $F(M)_k$, and the actions of $\Sph$ and $\Sigma_k$ is similar to those on $\Gamma(M)$.
\\
\indent  The morphism $\kappa_1$ is given  by the canonical projection and $\kappa_2$ is the morphism induced  by collapsing map $S^k\to \bar B_{\eps}(\phi(x))/\partial \bar B_{\eps}(\phi(x))$.
Finally, the morphism $\rho$ is given  by
\[
\rho(\langle \phi,\eps,v\rangle )=\langle \phi,\eps,\theta_v\rangle,\quad
\theta_v(y)=\left\{
\begin{array}{ll}
v &\text{ if }|v-\phi(y)| <\eps \\
\text{basepoint}  & \text{ if }|v-\phi(y)| \geq \eps
\end{array}\right.
\]
for $\langle \phi,\eps,v\rangle \in\MT_k$.\\
\indent The following theorem is a realization of the Atiyah duality in symmetric spectra and easily follows from the original Atiyah duality (\cite{atiyah, cohen}) and the fact that $\TV_k$ is $k/2-1$-connected.
\begin{thm}\label{Tatiyah}
The morphisms $\kappa_1$, $\kappa_2$ and $\rho$ defined above are $\pi_*$-isomorphisms. \qed
\end{thm}

\subsection{Problem about  construction in \cite{cohen}}\label{SSproblem}
In this subsection, we explain where the problem is in construction in \cite{cohen}.  We begin with explaining relation between two papers \cite{cj, cohen}.  In \cite{cj},   (an outline of ) construction of the isomorphism between the products on the homology of $LM$ and on Hochschild cohomology of $C^*(M)$ is given. Spectra are used in the consruction. Historically, the homotopy category of spectra was first defined in ad hoc way, and some sophisticated model categories of spectra \cite{ekmm, hss, mms} whose homotopy category is equivalent to the category were constructed later. Construction in \cite{cj} is given in terms of the former (classical) category but it is not sufficient to justify their argument, and one needs to give construction in a model category of spectra because of the following reason. In the construction of the isomorphism, the authors of \cite{cj} use the totalization of  a cosimplicial spectrum and a product on it (see Introduction). The totalization, a kind of homotopy limit, is generally defined for a functor to a model category not  to its homotopy category (see \cite{hirschhorn} for relation between   totalization and  homotopy limit).  For a model category $\M$ and a category $C$, let $\M^C$ denote the category of functors $C\to \M$ and natural transformations. Cosimplicial objects are the case of $C=\Delta$. $\M^C$ has a model category structure where weak equivalences are defined in the termwise manner. The homotopy limit is defined by taking  limit after replacing given functor with a well-behaved functor up to weak equivalences. Without  suitable replacement, weakly equivalent functors do not always give weakly equivalent limits. This replacement essentially  uses the model  structure and  this invariance is  crucial to the proof of the isomorphism in \cite{cj} since the isomorphism is induced by weak equivalences between cosimplicial objects (see Introduction). The cosimplicial object defined in \cite{cj} is a cosimplical object in the homotopy category. Intuitively, this means the cosimplicial identity holds only up to homotopy, and we can not take homotopy limit (or totalization) as it is. One may want to pick a cosimplical object in a model category of spectra such that it represents the cosimplical object in \cite{cj} in the homotopy category. However, such an object may not exist in general, see\cite{cooke, dugger}. Since we  also need to define a product on the cosimplicial object and morphisms between them on the level of model category, it is not sufficient to know existence of such an object even if one can pick. It is  better to give its concrete description.  One may think one could adopt (categorical) limit in the homotopy category instead of totalization but that category  does not have limits in general except for very simple diagram category $C$ such as a discrete category. These features of homotopy limit and homotopy category are well-known. Thus, we need to work with a model category (not with the homotopy category) to justify the argument in \cite{cj}. We also need to consider the (topological) Hochschild cohomology,   as it is on one side of the isomorphism. Since its definition uses the symmetry isomorphism of the monoidal product, our model category of spectra must have a   symmetric monoidal structure which is well-behaved with replacements. Model categories with such a structure are called (symmetric) monoidal model categories (see \cite{hovey}), and monoidal model categories  of spectra are given in \cite{ekmm, hss, mms}.   We can not carry over the construction in \cite{cj} to such a  category as it is, since key ingredients of construction in \cite{cj} are defined only up to homotopy. For example, the definition of $M^{-TM}$ in \cite{cj} depends on an embedding $\iota :M\to \R^{N+d}$ (in the notations of \cite{cj}), and   in the definition of the product 
$M^{-TM}\wedge M^{-TM}\to M^{-TM}$ (as well as the loop product on $LM^{-TM}$), $M^{-TM}$'s on the domain and the codomain are defined by different embeddings. If we define two  $M^{-TM}$'s on the domain via $\iota$, we must use $\iota\times \iota:M\to \R^{2(N+d)}$ on the codomain.   These different choices give isomorphic spectra in the homotopy category but not   in a model category. An isomorphism between the different choices is not given in the paper explicitly, but it relies on an isotopy between $\iota\times \iota$ and $0\times \iota$  (see subsection \ref{SScohenjones} where we implicitly used such an isotopy to define the isomorphism $Th(ev^*\bar \nu_{\eps_0}(\phi_0\times \phi_0))\cong \Sigma^NTh(ev^*\bar\nu_{\eps_0}(\phi_0))$). It is difficult to choose such an isotopy so that the product is strictly associative when we work in a model category of spectra. Another reason why we can not use the construction in \cite{cj} as it is, is that definitions of object and monoidal product in such a category is more complicated  than those in the homotopy category. The concern we have seen so far seems to be not only the present author's but also Cohen's.   Actually, He gave  more detailed construction in the category of symmetric spectra, one of the monoidal model categories, in order to fill in details  of the construction of the isomorphism in \cite{cj}. (He wrote ``We will use the multiplicative properties of Atiyah duality proven here to fill in details of the argument given in [3]'' in the introduction of \cite{cohen}.) He resolved the problem of choosing isotopy by including embeddings in the definition of the Thom spectra.
\par 
He also  described a product on a cs-spectrum which should induce the loop product on the totalization in \cite{cohen}. But it is unclear whether this structure induces a product on its totalization (even after some modification) and this is why we do not use it ( and this is a motivation for writing the present paper). To complete the exposition, we explain this point more precisely. In the rest of this subsection, we follow the notations in \cite{cohen}. He gave two models of $\MTM$ both of which are slightly different from one given in the present paper. One is unital and the other is non-unital. To define a cs-spectrum, he used the unital model denoted by $\MT(e)$. We shall review outline of Cohen's construction.  Fix a smooth embedding $e:M\to \R^k$ and a number $\eps<L_{e}$, and consider the tubular neighborhood $\nu_\eps(e)$ in subsection \ref{SScohenjones}. $\MT(e)_m$ is the space made by collapsing outside of a fiberwise disk bundle of $\nu_\eps(e)$. More precisely, we set
\[
\MT(e)_m=\{(\phi, x)\mid \phi:\R^k\to \R^{mk}\ \text{linear isometry},\ x\in \R^{km}\}/\{(\phi,x)\mid d(x,\phi(\nu_\eps(e)))\geq \eps\}.
\] 
The sequence $\{\MT(e)_m\}_m$ is regarded as an object of a category of spectra defined using $\{S^{mk}\}_m$ instead of the sphere spectrum. This category is Quillen-equivalent to the standard category of symmetric spectra. The action (or structure map) $S^k\wedge \MT(e)_m\to \MT(e)_{m+1}$ is given by
\[
t\wedge (\phi, x)\mapsto (\id \times \phi, (t,x))\  \qquad (t\in S^k=\R^k\cup\{\infty\}).
\] 
A map 
$\Delta_r:\MT(e)\longrightarrow \MT(e)\wedge (\nu_{2\eps}(e)_+)$ is used to define a coface map (see P.280 of \cite{cohen}). Here, $\wedge$ is $\hotimes$ in our notation while the position is switched. The problem is that this map is {\em not} a morphism in the category of  spectra which Cohen uses as it does not commute with the structure map.  In fact, by definition, $\Delta_r$ takes a point $(\phi,x)\in\MT(e)$ to the point $(\phi,x)\wedge x_1$ where $x_1\in\nu_{2\eps}(e)$ is the point determined by the unique decomposition $x=\phi(x_1)+x_2$ with $x_2\in Im(\phi)^{\perp}$ but $x_1$ for $(\id\times \phi,(t,x))$ is in general different from $x_1$ for $(\phi,x)$.  We write $x_1$ for $(\phi,x)$ (resp. $(\id\times\phi,(t,x))$) as $x_1^\phi$ (resp. $x_1^{\id\times\phi}$). $x_1^{\phi}$ is (the preimage of) the  point in $Im(\phi)$ closest to $x$, so if $t=x_1^\phi$, $x^{\id\times \phi}_1=x^\phi_1$ but if $t\not= x_1^\phi$, $t$ affects the position of the closest point according to the valance between distances for `$t$-part' and `$x$-part' even if it is sufficiently close to $x_1$ .  We shall show  a rigorous  counter-example.  Let $M$ be any closed manifold of dimension $\geq 1$,  $m$  a number such that $\MT(e)_m\not=*$, and $(\phi, x)\in \MT(e)_m$ a non  basepoint element.  Let $x_0\in \nu_{\eps}(e)$ be a point such that $|x-\phi(x_0)|<\eps$. Since $(\phi, x)$ is not the basepoint, we can take such a point.  If $|t-x_0|<\eps-|x-\phi(x_0)|$, $(\id\times \phi, (t,x))$ is not the basepoint  by the following inequality. 
\[
d((t,x),\id\times \phi(\nu_{\eps}(e)))  \leq \sqrt[]{|t-x_0|^2+|x-\phi(x_0)|^2} 
 \leq |t-x_0|+|x-\phi(x_0)| <\eps
\] Choose   $t$ such that  both  $|t-x_0|<\eps-|x-\phi(x_0)|$ and $t\not=x_1^\phi$ hold. Clearly this is possible. If $x_1^{\id\times \phi}=x_1^\phi$, since $x_1^{\id\times \phi}$ is given by orthogonal projection, the Euclidean (standard) inner product
\[
\langle (t,x)-(x_1^\phi,\phi (x_1^\phi)),\ (v,\phi (v))\rangle
\]
 is zero for any $v\in \R^k$. We have
\[
\begin{split}
\langle (t,x)-(x_1^\phi,\phi (x_1^\phi)),\ (v,\phi (v))\rangle
& =\langle t-x_1^\phi, v\rangle +\langle x-\phi (x_1^\phi), \phi (v)\rangle
\end{split}
\]
By the definition of $x_1^\phi$, the second inner product on the right hand side is zero. So we have $\langle t-x_1^\phi, v\rangle=0$. This is a contradiction since it is non-zero if we set $v=t-x_1^\phi$.  We have proved $x_1^{\id\times \phi}\not =x_1^\phi$.
 Thus, if we take $t$ such that both $t\not=x_1^\phi$ and $|t-x_0|<\eps-|x-\phi(x_0)|$ hold,  the two points $t\wedge \Delta_r(\phi, x)=x_1^\phi\wedge (\id\times \phi, (t,x))$ and $\Delta_r(t\wedge (\phi, x))=x^{\id\times \phi}_1\wedge (\id\times \phi, (t,x))$ are different, which means $\Delta_r$ is not a morphism in the category of spectra.
 \par
If we replace the unital model with the non-unital one denoted by $\MT$, a map defined similarly to $\Delta_r$ is a well-defined  morphism of symmetric spectra and we obtain a well-defined cs-spectrum. Here, $\MT$ is different from the object defined in subsection \ref{SSatiyahduality}. We are using the notations in \cite{cohen}. The difference is that $\MT$ here, is defined using  smooth embeddings $M\to \R^k$, so here,
\[
\MT_k=\{(e,\eps, x)\mid e:M\to \R^k,\ 0<\eps<L_e,\ x\in \R^k\}/\{(e,\eps,x)\mid x\not\in\nu_\eps(e)\}
\]. The structure map $S^1\wedge \MT_k\to \MT_k$ is defined exactly similar to the one in subsection \ref{SSatiyahduality}, and the map corresponding to $\Delta_r$ is given by 
\[
\MT_k\to \MT_k\wedge (M_+),\qquad (e,\eps,x)\longmapsto (e,\eps, x)\wedge \pi_e(x), 
\] 
where $\pi_e$ is defined in subsection \ref{SScohenjones}. This map is compatible with the action of $S^1$ because the structure map puts the $0$-map at the first component of the embedding and the first component of the vector does not affect the position of the  point in $0\times e(M)$ closest to the vector.  More precisely, we have
\[
|(t,x)-0\times e(y)|=\, \sqrt[]{t^2+|x-e(y)|^2},
\]
and the left hand side is smaller if $|x-e(y)|^2$ is smaller, which implies $\pi_{0\times e}((t,x))=\pi_e(x)$. However, even in this case,  (a non-unital version of ) the product on the cosimplicial objects  defined  in \cite{cohen} does not satisfy the conditions of McClure-Smith which ensure a product on a cosimplicial object  to induce a product on its totalization, see subsection \ref{SSms}.  ( In \cite{cohen}, it is  mentioned that the product on the cosimplicial object of topological Hochschild cohomology satisfies the conditions but it is  not mentioned what condition the product on the cosimplicial object of $LM^{-TM}$ satisfies. )  We shall look at this more closely.  By straight analogy with \cite{cohen}, we should define a cosimplicial object $L_M^\bullet$ by the exactly same formula as $\CL^\bullet$ in subsection \ref{SScosimplicialLMT} with replacing $\phi$ and $v$ with $e$ and $x$. We should define the non-unital version of the product in \cite{cohen} by
\[
\begin{split}
(x_1,\dots, x_p, &\langle e_1,\eps_1,z_1\rangle)\, \cdot\, (y_1,\dots, y_q, \langle e_2,\eps_2,z_2\rangle)\\
&=(x_1,\dots, x_p,y_1,\dots,y_q, \langle e_1\times e_2, \min\{\eps_1,\eps_2\}, (z_1,z_2)\rangle
\end{split}
\] Here, $\langle -\rangle$ means the class in the Thom space and we use comma instead of $\wedge$, compromising the notations with subsection \ref{SScosimplicialLMT}. 
This does not satisfy the McClure-Smith conditions.  Philosophically, this is because this product does not includes any perturbation procedure while the loop product includes one. Actually, the homeomorphism $\alpha$ in Lemma \ref{Llooptube} encodes perturbation of two loops. $\alpha$ is a map between the space of two loops whose basepoints (values at $0\in S^1$) belong to a tubular neighborhood of the diagonal of $M\times M$ and the space  of two loops having a common basepoint with a normal vector to the diagonal at the basepoint. If one wants to define $\alpha$ explicitly, some perturbation making the basepoints of two loops slide to each other is necessary. Practically, we shall show one of the McClure-Smith condition, $d^0(x\cdot y)=(d^0x)\cdot y$, is not satisfied. We take two elements $\langle e_i,\eps_i,z_i\rangle\in L^0_M$ ($i=1,2$). We have
\[
\begin{split}
d^0(\langle e_1,\eps_1,z_1\rangle\cdot \langle e_2,\eps_2,z_2\rangle)& =(\pi_{e_1\times e_2}(z_1,z_2),\langle e_1\times e_2, \min\{\eps_1,\eps_2\}, (z_1,z_2)\rangle)\\ 
(d^0\langle e_1,\eps_1,z_1\rangle)\cdot \langle e_2,\eps_2,z_2\rangle)& =(\pi_{e_1}(z_1),\langle e_1\times e_2, \min\{\eps_1,\eps_2\}, (z_1,z_2)\rangle)
\end{split}
\]
We easily see $\pi_{e_1\times e_2}(z_1,z_2)\not= \pi_{e_1}(z_1)$ if $\pi_{e_1}(z_1)\not =\pi_{e_2}(z_2)$ by argument similar to the above unital case. So this product does not satisfy the condition. This is why we do not use this product. A version of this product including perturbation is defined in subsection \ref{SScosimplicialLMT}. This product satisfies the condition for $d^0$ and used to establish the main theorem, but we also need  
a product which satisfies it only up to homotopy coherency on the way. \par
As far as the author knows, no paper which appeared after \cite{cohen} does not resolve the problem explained here.  In the preprint \cite{malm}, Malm  presented a proof of an isomorphism similar to the Cohen-Jones'   but his construction is given in the classical homotopy category of spectra similarly to \cite{cj}. The author could not find a  rigorous proof of the Cohen-Jones isomorphism (or its spectral version) anywhere except for the case of rational coefficient.

\subsection{Stasheff's associahedra}\label{SSK}
In this section we review the Stasheff's associahedral operad.  We define it as the geometric realization of an operad in the category of posets. This definition is well-known, see Sinha \cite[4.4]{sinha} or Fiedrowicz-Gubkin-Vogt \cite{FGV} for example. (In \cite{FGV} the authors use ``parenthized words''  instead of trees which we use here.) We will use the description of the associahedra by trees, explained below, to construct a $A_\infty$-structure which governs the loop product. We give a proof of the consistency of our description and the original Stasheff's definition in some detail since we use a similar argument in more complicated situation in subsections \ref{SSScofacialtree} and \ref{SSSCK}.
\begin{defi}\label{deftree}
A \textit{tree} is a finite connected acyclic graph. (We do not distinguish the source or target from two endpoints of an edge at this point.)  For an integer $n\geq 2$, an \textit{embedded $n$-tree}  is a pair $(T,f)$ of a tree $T$ and a continuous injective map $f$ from the geometric realization of $T$ to the plane $\R\times [0,1]$ such that
 $f(T)\cap\R\times\{0\}$ consists of a unique vertex called the \textit{root}, which is at least bivalent, and $f(T)\cap\R\times\{1\}$ consists of $n$ univalent vertices  called the \textit{leaves}, and all vertices different from the root and leaves are at least trivalent. An isotopy between two embedded $n$-trees $(T_0,f_0)$ and $(T_1,f_1)$ is continuous family of homeomorphisms $\{g_t:\R\times [0,1]\to \R\times [0,1] \mid 0\leq t\leq 1\}$ such that $g_0=\id_{\R \times[0,1]}$ and $g_1$ maps $f_0(T_0)$ homeomorphically onto $f_1(T_1)$. An \textit{$n$-tree} is an isotopy class of embedded $n$-tree $T$. We label the leaves with the numbers $1,\dots,n$ according to the usual order on $\R\times\{1\}=\R$. 
The vertex of an edge which is farther from the root is called its \textit{source}, and the other vertex is  called its \textit{target}. For two numbers $i,j$ with $1\leq i<j\leq n$,  the \textit{$(i,j)$-join} is the  vertex  at which  the root paths of the $i$-th and the $j$-th leaves join first starting from the leaves. Here the \textit{root path} is the unique shortest path from the leaf to the root.  The \textit{$(i,j)$-bunch} is the vertex which is the $(i,j)$-join and not the $(k,l)$-join for any pair $(k,l)$ satisfying $k\leq i<j\leq l$ and $(k,l)\not=(i,j)$.  \par
The set of all $n$-trees is denoted by $\T(n)$.  $\T(0)$ and $\T(1)$ are defined as the empty set and one point set respectively, and the unique element of $\T(1)$ called the $1$-tree, has only one vertex which is both the root and leaf, and has no edges.  We define a partial order on $\T(n)$ by declaring $T\leq T'$ if $T'$ is obtained from $T$ by successive contractions of internal edges (edges whose source is not a leaf). We  give the collection $\T=\{\T(n)\}_{n}$ a structure of operad over the category of posets and order-preserving maps as follows. Let $T_1\in\T(n_1)$ and $T_2\in\T(n_2)$ and $1\leq i\leq n_1$. The $n_1+n_2-1$-tree $T_1\circ_iT_2$  is obtained by identifying the $i$-th leaf of $T_1$ with the root of $T_2$. When we consider  $T_1$ and $T_2$ as a subtree of $T_1\circ_iT_2$, The root of $T_1\circ_iT_2$ is the root of $T_1$ and $j$-th leaf of $T_1\circ_iT_2$ is $j$-th leaf of $T_1$ if $j\leq i-1$, $j-i+1$-th leaf of $T_2$ if $i\leq j\leq i+n_2-1$ and $j-n_2+1$-th leaf of $T_1$ if $i+n_2\leq j\leq n_1+n_2-1$. (Thus, our partial composition does not include edge contraction.)
Finally, we define a topological operad $\K$ by  $\K(n)=|\T(n)|$ with the induced operad structure.
\end{defi}
\begin{exa}\label{Ebunch}
Consider the following trees $T_1$ and $T_2$.
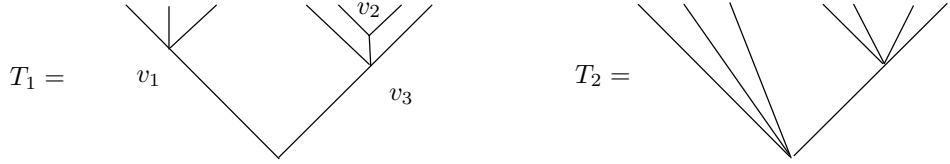
\begin{center}
{\unitlength 0.1in%
\begin{picture}(49.5000,8.4500)(-0.0800,-8.8200)%
%
\special{pn 8}%
\special{pa 646 78}%
\special{pa 1446 882}%
\special{fp}%
\special{pa 1118 78}%
\special{pa 878 300}%
\special{fp}%
\special{pa 870 86}%
\special{pa 870 308}%
\special{fp}%
%
\special{pn 8}%
\special{pa 1446 874}%
\special{pa 2246 78}%
\special{fp}%
%
\special{pn 8}%
\special{pa 1590 78}%
\special{pa 1918 388}%
\special{fp}%
\special{pa 1758 78}%
\special{pa 1918 237}%
\special{fp}%
\special{pa 2086 78}%
\special{pa 1918 237}%
\special{fp}%
\special{pa 1918 237}%
\special{pa 1926 396}%
\special{fp}%
\put(7.6600,-4.6000){\makebox(0,0){$v_1$}}%
\put(19.1800,-1.0200){\makebox(0,0){$v_2$}}%
\put(20.8600,-5.5600){\makebox(0,0){$v_3$}}%
\put(1.8200,-4.6000){\makebox(0,0){$T_1=$}}%
%
\special{pn 8}%
\special{pa 4142 866}%
\special{pa 4942 70}%
\special{fp}%
%
\special{pn 8}%
\special{pa 4294 74}%
\special{pa 4606 384}%
\special{fp}%
%
\special{pn 8}%
\special{pa 4454 74}%
\special{pa 4614 384}%
\special{fp}%
\special{pa 4766 82}%
\special{pa 4614 384}%
\special{fp}%
%
\special{pn 8}%
\special{pa 3326 74}%
\special{pa 4118 870}%
\special{fp}%
%
\special{pn 8}%
\special{pa 3806 66}%
\special{pa 4126 878}%
\special{fp}%
%
\special{pn 8}%
\special{pa 3566 74}%
\special{pa 4118 862}%
\special{fp}%
\put(29.9800,-5.0400){\makebox(0,0)[lb]{$T_2=$}}%
\end{picture}}%

\end{center}
$T_1$ has 3-vertices $v_1,v_2,v_3$ other than the root and leaves. We assume the leaves are labeled from left to right so the leftmost leaf is the first one. $v_1$ is the $(1,3)$-join (as well as the $(1,2)$-and the $(2,3)$-join), and  not the $(1,4)$-join, so it is the $(1,3)$-bunch. $v_3$ is the $(4,7)$-bunch. $v_2$ is the $(5,6)$-bunch.  The root vertex is the $(1,7)$-bunch. Also, $v_3$ is the $(4,5)$-, $(4,6)$-, $(5,7)$- and $(6,7)$-join but not the $(5,6)$-join. $T_2$ has the $(1,7)$- and the $(4,7)$-bunches (and neather the $(1,3)$- nor the $(5,6)$-bunch).  We see 
 $T_1\leq T_2$ and if $T_2$ has the $(i,j)$-bunch, $T_1$ also has.
\end{exa}
As in this example, in general, if $T_1\leq T_2$ and $T_2$ has the $(i,j)$-bunch, $T_1$ also has the $(i,j)$-bunch.
\begin{exa}
The following is an example of partial composition.
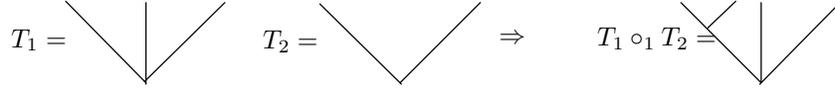
\begin{center}
{\unitlength 0.1in%
\begin{picture}(43.9000,4.3400)(0.6000,-4.8400)%
%
\special{pn 8}%
\special{pa 390 50}%
\special{pa 810 477}%
\special{fp}%
%
\special{pn 8}%
\special{pa 1223 57}%
\special{pa 803 470}%
\special{fp}%
%
\special{pn 8}%
\special{pa 810 57}%
\special{pa 810 484}%
\special{fp}%
%
\special{pn 8}%
\special{pa 1720 64}%
\special{pa 2147 484}%
\special{fp}%
%
\special{pn 8}%
\special{pa 2560 57}%
\special{pa 2140 477}%
\special{fp}%
%
\special{pn 8}%
\special{pa 3610 57}%
\special{pa 4030 477}%
\special{fp}%
%
\special{pn 8}%
\special{pa 4450 50}%
\special{pa 4023 484}%
\special{fp}%
%
\special{pn 8}%
\special{pa 4030 57}%
\special{pa 4030 477}%
\special{fp}%
\put(2.5000,-2.5300){\makebox(0,0){$T_1=$}}%
\put(15.6600,-2.6000){\makebox(0,0){$T_2=$}}%
\put(27.2800,-2.4600){\makebox(0,0){$\Rightarrow$}}%
\put(34.8400,-2.4600){\makebox(0,0){$T_1\circ_1T_2=$}}%
%
\special{pn 8}%
\special{pa 3897 50}%
\special{pa 3750 190}%
\special{fp}%
\end{picture}}%

\end{center}
\end{exa}

Let $P$ be a poset and $p\in P$ an elment. We denote by $\langle p\rangle$ the subposet $\{q\in P|q\leq p\}$. The \textit{codimension} of $p$ is the muximum of numbers $N$ such that a chain $p<p_1<\cdots<p_N\in P$ exists.
 The following is the fundamental property of $\T$. Verification is trivial.
\begin{lem}\label{Ltree}
\begin{itemize2}
\item[(1)] For each $n\geq 1$, $\T(n)$ has the muximum. We  denote it by $T(n)$. $\T$ is generated by the set of maximum elements of arity $\geq 2$ with operad composition.
\item[(2)] For each $n_1,n_2\geq 1$, and $1\leq i\leq n_1$, the composition $(-\circ_i-):\T(n_1)\times \T(n_2)\longrightarrow \T(n_1+n_2-1)$ induces a bijection from $\T(n_1)\times \T(n_2)$ onto $\langle T(n_1)\circ_i T(n_2)\rangle$.
\item[(3)] An element $T$ of $\T(n)$ is of codimension one if and only if it is equal to a composition of two maximum elements (of arity $\geq 2$). $T$ is of codimension two if and only if it is equal to a composition of three maximum elements (of arity $\geq 2$). 
\item[(4)] Let $T_1$ and $T_2$ be two different elements of $\T(n)$ of codimension one. If $\langle T_1\rangle \cap \langle T_2\rangle$ is not empty, there exists an element $T_3$ of codimension two such that $\langle  T_1\rangle\cap \langle T_2\rangle=\langle T_3\rangle$.\qed
\end{itemize2}
\end{lem}
We shall show the operad $\K$ is isomorphic  to the Stasheff's associahedral operad using these properties. 
\\
\indent Let $\T(n)_{1,2}$ be the subposet of $\T(n)$ consisting of elements of codimension one or two. We define a diagram
\[
B_n:\T(n)_{1,2}\longrightarrow \CG
\]
as follows. An element of codimension one is uniquely presented as  $T_1\circ_iT_2$ with $T_1$, $T_2$ muximum trees of arity $\geq 2$ by (2) and (3) of Lemma \ref{Ltree}. We put $B_n(T_1\circ_i T_2)=\K(n_1)\times \K(n_2)$,where $n_t$ is the arity of $T_t$. Similary, an element of codimension two is uniquely presented as  $(S_1\circ_j S_2)\circ_k S_3$ with $S_1,S_2,S_3$ muximul trees of arity $\geq 2$, and $j\leq k$.  We put 
$B_n((S_1\circ_j S_2)\circ_k S_3)=\K(m_1)\times \K(m_2)\times \K(m_3)$, where $m_t$ is the arity of $S_t$.\\
\indent On morphisms, suppose $(S_1\circ_j S_2)\circ_k S_3\leq T_1\circ_iT_2$. By (2) of Lemma\ref{Ltree}, only one of the following three cases occurs. 
\begin{enumerate}
\item $S_1\circ_jS_2\leq T_1$, $S_3=T_2$, $k=i$,
\item $S_1=T_1$, $S_2\circ_{k-j+1} S_3\leq T_2$, $j=i$, 
\item $S_1\circ_{k-m_2+1} S_3\leq T_1$, $S_2=T_2$, $i=j$.
\end{enumerate}
Using these relations, we define the map $B_n((S_1\circ_jS_2)\circ_kS_3)\to B_n(T_1\circ_i T_2)$. For example, in the first case, we define the map as the following map
\[
(-\circ_j-)\times \id: \K(m_1)\times \K(m_2)\times \K(m_3)\to \K(n_1)\times \K(n_2)
\] and similarly for the remaining cases. \\
\indent  A natural transformation $
B_n\Rightarrow \K(n):\T(n)_{1,2}\longrightarrow \CG
$
is defined by using the composition of $\K$ (the map $B_n(T_1\circ_iT_2)\to \K(n)$ is $(-\circ_i-)$). Here, $\K(n)$ is considered as  the constant diagram over $\T(n)_{1,2}$. So we obtain the induced  map
$
\theta_n:\colim_{\T(n)_{1,2}} B_n\longrightarrow \K(n)
$. The image of $\theta_n$ is $\partial\K(n)$, the subcomplex spanned by all $n$-trees of codimension one.\\
\indent A point of $\K(n)$ is  presented as $t_0T_0+\cdots +t_kT_k$ with $T_0<\cdots <T_k\in\T(n)$,  $t_0+\cdots +t_k=1$ and  $t_i\geq 0$. Using this presentation,
we define a map 
\[
\tilde\theta_n: Cone(\colim B_n)\longrightarrow \K(n)
\]
by $\tilde\theta_n(t\cdot u)= t\theta_n(u)+(1-t)T(n)$. Here, $Cone(X)=[0,1]\times X/\{0\}\times X$ and $t\cdot u$ is the point represented by $(t,u)$. Note that the construction $Cone(\colim B_n)$ coincides with  the definition of the associahedra given by Stasheff \cite{stasheff}.
\begin{prop}\label{PK}
Under the above notations, the maps $\theta_n :\colim B_n\to \partial\K(n)$ and $\tilde\theta_n:Cone(\colim B_n)\to \K(n)$ are   homeomorphisms.
\end{prop}
\begin{proof}
To ease the notations, let $\K(T)$ denote the subspace $|\langle T\rangle|\subset \K(n)$. By (2) of Lemma \ref{Ltree}, the map $(-\circ_i-):\K(n_1)\times\K(n_2)\longrightarrow \K(T(n_1)\circ_iT(n_2))$ is a homeomorphism for each $n_1,n_2$ and $i$. Since $\K(T_1)\cap \K(T_2)=\cup_{T\leq T_1,T_2}\K(T)$, (3) and (4) of the same lemma imply $\theta_n$ is a homeomorphism onto $\cup_{T<T(n)}\K(T)=\partial\K(n)$. This fact and (1) of the same lemma imply $\tilde\theta_n$ is a homeomorphism.
\end{proof}
\subsection{McClure-Smith product}\label{SSms}
In this subsection, we review (non-unital version of) the McClure-Smith product for $A_\infty$-structures.  
\begin{defi}
Let $X^\bullet$ be a cosimplicial object over  $\C$. A \textup{McClure-Smith product} (MS product, for short)  on $X^\bullet$ is a family of morphisms
$\{\mu_{p,q}:X^p\otimes X^q\to X^{p+q}\mid p,q\geq 0\}$ which satisfies the following conditions:
\[
\begin{split}
d^i(x\cdot y)&=\left\{
\begin{array}{ll}
(d^ix)\cdot y& \text{if }0\leq i\leq p\\
x\cdot (d^{i-p}y)& \text{if } p+1\leq i\leq p+q+1
\end{array}\right.\\
(d^{p+1}x)\cdot y&= x\cdot (d^{0}y)\\
s^i(x\cdot y)&=\left\{
\begin{array}{ll}
(s^ix)\cdot y& \text{if }0\leq i\leq p-1\\
x\cdot s^{i-p}y& \text{if }p\leq i\leq p+q
\end{array}
\right.\\
x\cdot (y\cdot z)&=(x\cdot y)\cdot z
\end{split}
\]
where $p=\deg x, \ q=\deg y$. Here, we  denote $\mu_{a,b}(z\otimes w)$ by $z\cdot w$, and we  interpret these equations as equations of morphisms if $\C=\SP$.
\end{defi}
We define a (non-symmetric) monoidal structure $\, \square \,$ on $\CC$ which is closely related to MS-product. For $X^\bullet$, $Y^\bullet\in\CC$, we put
\[
(X^\bullet\, \square \, Y^\bullet )^r=\bigsqcup_{p+q=r}X^p\otimes Y^q/\sim,\qquad d^{p+1}x\otimes y\sim x\otimes d^0y.
\]
The cosimplicial operators are defined similary to the above formulae of MS-conditions.
We call a semigroup (an object with associative product but without unit) with respect to the monoidal product $\, \square \,$ an  $\square$-object. The following is clear.
\begin{prop}
Let $X^\bullet$ be a cosimplicial object over $\C$. MS-products on $X^\bullet$ and structures of an $\square$-object on $X^\bullet$ are in one to one correspondence.\qed
\end{prop}
We denote by $\B$ the co-endmorphism operad of the cosimplicial space  $\Delta^\bullet$, i.e.,  $\B(n)=\Map_{\CCG}(\Delta^\bullet, \, (\Delta^\bullet)^{ \square \, n})$ with a natural composition product. The following is proved in \cite{ms}
\begin{prop}
The operad $\B$ is an $A_\infty$-operad. \qed
\end{prop}
In fact, $\B(n)$ is homeomorphic to the space of weakly order preserving  surjections from the interval $[0,1]$ to itself. For later use, we shall describe such a homeomorphism explicitly.\\
\indent We take a presentation of the standard topological simplex $\Delta^n$ as 
\[
\Delta^n=\{(t_1,\dots, t_n)\mid 
0\leq t_1\leq\dots\leq t_n\leq 1\}.
\]
We define a morphism of cosimplicial spaces  $\zeta_n:\Delta^\bullet\, \square \,\cdots\, \square \, \Delta^\bullet\to \Delta^\bullet$ ($n-1$-times $\square$) by
\[
\begin{split}
[(t_{11},\dots, t_{1p_1}),\dots, & (t_{n,1},\dots t_{n,p_n})] \\
& \longmapsto \left( \frac{t_{11}}{n},\dots, \frac{t_{1p_1}}{n},\dots ,
\frac{n-1+t_{n,1}}{n},\dots,\frac{n-1+t_{n,p_n}}{n} \right)
\end{split}
\]
The following is proved in \cite{ms}.
\begin{lem}[Lemma 3.6 of\cite{ms}]\label{Lzetan}
$\zeta_n$ is an isomorphism of cosimplicial spaces. \qed
\end{lem}
Thus  we obtain an identification
\[
\B(n)\stackrel{(\zeta_n)_*}{\cong} \Map_{\CG^\Delta}(\Delta^\bullet, \Delta^\bullet)\cong
\{ u:[0,1]\to [0,1]\mid u\text{ is a weakly monotone surjection }\}
\]
The second homeomorphism is given by the natural projection $\Map(\Delta^\bullet, \Delta^\bullet)\to \Map(\Delta^1,\Delta^1)$. For an element $f\in \B(n)$, we denote by $\underline{f}$ the corresponding weakly monotone surjection. We shall describe the composition product of $\B$ using this identification. For $f\in \B(n)$ and $g\in \B(m)$,
\begin{equation}\label{EQparameter}
\ul{f\circ_ig}(t)=\left\{
\begin{array}{ll}
\frac{n}{m+n-1}\,\ul{f}(t) & \text{if } t\in \ul{f}^{-1}[0,\frac{i-1}{n}]\vspace{2mm}\\
\frac{1}{m+n-1}[i-1+m\ul{g}(\, n\ul{f}(t)-i+1)\, ]& \text{if } t\in \ul{f}^{-1}[\frac{i-1}{n},\frac{i}{n}]\vspace{2mm}\\
\frac{1}{m+n-1}[m-1+n\ul{f}(t)\,]& \text{if } t\in \ul{f}^{-1}[\frac{i}{n},1]
\end{array}\right.
\end{equation}
The following is  proved in \cite{ms}.
\begin{prop}[Theorem 3.1 of \cite{ms}]\label{Pms}
For an $\square$-object $X^\bullet$, $\tot(X^\bullet)$ has a natural action of the operad $\B$. In other words, $\tot$ induces a functor from the category of $\square$-objects to the category of $\B$-algebras.\qed
\end{prop}
\subsubsection{Slight generalization of McClure-Smith product}\label{SSSslightgenofms}
We will use a straightforward generalization of MS-product described as follows.\\
\indent In general, the monoidal product $\square$ is not symmetric but if one of variables is a constant cosimplicial object, there exist obvious natural isomorphisms 
\[
c(X)\, \square \, c(Y)\cong c(X\otimes Y), \quad c(X)\, \square \, Z^\bullet \cong Z^\bullet \, \square \, c(X)\cong  X\otimes Z^\bullet .
\]
for $X$, $Y\in \C$ and $Z^\bullet \in \CC$. Here $c(-)$ is the constant cosimplical object, and  $X\otimes Z^\bullet$ denotes the cosimplicial object defined by $(X\otimes Z^\bullet)^n=X\otimes Z^n$.  \\
\indent Recall that there is a canonical procedure which associate a monad to an operad in a symmetric monoidal category. We shall define  a monad $\MK$ over $\C^\Delta$ by a similar way.  As a functor $\MK:\CC\longrightarrow \CC$, 
\[
\MK(X^\bullet)=\bigsqcup_{n\geq 1}\K(n)\hotimes (X^\bullet)^{\, \square \, n},
\]
where $\hotimes$ means the external tensor at each cosimplicial degree, and structure morphisms $\MK\circ\MK\Rightarrow \MK$ and $\id_{\CC}\Rightarrow \MK$ is defined by the same formula as in the symmetric monoidal case  by using the above symmetry isomorphisms.
\begin{defi}
A topological operad $\tB$ is defined as follows.
As a topological space, 
\[
\tB(n)=\Map_{\CCG}(\Delta^\bullet,\, \K(n)\otimes (\Delta^\bullet)^{ \square \, n})
\] 
for each $n$. For an element $(f,g_1,\dots, g_n)\in \tB(n)\times \tB (m_1)\times\cdots\times \tB(m_n)$ the composition product $f\circ (g_1,\dots, g_n)$ is given by
\[
\begin{split}
\Delta^\bullet& \xrightarrow{f} \K(n)\otimes (\Delta^\bullet)^{\, \square \, n}
\xrightarrow{\id\otimes (g_1\, \square \,\cdots\, \square \, g_n)}
\K(n)\otimes \big(\K(m_1)\otimes (\Delta^\bullet)^{\, \square \, m_1}\big)\, \square \,\cdots\, \square \, \big(\K(m_n)\otimes (\Delta^\bullet)^{\, \square \, m_n}
\big)
\vs{2mm}\\
&\cong \K(n)\otimes \K(m_1)\otimes\cdots\otimes \K(m_n)\otimes (\Delta^\bullet)^{\, \square \, m_1+\cdots+ m_n}\vs{2mm}\\
&\xrightarrow{(-\circ -)\otimes \id} \K(m_1+\cdots+ m_n )\otimes (\Delta^\bullet)^{\, \square \, m_1+\cdots+ m_n}.
\end{split}
\]
\end{defi}

\begin{prop}\label{PtB}
\begin{itemize2}
 \item[(1)] As operads, $\tB\cong \B\times \K$. In particular, $\tB$ is an $A_\infty$-operad.
 \item[(2)] For a $\MK$-algebra $X^\bullet$, $\tot(X^\bullet)$ has a natural action of the operad $\tB$. In other words, $\tot$ induces a functor $\Alg_{\MK}(\CC)\to \Alg_{\tB}(\C)$.
 \end{itemize2}
\end{prop}
\begin{proof}
(1) is clear.  The isomorphism is induced by the natural projection to $\B$ and evaluation at $\Delta^0$. For (2), the action of $\tB$ on $\tot(X^\bullet)$ for a $\MK$-algebra $X^\bullet$ is given by
\[
\begin{split}
\tB(n)\hotimes \tot(X^\bullet)^{\otimes n}&\to \tB(n)\hotimes\Map_{\CC}((\Delta^\bullet)^{ \square \, n},\, (X^\bullet)^{ \square \, n})\vs{1mm}\\
&\to\tB(n)\hotimes\Map_{\CC}(\K(n)\otimes(\Delta^\bullet)^{ \square \, n},\, \K(n)\hotimes (X^\bullet)^{ \square \, n})\vs{1mm}\\
&\to\Map_{\CC}(\Delta^\bullet,\, \K(n)\hotimes (X^\bullet)^{ \square \, n})\to \tot(X^\bullet).
\end{split}
\]
Here, the first morphism  is induced by the monoidal structure $\square$, the second by the tensor with the identity on $\K(n)$, and the third by the action of $\MK$. (Here we use the notation $\Map$ instead of $(-)^{(-)}$ for simplicity.)
\end{proof}
Finally, we state a similar claim for $\ttot$ (see subsection \ref{SSNT}). Let  $\tB'$ be an operad obtained by replacing $\Delta^\bullet$ with $\tilde\Delta^\bullet$ in the definition of $\tB$ (see subsection \ref{SSNT}).
\begin{prop}\label{PtB'}
 $\tB'$ is an $A_\infty$-operad, and $\ttot$ induces a functor $\Alg_{\MK}(\CC)\to \Alg_{\tB'}(\C)$.\qed
\end{prop}

\subsection{Topological Hochschild cohomology}\label{SStophochschild}
\begin{defi}
Let $A$ be a monoid in $\SP$ (i.e., an associative symmetric ring spectrum).  A $A-A$-bimodule in the category $\SP$ is as usual, an object $M$ of $\SP$ with a morphism 
$A\otimes M\otimes A \to M$ called a (two-sided) action of $A$, which satisfies the usual associativity and unity axioms. A morphism of $A-A$-bi-modules is a morphism in $\SP$ compatible with actions of $A$. Let $M$ be  a $A-A$-bimodule over $\SP$. We define a cs-spectrum $\thc^\bullet
(A,M)$ as follows. $\thc^p(A,M)=\Inhom_\SP(A^{\otimes p},M)$. This means $\thc^0(A,M)=\Inhom_\SP(\Sph, M)=M$. The coface operator $d^0:\thc^p(A,M)\to \thc^{p+1}(A,M)$ is defined as the adjoint of the following composition
\[
\Inhom_{\SP}(A^{\otimes p},M)\otimes A^{\otimes p+1}\stackrel{\cong}{\to} A\otimes \Inhom_{\SP}(A^{\otimes p},M)\otimes A^{\otimes p}\to A\otimes M\to M,
\]
where the first arrow is the transposition of the  first component of $A^{\otimes p+1}$, the second arrow is induced by the evaluation $\Inhom(X,Y)\otimes X\to Y$, and the third arrow is the left action of $A$. The last coface operator $d^p$ is defined similarly by using the right action. The other coface operators $d^i\  (1\leq i\leq p-1)$ and the codegeneracy operators $s^j:\thc^p(A,M)\to \thc^{p-1}(A,M)$\ ($0\leq j\leq p-1$) are defined as the pullback by the following morphisms 
\[
\begin{split}
A^{\otimes p+1} &= A^{\otimes i-1}\otimes (A\otimes A)\otimes A^{\otimes p-i}\to A^{\otimes i-1}\otimes A\otimes A^{\otimes p-i}=A^{\otimes p} \\
A^{\otimes p-1} &=A^{\otimes j}\otimes \mathbb{S}\otimes A^{p-j-1}\to A^{\otimes j}\otimes A\otimes A^{\otimes p-j-1}=A^{\otimes p}
\end{split}
\] 
induced by the product and unit morphism of $A$, respectively.  We call the totalization  $\tot(\thc^\bullet(A,M))$ the \textit{topological Hochschild cohomology of $A$ with coefficients in $M$} and denote it  by $\thc(A,M)$. 
\end{defi}
Let $A$ be a monoid in $\SP$. Note that the category of $A-A$ bimodules has natural structure of (non-symmetric) monoidal category with its monoidal product  given by usual tensor over $A$ of left and right $A$-modules. We consider semigroups in this monoidal category, which we call \textit{non-unital $A$-algebras}.
\begin{lem}\label{LmonoidTHC}
Let $A$ be a monoid in $\SP$ and $B$  a non-unital $A$-algebra. $\thc^\bullet(A,B)$ has a natural structure of a $\square$-object. In particular, $\thc(A,B)$ has the induced structure of   $\B$-algebra (see Proposition \ref{Pms}).
\end{lem}
\begin{proof} The associative product is given as the adjoint of 
\[
\Inhom(A^p,B)\otimes\Inhom(A^q,B)\otimes A^{p+q}\cong \Inhom(A^p,B)\otimes A^p\otimes\Inhom(A^q,B)\otimes A^q\xrightarrow{\text{evaluation}^2}B\otimes B\to B
\]
\end{proof} 
Note that an $\square$-object is regarded as an $\MK$-algebra by pulling back the structure by the map $\K(n)\hotimes (-)^{\square n}\to (-)^{\square n}$ which is the tensor  of the map $\K(n)\to*$ and the identity. Hence we may regard  the cs-spectrum $\thc^\bullet(A,B)$  as a $\MK$-algebra using the $\square$-object structure in  Lemma \ref{LmonoidTHC}.

\indent In the rest of this subsection, we prepare some technical results used later . The part one of the following is proved in \cite{mms} and the proof of part two is similar.
\begin{prop}\label{propnonunitalmodel}
\begin{itemize2}
\item[(1)] The category of monoids in $\SP$ has the model category structure where a morphism is weak equivalence or  fibration if and only if so it is as a morphism of   $\SP$ with the stable model structure.
\item[(2)] Let $A$ be a monoid in $\SP$. The category of non-unital $A$-algebras has a model category structure where a morphism is a weak equivalence or  fibration if and only if so it is as a morphism of   $\SP$ with the stable model structure.\qed
\end{itemize2}
\end{prop}

Note that the function spectrum $F(M)$ has a structure of  monoid given by $f_1\cdot f_2(x)=(f_1(x), f_2(x))\in S^{k_1+k_2}=\R^{k_1+k_2}\cup\{\infty\}$ for $f_i\in F(M)_{k_i}, x\in M$. This product is the same as that induced by the product on $\Sph$ at the codomain.  We define a structure of non-unital $F(M)$-algebra on $\Gamma(M)$ as follows (see subsection \ref{SSatiyahduality} for the definition of $\Gamma(M)$): For $\langle \phi_1,\eps_1,\theta_1\rangle$, 
$\langle \phi_2,\eps_2,\theta_2\rangle\in\Gamma(M)$, we put $\langle \phi_1,\eps_1,\theta_1\rangle \cdot \langle \phi_2,\eps_2,\theta_2\rangle=\langle \phi_1\times\phi_2, \min\{\eps_1,\eps_2\}, \theta_1\wedge \theta_2\rangle$, where $\theta_1\wedge \theta_2$ denotes the section taking $x\in M$ to the point represented by $(\theta_1(x),\theta_2(x))$. ( If this element does not belong to the interior of $\bar B_{\min\{\eps_1,\eps_2\}}(\phi_1\times \phi_2(x))$, we set $\theta_1\wedge \theta_2(x)=*$. ) For $f \in F(M)$ and $\langle\phi,\eps,\theta\rangle\in\Gamma(M)$, we put  $f\cdot (\phi,\eps,s)= (0\times\phi, \eps, f\wedge \theta)$ where $f\wedge \theta$ is understood similarly to $\theta_1\wedge \theta_2$. We define a right action by using this left action and the symmetry isomorphism of $\SP$, as the compositon
\[
\Gamma(M)\otimes F(M)\cong F(M)\otimes \Gamma(M)\to \Gamma(M).
\]
This  is well-defined since the product on $F(M)$ is commutative.
\begin{prop}\label{propinvarianceofthc}
 Let $A'\to F(M)$ be a cofibrant replacement of $F(M)$ as a monoid in $\SP$. We consider $\Gamma(M)$ as a non-unital $A'$-algebra by pulling back the structure of $F(M)$-algebra defined above.  Let $\Gamma(M)\to B$ be a fibrant replacement of $\Gamma(M)$ as $A'$-algebra and $Q$ be any fibrant cofibrant replacement of $F(M)$ (as a  monoid in $\SP$). Then 
$\thc(Q,Q)$ and $\thc(A',B)$ are weak equivalent as nu-$A_\infty$-ring spectra. Here, all replacements are taken with respect to the model structures in Proposition \ref{propnonunitalmodel}.

\end{prop}
\begin{proof}
It is easy to see the weak equivalence class of $\thc(Q,Q)$ as a $\B$-algebra is independent of a fibrant-cofibrant replacement $Q$. 
So all we have to do is to prove the claim for one particular $Q$.
Take a trivial cofibration $i:A'\to Q$  with $Q$ fibrant in the category of monoids. 
Clearly, $Q$ is also cofibrant. 
We have a zig-zag of weak equivalences of non-unital $A'$-algebras as follows:
\[
Q\leftarrow A'\rightarrow F(M)\xleftarrow{\kappa_1} \Gamma'(M)\xrightarrow{\kappa_2} \Gamma(M)\rightarrow B,
\]
where $\kappa_1$, $\kappa_2$, and $\Gamma'(M)$ is defined in subsection \ref{SSatiyahduality}, and the non-unital $A'$-algebra structure on $\Gamma'(M)$ is defined completely analogously to $\Gamma(M)$. Using Proposition \ref{propnonunitalmodel}, we can replace this chain by  the following zig-zag of weak equivalences of nu $A'$-algebras:
\[
Q\leftarrow B'\rightarrow B,
\]
where  $B'$ is a fibrant object. By this and homotopy invariance properties of $\otimes$\  (and $\Inhom$) in \cite{mms}, we obtain a zig-zag of termwise  level equivalences of $\square$-objects:
\[ 
\thc^\bullet(Q,Q)\to \thc^\bullet(A',Q)\leftarrow \thc^\bullet(A',B')\rightarrow \thc^\bullet(A',B),
\]
where the first map is induced by pullback by $A'\to Q$. Again by properties of $\otimes$\ in \cite{mms},  any cs-spectra in this zig-zag are Reedy fibrant, so the application of $\tot$ produces a zig-zag of level equvalences between $\thc(Q,Q)$ and $\thc(A',B)$ as $\B$-algebras.
\end{proof}
\section{Non-unital $A_\infty$-symmetric ring spectrum in string topology}\label{Sainfty}
In this section we refine the Cohen-Jones ring spectra recalled in subsection \ref{SScohenjones} to a nu-$A_\infty$-ring spectrum. It is realized as an action of the operad $\tB$ defined in sub-subsection \ref{SSSslightgenofms} on a symmetric spectrum $\LMT$ which is isomorphic to $\LMTM$ in the homotopy category $\Ho(\SP)$.
We begin by  defining the symmetric spectrum  $LM^{-\tau}$. 
\begin{defi}\label{DLMT}
Recall from subsection \ref{SSatiyahduality} the definitions of $\TV_k$, $\M_k$, and $\MT$. Let $ev:\TV_k\times LM\to \TV_k\times M$ denote the product of the identity  and the evaluation at $0\in S^1$ for each $k\geq 0$. Let $ev^*\M_k$ denote the space defined by the following pullback diagram.
\[
\xymatrix{ev^*\M_k\ar[r]\ar[d]&\TV_k\times LM\ar[d]^{ev}\\
\M_k\ar[r]&\TV_k\times M.}
\]
Here, the bottom horizontal arrow is the projection of the disk bundle defined in section \ref{SSatiyahduality}. With the top horizontal arrow , we regard $ev^*\M_k$ as a disk bundle over $\TV_k\times LM$. An element $(\phi,\eps,v,c)\in ev^*\M_k$ is a pair of an element $(\phi,\eps,v)\in \M_k$ and $c\in LM$ such that $c(0)=\pi_{\phi}(v)$.
We define a space $\LMT_k$ as the Thom space associated to this   disk bundle $ev^*\M_k$. We endow the sequence $\LMT=\{\LMT_k\}_{k\geq 0}$ with a structure of  symmetric spectrum  exactly analogous to that of $\MT$. (The structure map $S^1\wedge \LMT_k\to \LMT_{k+1}$ does not change the loop component.) 
\end{defi}
\textbf{Notation.}\quad In the rest of paper, the bold letter $\ve$  (or $\ve_i$) denotes an element of $\M_k$, and the components of $\ve$ (or $\ve_i$) are denoted by
\[
(\phi,\eps,v)\quad (\,\text{or} \ (\phi_i,\eps_i,v_i)\,)
\]
where $\phi\in V_k$, $\eps\in (0,L_{e_0}/16)$, and  $v\in\bar\nu_\eps(\phi)$ ( or $\phi_i\in V_k$, $\eps_i\in (0,L_{e_0}/16)$, and $v_i\in\bar\nu_{\eps_i}(\phi_i)$ ). Similarly, the bold letter $\ci$ ( or $\ci_i$ ) denotes an element of $\EVN_k$. The components of $\ci$ ( or $\ci_i$ ) are denoted by
\[
(\phi,\eps,v,c)\quad (\,\text{or} \ (\phi_i,\eps_i,v_i,c_i)\,)
\]
where $(\phi,\eps,v)\in \M_k$  and  $c\in LM$ ( or $(\phi_i,\eps_i,v_i)\in \M_k$ and $c_i\in LM$ ). \\
For a sequence of spaces $\{X_k\}_{k\geq 0}$ and an integer $n\geq 1$ we put 
\[
X[k_1,\dots, k_n]=X_{k_1}\times\cdots\times X_{k_n}.
\]

\subsection{Action of $\tB$ on $\LMT$}\label{SSactiononLMT}

\subsubsection{Sketch of the construction}\label{SSSroughoutline}
We shall sketch the idea of the construction of the product (or  action) on $\LMT$. An element of $\LMT_k$ includes a loop in $M$. For two loops $c_1,\ c_2$, their basepoints (the values at $0\in S^1$) may not coincide, but if these points are sufficiently close, we can perturb these loops so that they have a common basepoint in a uniform way, and concatenate them. The perturbation is given by  usual parallel transport in the Euclidean space where $M$ is embedded.  So the loops go outside of $M$ once but if the loops are contained in a tubular neighborhood of $M$, we can get back them into $M$ by the projection of the neighborhood. If the basepoints of two loops are not sufficiently close, and the transported loops are not contained in the neighborhood, we collapse them to the basepoint of the Thom space.  More precisely, the perturbation of two loops is done by two self maps $\psi_1(d),\psi_2(d):M\to M$ on $M$ ($d$ will be explained later) which are homotopic to the identity on $M$. We define these maps so that $\psi_1(d)\circ c_1$ and $\psi_2(d)\circ c_2$ have a common basepoint. We take a common basepoint $p_0\in M$ appropriately, and $\psi_i(d)$ is given by projection  after the parallel transport by the vector $p_0-c_i(0)$, where each point means its image by the embedding and the difference is taken in the Euclidean space.  This map clearly satisfies $\psi_i(d)(c_i(0))=p_0$ so the two loops can be concatenated. Actually, we will use a point projected to $p_0$ in the tubular neighborhood instead of $p_0$. As is seen below, the embedding  and the radius of the tubular neighborhood which have appeared here, depend on elements of $\LMT$. \\
\indent We shall go into details. Take elements $\ci_1, \ci_2 \in \LMT$. $\ci_i$ consists of a map $\phi_i:\R^{k_0}\to \R^{k_i}$, a number $\eps_i>0$, an element $v_i$ of tubular neighborhood of $\phi_i|_M:M\to \R^{k_i}$, and a loop $c_i$ with $c_i(0)=\pi_{\phi_i}(v_i)$. The common point used instead of $p_0$ is  $(v_1,v_2)$, and $\psi_i(d)$ is given by
\[
M\ni y\longmapsto y+(v_1,v_2)-c_i(0)=y+(v_1,v_2)-\pi_{\phi_i}(v_i) \in \R^{k_1+k_2},
\]
followed by the projection $\pi_{\phi_1\times \phi_2}$. Here, $y,\ c_i(0),$ and  $\pi_{\phi_i}(v_i)$ are regarded as elements of $\R^{k_1+k_2}$ via $\phi_1\times \phi_2$
  (see Figure \ref{Fpsiarity2}). We see $\psi_i(d)(\pi_{\phi_i}(v_i))=\pi_{\phi_1\times \phi_2}(v_1,v_2)$ since the right hand side of the `$\longmapsto$' is $(v_1,v_2)$ if $y=\pi_{\phi_i}(v_i)$, so the loops perturbed by these maps have the  common basepoint $\pi_{\phi_1\times \phi_2}(v_1,v_2)$.  Obviously  $\psi_i(d)$ depends on $\phi_i,\ v_i$. Also, to justify this definition, the left hand side of this map must belong to a tubular neighborhood of $\phi_1\times \phi_2(M)$. This condition is satisfied if $(v_1,v_2)$ is sufficiently close to the image of $M$ since if so, the norm of the vector used in the parallel transport is sufficiently small. Thus, $\psi_i(d)$ depends on an element $d$ of  
\[
\D_{k_1,k_2}^2=\{(\phi_1,\eps_1,v_1),(\phi_1,\eps_1,v_2)\mid d((v_1,v_2), M)<\teps\}.
\]
Here, $M$ denotes its image by $\phi_1\times \phi_2$, and $\teps$ is a sufficient small number. We may regard the space
 $\D_{k_1,k_2}^2$ as a `parametrized tubular neighborhood of the diagonal $M\subset M^{\times 2}$ (in the Euclidean space)' and we concatenate two loops only when their basepoints (or more precisely their vectors at the point) belong to $\D^2_{k_1,k_2}$, and otherwise, we collapse them to the basepoint of the Thom space. Actually, the standard of ``sufficient closeness'', depends on $\phi_i$. We also need to send the points outside of the tubular neighborhood  to the outside of that as they are collapsed, so $\teps$ should  be smaller than $\eps_i$. By these demands,  we must define a function 
\[
\teps:\{((\phi_1,\eps_1),(\phi_2,\eps_2))\}\to \R
\] 
and $\teps$ in the definition is replaced with the value $\teps((\phi_1,\eps_1),(\phi_2,\eps_2))$. Finally, we set
\[
 \ci_1\cdot\ci_2=(\phi_1\times \phi_2,\,\teps,\,(v_1,v_2),\, (\psi_1(d)\circ c_1)\cdot(\psi_2(d)\circ c_2))
\]
if $(\phi_i,\eps_i,v_i)_{i=1,2}\in \D_{k_1,k_2}^2$ and $\ci_1\cdot \ci_2=*$ otherwise. 
Here, $\teps=\teps((\phi_1,\eps_1),(\phi_2,\eps_2))$ and the dot between  loops denotes the concatenation (with rescaling of the interval (or $S^1$) depending on an element of $\B(2)$).  \par
 Next, consider the product of three elements $\ci_1,\ci_2,$ and $\ci_3$. For the above product, $(\ci_1\cdot\ci_2)\cdot \ci_3$\ and $\ci_1\cdot(\ci_2\cdot\ci_3)$ do not coincide even if we ignore the parametrization of $S^1$. The perturbations used in the former product are  $\psi_1\circ\psi_1,\ \psi_1\circ\psi_2,\ \psi_2$ and the latter are $\psi_1,\ \psi_2\circ\psi_1,\ \psi_2\circ\psi_2$ since for example, $\ci_1$ is perturbed twice by $\psi_1$ in the former one. More precisely, by the definition of the product, $\psi_1\circ\psi_1$, the first map in the former triple (resp. $\psi_1$, the first map  in the latter triple) means $\psi_1(d_1)\circ\psi_1(d_2)$ (resp. $\psi_1(d_3)$) with 
\[
d_1=((\phi_1\times \phi_2, \teps),(\phi_3,\eps_3)),\quad d_2=((\phi_1,\eps_1),(\phi_2,\eps_2)),\ 
\ \text{and}\  \ d_3=((\phi_1,\eps_1),(\phi_2\times \phi_3,\teps'))
\]
 where $\teps=\teps((\phi_1,\eps_1),(\phi_2,\eps_2))$, $\teps'=\teps((\phi_2,\eps_2),(\phi_3,\eps_3))$. The other maps are similar abbreviations.  To construct a homotopy between the two products, it is enough to construct a homotopy between these triples. The construction is similar to that of $\psi_i$. Similarly to $\D^2$, we define a parameter space by
\[
\D^3_{k_1,k_2,k_3}=\{(u\,;(\phi_i,\eps_i,v_i))_{i=1,2,3}\mid u\in \K(3), \ d((v_1,v_2,v_3), M)<\teps(u\, ;(\phi_1,\eps_1),(\phi_2,\eps_2),(\phi_3,\eps_3))\}
\]
Here $\K(3)$ is the arity 3 part of the associahedra, and $M$ denotes its image by $\phi_1\times \phi_2\times \phi_3$, and $\teps$ is suitable function. We shall look at the homotopy between the first maps $\psi_1\circ\psi_1$ and $\psi_1$. It is defined as  a map
\[
\psi^3_1:\D^3_{k_1,k_2,k_3}\longrightarrow \Map(M,M).
\]
We use the cone presentation of $\K(3)$ in subsection \ref{SSK}. Let $u_L$ (resp. $u_R$) be the vertex represented by the tree $T\circ_1T$ (resp. $T\circ_2T$), where $T$ is the unique 2-tree. An element of $\K(3)$ is expressed as $tu_L$ or $tu_R$ for some $t\in [0,1]$, and $0u_L$ and $0u_R$ represent the same point. $\psi^3_1(tu_L)$, where the other components $(\phi_i, \eps_i, v_i)$ of $\D^3$ are omitted, is defined by
\begin{equation}\label{EQarity3psi}
M\ni y\ \longmapsto \  t\,\{ \psi_1(d_1)\circ\psi_1(d_2)(y)\}+(1-t)\{y+v-\pi_{\phi_1}(v_1)\}\ \in \R^{k_1+k_2+k_3}
\end{equation}
followed by $\pi_{\phi}$, where we set $v=(v_1,v_2,v_3)$ and $\phi=\phi_1\times \phi_2\times \phi_3$, and elements in $M$ denotes its image by $\phi$, and $t$ and $(1-t)$ mean the usual scalar multiplication, so this is the straight line homotopy.  Similarly, $\psi^3_1(tu_R)$ is defined by 
\[
y\ \longmapsto\  t\,\{ \psi_1(d_3)(y)\}+(1-t)\{y+v-\pi_{\phi_1}(v_1)\}\in \R^{k_1+k_2+k_3}
\]
followed by $\pi_{\phi}$. The two definitions coincide on $t=0$ obviously. The equation 
\begin{equation}\label{EQbasepointspsi}
\psi^3_1(tu_L\text{ or }tu_R)(\pi_{\phi_1}(v_1))=\pi_{\phi}(v)
\end{equation}
 also holds. We shall verify this for $tu_L$. Put $y_0:=\pi_\phi(v)$. We have $\psi_1(d_1) \circ \psi_1(d_2)(\pi_{\phi_1}(v_1))=\psi_1(d_1)(\pi_{\phi_1\times \phi_2}(v_1,v_2))=\pi_{\phi}(v)=y_0$. So if the norm of $v-\pi_{\phi_1}(v_1)$ is sufficiently small and $y=\pi_{\phi_1}(v_1)$ holds,  the two endpoints of the line segment parametrized by $t$ in the formula (\ref{EQarity3psi}) belong to the fiber over $y_0$  in the tubular neighborhood $\nu_{\eps}(\phi)$ for some $\eps>0$. The fiber is the $\eps$-disk in the affine subspace $\phi(y_0)+(T_{\phi(y_0)}\phi(M))^\perp$ by definition. In particular it is a convex set. So any point in the line segment belongs to the same fiber, which implies the equation (\ref{EQbasepointspsi}). This equation ensures the perturbed loops to have a common basepoint if we define homotopies between the remaining pairs in the two triples  similarly. In the higher arity case, we define the perturbation map $\psi$ similarly. We proceed on inductively for the arity. In arity $n$, we first define the map on $\partial \K(n)$ using  composition of the maps in lower arity, then using the cone presentation of $\K(n)$, we define the map on the interior. We also construct $\teps$ similarly and actually we construct  $\teps$ first. In the following, we omit the superscript of $\psi$.

\subsubsection {Formula for the action}\label{SSSoutlineLMT}
The action of $\tB$ on $\LMT$ which we will construct is denoted by $\Phi=\{\Phi_n\}_{n\geq 1}$ where $\Phi_n$ is a morphism
\[
\Phi_n:\tB(n)\hotimes (\LMT)^{\otimes n}\longrightarrow \LMT.
\] 

We will define a continuous function
\[
\teps=\teps_{k^1,\dots, k^n}:\K(n)\times\TV[k_1,\dots, k_n]\longrightarrow (0,L_e/16)
\]
(which satisfies some conditions) for each $k_1\dots k_n\geq 0$. Roughly speaking, this function gives the upper bound of  the distance between basepoints of loops which the loop product does not collapse loops to the basepoint of the Thom space.  We define a space $\D_{k_1,\dots, k_n}^n\subset \K(n)\times \M[k_1,\dots,k_n]$ by 
\[
\D_{k_1,\dots, k_n}^n=\{(u;\ve_i)_{i=1}^n\mid d((v_1,\dots, v_n),\, \phi(M))\leq \teps (u;(\phi_i,\eps_i)_{i=1}^n)\},
\]
where $\phi=\phi_1\times\cdots\times \phi_n:\R^{k_0}\to\R^{k_{\leq n}}$, and $d(-,-)$ is the Euclidean distance. 
As in the previous sub-subsection, we may consider the space
 $\D_{k_1,\dots, k_n}^n$ as a `parametrized tubular neighborhood of the diagonal $M\subset M^{\times n}$ (in the Euclidean space)' and we concatenate $n$ loops only when their basepoints (or more precisely their vector at the point) belongs to $\D$, and otherwise, we collapse them to the basepoint of the Thom space. \\
\indent  We will  define a continuous map 
\[
\tPhi_{k_1,\dots, k_n}:\B(n)\times ev^*\D_{k_1,\dots, k_n}^n\longrightarrow ev^*\M_{k_{\leq  n}}
\]
Here, $ev^*\D_{k_1,\dots,k_n}$ is defined by the following pullback diagram
\[
\xymatrix{ev^*\D_{k_1,\dots,k_n}\ar[r]\ar[d]&\D_{k_1,\dots,k_n}\ar[d]^{\cap}\\
\K(n)\times ev^*\M[k_1,\dots,k_n]\ar[r]^{\id\times ev}&\K(n)\times \M [k_1,\dots,k_n],}
\]
or more explicitly
\[
\begin{split}
ev^*\D_{k_1,\dots,k_n}=\{(u;\ci_i)_{i=1}^n\mid d((v_1,\dots, v_n),\, \phi(M))&    \leq \teps (u;(\phi_i,\eps_i)_{i=1}^n)\}\vspace{1mm}\\
&(\,\subset \K(n)\times ev^*\M[k_1,\dots, k_n]\,).
\end{split}
\]
Assuming the existence of $\tPhi$, the action $\Phi_n:\tB(n)\hotimes (\LMT)^{\otimes n}\longrightarrow \LMT$ is defined using $\tPhi$ and a collapsing map as follows. 
If the values of $\teps$ are  sufficiently small, by  Proposition\ref{PtB} (1), we may regard the space $\B(n)\times ev^*\D_{k_1,\dots,k_n}$ as a subspace of $\tB(n)\times \EVN[k_1,\dots,k_n]$ by
\[
\B(n)\times ev^*\D_{k_1,\dots,k_n}\,\subset \,
\B(n)\times \K(n)\times ev^*\M[k_1,\dots,k_n]\,=\,\tB(n)\times \EVN[k_1,\dots,k_n].
\] Using this inclusion, the morphism $\Phi_n$ is defined by the following composition
\[
\begin{split}
(\tB (n)_+)\wedge \LMT_{k_1}&\wedge \LMT_{k_2}\wedge\cdots\wedge \LMT_{k_n}\\
&\cong \, (\tB(n)\times\EVN[k_1,\dots, k_n])\, /\, \tB(n)\times\partial(\EVN[k_1,\dots, k_n])\\
&\xrightarrow{\text{collapse}}\,   
(\B(n)\times ev^*\D_{k_1,\dots, k_n})\, /\, (\B(n)\times\partial ev^*\D_{k_1,\dots,k_n})\\
&\xrightarrow[]{\tPhi}\, (\EVN_{k_{\leq n}})\, /\, (\partial\EVN_{k_{\leq n}})=
\LMT_{k_{\leq n}}
\end{split}
\]
Here, $\tB(n)_+$ denotes $\tB(n)$ with disjoint basepoint,  $\partial ev^*\D_{k_1,\dots, k_n}$, $\partial(\EVN[k_1,\dots,k_n])$, 
 and  $\partial\EVN_{k_{\leq n}}$ (the total space of ) the boundary sphere bundles of $ev^*\D_{k_1,\dots,k_n}$, $\EVN[k_1,\dots,k_n]$, and  $\EVN_{k_{\leq n}}$ respectively. \\
\indent For elements $(f,u)\in\tB(n)=\B(n)\times\K(n)$, $\ci_i\in\EVN$,  $\tPhi(f,u\,;\ci_1,\dots,\ci_n)$ will be of the following form 
\begin{equation}
(\phi_1\times\cdots\times \phi_n,\, \teps,\, (v_1,\dots,v_n),\, \tc) \label{EtPhi}
\end{equation} 
for some loop $\tc$  ($\teps=\teps\,(u\,;(\phi_1,\eps_1),\dots, (\phi_n,\eps_n))$). To define $\tc$, we will need a  continuous map
\[
\psi=(\psi_1,\dots,\psi_n):\D_{k_1,\dots, k_n}\longrightarrow \Map_{\CG}(M, M)^{\times n}.
\]
for each $k_1,\dots, k_n$. As mentioned in sub-subsection \ref{SSSroughoutline}, this map  encodes the perturbation  procedure. Actually,  $\psi$ will satisfy the equation 
\begin{equation}
\psi_1(d)(\pi_{\phi_1}(v_1))=\psi_2(d)(\pi_{\phi_2}(v_2))=\cdots =\psi_n(d)(\pi_{\phi_n}(v_n))
=\pi_{\phi_1\times\cdots\times\phi_n}(v_1,\dots,v_n) \label{Ebp}
\end{equation}
 for any $d=(u\,;\ve_1,\dots,\ve_n) \in \D$. Here $\psi_i(d)\in \Map(M,M)$ denotes the value of $\psi_i$ at $d$. So if $(u\,;\ci_1,\dots,\ci_n)\in ev^*\D$, as $c_i(0)=\pi_{\phi_i}(v_i)$,  the $n$ loops $\psi_1(d)\circ c_1,\dots,\psi_n(d)\circ c_n$ has common basepoint $\pi_{\phi_1\times\cdots\times \phi_n}(v_1,\dots,v_n)$ and we can concatenate these loops at the basepoint. Here, $d$ is the image of $(u\,;\ci_1,\dots,\ci_n)$ by the projection $ev^*\D\to\D$. $\tc$ is defined as the  loop created by concatenation of the loops rescaled with the rate determined by an element $f\in\B(n)$. Precisely speaking, we put
\begin{equation}\label{Etc}
\tc(t)=\psi_i(d)\circ c_i(nf(t)-i+1) \quad \text{ for }\  t\in f^{-1}\left[ \frac{i-1}{n},\frac{i}{n}\right], \ \  1\leq i\leq n.
\end{equation}
 (Here we identify $f$ with a weakly monotone surjection on $[0,1]$, see subsection \ref{SSms}.) See Figure \ref{Fperturb} for the case of $n=2$.
This $\tc$ is a well-defined single loop since $\psi_i(d)$'s have the common basepoint, and $nf(t)-i+1=0$ (resp. $1$) if $t\in f^{-1}(\frac{i-1}{n})$ (resp. $f^{-1}(\frac{i}{n})$).

\begin{figure}\label{Fperturb}
\begin{center}
 \input{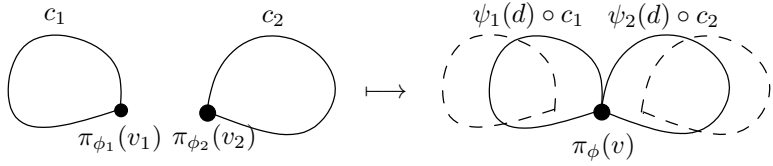}
\end{center}
\caption{Perturbation of loops :  the dot represents the value at $0$ of each loop, and we set $\phi=\phi_1\times\phi_2$,\ $v=(v_1,v_2)$}
\end{figure}

 In summary, according to the decomposition $\tB=\B\times\K$ of Proposition \ref{PtB} (1), we may say $\K$ controls the purterbation and $\B$ does the rescaling. \\
\indent Thus, the construction of the action of $\tB$ on $\LMT$ is reduced to the construction of maps $\teps$ and $\psi$.
\begin{figure}\label{Fpsiarity2}
\begin{center}
{\unitlength 0.1in%
\begin{picture}(22.4800,10.9600)(-1.8000,-11.7600)%
%
\special{pn 8}%
\special{ar 1114 2405 1527 1527 4.1053411 5.3274901}%
%
\special{pn 8}%
\special{ar 1110 2395 2262 2262 4.2448275 5.1498209}%
%
\special{pn 20}%
\special{pa 810 995}%
\special{pa 1404 305}%
\special{fp}%
\special{sh 1}%
\special{pa 1404 305}%
\special{pa 1345 342}%
\special{pa 1369 345}%
\special{pa 1376 369}%
\special{pa 1404 305}%
\special{fp}%
%
\special{sh 1.000}%
\special{ia 814 995 33 33 0.0000000 6.2831853}%
\special{pn 8}%
\special{ar 814 995 33 33 0.0000000 6.2831853}%
%
\special{sh 1.000}%
\special{ia 1430 281 35 35 0.0000000 6.2831853}%
\special{pn 8}%
\special{ar 1430 281 35 35 0.0000000 6.2831853}%
%
\special{pn 8}%
\special{pa 1430 281}%
\special{pa 1430 1101}%
\special{fp}%
%
\special{pn 8}%
\special{pa 1430 1091}%
\special{pa 1584 1091}%
\special{fp}%
%
\special{pn 8}%
\special{pa 1430 1091}%
\special{pa 1490 1091}%
\special{pa 1490 1031}%
\special{pa 1430 1031}%
\special{pa 1430 1091}%
\special{pa 1490 1091}%
\special{fp}%
%
\special{sh 1.000}%
\special{ia 1414 1111 31 31 0.0000000 6.2831853}%
\special{pn 8}%
\special{ar 1414 1111 31 31 0.0000000 6.2831853}%
%
\special{pn 8}%
\special{pa 1134 615}%
\special{pa 884 615}%
\special{fp}%
\put(4.7000,-6.1500){\makebox(0,0){$v-\phi(\pi_{\phi_i}(v_i))$}}%
\put(6.8000,-11.5500){\makebox(0,0){$\phi(y)$}}%
\put(17.7000,-4.4000){\makebox(0,0){$\widehat{\psi}_i(d)(y)$}}%
\put(14.2400,-12.4100){\makebox(0,0){$\psi_i(d)(y)$}}%
\put(20.9000,-1.4500){\makebox(0,0){$\nu_\eps(\phi)$}}%
\put(20.0000,-9.7000){\makebox(0,0){$\phi(M)$}}%
\end{picture}}%
\end{center}
\caption{the definition of $\psi$ for the arity 2 :  we set $d=(pt;\ve_1,\ve_2)$}
\end{figure}
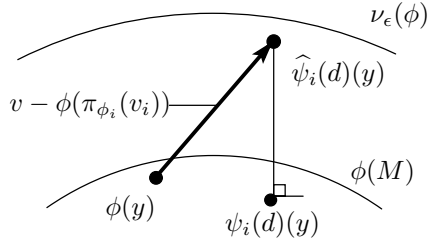
In sub-subsection \ref{SSSconditionsLMT}, we state conditions which $\teps$ and $\psi$ satisfy and prove these conditions ensure the action $\Phi$ is well-defined. In sub-subsection \ref{SSSconstructionLMT}, we construct these maps. The construction is a standard inductive one using the cone presentation of the associahedra given in subsection \ref{SSK}.  
\subsubsection{Conditions which $\teps$ and $\psi$ satisfy}\label{SSSconditionsLMT}
In this subsection, we list the conditions which $\teps$ and $\psi$ satisfy and prove the conditions ensure the formulae of the previous section give a well-defined operad action ( see Proposition \ref{Pwelldefined}). The construction is given in next subsection.\\

\indent \textbf{Conditions on $\teps$.}\quad $\teps$ satisfies the following five conditions \epsS - \epscone \ for any $n\geq 2$ and $(u\,;(\phi_1,\eps_1),\dots,(\phi_n,\eps_n))\in\K(n)\times\TV[k_1,\dots,k_n]$.
\begin{itemize}
\item[\epsS] For $i$ with $1\leq i\leq n$, $\teps(u\,;(\phi_1,\eps_1),\dots, (\phi_n,\eps_n))=\teps(u\,;(\phi_1,\eps_1),\dots,(0\times\phi_i,\eps_i),\dots, (\phi_n,\eps_n))$, where $0\times\phi_i(x)=(0,\phi_i(x))\in\R^{k_i+1}$.
\item[\epsSigma] For any permutations $\sigma_1,\dots,\sigma_n$,\ $\teps(u\,;(\phi_1,\eps_1),\dots, (\phi_n,\eps_n))=\teps(u\,;(\sigma_1\phi_1,\eps_1),\dots,(\sigma_n\phi_n,\eps_n))$, where $\sigma_i\phi_i$ is the composition of $\phi_i$ with the permutation of the components of $\R^{k_i}$ associated to $\sigma_i$.
\item[\epscomp] If $u$ is presented as $u=u_1\circ_iu_2$ for some $i$ and   $u_s\in \K(n_s)$ ($s=0,1$,\ $n_s\geq 2,\ n_1+n_2-1=n$) , 
\[
\teps(u\,;(\phi_1,\eps_1),\dots, (\phi_{n},\eps_{n}))=\teps(u_1\,;(\phi_1,\eps_1),\dots,(\phi_{i-1},\eps_{i-1}),(\phi',\eps'),(\phi_{i+n_2},\eps_{i+n_2}),\dots,(\phi_{n},\eps_{n})\,)
\]
where $(\phi',\eps')=(\phi_i\times\cdots\times\phi_{i+n_2-1},\,
\teps(u_2\,;(\phi_i,\eps_i),\dots,(\phi_{i+n_2-1},\eps_{i+n_2-1}))\,)$.
\item[\epsev]
$\teps(u\,;(\phi_1,\eps_1),\dots, (\phi_n,\eps_n))\leq \frac{1}{10^{A}}\min\{\eps_1,\dots,\eps_n\}$, where $A=|\phi_1\times\cdots\times \phi_n|^2$.
\item[\epscone]  If $u\in\partial \K(n)$, 
$\teps(tu\,;(\phi_1,\eps_1),\dots,(\phi_n,\eps_n))\leq \teps(u\,;(\phi_1,\eps_1),\dots,(\phi_n,\eps_n))$ for any $t\in [0,1]$.
\end{itemize}
Note that the first three conditions are natural ones to ensure equivariance with the sphere spectrum, with the symmetric group, and compatibility with the operad composition, respectively. The last two conditions are ones to make the construction of $\teps$ well-defined at each inductive step.\\

\textbf{Conditions on $\psi$.}\quad To state conditions which $\psi$ satisfies, we need three maps:
\[
\begin{split}
\alpha_i:
&
\K (n)\times \K (m)\times \M[k_1,\dots,k_{n+m-1}]\longrightarrow \K(n)\times \M[k_1,\dots,k_{i-1}, k_{(i)}, k_{i+m}\dots, k_{n+m-1}], 
\\
\beta_i:
&
 \K (n)\times\K(m)\times\M[k_1,\dots, k_{n+m-1}]\longrightarrow \K(n+m-1)\times\M[k_1,\dots, k_{n+m-1}],\ \text{and}
\\
\gamma_i :
&
\K (n)\times\K(m)\times
\M[k_1,\dots, k_{n+m-1}]\longrightarrow\K(m)\times\M[k_i,\dots, k_{i+m-1}]
\end{split}
\]
 defined by 
\[
\begin{split}
\alpha_i(u_1,u_2\,;\ve_1,\dots,\ve_{n+m-1})=(u_1\,;& \ve_1,\dots,\ve_{i-1},\ve_{(i)},\ve_{i+m},\dots,\ve_{n+m-1}),\\
\beta_i =(-\circ_i-)\times  \id, \quad \text{and}\quad & \gamma_i =proj_2\times proj_{i,\dots,i+m-1},
\end{split}
\]
where 
\[
\begin{split}
k_{(i)}&=k_i+\cdots +k_{i+m-1},\\
\ve_{(i)}&=(\phi_i\times\cdots\times\phi_{i+m-1},\ \teps (u_2\,;(\phi_i,\eps_i),\dots, (\phi_{i+m-1},\eps_{i+m-1})),\ (v_i,\dots,v_{i+m-1}))
\end{split}
\], and $proj_2$ and $proj_{i,\dots,i+m-1}$ are the usual projections to the components indicated by subscripts. \par
 $\psi$ satisfies the following five conditions \psiS - \psiev \, for any set of indexes such that the involved notations make sense. Concretely speaking, the condition \psicomp\, is satisfied for any $n,m \geq 2,\ i,\ k_1,\dots, k_{n+m-1}\geq 0$ with $1\leq i\leq n$. The others, where $\D^n$ is abbreviation of $\D^n_{k_1,\dots, k_n}$, are satisfied for $n\geq 2, k_1,\dots, k_n\geq 0$.
\begin{enumerate}
\item[\psiS] For $t\in\R$, $i$ with $1\leq i\leq n$, and $(u\,;\ve_1,\dots,\ve_n)\in \D^n$, if $(u\,;\ve_1,\dots,t\times\ve_i,\dots,\ve_n)\in \D^n$, $\psi(u\,;\dots,\ve_i,\dots)=\psi(u\,;\dots,t\times\ve_i,\dots)$. Here $t\times \ve_i$ denotes the element $(0\times \phi_i,\eps_i,(t,v_i))$.
\item[\psiSigma] For $(u\, ;\ve_1,\dots,\ve_n)\in \D^n$ and $\sigma_i\in \Sigma_{k_i}$,
$\psi(u\, ;\sigma_1\cdot\ve_1,\dots,\sigma_n\cdot\ve_n)=\psi(u\,;\ve_1,\dots,\ve_n)$. Here, the $k$-th symmetric group acts on $\M_k$ by permutations of components.

\item[\psicomp] The following diagram commutes:
\[
\xymatrix{
\beta_i^{-1}\D^{n+m-1}_{k_1,\dots, k_{n+m-1}}\ar[rr]^{\alpha_i\times \gamma_i\qquad \qquad} \ar[dd]^{\beta_i} &&\D^n_{k_1,\dots, k_{(i)},\dots k_{n+m-1}}\times\D^m_{k_i,\dots, k_{i+m-1}}\ar[d]^{\psi\times\psi}\\
&&\Map(M,M)^n\times\Map(M,M)^m\ar[d]^{Comp_i}\\
 \D^{n+m-1}_{k_1,\dots, k_{n+m-1}}\ar[rr]^{\psi} &&\Map(M,M)^{n+m-1}.}
\]
Here, $k_{(i)}=k_i+\cdots +k_{i+m-1}$, and $Comp_i$ is defined as follows:
\[
Comp_i (g_1,\dots,g_n;f_1,\dots, f_m)=(g_1,\dots, g_{i-1}, g_i \circ f_1,\dots,g_i \circ f_m,g_{i+1},\dots,g_n).
\]
\item[\psibp] For $(u\,;\ve_1,\dots,\ve_n)\in \D^n$, 
$\psi_i(u\,;\ve_1,\dots,\ve_n)(\pi_{\phi_i}(v_i))=\pi_{\phi_1\times\cdots\times\phi_n}(v_1,\dots, v_n)$.
\item[\psiev] For $(u\,;\ve_1,\dots,\ve_n)\in \D^n$, $|\psi(u\,;\ve_1,\dots,\ve_n)(y)- y|\leq
6n^2d(v,\phi(M))$ where $v=(v_1,\dots,v_n)$, and $\phi=\phi_1\times\cdots\times\phi_n$, and $d(-,-)$ denotes the Euclid distance, and the minus and the norm in the left hand side are taken in $\R^{k_0}$, see subsection \ref{SScohenjones}.
\end{enumerate}
Note that by the conditions \epscomp and \epsev , we have
\[
\begin{split}
\alpha_i(\beta_i^{-1}\D^{n+m-1}_{k_1,\dots, k_{n+m-1}})&\subset \D^n_{k_1,\dots,k_{(i)},k_{i+m},\dots, k_{n+m-1}}\\
\gamma_i(\beta_i^{-1}\D^{n+m-1}_{k_1,\dots, k_{n+m-1}})&\subset \D^m_{k_i,\dots, k_{i+m-1}}.
\end{split}
\]
By these inclusions, the top horizontal map of \psicomp\  is well-defined. \\
\indent The role of first three and last conditions are similar to those on $\teps$. The forth condition ensures  loops perturbed by $\psi$ to have common basepoint as is seen for $n=2,3$ in sub-subsection \ref{SSSroughoutline}.
\begin{prop}\label{Pwelldefined}
If $\teps$ and $\psi$ satisfy the conditions \epsS, \epsSigma, \epscomp, \psiS, \psiSigma, \psicomp, and \psibp, the formulae (\ref{EtPhi}) and (\ref{Etc}) in sub-subsection \ref{SSSoutlineLMT} define an action of $\tB$ on $\LMT$.
\end{prop}
\begin{proof}
By \psibp,  $\tc$ is a well-defined loop (see sub-subsection \ref{SSSoutlineLMT}). By conditions \epsS\  and \psiS\  (resp.\epsSigma\  and \psiSigma)  the map $\Phi:(\tB (n)_+)\wedge \LMT_{k_1}\wedge \LMT_{k_2}\wedge\cdots\wedge \LMT_{k_n}\longrightarrow \LMT_{k_{\leq n}}$ commutes with the action of $\Sph$ (resp. $\Sigma_*$) so defines a morphism of symmetric spectra. The remaining  thing is compatibility of $\Phi$ and composition of $\tB$($=\B\times \K$). 
Let $(f,u)\in \B(n)\times \K(n)$ and $(g,w)\in \B(m)\times \K(m)$ be two elements. We must show
\[
\tc( f\circ_i g,u\circ_i w\,;\ci_1,\dots, \ci_n)=\tc(f,u\,;\ci_1,\dots,\ci_{i-1},\ci_{(i)},\ci_{i+m},\dots, \ci_{n+m-1}), 
\]
where $\ci_{(i)}=(\phi_i\times\cdots\times \phi_{i+m-1},\, \eps_{(i)},\, (v_i,\dots, v_{i+m-1}),\,  \tc(g, w\,;\ci_i,\dots,\ci_{i+m-1}))$ and we regard $\tc$ as a map $ \B(n')\times ev^*\D^{n'}\to LM$ ($n'\geq 2$). The compatibility with composition obviously follows from this equation and the condition \epscomp. To prove the equation, note that
\[
\begin{split}
t\in (f\circ_i g)^{-1}&\left[\frac{k-1}{n+m-1}, \frac{k}{n+m-1}\right] \\
&\iff
\left\{
\begin{array}{ll}
t\in f^{-1}[\frac{k-1}{n},\frac{k}{n}]& \text{ if }k\leq i-1\vs{2mm}\\
nf(t)-i+1\in g^{-1}
[\frac{k-i}{m},\frac{k-i+1}{m}]& \text{ if }i\leq k\leq i+m-1\vs{2mm}\\
f^{-1}[\frac{k-m}{n},\frac{k-m+1}{n}]& \text{ if }i+m\leq k\leq n+m-1
\end{array}\right.
\end{split}
\]
We shall show the equation in the  case of $i\leq k\leq i+m-1$. Proofs of the other cases are similar.
 By the above condition on $t$, we have $t\in f^{-1}[\frac{i-1}{n},\frac{i}{n}]$ so the right hand side of the above equation is equal to
\[
\begin{split}
\psi_i(u)\circ c_{(i)}(nf(t)-i+1)&=\psi_i(u)\circ\psi_{k-i+1}(w)(c_k[mg(nf(t)-i+1)-(k-i+1)+1])\\
&=\psi_k(u \circ_i w )(c_k((n+m-1)f\circ_ig(t)-k+1)) \qquad \text{(by \psicomp)}\\
&= \text{ the left hand side.}
\end{split}
\]
Here, for example, $\psi_i(u)$ is abbreviation of $\psi_i(u\,;\ve_1,\dots, \ve_{(i)},\dots\ve_{n+m-1})$, and see also the equation (\ref{EQparameter}) in subsection \ref{SSms}.
\end{proof}
\subsubsection{Construction of $\teps$ and $\psi$}\label{SSSconstructionLMT}
We first describe construction of $\teps$ and $\psi$, then verify well-definedness of the construction. \\

\textbf{Construction of $\teps$.}\quad 
We construct $\teps$ by induction on the arity $n$. When $n=2$, $\K(2)$ is a one-point set $\{pt\}$. We put
\[
\teps( pt;(\phi_1,\eps_1),(\phi_2,\eps_2))=\frac{1}{10^{|\phi_1\times\phi_2|^2}}\min\{\eps_1,\eps_2\}.
\]
Suppose $\teps$ is constructed up to $n-1$ and satisfies conditions \epsS\ -\epscone\ as far as they make sense. Let $B_V:\T(n)_{1,2}\longrightarrow \CG$ be the diagram given by
\[
B_V(T)=B_n(T)\times \TV[k_1,\dots,k_n]
\]
(see subsection \ref{SSK}). Let $\partial\K(n)$ denote the subspace of $\K(n)$ which is the union of  all faces of codimension one. To define $\teps$ on $\partial\K(n)\times\TV[k_1,\dots,k_n]$,
we define a natural transformation $\hat\eps:B_V\to (0,L_{e_0}/16)$, where $(0,L_{e_0}/16)$ denotes the constant functor taking the interval as its the value on each object. For an element $T$ of codimension one, we define $\hat\eps_T$ by the  equation  in \epscomp. In other words, if $T=T_1\circ_iT_2$, and the arity of $T_t$ is $n_t$,  we put 
\[
\begin{split}
\hat\eps_T(u,w\,;(\phi_1,\eps_1),\dots, (\phi_{n_1+n_2-1},\eps_{n_1+n_2-1}))&   
\\
=\teps(u\,;(\phi_1,\eps_1),\dots,(\phi_{i-1},\eps_{i-1}),(\phi',\eps')
&,(\phi_{i+n_2},\eps_{i+n_2}),\dots,(\phi_{n_1+n_2-1},\eps_{n_1+n_2-1}))
\end{split}
\]
for $u\in \K(n_1),\ w\in \K(n_2)$. If $T$ is of codimension two, we take an element $T'$ of codimension one such that $T\leq T'$, and define $\hat\eps_{T}$ to be the composition
\[
B_V(T)\to B_V(T')\xrightarrow{\hat\eps_{T'}} (0,L_{e_0}/16)
\]
Well-definedness of $\hat\eps$ is verified below.  We define $\teps$ on $\partial\K(n)\times\TV[k_1,\dots,k_n]$ as the composition 
\[
\partial\K(n)\times\TV[k_1,\dots,k_n] \cong \underset{\T(n)_{1,2}}{\colim} B_V\xrightarrow{\hat\eps}(0,L_e/16),
\]
where the homeomorphism $\cong $ is induced from the homeomorphism $\theta_n:\colim B_n\cong \partial\K(n)$. Recall that any point of $\K(n)$ is of the form $tu$ for some $t\in [0,1]$ and $u\in \partial\K(n)$. Using $\teps|_{\partial\K(n)\times\TV[k_1,\dots,k_n]}$, we put
\[
\begin{split}
\teps(tu\,;(\phi_1,\eps_1),\dots,(\phi_n,\eps_n))
&
\\
=
(1-t)\min_{w\in\partial K(n)}\{\teps
&
(w;(\phi_1,\eps_1)
,\dots,(\phi_n,\eps_n))\}
+t\teps(u\,;(\phi_1,\eps_1),\dots,(\phi_n,\eps_n)).\qed
\end{split}
\]
 
\textbf{Construction of $\psi$.}\quad 
We also construct $\psi$ by induction on the arity $n$. When $n=2$, we put 
\[
\psi_i(pt;\ve_1,\ve_2)(y)
=\pi_{\phi}(\widehat{\psi}_i(pt;\ve_1, \ve_2)(y)), \ 
\text{where } \widehat{\psi}_i(pt;\ve_1,  \ve_2)(y)= v +\phi (y)-\phi(\pi_{\phi_i}(v_i)).
\]
Here, we set $\phi=\phi_1\times \phi_2$, $v=(v_1,v_2)$, and $\pm$ is taken in $\R^{k_1+k_2}$. This is the same as the one given in sub-subsection \ref{SSSroughoutline}. For later use, we introduced the notation $\widehat{\psi}$. \\

 \indent Suppose $\psi$ is constructed up to arity $n-1$ and satisfies the conditions \psiS\ -\psiev\ as far as they make sense. 
First, we shall construct $\psi$ on
\[
\delta\D^n=\{(u\,;\ve_1,\dots,\ve_n)\in\D^n\mid u\in\partial\K(n)\},
\]
where $\D^n=\D^n_{k_1,\dots,k_n}$. We define a diagram $B_{\D}:\T(n)_{1,2}\longrightarrow \CG$ by
\[
B_{\D}(T_1\circ_iT_2)=\beta_{i}^{-1}\D^n,\quad
B_{\D}((S_1\circ_jS_2)\circ_kS_3)=\beta_{j,k}^{-1}\D^n.
\]
Here, $T_1\circ_iT_2$ denotes an element of codimension one and $\beta_i$ is the map given in the statement of conditions on $\psi$. (So if $\ari(T_l)=n_l$, we have $n_1+n_2-1=n$ and the domain of $\beta_i$ is $\K(n_1)\times \K(n_2)\times \M[k_1,\dots k_n]$.) $(S_1\circ_jS_2)\circ_kS_3$ denotes an element of codimension two and $\beta_{j,k}$ is defined as follows:
\[
\begin{split}
\beta_{j,k}
:=&((-\circ_j-)\circ_k-)\times \id:\\
&\K (m_1)\times\K(m_2)\times\K(m_3)\times\M[k_1,\dots, k_{m_1+m_2+m_3-2}]\\
&\longrightarrow \K(m_1+m_2+m_3-2)\times\M[k_1,\dots, k_{m_1+m_2+m_3-2}]
\end{split}
\]
where $m_t$ is the arity of $S_t$ (so $m_1+m_2+m_3-2=n$). 
On morphisms, $B_{\D}$ is defined by the exactly  same way as $B_n$. Clearly $\theta_n$ in subsection \ref{SSK} induces a homeomorphism $\colim B_\D\cong\delta\D^n$. Similarly to the construction of $\teps$, we shall define a 
natural transformation 
\[
\bpsi:B_\D\longrightarrow \Map(M,M)^n.
\]
Let $T$ be an element of $\T(n)_{1,2}$. If $T$ is of codimension one, we can write $T=T_1\circ_iT_2$  uniquely. With this expression,  $\bpsi_T$ is the composition through the top right angle of the square in the  condition \psicomp, $Comp_i\circ (\psi\times \psi)\circ (\alpha_i\times \gamma_i)$, where  $n$ and $m$ in the condition are replaced with $\ari (T_1)$ and $\ari (T_2)$ respectively.  If $T$ is of codimension two, we take an element $T'\geq T$ of codimension one, and define $\bpsi_{T}$ to be the composition $B_{\D}(T)\to B_{\D}(T')\stackrel{\bpsi_{T'}}{\to}\Map(M,M)^n$. 
Well-definedness of $\bpsi$ is verified bellow. Define $\psi$ on $\delta\D^n$ as the composition
$\delta\D^n\cong \colim B_\D\xrightarrow{\bpsi} \Map(M,M)^n$\\
\indent Let $tu$ denote the element of $\K(n)$ with $t\in [0,1]$ and $u\in \partial\K(n)$. Using $\psi|_{\delta\D^n}$, we put
\begin{equation}\label{EQdefpsi}
\begin{split}
\psi_i(tu\,;\ve_1,\dots, \ve_n)(y)
&=\pi_{\phi}(\widehat{\psi}_i(tu\,;\ve_1,\dots, \ve_n)(y))\\
\text{where } \widehat{\psi}_i(tu\,;\ve_1,\dots, 
& \ve_n)(y)= (1-t)\{v +\phi (y)-\phi(\pi_{\phi_i}(v_i))\}
+t\{\phi(\psi_i(u\,;\ve_1,\dots,\ve_n)(y))\}.
\end{split}
\end{equation}
Here, $\phi=\phi_1\times\cdots\times \phi_n$, $v=(v_1,\dots,v_n)$, and sum and scalar multiplication are taken in $\R^{k_{\leq n}}$, see Figure \ref{Fpsi}. (We will verify the image of  $\widehat{\psi}_i$ belongs to the domain of $\pi_\phi$ below.)
\qed\\
\begin{figure}
\begin{center}
\begin{picture}(300,250)(-150,0)
\qbezier(-150,170)(0,240)(150,170)
\put(0,230){\circle*{4}}
\qbezier(0,230)(-130,100)(-130,100)
\qbezier(0,230)(130,100)(130,100)
\qbezier(0,230)(0,180)(0,180)
\qbezier(0,230)(50,30)(50,30)
\qbezier(0,230)(-50,30)(-50,30)
\put(-70,160){\circle*{4}}
\qbezier(-70,160)(-70,100)(-70,100)
\put(-70,100){\circle*{4}}
\put(-75,100){\framebox(5,5){}}
\qbezier(-80,100)(-70,100)(-70,100)
\put(-85,85){$\psi_i(tu)(y)$}
\put(-155,85){$\psi_i(u)(y)$}
\put(-100,170){$\widehat{\psi}_i(tu)(y)$}
\put(-50,190){$t$}
\put(-120,140){$1-t$}
\put(-35,240){$v+\phi(y-\pi_{\phi_i}(v_i))$}
\put(130,150){$\phi(M)$}
\put(120,200){$\R^{k_{\leq n}}$}
\put(-15,5){$\psi(\partial\K(4))(y)$}
\thicklines
\qbezier(-50,30)(0,40)(50,30)
\qbezier(-50,30)(-80,70)(-130,100)
\qbezier(50,30)(80,70)(130,100)
\qbezier(-130,100)(-65,150)(0,180)
\qbezier(130,100)(65,150)(0,180)
\end{picture}
\end{center}
\caption{construction of $\psi$ for $n=4$ ($\psi_i(u)(y)$ is the abbreviation of $\psi_i(u;\ve_1,\dots,\ve_4)(y)$, $\psi_i(tu)(y)$ and $\widehat{\psi}_i(tu)(y)$ are similar abbreviations.)}\label{Fpsi}
\end{figure}\\
\textbf{Verification on construction of $\teps$ and $\psi$.}\quad (On $\teps$)   
We shall verify the well-definedness of the natural transformation $\hat\eps$. 
 An element  $T$ of codimension two is presented as $(S_1\circ_jS_2)\circ_k S_3$ for unique 5-tuple $(S_1,S_2,S_3,j,k)$ such that $j\leq k$. Let $T'=T_1\circ_iT_2$ be an element of codimension one such that $T\leq T'$. We must show the map $\hat\eps_T$ is independent of a choice of $T'$.  We shall consider the case of  $ k\leq j+m_2-1$ ($m_t$ is the arity of $S_t$). In this case there are exactly two possibilities: (i) $S_1\circ_jS_2< T_1$, $S_3=T_2$ and $k=i$, or (ii)$S_1=T_1$, $S_2\circ_{k-j+1}S_3< T_2$, and $j=i$. The values $\hat\eps_T$ at $(u_1,u_2,u_3\,;(\phi_1,\eps_1),\dots, (\phi_n,\eps_n))$ corresponding to (i) and (ii) are
\[
\begin{split}
\teps(u_1&\circ_ju_2\, ;(\phi_1,\eps_1),\dots, (\phi_k\times\cdots\times\phi_{k+m_3-1},\teps(u_3\,;(\phi_k,\eps_k),\dots,(\phi_{k+m_3-1},\eps_{k+m_3-1})),\dots) \quad \text{and}\\
\teps(u_1\,&;\dots,( \phi_j\times\cdots\times\phi_{j+m_2+m_3-2},\teps(u_2\circ_{k-j+1}u_3\,;(\phi_j,\eps_j),\dots, (\phi_{j+m_2+m_3-2},\eps_{j+m_2+m_3-2})),
\dots)
\end{split}
\] respectively . As $m_t\leq n-2$, using the inductive hypothesis concerning \epscomp, we can obviously verify these two 
values are equal. We can also verify the case of $k> j+m_2$ similarly. \\
\indent The condition \epscomp\ is satisfied by the assertion we just have verified, and verification of the remaining conditions is straightforward.\\
\indent (On $\psi$)  We shall verify the natural transformation $\bpsi$ is well-defined.  This is almost clear from the definition of $\bpsi$ and \epscomp we have already verified but we give a formal proof to make sure. Let $T=(S_1\circ_jS_2)\circ_kS_3$  be an element of codimension  two. We must verify the map $\bpsi_T$ is independent of a choice of an element $T'=T_1\circ_iT_2$ of codimension one with $T'>T$.We shall consider the case of $k\leq j+m_2-1$.  The maps $\bpsi_T$ corresponding to the above two choices (i),(ii) in the verification on $\teps$  are compositions
\[
\begin{split}
\beta_{j,k}^{-1}\D^n
&
\xrightarrow{\beta_j\times\id}\beta_k^{-1}\D^n\xrightarrow{\alpha_k\times\gamma_k}\D^{m_1+m_2-1}\times\D^{m_3}\xrightarrow{Comp_k\circ(\psi^2)}\Map(M,M)^n \ \text{and}
\\
\beta_{j,k}^{-1}\D^n
&
\xrightarrow{\id\times\beta_{k-j+1}}
\beta_j^{-1}\D^n\xrightarrow{\alpha_j\times\gamma_j}
\D^{m_1}\times\D^{m_2+m_3-1}
\xrightarrow{Comp_j\circ(\psi^2)}\Map(M,M)^n,
\end{split}
\]
respectively.
These two maps fit into the top-right and bottom-left corners of the following diagram:
\[
\xymatrix{
\beta^{-1}_{j,k}\D^n\ar[rr]^{\beta_j\times\id}
\ar[rrd]^{\alpha'\times\gamma'}
\ar[dd]_{\id\times\beta_{k-j+1}}
\ar[rdd]^{\alpha''\times\gamma''}
&&
\beta^{-1}_k\D^n
\ar[r]^{\alpha_k\times\gamma_k}
\ar@{}[d]^{\text{(A)}}
&
\D^{m_{12}}\times \D^{m_3}\ar[ddd]^{\psi^2}
\\
&&
\beta^{-1}_j\D^{m_{12}}\times\D^{m_3}
\ar[ur]^{\beta_j\times\id}
\ar[d]^{\alpha_j\times\gamma_j\times\id}
\ar@{}[ld]_{\text{(B)}}
&
\\
\beta^{-1}_j\D^n
\ar[dd]_{\alpha_j\times\gamma_j}
\ar@{}[r]^{\text{(C)}}
&
\D^{m_1}\times \beta^{-1}_k\D^{m_{23}}
\ar[ldd]^{\id\times\beta_k}
\ar[r]_{\id\times\alpha\times\gamma_{k-j+1}}
&
\D^{m_1}\times\D^{m_2}\times\D^{m_3}
\ar[d]^{\psi^3}
\ar@{}[rd]^{\text{(E)}}
&
\\
&
&
\Map^{m_1}\times\Map^{m_2}\times\Map^{m_3}
\ar[d]_{\id\times C_{k-j+1}}
\ar[r]_{C_j\times \id}
&
\Map^{m_{12}}\times\Map^{m_3}
\ar[d]^{C_k}
\\
\D^{m_1}\times\D^{m_{23}}
\ar@{}[rru]^{\text{(D)}}
\ar[rr]_{\psi^2}
&&
\Map^{m_1}\times\Map^{m_{23}}
\ar[r]_{C_j}
\ar@{}[ur]_{\text{(F)}}
&
\Map^{n}.
}
\]
Here, 
\begin{itemize}
\item
$\Map$, $m_{st}$, and $C_l$ are abbreviations of $\Map(M,M)$, $m_s+m_t-1$, and $Comp_l$, and
\item the maps $\alpha'$, $\gamma'$, $\alpha''$,  and$\gamma''$ are defined by the following formulae: \\
$\alpha'(u_1,u_2,u_3\,;\ve_1,\dots\ve_n)=(u_1,u_2\,;\ve_1,\dots,\ve_{(k)},\dots,\ve_n)$,  
$\gamma'(u_1,u_2,u_3\,;\dots)=(u_3\,;\ve_k,\dots\ve_{k+m_3-1})$, \\
$\alpha''(u_1,u_2,u_3\,;\ve_1,\dots,\ve_n)=(u_1\,;\ve_1\dots, \ve_{(j)},
\dots \ve_n)$, \\
and $\gamma''(u_1,u_2,u_3\,;\ve_1,\dots,\ve_n)=(u_2,u_3\,;\ve_j,\dots,\ve_{j+m_2+m_3-2})$,
where $\ve_{(k)}$ (resp. $\ve_{(j)}$) is the element defined by the same formula as $\ve_{(i)}$ in the definition of $\alpha_i$ from $(u_3\,;\ve_k,\dots,\ve_{k+m_3-1})$  (resp.$(u_2\circ_{k-j+1}u_3\,;\ve_j,\dots \ve_{j+m_2+m_3-2})$).
\end{itemize}
Commutativity of this diagram is verified as follows. 
Commutativity of the triangles (A), (C) and square (B) follows from the following obvious equations.
\[
\alpha\circ \alpha'=\alpha'',\quad \gamma'=\gamma_{k-j+1}\circ\gamma'',\quad \gamma_j\circ \alpha'=\alpha\circ \gamma''.
\]
  Commutativity of the squares (D) and (E) follows from the  inductive hypothesis concerning \psicomp\  as $m_s, m_{st}\leq n-2$, and commutativity of (F) is clear.   Thus, we have verified the map $\bpsi_T$ is independent of choices of $T'$ for the case $k\leq j+m_2-1$. For the case $k\geq j+m_2$, we can verify the well-definedness of $\psi_T$ by a similar way.\\
\indent  
To show $\psi_i$ is well-defined on whole $\D^n$, we must show  values of $\widehat{\psi}_i$ belong to the tubular neighborhood $\nu_\eps(\phi)$ for some $\eps<L_{\phi\circ e_0}(=|\phi|L_{e_0})$. Otherwise, we cannot compose the projection $\pi_\phi$ with $\widehat{\psi}$. This is the point where we need the condition \psiev.  We must show 
$
d(\widehat{\psi}_i(tu\,;\ve_1,\dots,\ve_n)(y),\phi(M))< |\phi|\cdot L_{e_0}
$ for any element $(tu\,;\ve_1,\dots,\ve_n)\in\D^n$, and $y\in M$. By the defining equation (\ref{EQdefpsi}), we have
\[
|\widehat{\psi}_i(tu)(y)-\phi(y)|
\leq
(1-t)|v-\pi_{\phi_i}(v_i)|+t|\phi|\cdot |\psi_i(u)(y)-y|.
\]
Here, $\phi=\phi_1\times\cdots\times \phi_n$, $v=(v_1,\dots,v_n)$, and for example, $\widehat{\psi}_i(tu)(y)$ is an abbreviation of $\widehat{\psi}_i(tu\,;\ve_1,\dots,\ve_n)(y)$. We can write $u=u_1\circ_ju_2$ for some $u_1$, $u_2$ with $n_1=\ari 
u_1\geq 2$, $n_2=\ari u_2\geq 2$. By inductive hypothesis, we have
\[
\begin{split}
|\psi_i(u)(y)-y|
&\leq|\psi_{i_1}(u_1)(\psi_{i_2}(u_2)(y))-\psi_{i_2}(u_2)(y)+\psi_{i_2}(u_2)(y)-y|\\
&\leq (6n_1^2+ 6n_2^2)d(v,\phi(M)).
\end{split}
\]
($i_1$, $i_2$ are suitable integers.) We also have
\[
\begin{split}
|v-\phi\pi_{\phi_i}(v_i)|
&
\leq |v-\phi\pi_{\phi}(v)|+|\phi\pi_{\phi}(v)-\phi\pi_{\phi_i}(v_i)|
\\
&=d(v,\phi(M))+|\phi|\cdot|\pi_{\phi}(v)-\pi_{\phi_i}(v_i)|
\\
&
\leq d(v,\phi(M))+\frac{|\phi|}{|\phi_i|}
(|\phi_i\pi_{\phi}(v)-v_i|+|v_i-\phi_i\pi_{\phi_i}(v_i)|)
\\
&
\leq (1+2|\phi|)d(v,\phi(M))\quad (\because |\phi_i(\pi_\phi(v))-v_i|\leq |\phi(\pi_\phi(v))-v|)
\end{split}
\]
So we have 
\[
\begin{split}
d(\widehat{\psi}_i(tu\,;\ve_1,\dots,\ve_n)(y),\phi(M))
&
\leq
|\widehat{\psi}_i(tu)(y)-\phi (y)|
\\
&
\leq \max\{6(n_1^2+n_2^2)|\phi|, 1+2|\phi|\}\cdot d(v,\phi(M))
\\
&
\leq \frac{\max\{6(n_1^2+n_2^2)|\phi|, 1+2|\phi|\}}{10^{|\phi|^2}}\min\{\eps_1,\dots,\eps_n\} \quad\text{(by \epsev)}
\\
&
\leq |\phi|L_{e_0}.
\end{split}
\]
\psicomp\ is satisfied by well-definedness of $\bpsi$ which we have already verified. \psiS \ holds since the $0$-constant component of the linear map does not affect the position of the closest point, see subsection \ref{SSproblem}.  Verification of \psiSigma\  is trivial. Verification of \psibp\ is completely analogous to verification of this condition for the case $n=3$ done in sub-subsection \ref{SSSroughoutline}, and we omit it.  We shall verify  the condition \psiev. By the equation (\ref{EQdefpsi}), we have
\[
|\widehat{\psi}_i(tu)(y)-\phi(\psi_i(u)(y))|
\leq (1-t)|v-\phi\pi_{\phi_i}(v_i)|+(1-t)|\phi |\cdot 
|\psi_i(u)(y)-y|.
\]
So with the above inequalities, we have
\[
\begin{split}
|\psi_i(tu)(y)-y|
&\leq \frac{1}{|\phi |}(|\widehat{\psi}_i(tu)(y)-\phi (y)|+|\widehat{\psi}_i(tu)(y)-\phi(\psi_i(tu)(y))|)
\\
&\leq \frac{1}{|\phi |}(|\widehat{\psi}_i(tu)(y)-\phi (y)|+|\widehat{\psi}_i(tu)(y)-\phi(\psi_i(u)(y))|)\\
&
\leq\frac{1}{|\phi|}
(2(1-t)|v-\phi\pi_{\phi_i}(v_i)|+|\phi|\cdot|\psi_i(u)(y)-y|)
\\
&\leq\frac{1}{|\phi|}(2(1+2|\phi|)+6(n_1^2+n_2^2))d(v,\phi(M))\\
&\leq (6+6n_1^2+ 6n_2^2)\cdot d(v,\phi(M)).
\end{split}
\]
Here, we use the fact that $\phi(\psi_i(tu)(y))$ is the closest to $\widehat{\psi}_i(tu)(y)$ in $\phi(M)$ in the second inequality.
As $n_1,n_2\geq 2$ and $n=n_1+n_2-1$,
\[
6n^2-(6+6n_1^2+ 6n_2^2)=-12n_1^2+12n\cdot n_1+12n_1-12n-8=12(n_1-1)(n-n_1)-12 \geq 0. 
\]
Thus, we have completed the construction of an action of $\tB$ on $\LMT$. \qed

\subsection{Comparison with Cohen-Jones spectrum}\label{SScomparison}
Note that the classical stable homotopy category  is equivalent to the homotopy category of symmetric spectra $\Ho( \SP)$ as symmetric monoidal category, so the Cohen-Jones ring spectrum $\LMTM$ is considered as a monoid in $\Ho(\SP)$.
\begin{thm}\label{TLMTLMTM}
If we consider $\LMT$ as a semigroup in the homotopy category $\Ho(\SP)$, it is isomorphic to the Cohen-Jones ring spectrum. (A semigroup in a monoidal category is an object equipped with an associative product (but without a unit))
\end{thm}
\begin{proof}
Define a map
$
\alpha: E^{-1}(\bnu_{\eps_1}(\phi_0\times\phi_0))\longrightarrow ev_\infty^*\bnu_{\eps_1}(\phi_0\times\phi_0)
$  as follows: (For the definitions of $E$ and $ev_\infty$, see subsection \ref{SScohenjones})
\[
\alpha((c_1,v_1),(c_2,v_2))=((\tc_1,\tc_2), (v_1,v_2)),\quad  
\tc_i(t)=
\psi_i((\phi_0,v_1),(\phi_0,v_2))(c_i(t))\quad  (i=1,2).
\]
Here, $\psi_i$ is the one constructed in sub-subsection \ref{SSSconstructionLMT} and note that the formula of $\psi_i$ does not depends on $\eps$ so we omit it here.
We shall prove $\alpha$ is homeomorphism for sufficiently small $\eps_1$. 
First we show $\alpha$ is injective. Let $((c_1,v_1),(c_2,v_2))$ and $((c_1',v_1'), (c_2',v_2'))$ be two elements of $E^{-1}(\bnu_{\eps_1}(\phi))$ ($\phi=\phi_0\times\phi_0$). If $\alpha((c_1,v_1),(c_2,v_2))=\alpha((c_1',v_1'), (c_2',v_2'))$, we have  $(v_1,v_2)=(v'_1,v'_2)$, and
$\widehat{\psi}_i((\phi_0,v_1),(\phi_0,v_2))(c_i(t))$ and $\widehat{\psi}_i((\phi_0,v'_1),(\phi_0,v'_2))(c_i'(t))$ belong to the same fiber of the tubular neighborhood of $\phi(M)$ for each $t\in [0,1]$, and $i=1,2$ (see the construction of $\psi$ in sub-subsection \ref{SSSconstructionLMT}).
It follows that
\[
\phi(c_i(t)-c'_i(t))\in T_{\phi(\tc_i(t))}(\phi(M))^{\perp}.
\] If $\eps_1$ is sufficiently small, three points $c_i(t), c'_i(t)$ and $\tc_i(t)$ belong to sufficiently small open set of $M$ which is approximated by the affine space 
$\phi(\tc_i(t))+T_{\phi(\tc_i(T))}\phi(M)$, so the above relation can not occur if $c_i(t)\not=c_i'(t)$. This implies $\alpha$ is injective. \par
Let $(\bar c_1,\bar c_2,(v_1,v_2))$ be an element of $ev_\infty^*\bnu_{\eps_1}(\phi)$. We shall show there exists a unique element $c_i(t)\in B_{10\eps_1}(\phi(\bar c_i(t))\cap \phi(M)$  such that $\widehat{\psi}_i((\phi_0,v_0),(\phi_0,v_0))(c_i(t))$ belongs to the fiber over  $\phi(\bar c_i(t))$ of the tubular neighborhood of $\phi(M)$ for each $i$ and $t$, 
 where $B_{10\eps_1}(\phi(\bar c_i(t)))$ is the neighborhood of  $\phi(\bar c_i(t))$ with radius $10\eps_1$. If $\eps_1$ is sufficiently small, the two points $(v_1,v_2)$ and  $\phi(\pi_{\phi_0}(v_i))$ are sufficiently close, and the set 
$B_{10\eps_1}(\phi(\bar c_i(t)))\cap \phi(M)$ is approximated by the affine space $\phi(\bar c_i(t))+T_{\phi(\bar c_i(t))}\phi(M)$. So the intersection \[
[\phi(\pi_{\phi_0}(v_i))-(v_1,v_2)+\phi(\bar c_i(t))+T_{\phi(\bar c_i(t))}\phi(M)^{\perp}]\cap\phi(M)\cap B_{10\eps_1}(\phi(\bar c_i(t)))
\] 
consists of exactly one point. We denote this point by $\phi(c_i(t))$, which is easily seen to satisfy the above condition.  It is easy to see $c_i(t)$  continuously depends on $t$. Thus we obtain an element $((c_1,v_1),(c_2,v_2))\in E^{-1}(\bnu_{\eps_1}(\phi_0\times\phi_0))$ with $\alpha((c_1,v_1),(c_2,v_2))=(\bar c_1,\bar c_2,(v_1,v_2))$.  It is also easy to see $c_i$  continuously depends on $\bar c_i$. So the left inverse $(\bar c_1,\bar c_2, (v_1,v_2))\mapsto ((c_1,v_1),(c_2,v_2))$ is continuous  and as $\alpha$ is injective, this is the continuous inverse of $\alpha$ .\par
 This homeomorphism $\alpha$ makes the diagram in Lemma \ref{Llooptube}  commutative so by the same lemma, we can use this homeomorphism in the definition of the Cohen-Jones ring spectrum. 
An isomorphism $\LMTM\cong \LMT$ in the homotopy category of spectra is given by
$Th(ev^*\nu(\phi_0))\to \LMT_N ,\ (c,v)\mapsto (\phi_0,\eps_0,c,v).$
So what we have to prove is homotopy commutativity of the following diagram:
\[
\xymatrix{
Th(ev^*\nu(\phi_0)\times ev^*\nu(\phi_0))\ar[r]\ar[d]& \LMT_N\wedge\LMT_N\ar[d] \\
S^NTh(ev^*\nu(\phi_0))\ar[r]&\LMT_{2N}
}
\]
But we can easily construct a homotopy for the commutativity  using a linear isotopy $\R^{2N}\times I\to \R^{2N}$ covering an isotopy $\R^N\times I\to \R^{2N}$ between the diagonal and $0\times \id$.
\end{proof}

\section{Equivalence with topological Hochschild cohomology}\label{Sequivalence}
\subsection{Outline}\label{SSoutlineEQ}
To prove Theorem \ref{mainthm1} (2),  it is enough to prove $\LMT$ and $\thc(A',B)$ are weak equivalent as nu-$A_\infty$-ring spectra (see Proposition \ref{propinvarianceofthc} for the definition of $A'$ and $B$). We will prove this claim by constructing a zig-zag of termwise weak equivalences between two cosimplicial objects whose totalizations are isomorphic to $\LMT$ and $\thc(A',B)$. 
In subsection \ref{SScosimplicialLMT},  we define a cs-spectrum $\CL^\bullet$  with an action of  the monad $\MK$  such that the totalization $\tot\CL^\bullet$ is isomorphic to $\LMT$ as $\tB$-algebras (see sub-subsection \ref{SSSslightgenofms} for $\MK$ and $\tB$).  We want to connect two cs-spectra $\CL^\bullet$ and $\thc^\bullet(A',B)$ by a zig-zag of termwise stable equivalences which preserve  actions of $\MK$, but this turns out  difficult (see sub-subsection \ref{SSSmotivation} for an explanation). To avoid this difficulty, we define a monad $\MCK$ over cosimplicial objects which is considered as a "up to homotoy coherency version" of $\MK$ and connect $\CL^\bullet$  and $\thc^\bullet (A',B)$ by a zig-zag of termwise stable equivalences which preserve  actions of $\MCK$ as follows:
\[
\CL^\bullet \xleftarrow{p_0} \IM^\bullet \xrightarrow{\bar q_2} \thc^\bullet (A',B),
\] 
where the actions of $\MCK$ on the left and right are induced by a natural morphism $U:\MCK\to \MK$ of monads, see subsections \ref{SSintermediate} and \ref{SSMCKalgebra}. In sub-subsection \ref{SSSCB}, we prove an action of $\MCK$ induces an action of an $A_\infty$-operad on $\ttot$, but we cannot prove it does on $\tot$. So we must prove $A_\infty$-structures on $\tot$ and  $\ttot$ both induced by an action of $\MK$ are equivalent, and this is done in subsections \ref{SScomparisontot} and \ref{SSproofTmain2}. 
 The remaining thing we have to prove  is that the two morphisms $\IM^\bullet\to \CL^\bullet$ and $\IM^\bullet\to \thc^\bullet (A',B)$ induce weak equivalences on $\ttot$. for the first morphism, this is clear as it is a termwise levelwise weak equivalence. For the latter morphism, we need more care and prove it in subsection \ref{SSproofTmain2} (note that $\ttot$ is not always homotopy invariant for symmetric spectra as the model category $\SP$ is not fibrant).
\subsection{Cosimplicial model of $\LMT$}\label{SScosimplicialLMT}
We shall define a cs-spectrum $\CL^\bullet$ such that $\tot(\CL^\bullet)$ is isomorphic to $\LMT$. We put
\[
\CL^p=(M^{\times p})\hotimes \MT (=\{ M^{\times p}_+\wedge \MT_k\}_k).
\]
(See subsection \ref{SSatiyahduality} for $\MT$.) If $1\leq i\leq p$, the coface morphism $d^i:\CL^p\to \CL^{p+1}$ is defined as the morphism repeating the $i$-th component of $M^{\times p}$, and $d^0$ and $d^{p+1}$ are defined as the ones taking an element $(x_1,\dots, x_p, \langle \phi,\eps, v\rangle)$ to $(\pi_\phi(v), x_1,\dots, x_p, \langle \phi,\eps, v\rangle)$ and $(x_1,\dots, x_p, \pi_\phi(v), \langle \phi,\eps, v\rangle)$ respectively. 
The codegeneracy morphism $s^i:\CL^p\to\CL^{p-1}$ is the one skipping the $i+1$-th component for each $i=0,\dots,p-1$, $(x_1,\dots, x_p, \langle \phi,\eps, v\rangle)\mapsto (x_1,\dots, x_i,x_{i+2},\dots, x_p, \langle \phi,\eps, v\rangle)$. \par
 Let $\Delta^p=\{(t_1,\dots, t_p)\in\R^p\mid 0\leq t_1\leq\cdots\leq t_p\leq 1\}$ be the standard $p$-simplex for $p\geq 0$. 
We define a morphism 
\[
\varphi_p :\Delta^p\hotimes \LMT\to \CL^p\  \text{by} \ 
[(t_1,\dots, t_p),\langle \phi,\eps,v, c\rangle ]\mapsto (c(t_1),\dots c(t_p),\langle \phi,\eps,v \rangle ).
\] As the collection $\{\varphi_p\}_{p\geq 0}$ forms a morphism between cosimplicial spectra, where $\{\Delta^p\hotimes \LMT\}_p$ has the structure induced from $\Delta^\bullet$, taking the adjoint we obtain a morphism $\varphi:\LMT\to\tot(\CL^\bullet)\in\SP$. More explicitly, recall that $\tot(\CL^\bullet)$ is a subspectrum of $\prod_{p\geq 0}(\CL^p)^{\Delta^p}$ by definition.  $\varphi_p$ induces a morphism $\LMT\to (\CL^p)^{\Delta^p}$ for each $p$ by adjoint.  These morphisms for various $p$ induce the morphism to the totalization by compatibility with cofaces and codegeneracies. 
\begin{prop}\label{PisoLMT}
Under the above notations, $\varphi:\LMT\to \tot\CL^\bullet $ is an isomorphism in $\SP$ (not only in $\Ho(\SP)$).\qed
\end{prop}
\begin{proof}
This is proved completely analogously to the corresponding statement for the well-known cosimplicial model of a free loop space by point-set level verification so we only sketch the outline. For an element $\alpha \in \tot\CL^\bullet_k$, take the component $\alpha_1\in (\CL^1)_k^{\Delta^1}=\Map_{\CG}(\Delta^1, (M_+)\wedge\MT_k)$. The compatibility with a codegeneracy implies the component in $\MT_k$ of $\alpha^1$ takes a constant value, say $(\phi, \eps, v)$. The compatibility with cofaces implies the components of $\alpha^1(0)$ and $\alpha^1(1)$ in $M$ are $\pi_\phi(v)$. By these observation, $\alpha^1$ gives an element of $\LMT_k$. This procedure gives the inverse of $\varphi$.
\end{proof}
We shall define an action of $\MK$ on $\CL^\bullet$. We define a morphism $\Psi:\MK(\CL^\bullet)\longrightarrow \CL^\bullet$ by
\[
\begin{split}
\Psi&(u\,;(x_{11},\dots, x_{1p_1},\langle \phi_1 ,\eps_1 ,v_1\rangle ),\dots (x_{n,1},\dots, x_{n,p_n}, \langle \phi_n,\eps_n, v_n\rangle)\\
&=(\psi_1(x_{11}),\dots, \psi_1(x_{1p_1}),\dots, \psi_n(x_{n,1}),\dots, \psi_n (x_{n,p_n}),\langle \phi_1\times\cdots\times\phi_n,\teps, (v_1,\dots ,v_n)\rangle)
\end{split}
\]
for $n\geq 2$, where $\psi_i$ and $\teps$ mean $\psi_i(u\,;(\phi_1,\eps_1,v_1),\dots, (\phi_n,\eps_n,v_n))$ and $\teps(u\,;(\phi_1,\eps_1),\dots,(\phi_n,\eps_n))$  respectively, see subsection \ref{SSactiononLMT} for $\psi$ and $\teps$, and by the identity on the arity $1$ part.
\begin{prop}\label{PwelldefinedPsi}
$\Psi$ is an action of $\MK$.
\end{prop}
\begin{proof}
$\Psi$ actually factors through the quotient $\MK(\CL^\bullet)$ because  $\psi_i(\pi_{\phi_i}v_i)=\pi_{\phi}(v_1,\dots,v_n)$ for each $i$ (see the condition \psibp \ in sub-section \ref{SSSconditionsLMT}).
Verification of compatibility of $\Psi$ with coface and codegeneracy morphisms is a routine work. (Compatibility with the first and last coface morphisms also follows from the identity $\psi_i(\pi_{\phi_i}v_i)=\pi_{\phi}(v_1,\dots,v_n)$. )
Compatibility with the product of the monad follows from the condition \psicomp \ stated in sub-subsection \ref{SSSconditionsLMT}.
\end{proof}

\begin{prop}\label{PcompatibilityLMT}
The isomorphism $\varphi:\LMT\to \tot\CL^\bullet$ in Proposition \ref{PisoLMT} is compatible with the actions of $\tB$. Here, the action on $\LMT$ is the one defined in subsection \ref{SSactiononLMT} and the action on $\tot\CL^\bullet$ is the one induced by $\Psi$ via Proposition \ref{PtB}.
\end{prop}
\begin{proof}
We have to prove the following diagram is commutative
\[
\xymatrix{
\tB(n)\hotimes (\LMT)^{\otimes n}\ar[r]\ar[d]&\LMT\ar[d]\\
\tB(n)\hotimes(\tot\CL^\bullet)^{\otimes n}\ar[r]& \tot\CL^\bullet
}
\]
Let $(f,u)\in\tB(n)=\B(n)\times \K(n)$, $\ci_1,\dots,\ci_n\in\LMT$, and $(t_1,\dots t_k)\in\Delta^k$ be elements. Let 
\[
[(t_{11},\dots, t_{1,k_1}),\dots,(t_{n,1},\dots, t_{n,k_n})]\in(\Delta^\bullet\, \square \,\cdots\, \square \,\Delta^\bullet)^k
\] be a representative of $f(t_1,\dots t_k)$ 
. Here, the superscript $k$ denotes the cosimplicial degree. The image of $((f,u)\,;\ci_1,\dots,\ci_n)$ by the left-bottom corner composition takes $(t_1,\dots,t_k)$ to 
\[
\psi_1(c_1(t_{11})),\dots,\psi_1(c_1(t_{1,k_1})),\dots, \psi_n(c_n(t_{n,1})),\dots, \psi_n(c_n (t_{n,k_n})),\phi_1\times\cdots\times \phi_n,\teps, (v_1,\dots, v_n)).
\]
Since 
$(i-1+t_{i,l})/n=\ul{f}(t_L)$, where $L=\sum_{j=1}^{i-1}k_j+l$, for each $1\leq i\leq n$ and $ 1\leq l\leq k_i$, 
by the definition of $\tilde c$ in sub-subsection \ref{SSSoutlineLMT}, the above value is equal to the image by right-top composition. 
\end{proof}


\subsection{Intermediate cosimplicial object}\label{SSintermediate}
In this subsection, we define a cs-spectrum $\IM^\bullet$ which intermediates $\CL^\bullet$ and $\thc^\bullet(A',B)$ (see subsection \ref{SSoutlineEQ}). In the rest of this paper, we fix $A'$ an $B$ as in Proposition \ref{propinvarianceofthc}
\subsubsection{Point-set description of hom-spectrum}\label{SSSpointset}
To write down construction efficiently in the following, we shall describe the internal hom object $\Inhom$ in $\SP$ at the point-set level. This description is used in later (sub)sections implicitly. Of course, the contents of this subsection are not new, and these are easily deduced by unwinding the definition of $\otimes=\wedge_S$ in subsection \ref{SSNT} and the adoint property $Mor_{\SP}(X\otimes Y,Z)\cong Mor_{\SP}(X,\Map(Y,Z))$, where $Mor$ denotes the set of morphisms.\par 
Let $X_1,\dots, X_n$ and $Y$ be symmetric spectra. The $k$-th space of the spectrum 
$\Inhom(X_1\otimes\cdots\otimes X_n,Y)$ consisting of sequences of continuous maps 
\[
f_l:\bigsqcup_{l_1+\cdots+ l_n=l} (X_1)_{l_1}\times\cdots\times (X_n)_{l_n}
\to Y_{l+k},\quad l\geq 0
\]
which satisfy the following conditions.
\begin{enumerate}
\item For any $(x_1,\dots,x_n)\in (X_1)_{l_1}\times\cdots\times (X_n)_{l_n}$, if at least one of $x_i$ is the basepoint, $f_l(x_1,\dots,x_n)$ is also the basepoint.
\item If the previous condition are satisfied, $f_l$ induces a map on the corresponding smash product. Then the following diagram commutes
\[
\xymatrix{
S^1\wedge (X_1)_{l_1}\wedge\cdots\wedge (X_n)_{l_n}\ar[r]^{T\qquad}\ar[d]^{f_l}&
(X_1)_{l_1}\wedge\cdots\wedge S^1\wedge (X_i)_{l_i}\wedge\cdots\wedge (X_n)_{l_n}\ar[d]\\
S^1\wedge Y_{l+k}\ar[d]&(X_1)_{l_1}\wedge\cdots\wedge (X_i)_{l_i+1}\wedge\cdots\wedge (X_n)_{l_n}\ar[d]^{f_{l+1}}\\
Y_{l+k+1}\ar[r]^{\sigma^{-1}\cdot}&
Y_{l+k+1}
}
\]
Here, $T$ is the obvious transposition, and the permutation $\sigma\in \Sigma_{l+k+1}$ is the cyclic permutation of the letters $k+1,\dots, k+l_1+\dots +l_i+1$ which takes the first to the last and others to its previous letter $(k+1, k+\cdots +l_i+1, k+\cdots +l_i,\dots, k+2)$.  (The rest arrows are induced by the action of the sphere spectrum).
\item When we regard $\Sigma_{l_1}\times\cdots\times \Sigma_{l_n}\subset \Sigma_{l}\subset \Sigma_{k+l}$ (permutations on the last $l$-letters), the map $f_l$ is $\Sigma_{l_1}\times\cdots\times \Sigma_{l_n}$-equivariant.
\end{enumerate}
The subscript $l$ is omitted in the rest of the paper, and we always identify an element $f$ in $\Map(X_1\otimes \cdots\otimes X_n, Y)_k$ and the sequence of maps satisfying the above three conditions, which $f$ factors through. Via this identification, the action of $\Sigma_k$ is given by  permutations of the first $k$-letters on the codomain $Y_{l+k}$, and the action of $S^1$ is given by that on $Y$.
The adjoint bijection $Mor_{\SP}(Z\otimes X_1\otimes \cdots \otimes X_n, Y)\cong Mor_{\SP}(Z,\Map(X_1\otimes \cdots \otimes X_n, Y))$ is given by $f\mapsto (z\mapsto f(z,\dots))$ as usual. 
\begin{exa}\label{EXprodpointset}
 The product of the $\square$-object structure on $\thc^\bullet(A,B)$ in Lemma \ref{LmonoidTHC} is given by 
\[
h_1\cdot h_2(f_1,\dots, f_{p+q})=\sigma^{-1}\cdot\{h_1(f_1,\dots, f_p)\cdot h_2(f_{p+1},\dots, f_{p+q})\}.
\]
for $h_1\in \thc^p(A, B)_{k_1}, h_2\in \thc^{q}(A, B)_{k_2}$ and $f_j\in F(M)_{l_j}$. Here the second dot in the right hand side denotes the product of $B$, and $\sigma\in \Sigma_{k_{\leq 2}+l_{\leq p+q}}$ denotes the permutation corresponding to the transposition
\[
h_1,h_2,f_1,\dots, f_{p+q}\longmapsto h_1,f_1,\dots, f_p, h_2, f_{p+1},\dots, f_{p+q}
\]
taking the levels into account. More precisely, we set
\[
\sigma(i)=\left\{
\begin{array}{ll}
i &(1\leq i\leq k_1, \ k_{\leq 2}+l_{\leq p}+1\leq i\leq k_{\leq 2}+l_{\leq p+q})\vspace{1mm}\\
i+l_{\leq p} &(k_1+1\leq i\leq k_1+k_2)\vspace{1mm}\\
i-k_2 &(k_1+k_2+1\leq i\leq k_{\leq 2}+l_{\leq p}).
\end{array}\right.
\]
The appearance of $\sigma^{-1}$ is inevitable by definition of the symmetry isomorphism of the monoidal product.
\end{exa}

\subsubsection{Definition of intermediate cosimplicial object}\label{SSSdefIM}
Put 
\[
\athc^\bullet=\thc^\bullet (F(M),\Gamma(M))
\] 
(see subsection \ref{SSatiyahduality} for  $F(M)$ and $\Gamma(M)$). We define a morphism $\trho^{\, p}:\CL^p\to \athc^p$ by
\[
\trho^{\,p}(x_1,\dots,x_p,\langle \phi,\eps,v\rangle )(f_1,\dots, f_p)=\rho(\langle \phi,\eps,v\rangle)\cdot f_1(x_1)\cdots f_p(x_p).
\]
Here, see subsection \ref{SSatiyahduality} for $\rho$, and $(-\cdot -)$ denotes the action of the sphere spectrum. To be precise, for $f\in F(M)_k,\ x\in M$, by definition we have $f(x)\in S^k=\Sph_k$ so for $g\in \Gamma (M)$, $g\cdot f(x)$ is the image of the morphism $\Gamma(M)\otimes \Sph\cong \Gamma(M)$ which is a part of the data of  the monoidal structure of $\SP$, and more explicitly, the dot adds $k$ constant components to $0$ to the $\phi$-part of $g$, and the component $f(x)$ to each value of the section of $g$ from the right, and does not change the $\eps$-part of $g$.  A morphism similar to $\trho^{\, p}$ is also considered in \cite{cj, cohen}. A trouble is that the collection $\{\trho^{\, p}\}_p$ does \textit{not} commutes with coface morphisms. This is because the first and the last cofaces of $\CL^\bullet$ are defined using the projection of a tubular neighborhood while the definition of cofaces of $\athc^\bullet$ does not include it (the only geometric construction it includes is a collapsing map). More formally, we shall see this for $p=0$. Let $\langle \phi,\eps,v\rangle\in \CL^0=\MT$. By definition $d^0\langle \phi,\eps,v\rangle=(\pi_\phi(v),\langle \phi,\eps, v\rangle)$, so we have 
\[
\trho^{\,1}d^0\langle \phi,\eps,v\rangle(f_1)=\rho(\langle \phi,\eps,v\rangle)\cdot f_1(\pi_\phi(v)), \qquad d^0\trho^{\,0}\langle \phi,\eps,v\rangle(f_1)=\rho(\langle \phi,\eps,v\rangle)\cdot f_1
\]
The right hand side of the right equation comes from the structure of  the right $F(M)$-module. Explicitly, for $g\in \Gamma(M)_l, f\in F(M)_k$, $g\cdot f$ is given by $x\mapsto (g(x),f(x))\in \R^{k+l}$, where we omit the $\phi$- and the  $\eps$-parts. The difference between these two  is  whether the point in the domain of $f_1$ moves or not, and we have $\trho^{\, 1}d^0\langle \phi,\eps,v\rangle\not=d^0\trho^{\, 0}\langle \phi,\eps,v\rangle$. We easily see a similar inequality for the last coface.\\
 \indent To remedy this, we introduce an intermediate cs-spectrum $\IM^\bullet$. At each cosimplicial degree, this is a certain path object. The domain is the interval $[0,2]$, and at $0$, the (first and last) cofaces defined using the projection $\pi_\phi$ and at $2$ without using it. To define cofaces between $0$ and $2$, we use a family of homotopy between a self map on $M$ which shrinks a small disk to its center and the identity, parametrized by the center. Recall that an element of $\Gamma(M)$ is a section of a sphere bundle whose fiber is obtained by collapsing the boundary of  a disk with a given small radius, so if the point in the domain $M$ goes outside of a disk of the radius, the value of a section goes to the basepoint and becomes irrelevant to the point.  This means we can define the cofaces using the homotopy. We shall give a formal definition.  
\begin{defi}\label{DTIM} 
\begin{itemize2}
\item[(1)] For a linear map $\phi\in V_k$, $\phi^i:\R^{k_0}\to \R$ denotes its $i$-th component.   Similarly, $v^i\in \R$ denotes the $i$-th component of a vector $v\in \R^k$. 
\item[(2)] Let $\phi\in V_k$ be an element. A component of $\phi$ is said to be a {\em $0$-component}  if it is the constant map to $0\in \R$.  $\phi$ is said to {\em have no $0$-components} if  none of its components   is a $0$-component.  
\item[(3)] Let $p_0:(\athc^p)^{[0,1]}\to\athc^p$ be the evaluation at $0$ and $\CL^p_k\times_{\athc^p_k}(\athc^p_k)^{[0,1]}$ be the pullback of the diagram $\CL^p_k\xrightarrow{\trho^{\, p}}\athc^p_k\xleftarrow{p_0}(\athc^p_k)^{[0,1]}$. 
 An element $(\lambda, h)$ of $\CL^p_k\times_{\athc^p_k}(\athc^p_k)^{[0,1]}$ is said to {\em have a shaft}\, if there exists an element $(\bar \phi,\bar \eps,\bar v)\in \M_{k'}$ (for some $k'$) such that the following conditions hold.
\begin{itemize}
\item $\bar\phi$ has no $0$-components.
\item For elements $s\in [0,1]$, $f_1,\dots, f_p\in F(M)$,  let $(\phi_s,\eps_s,\theta_s)$ be an element representing $h_s(f_1,\dots,f_p)$.  \vs{1mm} Recall that $\theta_s$ is a section of the sphere bundle whose fiber at $y$ is $\bar B_{\eps_s}(y)/\partial \bar B_{\eps_s}(y)$.  Then, there exist  numbers $i_1<\dots <i_{k'}$ such that 
\begin{itemize}
\item  $\phi_s^i$ is a $0$-component for $i\not=i_1,\dots,i_{k'}$, and 
\item the following equations hold  for each $y\in M$, $s\in [0,1]$, and $r=1,\dots ,k'$,
\[
\phi_s^{\,i_r}=\bar \phi^{\, r}, \ \eps_s=\bar \eps, \text{ and\  if  $\theta_s(y)\not=*$, } \  \theta_s^{\,i_r}(y)=\bar v^{\, r}.
\] 

\end{itemize} Here, $*$ is the point represented\vs{1mm} by $\partial \bar B_{\eps_s}(y)$, and if $\theta_s(y)\not=*$,  $\theta_s(y)$ is naturally regarded as a point in the interior of $\bar  B_{\eps_s}(y)$. $\theta_s^{\,i_r}(y)$ denotes its $i_r$-th component.   These conditions imply $\phi_s$ has exactly $k'$ non-$0$-components, and the $r$-th non-$0$-component is equal to the $r$-th component of $\bar \phi$.
\end{itemize}
 To clarify, $k'$ can be different from $k$, and $i_1,\dots, i_{k'}$ can depend on $s, f_1,\dots, f_p$. 
If $h_s(f_1,\dots, f_p)\not=*$ for some $s, f_1,\dots, f_p$, the element $(\phi_s,\eps_s,\theta_s)$ is unique, so the element $(\bar \phi,\bar \eps,\bar v)$ is unique and called the {\em shaft} of $h$ (or $(\lambda, h)$).  Otherwise, i.e., if $(\lambda, h)$ is the basepoint of $\CL^p_k\times_{\athc^p_k}(\athc^p_k)^{[0,1]}$, \vs{1mm} the above condition is satisfied by any element satisfying  $\bar v\in \partial \bar \nu_{\bar \eps}(\bar \phi)$. Let 
\[
\TIM^p_k\subset \CL^p_k\times_{ \athc^p_k}(\athc^p_k)^{[0,1]}
\]
be the subspace consisting of the elements having a shaft.  The action of $S^1$ on $\CL^p\times_{ \athc^p}(\athc^p)^{[0,1]}$ adds a $0$-component on the $\phi$-part and does not change the original components, and the condition of having a shaft imposes no restrictions on the $0$-components, so the sequence $\{\TIM^p_k\}_k$ is closed under the action of $S^1$.  It is clear that it is also closed under the action of the symmetric groups on  $\CL^p\times_{ \athc^p}(\athc^p)^{[0,1]}$ , \vs{1mm}so it forms  a symmetric spectrum, denoted by $\TIM^p$. 
\item[(4)] We shall define the coface morphisms. Fix a continuous map $H_{x,s}:M\to M$  depending on $x\in M$ and $s \in [0,1]$ continuosly such that
\begin{itemize}
\item $H_{x,0}$ takes $\bar B_{r}(x)\cap M$ to the one point set $\{x\}$, where $r=L_{e_0}/2$, and 
\item  $H_{x,s}$ is the identity on the whole $M$ if $s=1$.
\end{itemize}
For an element $(\lambda,h)\in \TIM^p_k$, $(\lambda^i,h^i)=d^i(\lambda,h)$ is defined as follows. We set $\lambda^i=d^i\lambda$, the coface  of $\CL^\bullet$.  If $h=*$, we set $h^i=*$.  Otherwise, let $(\phi,\eps, v)$ be the shaft of $(\lambda, h)$ and  put 
\[
h^i(f_1,\dots, f_{p+1})
=\left\{
\begin{array}{ll}
\sigma_{l_1,k}\cdot \{ H_{x_0}^*f_1\cdot h(f_2,\dots, f_{p+1}) \}&( i=0)\vs{1mm}\\
h(f_1,\dots, f_i\cdot f_{i+1},\dots, f_{p+1})& (1\leq i\leq p)\vs{1mm}\\
h(f_1,\dots, f_p)\cdot H_{x_0}^*f_{p+1}& (i=p+1).
\end{array}
\right.
\]
Here, $x_0=\pi_{\phi}(v)$, and  $l_1$ is the level of $f_1$ and $\sigma_{l_1,k}$ is the permutation which transposes the first $l_1$ letters and the next $k$ letters, and does not move others  (this is the natural permutation appearing when one transpose elements of level $k$ and $l_1$). $\alpha^*f$ denotes $f\circ \alpha$ for $\alpha\in \Map(M, M),\ f\in F(M)$. The second dot in the first line and the dot in the third line denote the action of $F(M)$ on $\Gamma(M)$, and the dot in the second line denotes the product on $F(M)$. We also omit the subscript $s$ of $h^i_{s},\ h_{s},\ $ and $H_{x_0,\, s}$. The codegeneracy morphisms on $\TIM^\bullet$ is the one induced by those on $\CL^\bullet$ and $\athc^\bullet$  in the component-wise manner.  
\item[(5)] We set $\BIM^p=(\athc^p)^{[1,2]}$. The cofaces and the codegneracies on $\BIM^\bullet=\{\BIM^p\}_p$ is induced by those of $\athc^\bullet$ in the obvious manner.
\item[(6)] We set $\IM^\bullet =\TIM^\bullet \times_{\athc^\bullet}\BIM^\bullet$, the pullback of the  evaluations $:\TIM\to \athc$ and  $:\BIM\to \athc$ at $1\in [0,1]$ (or $[1,2]$). (Thus, an element of $\IM$ is naturally considered as a pair of an element of $\CL$ and a path $[0,2]\to \athc$.)
\end{itemize2}
\end{defi}
Before verifying well-definedness of $\TIM$, we shall observe the condition of having a shaft. The  action of $S^1$ does not change the shaft (if the result is not the basepoint).  The effect of the action of a permutaion $\sigma$ on the shaft is given by the action of the permutation among non-$0$-components induced by $\sigma$.  An example of an element of $\TIM$ is the element $(\lambda, \trho^{\,p}(\lambda))$ where $\lambda$ is an element of $\CL^p$ and $\trho^{\, p}(\lambda)$ is considered as the constant path in $\athc^p$. For an element $(\lambda, h)$ having a shaft, If $\lambda\not=*$, the shaft is the same as the element made from the element representing $\lambda$ by removing the $0$-components from its map and the components of its vector with the same numbering as a $0$-component of the map. It can hold that $\lambda=*$ and $h\not=*$. Even in this case, we want to use an element $(\phi, \eps, v)$ to adapt the maps $\teps$ and $\psi$ to construction on $\TIM$. This is because we give the condition. A fact which needs attention is that the dimension of the vector or equivalently, of  the codomain of the map of a shaft can jump discontinuously. For example, we fix a map $\phi\in V_k$ having no $0$-components and set $\phi_\alpha(x)=(\alpha x,\phi(x))$ for $\alpha\in \R$. If $\alpha\not=0$, $\phi_\alpha :\R^{k_0}\to \R^{k_0+k}$ has no $0$-components. We also fix $(w,v)\in \cap_{|\alpha|<\delta}\nu_\eps(\phi_\alpha)$ with $\dim v=k$, where $\delta>0$ is so small that the intersection is non-empty. We regard $\rho(\langle \phi_\alpha, \eps, (w,v)\rangle)$ as the constant path in $\Gamma(M)$ and as an element of $\TIM^0_{k_0+k}$. If $\alpha\not=0$, the shaft of the element is $(\phi_\alpha, \eps, (w,v))$. but if $\alpha=0$, the shaft is $(\phi,\eps, v)$. However, the equation $\underset{\alpha\to 0}{\lim}\pi_{\phi_\alpha}(w,v)=\pi_{\phi}(v)$ holds since  the effect of $w$ on the position of the closest point gets smaller when $\alpha\to 0$. More generally, ingredients which is defined using the shaft and used in the later costruction depend on an element of $\TIM$ continuously. The following lemma gives instances.
\begin{lem}\label{Lshaftconti}
\begin{itemize2}
\item[(1)] The maps $\TIM_k^p-\{*\}\ni (\lambda, h)\mapsto \{\text{ the shaft }(\phi,\eps,v)\text{ of }(\lambda, h)\}\mapsto \pi_\phi(v)\in M$ and $ \TIM_k^p-\{*\}\ni (\lambda, h)\mapsto \{\text{ the shaft }(\phi,\eps,v)\text{ of }(\lambda, h)\}\mapsto |\phi|\in \R$ are continuous.
\item[(2)] Let $u\in \K(n)$ and $(\lambda_1,h_1),\dots ,(\lambda_n, h_n)\in \TIM-\{*\}$. Let $\ve_l=(\phi_l,\eps_l,v_l)$ be the shaft of $(\lambda_l,h_l)$ for $1\leq l\leq n$. The map $(u\,;(\lambda_l,h_l)_l)\mapsto \teps(u\,;(\phi_l,\eps_l)_l)\in \R$ is continuous (for the product topology on the domain). When $(u\,;(\lambda_l,h_l)_l)$ runs through the range such that $(u\,;(\ve_l)_l)$ belongs to $\D^n$, the map $(u\,;(\lambda_l,h_l)_l)\mapsto \psi_i(u;(\ve_l)_l)\in \Map(M,M)$ is continuous. 
\end{itemize2}
\end{lem}
\begin{proof}
We shall prove the continuity of $\pi_\phi(v)$. In short, this holds since  the $0$-components does not affect the position of the closest point $\pi_{\phi}(v)$, see subsection  \ref{SSproblem}. We shall give a detailed proof. Let $N\subset M$ be a neighborhood of $\pi_\phi(v)$.  Let $(\lambda_1,h_1)$ be another element and $(\phi_1,\eps_1,v_1)$ be its shaft. We substitute  fixed elements $f_1,\dots, f_p\in F(M),$ and $s\in [0,1]$ and compare components of $\phi$ and $\phi_1$. The components of $\phi$ (resp. $\phi_1$) are included in the components of $h_s(f_1,\dots,f_p)$ (resp. $h_{1,s}(f_1,\dots,f_p)$). If a pair of components of $\phi$ and $\phi_1$ is a pair of  components of $h_s(f_1,\dots,f_p)$ and $h_{1,s}(f_1,\dots,f_p)$ with the same numbering, we say these components of $\phi$ and $\phi_1$ correspond.  More precisely, we can take elements $f_1,\dots,f_p\in F(M),$ and $s\in [0,1]$ such that  $h_s(f_1,\dots, f_p)\not=*$. If $h_1$ is sufficiently close to $h$, we have $h_{1,s}(f_1,\dots,f_p)\not=*$. Let $\phi'$ (resp. $\phi'_{1}$) denote the $\phi$-part of  $h_s(f_1,\dots, f_p)$ (resp.  $h_{1,s}(f_1,\dots,f_p)$).  If  neither the $j$-th component $(\phi')^j$ nor  $(\phi'_1)^j$ is a $0$-component, and $(\phi')^j$ (resp. $(\phi'_1)^j$) is the $t$-th (resp. $t_1$-th) non-$0$-component of $\phi'$ (resp. $\phi_1'$) for some $j$, we say the components $\phi^t$ and $\phi_1^{t_1}$ correspond. Suppose $(\lambda_1,h_1)$ is sufficiently close to $(\lambda, h)$.  We may assume any component of $\phi$ corresponds to a component of $\phi_1$ since a map sufficiently close to a non-zero map is also non-zero. Let $\phi_2$  be the map made from $\phi_1$  by removing the components which do not correspond to a component of $\phi$, and $v_2$ be the vector made from $v_1$ by removing the components with the same numbering as a removed component of $\phi_1$.  If $(\lambda_1, h_1)$ runs through a sufficiently small neighborhood of $(\lambda, h)$,  the codomain of $\phi_2$ is stationary and  equal to the codomain of $\phi$ (say $\R^{k'}$) , and the map $(\lambda_1, h_1)\mapsto (\phi_2,v_2)\in V_{k'}\times \R^{k'}$  is continuous. So if $(\lambda_1,h_1)$ belongs to a sufficiently small neighborhood of $(\lambda, h)$, the following conditions are satisfied for some $\delta>0$.
\begin{enumerate}
\item $\pi_{\phi_2}(v_2)\in N$.
\item If $x\in M-N$, $|v_2-\phi_2(x)|\geq |v_2-\phi_2(\pi_{\phi_2}(v_2))|+\delta$.
\item If the $j$-th component of $\phi_1$ does not correspond to a component of $\phi$,    the inequality $|\phi_1^j(x)|<\frac{\delta}{4k'}$ holds for any $x\in M$.

\end{enumerate} The first condition holds since $(\phi_2,v_2)=(\phi,v)$ if $(\lambda_1,h_1)=(\lambda, h)$ and the third one holds since $\phi_1^j$ is sufficiently close to the constant map to $0$. By definition of the norm, we have the equation
\[
|v_1-\phi_1(x)|=\sqrt[]{|v_2-\phi_2(x)|^2+\sum_j|v_1^j-\phi_1^j(x)|^2}
\]
for $x\in M$. Here $j$ runs through the numbering not corresponding to a component of $\phi$. Putting all together this equation for the cases of  $x=\pi_{\phi_2}(v_2)$ and $x\in M-N$, the second condition, and the inequality
\[
||v_2^j-\phi_1^j(x)|-|v_2^j-\phi_1^j(\pi_{\phi_2}(v_2))||<\frac{\delta}{2k'}
\]
deduced from the third condition, we see the following inequality by easy calculation.
\[
|v_1-\phi_1(x)|\geq |v_1-\phi_1(\pi_{\phi_2}(v_2))|+\delta/2\quad \text{ if } x\in M-N. 
\] This implies $\pi_{\phi_1}(v_1)$ belongs to $N$. Continuity of $|\phi|$, $\teps$, and $\psi_i$ is similarly proved.  The definition of $\psi_i$ looks complicated but it is defined by using $\pi_{\phi}$ so the above proof works mostly as it is.
\end{proof}
We shall verify well-definedness of $\TIM^\bullet$. We use the notations in Definition \ref{DTIM}.
The action of $F(M)$ on $\Gamma(M)$ adds $0$-components on the map and does not change the original component. The defining  condition on the shaft is given in the same form for any set of element $(f_1,\dots, f_p\,; s)$ so the operations on these substituted elements does not affect the condition. By these observations,  $(\lambda^i, h^i)$ has a shaft. 
We shall verify the above formula of the first  coface map is a well-defined map between pullbacks. Let $(\phi',\eps', v')$ be the element representing $\lambda$. $\phi$ (resp. $v$) is the same as the element made from $\phi'$ (resp. $v'$) by removing the $0$-components (resp. the components with the same numbering as the $0$-components). We  have $\pi_{\phi'}(v')=x_0(=\pi_\phi(v))$ since $0$-components of $\phi'$ do not affect the position of the closest point, see subsection \ref{SSproblem}. We have $\eps'=\eps$.  If $|v'-\phi'(y)|<\eps$, we have 
\[
|x_0-y|= \frac{1}{|\phi'|}|\phi'(x_0)-\phi'(y)|\leq \frac{1}{|\phi'|}|\phi'(x_0)-v'|+|v'-\phi'(y)|<2\eps<L_{e_0}/2,
\] which implies  $H_{x_0,\,0}^*f_1(y)=f_1(x_0)$. If $|v'-\phi'(y)|\geq \eps$, \vs{1mm}we have $\rho(\langle \phi', \eps', v'\rangle)(y)=*$ by definition. By these observations, we have $\rho(\langle \phi',\eps',v'\rangle)\cdot f_1(x_0)=\rho(\langle \phi',\eps',v'\rangle)\cdot H_{x_0,\,0}^*f_1$.  By this equation, we see the two elements 
\[
\begin{split}
\tilde \rho^{\,p+1}\lambda^0(f_1,\dots, f_{p+1}) & =\rho(\langle \phi',\eps',v'\rangle)f_1(x_0)f_2(x_1)\cdots f_{p+1}(x_p)\ \  \text{and} \\
p_0h^0(f_1,\dots, f_{p+1}) & =\sigma_{l_1,k}\cdot H^*_{x_0,\,0}f_1\cdot p_0h(f_2,\dots, f_{p+1})\\
&=\sigma_{l_1,k}\cdot H^*_{x_0,\, 0}f_1\cdot \rho(\langle \phi',\eps',v'\rangle)f_2(x_1)\cdots f_{p+1}(x_p) \\
&= \rho(\langle \phi',\eps',v'\rangle)\cdot H^*_{x_0,\,0}f_1\cdot f_2(x_1)\cdots f_{p+1}(x_p)
\end{split} 
\] coincide. We also use $p_0(h)=\trho^{\,p}(\lambda)$ in the equation in the third line. The coface map is continuous by Lemma \ref{Lshaftconti}. The case of the last coface is similarly verified.  Well-definedness of other cofaces and codegeneracies is obvious.   
Verification of cosimplicial identity is trivial as it is equivalent to verify the identity for $\thc^\bullet(F(M),H_{x, s}^*\Gamma(M))$ with $x, s$ fixed, and we have completed verification on the definition of $\IM^\bullet$. 
\begin{defi}\label{Dmorphismscosimplicial}
We define a morphism $p_0:\IM^\bullet \to \CL^\bullet$ as the projection to the first factor of the pullback and $q_2:\IM^\bullet \to \athc^\bullet$ as the evaluation at $2\in [0,2]$. Fix weak equivalences $ \eta :A'\to F(M)$ and $\zeta:\Gamma(M)\to B$, and let $\zeta_*\eta^*: \athc^\bullet\to \thc^\bullet (A',B)$ denote the morphism induced by the pushforward by $\zeta$ and the pullback $\eta^{\otimes n}:(A')^{\otimes n}\to F(M)^{\otimes n}$ at the $n$-th term for each $n\geq 0$, and $\bar q_2:\IM^\bullet \to \thc^\bullet(A',B)$  the composition $\zeta_*\eta^*\circ q_2$.
\end{defi}
The following lemma is proved by a way analogous to the proof of the fact that the  projection $X\times_YY^{[0,1]}\to X$ is a homotopy equivalence. A homotopy inverse of the projection is given by shrinking  paths to their values at $0$. 
\begin{lem}\label{LweqIM}
The morphism $p_0:\IM^\bullet \to \CL^\bullet$ is a termwise level equivalence.
\qed
\end{lem}

\subsection{Generalization of McClure-Smith product for $A_\infty$-structure}\label{SSgeneralizationMS}
\subsubsection{Motivation}\label{SSSmotivation}
In this subsection, we introduce a generalization of McClure-Smith product reviewed in subsection \ref{SSms}.
This is somewhat complicated so we first explain the motivation. It is ideal that we could construct an action of $\MK$ on $\IM^\bullet$ such that the two morphisms $p_0 : \IM^\bullet \to \CL^\bullet$ and $q_2 : \IM^\bullet\to \athc^\bullet$ are compatible with the action. Here $\athc^\bullet$ is regarded as an $\MK$-algebra as in the paragraph below Lemma \ref{LmonoidTHC}. The main difference between the actions of $\MK$ on $\CL^\bullet$ and $\athc^\bullet$    is that the former uses the perturbation map $\psi_i(d):M\to M$, which is constructed using a tubular neighborhood whereas the latter does not. So it is a natural approach  to construct an action of  $\MK$ on $\IM^\bullet$ using a homotopy $\tpsi_i(d)$ between $\psi_i(d)$ and the identity. To explain the difficulty, we shall look at this (failed) approach more closely. We fix  a homotopy $\tpsi_{i,s}(d)$ with $\tpsi_{i,0}(d)=\psi_i(d)$ and $\tpsi_{i,2}(d)=\id_M$. We first consider the arity 2 case for simplicity and we write the action as product $(-\cdot-)$ below as $\K(2)$ is a one-point set. 
Recall that an element $(\lambda, h)$ of $\IM^p$ consists of $\lambda\in \CL^p$ and $h:[0,2]\to \Map_{\SP}(F(M)^{\otimes p},\Gamma(M))\in \athc^p$. We would define a product on $\IM^\bullet$ as follows. It is the same as the product of $\CL^\bullet$ defined in subsection \ref{SScosimplicialLMT} for the $\lambda$-part, and  is defined similarly to the product on $\athc^\bullet$ but pulling back elements of $F(M)$ by $\tpsi_{i,s}$ at $s$ for the $h$-part. In the notations, we set
\begin{equation}\label{EQfailedprod}
\begin{split}
(\lambda_1,h_1)\cdot (\lambda_2,h_2)&=(\lambda_1\cdot \lambda_2, \tilde h),\\
\text{where} \quad \tilde h(f_1,\dots, f_{p+q})&=\sigma^{-1} \cdot \{ h_{1}(\tpsi_{1}^*f_1,\dots, \tpsi_{1}^*f_p)\cdot h_{2}(\tpsi_{2}^*f_{p+1},\dots, \tpsi_{2}^*f_{p+q})\}\ \  \in\ \  \Gamma(M)
\end{split}
\end{equation}
for $(\lambda_1,h_1),\ (\lambda_2,h_2)\in \IM^{p, \text{or}\ q}$. Here, $\tpsi_i$ denotes $\tpsi_{i, s}(d)$ with $d=(u\,;\ve_1,\ve_2)$ where $u\in \K(2)$ is  the unique element and $\ve_i$ is the shaft of $h_i|_{[0,1]}$, and  $\sigma$ is the permutation given in Example \ref{EXprodpointset}, and the first and the second dots in the second line represent the action of the symmetric group and the product on $\Gamma(M)$ respectively. We also omit the subscript $s\in [0,2]$ of $\tilde h, h_i, \tpsi_i$ here and below. \par
 If this product extended to an action of $\MK$,   it should satisfy the equation
\[
d^0\{(\lambda_1, h_1)\cdot (\lambda_2,h_2)\}=\{d^0(\lambda_1,h_1)\}\cdot (\lambda_2,h_2).
\]
For $p=q=0$, the value of  the $h$-part of the left and the right hand sides at $f\in F(M)$ are 
\begin{equation}\label{EQcofacefail}
(H_{\pi_{\phi}(v) }^*f)\cdot h_1\cdot h_2,\quad \text{and} \quad (H_{\pi_{\phi_1}(v_1) }^*\tpsi_{1}^*f) \cdot h_1\cdot h_2,
\end{equation}
respectively, where we  set $\phi=\phi_1\times\phi_2$ and $v=(v_1,v_2)$, and we extends $H_{x,\,s}$ to $s\in [0,2]$ by the identity, and we omit permutations.
By comparison, we see the equation \begin{equation}\label{EQtpsifail}
H_{\pi_\phi(v),s}(y)=\tpsi_{1,s}\circ H_{\pi_{\phi_1}(v_1),s}(y)
\end{equation}
must be satisfied for each $s\in[0,2]$ and each $y\in M$ at which the value of the sections which are the components of the elements in (\ref{EQcofacefail})  is not $*$.  However, it is difficult to define $\tpsi$ so that it satisfies this equation or a similar equation in higher arity. If the product  extended to an action of $\MK$, the equation 
$
\{d^0(\lambda_1,h_1)\}\cdot (\lambda_2,h_2)\cdot (\lambda_3,h_3)=d^0\{(\lambda_1,h_1)\cdot (\lambda_2,h_2)\cdot (\lambda_3,h_3)\}$ should also hold. (The product of the three elements depends on an element of $\K(3)$ while we use dots.)  By an argument similar to the above, we would have an  equation of the same form as (\ref{EQtpsifail}) except for $\phi$ and $v$ representing the products of suitable three elements. In any choice of $H_{y, s}$, as $s$ gets closer to $1$, the subset of $M$ consisting of the points sent to $\{y\}$ by $H_{y,s}$ gets smaller (because $H_{y,1}=\id$). If $s$ is sufficiently close to $1$, the subset is strictly smaller than the neighborhood of $\pi_{\phi_1}(v_1)$ whose points  the value of the products of the three elements (similar to the products in (\ref{EQcofacefail}) at is not $*$. For such an $s$, 
 $\tpsi_{1,s}$ should be independent of $u\in \K(3)$ in the image of  the neighborhood of $\pi_{\phi_1}(v_1)$ by $H_{\pi_{\phi_1}(v_1)}$ in view of \vs{1mm} (the arity $3$ version of) the equation (\ref{EQtpsifail}) and the independence of $H_{y,s}$ on $u$. This means the product of three elements is strictly associative for such an $s$. It is difficult  to define $\tpsi$ so that the product defined using $\tpsi$ satisfies  this condition (as well as  others) as long as we use the method of projection after parallel transport. $\tpsi$ gives a homotopy which rewinds the perturbation given by $\psi$, and the difficulty is similar to the one in defining $\psi$   so that it induces a strictly associative product. To give intuition of the difficulty, we shall show a simple example. Consider the two decomposition of the diagonal map $\R\to \R^3$ as follows.
\[
\R\stackrel{\Delta}{\longrightarrow}\R\times \R\stackrel{\Delta\times\id}{\longrightarrow}\R\times \R\times \R, \qquad \R\stackrel{\Delta}{\longrightarrow}\R\times \R\stackrel{\id \times \Delta}{\longrightarrow}\R\times \R\times \R
\]
Here, $\Delta$ denotes the diagonal map. For a point $(a,b,c)$ of a tubular neighborhood of $\R$ in $\R^3$, we can define two elements of $\R$. One is the element sent by the projection of a tubular neighborhood of $\Delta$ after  the projection of a  neighborhood of $\Delta\times \id$, according to the left diagram. The other is the element sent by the projection of a tubular neighborhood of $\Delta$ after the  projection  of a neighborhood of $\id\times \Delta$, according to the right diagram.  Here the projection is actually the orthogonal projection to the image of the map $\Delta$. So it sends $(a, b)$ to $(a+b)/2$. The former point is $(a+b+2c)/4$ and the latter is $(2a+b+c)/4$ and these are not the same. We also use the orthogonal projection in the definition of $\psi$. By the essentially same reason as this  example, $\psi$ does not induce a strictly associative product.  It is difficult to modify $\psi$ so that it is strictly associative since  properties of the orthogonal projection are used in many places.
 It is also impossible to define $\tpsi$ so that  $\tpsi_{i, s}$ becomes $\id_M$ before $s$ gets  close to $1$ because of  (the arity $3$ version of) the equation (\ref{EQtpsifail}). \\
\indent Thus, it is difficult to construct an action of $\MK$ on $\IM^\bullet$. To avoid this difficulty, we introduce an `up to homotopy coherency version' of  the McClure-Smith product. Practically speaking, we relax the equations
\[
d^0(x\cdot y)=(d^0x)\cdot y,\quad d^{p+q+1}(x\cdot y)=x\cdot d^{q+1}y,\quad (d^{p+1}x)\cdot y=x\cdot d^0y\quad (|x|=p, |y|=q)
\]
by homotopy. We first define an operad which parametrize coherency homotopies concerning these equations, and then we define a monad over cosimplicial objects by using the operad, finally we prove an action of this monad induces a $A_\infty$-structure on $\ttot$.\par

\subsubsection{Cofacial trees}\label{SSScofacialtree}
In this sub-subsection, we introduce an operad $\CT$ over posets.
An element of  $\CT(n)$, which is called a  \textit{cofacial $n$-tree}, is a $n$-tree (see Definition \ref{deftree}) with some marks attached to its vertices and edges. We consider three kinds of marks $\df$, $\dl$, and $h_i$. ($i=1,\dots, n-1$). $\df$ represents the first coface map and $\dl$ the last one. $h_i$ represents a homotopy 
\[
x_1\cdots(\dl x_i)\cdot x_{i+1}\cdots x_n\simeq x_1\cdots x_i\cdot (\df x_{i+1})\cdots x_n.
\] 
For example, a correspondence between cofacial trees and variables or homotopies are presented in Figures \ref{FCT(1)} and \ref{Fpentagon2}.
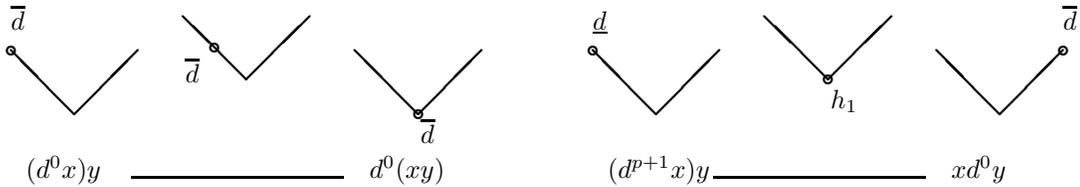
\begin{figure}[H]
\begin{picture}(400,80)(-200,-40)
\put(-200,-40){
\begin{picture}(180,80)(-90,-40)
\thicklines
\put(-40,-30){\line(1,0){80}}
\put(-80,-30){$(d^0x)y$}
\put(-90,-18){
\begin{picture}(50,50)(-25,-25)
\thicklines
\branchI
\branchIII
\saaaacirc
\saaaamark{$\df$}
\end{picture}
}
\put(50,-30){$d^0(xy)$}
\put(40,-18){
\begin{picture}(50,50)(-25,-25)
\thicklines
\branchI
\branchIII
\socirc
\somark{$\df$}
\end{picture}}
\put(-25,-5){
\begin{picture}(50,50)(-25,-25)
\thicklines
\branchI
\branchIII
\saacirc
\saamark{$\df$}
\end{picture}}
\end{picture}}

\put(20,-40){
\begin{picture}(180,80)(-90,-40)
\thicklines
\put(-40,-30){\line(1,0){80}}
\put(-80,-30){$(d^{p+1}x)y$}
\put(-90,-18){
\begin{picture}(50,50)(-25,-25)
\thicklines
\branchI
\branchIII
\saaaacirc
\saaaamark{$\dl$}
\end{picture}
}
\put(50,-30){$xd^0y$}
\put(40,-18){
\begin{picture}(50,50)(-25,-25)
\thicklines
\branchI
\branchIII
\sbbbbcirc
\sbbbbmark{$\df$}
\end{picture}}
\put(-25,-5){
\begin{picture}(50,50)(-25,-25)
\thicklines
\branchI
\branchIII
\socirc
\somark{$h_1$}
\end{picture}}
\end{picture}}
\end{picture}
\caption{two components of $\CT(2)$ ($p=|x|$, $\circ$ represents the presence of a mark on the vertex or edge)}\label{FCT(1)}
\end{figure}

\begin{figure}
\begin{picture}(380,300)(-189,-145)
\put(-15,118){\line(-5,-4){83}}
\put(-96,33){\line(1,-3){25}}
\put(15,118){\line(5,-4){83}}
\put(96,33){\line(-1,-3){25}}
\put(-43,-52){\line(1,0){96}}

\put(-20,125){$((d^0x)y)z$}
\put(-120,40){$(d^0x)(yz)$}
\put(-90,-55){$d^0(x(yz))$}

\put(80,40){$(d^0(xy))z$}
\put(60,-55){$d^0((xy)z)$}
\thicklines
\put(-25,135){
\begin{picture}(50,50)(-24,-24)
\thicklines
\branchI
\branchIII
\branchIV
\saaaacirc
\saaaamark{$\df$}
\end{picture}}
\put(-100,90){
\begin{picture}(50,50)(-24,-24)
\thicklines
\branchI
\branchIII
\branchII
\saaaacirc
\saaaamark{$\df$}
\end{picture}}

\put(-170,30){
\begin{picture}(50,50)(-24,-24)
\thicklines
\branchI
\branchIII
\branchV
\saaaacirc
\saaaamark{$\df$}
\end{picture}}

\put(-150,-40){
\begin{picture}(50,50)(-24,-24)
\thicklines
\branchI
\branchIII
\branchV
\saacirc
\saamark{$\df$}
\end{picture}}
\put(-130,-110)
{\begin{picture}(50,50)(-24,-24)
\thicklines
\branchI
\branchIII
\branchV
\socirc
\somark{$\df$}
\end{picture}}
\put(-25,-115){
\begin{picture}(50,50)(-24,-24)
\thicklines
\branchI
\branchIII
\branchII
\socirc
\somark{$\df$}
\end{picture}}
\put(60,90){
\begin{picture}(50,50)(-24,-24)
\thicklines
\branchI
\branchIII
\branchIV
\saaacirc
\saaamark{$\df$}
\end{picture}}

\put(130,30){
\begin{picture}(50,50)(-24,-24)
\thicklines
\branchI
\branchIII
\branchIV
\saacirc
\saamark{$\df$}
\end{picture}}

\put(110,-40){
\begin{picture}(50,50)(-24,-24)
\thicklines
\branchI
\branchIII
\branchIV
\sacirc
\samark{$\df$}
\end{picture}}

\put(90,-110)
{\begin{picture}(50,50)(-24,-24)
\thicklines
\branchI
\branchIII
\branchIV
\socirc
\somark{$\df$}
\end{picture}}
\put(-25,0){
\begin{picture}(50,50)(-24,-24)
\thicklines
\branchI
\branchIII
\branchII
\saacirc
\saamark{$\df$}
\end{picture}}
\end{picture}

\begin{picture}(380,300)(-189,-120)
\put(-15,118){\line(-5,-4){83}}
\put(-96,33){\line(1,-3){25}}
\put(15,118){\line(5,-4){83}}
\put(96,33){\line(-1,-3){25}}
\put(-43,-52){\line(1,0){96}}

\put(-20,125){$((d^{p+1}x)y)z$}
\put(-120,40){$(d^{p+1}x)(yz)$}
\put(-90,-55){$xd^0(yz)$}

\put(80,40){$(xd^0y)z$}
\put(60,-55){$x((d^0y)z)$}
\thicklines
\put(-25,135){
\begin{picture}(50,50)(-24,-24)
\thicklines
\branchI
\branchIII
\branchIV
\saaaacirc
\saaaamark{$\dl$}
\end{picture}}
\put(-100,90){
\begin{picture}(50,50)(-24,-24)
\thicklines
\branchI
\branchIII
\branchII
\saaaacirc
\saaaamark{$\dl$}
\end{picture}}

\put(-170,30){
\begin{picture}(50,50)(-24,-24)
\thicklines
\branchI
\branchIII
\branchV
\saaaacirc
\saaaamark{$\dl$}
\end{picture}}

\put(-150,-40){
\begin{picture}(50,50)(-24,-24)
\thicklines
\branchI
\branchIII
\branchV
\socirc
\somark{$h_1$}
\end{picture}}
\put(-130,-110)
{\begin{picture}(50,50)(-24,-24)
\thicklines
\branchI
\branchIII
\branchV
\sbbcirc
\sbbmark{$\df$}
\end{picture}}
\put(-25,-115){
\begin{picture}(50,50)(-24,-24)
\thicklines
\branchI
\branchIII
\branchV
\sbbacirc
\sbbamark{$\df$}
\end{picture}}
\put(60,90){
\begin{picture}(50,50)(-24,-24)
\thicklines
\branchI
\branchIII
\branchIV
\saacirc
\saamark{$h_1$}
\end{picture}}

\put(130,30){
\begin{picture}(50,50)(-24,-24)
\thicklines
\branchI
\branchIII
\branchIV
\saabbcirc
\saabbmark{$\df$}
\end{picture}}

\put(110,-40){
\begin{picture}(50,50)(-24,-24)
\thicklines
\branchI
\branchIII
\branchII
\saabbcirc
\saabbmark{$\df$}
\end{picture}}

\put(90,-110)
{\begin{picture}(50,50)(-24,-24)
\thicklines
\branchI
\branchIII
\branchV
\saabbcirc
\saabbmark{$\df$}
\end{picture}}
\put(-25,0){
\begin{picture}(50,50)(-24,-24)
\thicklines
\branchI
\branchIII
\branchII
\socirc
\somark{$h_1$}
\end{picture}}
\end{picture}
\caption{two components of $\CT(3)$}\label{Fpentagon2}
\end{figure}
We impose the following rules on the attached marks.
\begin{itemize}
\item $\df$ and $\dl$ can be attached to any vertex.
\item  $\df$ (resp. $\dl$) can be attached to an edge $e$ if and only if it is the leftmost  (resp. rightmost) one among edges which have the same target as $e$.

\item The only one vertex which $h_i$ can be attached to is the $(i,i+1)$-join. $h_i$ cannot be attached to any edge.
\item If a mark can be attached to a vertex or edge,  any number of copies of the mark can be attached to it.  The number of copies  attached is called the multiplicity.   
\item A $n$-tree with no marks is also considered as a cofacial $n$-tree. 
\end{itemize}
For example, trees with marks presented in Figures \ref{FCT(1)} and \ref{Fpentagon2} satisfy the above  rules. 
\begin{exa}\label{EXcofacialtree}
The following tree is another example which satisfies the rules:\\
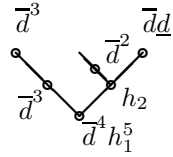
\begin{center}
\begin{picture}(50,50)(-24,-24)
\thicklines
\branchI
\branchIII
\branchV

\socirc
\somark{$\df^4 h_1^5$}

\saacirc
\saamark{$\df^3$}

\saaaacirc
\saaaamark{$\df^3$}

\sbbcirc
\sbbmark{$h_2$}

\sbbacirc
\sbbamark{$\df^2$}

\sbbbbcirc
\sbbbbmark{$\df \dl$}
\end{picture}
\end{center}
Here, each superscript represents the multiplicity. For example, the multiplicity of $\df$ on the root is 4 and that of $h_1$ is 5. 
\end{exa}
On the other hand, the following three trees with marks do \textit{not} satisfy the above rules:\\
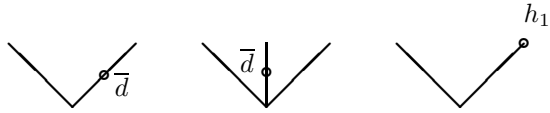
\begin{center}
\begin{picture}(50,50)(-24,-24)
\thicklines
\branchI
\branchIII
\sbbcirc
\sbbmark{$\df$}
\end{picture}
\qquad
\begin{picture}(50,50)(-24,-24)
\thicklines
\branchI
\branchII
\branchIII
\put(0,0){\circle{3}}
\put(-10,0){$\df$}
\end{picture}
\qquad
\begin{picture}(50,50)(-24,-24)
\thicklines
\branchI
\branchIII
\sbbbbcirc
\sbbbbmark{$h_1$}
\end{picture}
\end{center}
We introduce a partial order on $\CT(n)$. 
We define a relation $\leq$ on $\CT(n)$ for each $n$ by iteration of the following three kind of operations.
\begin{itemize}
\item[\ordI] Suppose a cofacial $n$-tree $T$ has an edge where  some marks are attached. We shift some  of the marks on the edge to one of its endpoints and denote the result by $T'$. Then $T'\leq T$. (Examples of this operation are presented on the left line segment of Figure \ref{FCT(1)}. If  $T$ is the tree at the middle of the segment, $T'$ is each of those  at the endpoints of the segment.) 
\item[\ordII] Suppose $T$ has the mark $h_i$. We remove a part of copies of $h_i$ and add the same number of $\df$ (resp. $\dl$) to the source of the edge  which has the $(i,i+1)$-join as its target  and is on the root path from $i+1$-th leaf (resp. $i$-th leaf),  and denote the result by $T'$. Then $T'\leq T$. (Examples are on the right line segment of Figure \ref{FCT(1)}. If $T$ is the tree at the middle of the segment, $T'$ is each of those  at the endpoints of the segment.)
\item[\ordIII] Suppose $T'$ has an internal edge $e$. If $e$ has the mark $\df$ (resp. $\dl$), we shift all copies of them to the leftmost (resp. rightmost) edge in the edges incoming to the source of $e$. (If $e$ has no marks, we do nothing and proceed to next edge contraction.) Then, we contract $e$ to a vertex (so two endpoints of $e$ are identified with the vertex), and we carry all the marks  over to the new tree. If this tree satisfies the rule in attaching marks, we set this tree as $T$ and  declare $T'\leq T$. Suppose this tree does not satisfy the rule in attaching marks. Concretely speaking, $\df$ (resp. $\dl$) is on  an edge which is not rightmost (resp. leftmost).  Say the mark is on the root path from the $i+1$-th leaf (resp. $i$-th leaf) and the vertex created by the edge contraction is the $(i,i+1)$-join. (This is always the case for some $i$.)  We remove all copies of the mark and add the same number of $h_{i}$  to the created vertex. We denote the result by $T$. Then $T'\leq T$. (Examples are the relation between the trees $T'$ at the middle   of the edges $((d^0x)y)z-(d^0(xy))z$, $(d^0(xy))z-d^0((xy)z)$, $d^0(x(yz))-(d^0x)(yz)$, $((d^{p+1}x)y)z-(xd^0y)z$, $x((d^0y)z)-xd^0(yz)$, $xd^0(yz)-(d^{p+1}x)(yz)$ in Figure \ref{Fpentagon2} and the tree $T$ at the center of the pentagon to which the edge belongs. )
\end{itemize}
\begin{exa}\label{EorderCT} In Figures \ref{FCT(1)} and \ref{Fpentagon2} each cofacial tree at the middle of each edge is larger than cofacial trees at its endpoints, and the one at the center of each pentagon is the largest among those in the same pentagon.
\end{exa}
We shall prepare some notations.
\begin{defi}\label{DnumberCT}
For a cofacial $n$-tree $T$ we define a number $m_i(T)$ for each $i=0,\dots, n$ as follows. $m_0(T)$ (resp. $m_n(T)$) is the number of $\df$'s (resp. $\dl$'s) on the root path of the first (resp. last) leaf (counted with multiplicity).  In the case $1\leq i\leq n-1$, consider the shortest path connecting the $i$-th and $i+1$-th leaves. Starting from the $i$-th leaf, we count the number of $\dl$'s on the path until we arrive at  the $(i,i+1)$-join. Next  we count the number of $h_i$'s on the join and then we count the number of $\df$'s from the join to the $i+1$-th leaf. (We do not count $\df$ and $\dl$ on the $(i,i+1)$-join.) $m_i(T)$ is the total number. (See Figure \ref{Fdij}.) 
\end{defi}
We can observe the sum $m_0(T)+\cdots +m_n(T)$ is equal to the total number of marks on $T$ (counted with multiplicity).
\begin{exa}\label{EXnumberCT}
For the cofacial tree $T$ in Example \ref{EXcofacialtree}, we see $m_0(T)=3+3+4=10,\ m_1(T)=5+2=7,\ m_2(T)=1+1=2,\ m_3(T)=1$. For cofacial trees $T$ on the left line segment 
in Figure \ref{FCT(1)}, $m_0(T)=1, m_1(T)=m_2(T)=0$, and for $T$ on the right, $m_0(T)=m_2(T)=0, m_1(T)=1$.
\end{exa}
We use the following trivial observation.
\begin{lem}\label{LorderCT}
If $T\leq T'$ in $\CT(n)$, $m_i(T)=m_i(T')$ for each $i=0,\dots,n$.\qed
\end{lem}
\begin{lem}\label{Lcofacialtree1}
For $n\geq 2$, the relation $\leq$ on $\CT(n)$ is a partial order. Each connected component of $\CT(n)$ has the maximum of the following form :
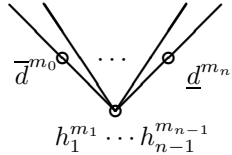
\begin{center} 
\begin{picture}(60,60)(-29,-29)
\thicklines
\put(0,-15){\circle{4}}
\put(-23,-28){$h_1^{m_1}\cdots h_{n-1}^{m_{n-1}}$}
\put(0,-15){\line(-1,1){40}}
\put(0,-15){\line(-2,3){27}}
\put(0,-15){\line(2,3){27}}
\put(0,-15){\line(1,1){40}}
\put(-20,5){\circle{4}}
\put(-38,-5){$\df^{m_0}$}
\put(20,5){\circle{4}}
\put(27,-5){$\dl^{m_n}$}
\put(-7,2){$\cdots$}
\end{picture}
\end{center}
whose underlying tree is the maximum of $\T(n)$. We denote this cofacial tree by $T(m_0,\dots, m_n)$. (A connected component of a poset is an equivalence class of the relation generated by $\leq $.)
\end{lem}
\begin{proof}
Clearly $\leq$ satisfies the transitivity and reflexivity laws.  We shall show the anti-symmetry law.  Let $T$ and $T'$ be two cofacial $n$-tree with $T\leq T'$. For our purpose, we may assume the underlying $n$-trees of $T,\ T'$ are the same.
 We identify the sets of vertices of $T$ and $T'$ by identifing the $i$-th leaf of $T'$  the root, and the $(i,j)$-bunch with those of $T$ respectively.  The anti-symmetry law immediately follows from the following observations.
\begin{enumerate}
\item  The multiplicity of $\df$ and $\dl$  at each vertex of  $T'$ is smaller than or equal to the multiplicity at the corresponding vertex of $T$,
\item the multiplicity of $h_i$  on $T'$ is larger than or equal to that on $T$ for $i=1,\dots,n-1$, and
\item the multiplicity of $\df$ and $\dl$ at each incoming edge of each vertex of $T'$ is larger than or equal to the multiplicity at each incoming edge of the corresponding vertex of $T$. 
\end{enumerate}
For any cofacial $n$-tree $T$ we can take a cofacial $n$-tree $T'$ such that $T\leq T'$ and the multiplicity of $\df$ and $\dl$ on any vertex of $T'$ are zero (using the operation \ordI ). Then we contract all internal edges of $T'$ and obtain an element of  a form $T(m_0,\dots,m_n)$ with $T'\leq T(m_0,\dots,m_n)$. This and Lemma \ref{LorderCT} imply $T(m_0,\dots,m_n)$ is the maximum of the connected component containing $T$ (and $m_i=m_i(T)$).\end{proof}
We denote by $\CT(m_0,\dots,m_n)$ the component including $T(m_0,\dots,m_n)$.
Since $1$-tree has only one vertex which is both root and leaf and no edges, $\CT(1)$ is a discrete poset consisting of formal symbols $\df^{m_0}\dl^{m_1}$ with $m_0, m_1\geq 0$. We denote the one point set $\{\df^{m_0}\dl^{m_1}\}$ by $\CT(m_0,m_1)$.\\
 \indent Similarly to $\T$, the collection $\CT=\{\CT(n)\}_{n\geq 1}$ has a structure of an operad in the category of posets.  Let $T_1\in \CT(n)$ and $T_2\in \CT(m)$ be two cofacial trees. The underlying tree of the composition $T_1\circ_iT_2$ is the composition of the underlying tree: $U(T_1\circ_iT_2)=U(T_1)\circ_iU(T_2)$. Let $T_1'$ be the $n$-tree with marks obtained from $T_1$ by replacing $h_j$ with $h_{j+m-1}$ for each $j=i,\dots,n-1$ and $T'_2$ be the $m$-tree with marks obtained from $T_2$ by replacing $h_j$ with $h_{j+i-1}$ for each $j=1,\dots, m-1$. The marks on the vertex which connects $T_1$ and $T_2$  is the union of the marks on $i$-th leaf of $T_1'$ and on the root of $T_2'$ (with taking multiplicity into account). The marks on other vertices or edges are equal to the marks on the corresponding vertices or edges of $T_1'$ or $T_2'$. For example,
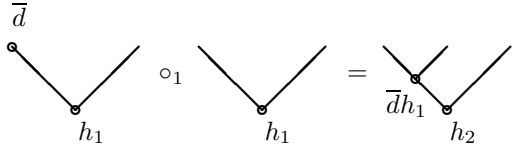
\[
\begin{minipage}{50pt}
\begin{picture}(50,50)(-24,-24)
\thicklines
\branchI
\branchIII
\socirc
\somark{$h_1$}
\saaaacirc
\saaaamark{$\df$}
\end{picture}
\end{minipage}
\ \circ_1\ 
\begin{minipage}{50pt}
\begin{picture}(50,50)(-24,-24)
\thicklines
\branchI
\branchIII
\socirc
\somark{$h_1$}
\end{picture}
\end{minipage}
\ =\ 
\begin{minipage}{50pt}
\begin{picture}(50,50)(-24,-24)
\thicklines
\branchI
\branchIII
\branchIV
\saacirc
\saamark{$\df h_1$}
\socirc
\somark{$h_2$}
\end{picture}
\end{minipage}
\]
If $\ari (T_1)=1$, $T_1\circ_1T_2$ has the same underlying tree as $T_2$  but has the marks of $T_1$ on its root in addition to the marks of $T_2$. If $\ari (T_2)=1$, $T_1\circ_iT_2$ has the same underlying tree as $T_1$  but has the marks of $T_2$ on its $i$-th leaf in addition to the marks of $T_1$.  \par
The following property of $\CT$ is analogous to the property of $\T$ given in Lemma.\ref{Ltree}
\begin{lem}\label{Lcofacialtree2}
\textup{(1)} The map $(-\circ_i-):\CT(m_0,\dots,m_{n_1})\times\CT(m'_0,\dots, m'_{n_2})\longrightarrow \CT(n_1+n_2-1)$ is a monomorphism of sets. Its image is contained in $\CT(\tilde m_k)_{k=0}^{n_1+n_2-1}$, where $(\tilde m_k)=(m_0,\dots, m_{i-2},m_{i-1}+m'_0, m'_1,\dots, m'_{n_2-1},m_{i}+m'_{n_2}, m_{i+1},\dots ,m_{n_1})$  \\
\textup{(2)} For three cofacial trees $T_1$, $T_2$, and $T'$ with $T'\leq T_1\circ_iT_2$ there exist two elements $T_1'$ and $T'_2$ such that $T_1'\leq T_1$, $T_2'\leq T_2$ and $T'=T'_1\circ_iT'_2$.\\
\textup{(3)} Let $n\geq 2$. Any element of $\CT(n)$ of codimension $1$ is a composition of two irreducible elements. Conversely, a composition of two irreducible elements is of codimension 1. If a cofacial tree $T$ is presented as $T_1\circ_iT_2$ with $T_1$, $T_2$ irreducible, the data $(T_1,T_2,i)$ is unique. Here a cofacial tree $T$ is \textit{irreducible} if $T=T_1\circ_iT_2$ implies $T_1=1$ or $T_2=1$. ($1$ is the 1-tree with no marks.)\\
\textup{(4)} Let $n\geq 2$. Any element of $\CT(n)$ of codimension $2$ is a composition of three irreducible elements. Conversely, a composition of three irreducible elements is of codimension 2. Suppose a cofacial tree $T$ is presented as $(S_1\circ_jS_2)\circ_k S_3$ with $S_1$, $S_2$, and $S_3$ irreducible and $j\leq k$. Then
\begin{enumerate}
\item[(a)] if  $S_2\not=S_3$ and $\ari(S_2)=\ari(S_3)=1$ and $j=k$, there are exactly two choices of the data $(S_1,S_2,S_3,j,k)$ and if we take one, the other is $(S_1,S_3,S_2,j,k)$, 
\item[(b)] if  $S_1\not=S_2$ and $\ari(S_1)=\ari(S_2)=1$, there are exactly two choices of the data $(S_1,S_2,S_3,j,k)$ and if we take one, the other is $(S_2,S_1,S_3,j,k)$, 
\item[(c)] otherwise the data $(S_1,S_2,S_3,j,k)$ is unique.
\end{enumerate}
\textup{(5)} If $T_1$ and $T_2\in \CT(n)$ are two different elements of codimension one such that $\langle T_1\rangle\cap\langle T_2\rangle\not=\emptyset$, there exists a (unique) $T_3$ of codimension 2 such that $\langle T_1\rangle\cap\langle T_2\rangle=\langle T_3\rangle$. 
\end{lem}
\begin{proof} 
(1) is clear. In terms of  the definition of the numbers $m_i(T)$, the numbers $\tilde m_i$ are natural. For example, $\df$'s on the root path of the first leaf of $T_2$ is counted on the path  from the $i-1$-th leaf to $i$-th leaf in $T_1\circ_iT_2$ so $m'_0$ is added to $m_{i-1}$.\par 
(2) We consider the case of $\ari(T_2)\geq 2$. The case $\ari(T_1)=1$ is similar and easier. We put $\ari(T_2)=n_2$. 
Both of $T_1\circ _iT_2$ and $T'$ have the $(i,i+n_2-1)$-bunch.  The point is that the  operations which realize $T'\leq T_1\circ_iT_2$ do not shift marks through the $(i,i+n_2-1)$-bunch. 
Precisely speaking, if we define numbers $\bar\mu$, $\underline{\mu}$, and $m_l'$ ($0\leq l\leq n_2$) by
\begin{enumerate}
\item $\bar\mu$ (resp. $\underline{\mu}$) being the multiplicity of $\df$ (resp. $\dl$) on the $(i,i+n_2-1)$-bunch, 
\item $m_0'$ (resp. $m_{n_2}'$) $=$  the number of $\df$'s  (resp $\dl$'s) on the path in $T'$ from  the $i$-th (resp. $i+n_2-1$-th) leaf to the $(i,i+n_2-1)$-bunch (counted with multiplicity),
\item for each $1\leq l\leq n-1$, $m'_l=m_{i+l
}(T')$,
\end{enumerate}
then we have
\[
m_0'\geq m_0(T_2)\geq m_0'-\bar\mu,\quad m_{n_2}'\geq m_{n_2}(T_2)\geq m'_{n_2}-\underline{\mu}, \quad m'_i=m_i(T_2) (1\leq i\leq n).
\]
 So we can take $T_1'$ and $T_2'$ such that $T'=T'_1\circ_i T'_2$ and $m_i(T_1')=m_i(T_1)$, $m_i(T_2')=m_i(T_2)$. The operations which realize $T'\leq T_1\circ_iT_2$ can be separated to  operations on $T'_1$ and $T'_2$ and these two  realize $T'_1\leq T_1$ and $T'_2\leq T_2$. 
(3) and (4) are clear.\par
We shall prove (5). Let $T$ and $T'$ be two  cofacial $n$-trees of codimension $1$. Write $T=T_1\circ_iT_2$ and $T'=T_1'\circ_{i'}T_2'$ where $T_k$ and $T'_k$ are irreducible for $k=1,2$. For simplicity, we assume $n_2=\ari(T_2)\geq 2$ and $n'_2=\ari(T'_2)\geq 2$ and by symmetry we may assume $i\leq i'$. By the condition $\langle T\rangle\cap \langle T'\rangle\not=\emptyset$, one of the following two cases occurs. (i) $ i'+n'_2-1\leq i+n_2-1$, (ii) $i+n_2-1\leq i'$. In the former case, the same condition and Lemma \ref{LorderCT} imply 
\[
\begin{split}
m_{i'-i}(T_2)\geq &m_0(T'_2),\quad m_{i'-i+l}(T_2)=m_l(T'_2) \ (1\leq l\leq n_2'-1),\\
& m_{i'-i+n_2'}(T_2)\geq m_{n'_2}(T'_2).
\end{split}
\]
By these inequalities, we can put 
\[
T'':=T(m_0(T_2),\dots, 
m_{i'-i}(T_2)-m_0(T_2'), m_{i'-i+n'_2}(T_2)-m_{n_2'}(T_2'),
\dots, m_{n_2}(T_2))\in \CT (n_2-n'_2+1),
\]
and $(T_1\circ_iT'')\circ_{i'}T'_2$ is the element of codimension 2 such that 
$\langle T\rangle\cap\langle T'\rangle =\langle (T_1\circ_iT'')\circ_{i'}T'_2\rangle $. The proof of the latter case (ii) is similar.
 \end{proof}

We shall define maps on $\CT$
\[
\begin{split}
 d^j_i:&\CT(m_0,\dots, m_i,\dots, m_n)\longrightarrow \CT(m_0,\dots, m_i+1,\dots, m_n),\\
 s^j_i:&\CT(m_0,\dots, m_i,\dots, m_n)\longrightarrow \CT(m_0,\dots, m_i-1,\dots, m_n)
\end{split}
\]
which will be used to define the monad $\MCK$.
( $i=0,\dots, n$. For $d^j_i$,  $j=0,\dots, m_0$ if $i=0$, and $j=1,\dots,m_n+1$ if $i=n$, otherwise  $j=1,\dots, m_i$. For $s^j_i$, $j=0,\dots, m_i-1$. )\\
\indent We first define $d_i^j$'s. We first consider the case of $i=0$.  In the case of $j=0$,  $d^0_0$ is the map which increases the multiplicity of $\df$ on the root by one (and does not change the other marks). In the case of $j>0$, consider the root path  of the first leaf. We count the number of $\df$ on the path, starting from the root (taking the multiplicity into account).  We define $d^j_0$ as the map which increases the multiplicity of $\df$ on the vertex or edge at which $j$-th $\df$ is, by one. \\
\indent In the case $0<i<n$, we count the number of  marks on the path connecting the $i$-th and $i+1$-th leaves as in Definition \ref{DnumberCT} (see also Figure \ref{Fdij}). So we  count only $\dl$ until we arrive at the $(i,i+1)$-join, for example.  $d^j_i$ is the map which increases the multiplicity of the $j$-th mark by one. \\
\indent In the case $i=n$, consider the root path of the last leaf. We count the number of $\dl$ on the path from the last leaf. For $1\leq j\leq m_n$,  $d^j_i$ is the map which increases the multiplicity of the $j$-th $\dl$. For $j=m_n+1$, $ d^{m_n+1}_n$ is the map which increases the multiplicity of $\dl$ on the root.\\
\indent For each $j=0,\dots, n$, To define $s_i^j$, we use the path used in defining $d_i^j$. $s_i^j$ is the map  which decreases the multiplicity of the $j+1$-th mark by one.
\begin{exa}
Let $T$ be an element of $\CT(1,2,1)$ as follows.
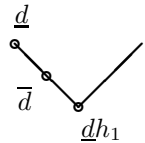
\begin{center}
\begin{picture}(50,50)(-24,-24)
\thicklines
\socirc
\somark{$\dl h_1$}
\saacirc
\saamark{$\df$}
\saaaacirc
\saaaamark{$\dl$}
\branchI
\branchIII
\end{picture}
\end{center}
$d_0^0(T),\ d_0^1(T),\ d_1^1(T),\ d_1^2(T),\ d_2^1(T),\ d_2^2(T)$ are equal to
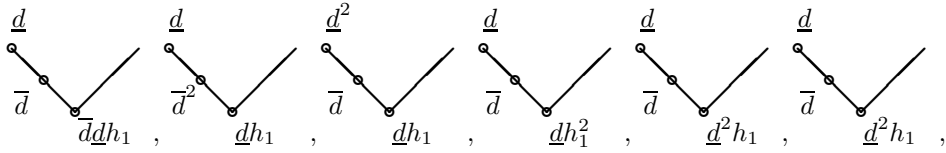
\begin{center}
\begin{picture}(50,50)(-24,-24)
\thicklines
\socirc
\somark{$\df\dl h_1$}
\saacirc
\saamark{$\df$}
\saaaacirc
\saaaamark{$\dl$}
\branchI
\branchIII
\end{picture}
,\
\begin{picture}(50,50)(-24,-24)
\thicklines
\socirc
\somark{$\dl h_1$}
\saacirc
\saamark{$\df^2$}
\saaaacirc
\saaaamark{$\dl$}
\branchI
\branchIII
\end{picture}
,\ 
\begin{picture}(50,50)(-24,-24)
\thicklines
\socirc
\somark{$\dl h_1$}
\saacirc
\saamark{$\df$}
\saaaacirc
\saaaamark{$\dl^2$}
\branchI
\branchIII
\end{picture}
,\ 
\begin{picture}(50,50)(-24,-24)
\thicklines
\socirc
\somark{$\dl h_1^2$}
\saacirc
\saamark{$\df$}
\saaaacirc
\saaaamark{$\dl$}
\branchI
\branchIII
\end{picture}
,\ 
\begin{picture}(50,50)(-24,-24)
\thicklines
\socirc
\somark{$\dl^2 h_1$}
\saacirc
\saamark{$\df$}
\saaaacirc
\saaaamark{$\dl$}
\branchI
\branchIII
\end{picture}
,\ 
\begin{picture}(50,50)(-24,-24)
\thicklines
\socirc
\somark{$\dl^2 h_1$}
\saacirc
\saamark{$\df$}
\saaaacirc
\saaaamark{$\dl$}
\branchI
\branchIII
\end{picture}
,
\end{center}
respectively. $s_0^0(T),\ s_1^0(T),\ s_1^1(T),\ s_2^0(T)$ are equal to
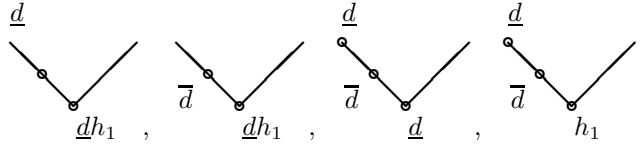
\begin{center}
\begin{picture}(50,50)(-24,-24)
\thicklines
\socirc
\somark{$\dl h_1$}
\saacirc
\saaaamark{$\dl$}
\branchI
\branchIII
\end{picture},\quad
\begin{picture}(50,50)(-24,-24)
\thicklines
\socirc
\somark{$\dl h_1$}
\saacirc
\saamark{$\df$}
\branchI
\branchIII
\end{picture},\quad
\begin{picture}(50,50)(-24,-24)
\thicklines
\socirc
\somark{$\dl$}
\saacirc
\saamark{$\df$}
\saaaacirc
\saaaamark{$\dl$}
\branchI
\branchIII
\end{picture},\quad
\begin{picture}(50,50)(-24,-24)
\thicklines
\socirc
\somark{$ h_1$}
\saacirc
\saamark{$\df$}
\saaaacirc
\saaaamark{$\dl$}
\branchI
\branchIII
\end{picture}.
\end{center}
\end{exa}
\begin{figure}
\begin{center}
\begin{picture}(200,100)(-99,-49)
\put(-8,-47){\vector(-1,1){80}}
\put(-62,-15){$d_0^j$}
\put(-75,-30){$m_0(T)$}

\put(88,32){\vector(-1,-1){80}}
\put(54,-12){$d_3^j$}
\put(54,-27){$m_3(T)$}

\put(-70,45){\line(1,-1){70}}
\put(0,-25){\line(1,1){25}}
\put(25,0){\vector(-1,1){35}}
\put(-5,-10){$d_1^j$}
\put(-20,5){$m_1(T)$}

\put(10,45){\line(1,-1){30}}
\put(40,15){\vector(1,1){30}}
\put(37,30){$d_2^j$}
\put(30,45){$m_2(T)$}

\thicklines
\put(0,-40){\line(-1,1){80}}
\put(0,-40){\line(1,1){80}}
\put(40,0){\line(-1,1){40}}
\end{picture}
\end{center}
\caption{paths used to define $m_i(T)$ and $\{d_i^j\}$}\label{Fdij}
\end{figure}
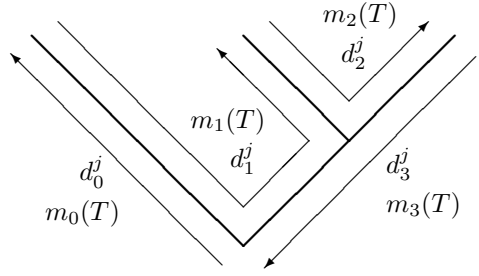
\indent The proof of the following lemma is routine.
\begin{lem}\label{Lcosimplicialid}
The maps $s^j_i$ and $d^j_i$ satisfy the following identities.
\[
\begin{split}
d^j_id^k_i&=d_i^kd_i^{j-1}\quad  (k<j)\\
s^j_id_i^k&=d_i^ks_i^{j-1}\quad (k<j)\\
          &=id            \quad (k=j,j+1)\\
          &=d^{k-1}_is_i^j    \quad (k>j+1)\\
s^j_is^k_i&=s^k_is^j_i    \quad (k>j)\\
d^j_id^k_l=d^k_ld^j_i,    
\quad
s^j_id^k_l&=d^k_ls^j_i    
,
\quad
s^j_is^k_l=s^k_ls^j_i    \quad (i\not= l)
\end{split}
\]
(Note that the first five identities have the same form as the cosimplicial identities.)\qed
\end{lem}
\subsubsection{Operad $\CK$}\label{SSSCK}
\begin{defi}\label{DCK}
We define a topological operad $\CK$ by $\CK(n)=|\CT(n)|$ with the induced composition. Let $\CK(m_0,\dots, m_n)$ denote the connected component of $\CK(n)$ corresponding to $\CT(m_0,\dots,m_n)$ and $d^j_i$ and
$s^j_i$ denote the maps on $\CK(m_0,\dots,m_n)$ induced by $d_i^j$ and $s_i^j$ defined in sub-subsection \ref
{SSScofacialtree} via the realization functor.  We also set $\CK(T)=|\langle T\rangle|$ for a cofacial $n$-tree $T$. Here, $\langle T\rangle$ is the subposet $\{T'\mid T'\leq T\}$ of $\CT(n)$. We  regard $\CK(T)$ as a subspace of $\CK(n)$ in the obvious manner. $U:\CK\to \K$ denotes the morphism of operads induced by the morphism $:\CT\to \T$  forgetting marks. 
\end{defi}
We shall give a description of $\CK$ analogous to the description of $\K$ given in subsection \ref{SSK}.
 Let $\CT(m_0,\dots,m_n)_{1,2}$ be the subposet of $\CT(m_0,\dots,m_n)$ consisting of elements of codimension one or two. We define a diagram $
B^c=B^c_{m_0,\dots, m_n}:\CT(m_0,\dots,m_n)_{1,2}\longrightarrow \CG
$
as follows (compare with the definition of $B_n$ in subsection \ref{SSK}). As an element of codimension one is uniquely presented as $T_1\circ_iT_2$ by Lemma \ref{Lcofacialtree2}, we put $B^c(T_1\circ_i T_2)=\CK(T_1)\times \CK(T_2)$. An element of codimension two is of the form $(S_1\circ_j S_2)\circ_k S_3$. If this presentation is not unique, we fix a presentation once and for all, and put 
$B^c((S_1\circ_j S_2)\circ_k S_3)=\CK(S_1)\times \CK(S_2)\times \CK(S_3)$.  Suppose $(S_1\circ_j S_2)\circ_k S_3\leq T_1\circ_iT_2$.
For example, in the case (a) of Lemma \ref{Lcofacialtree2},(4), by (2) of the same lemma, one of the following cases occurs.
\begin{enumerate}
\item $S_1\circ_jS_2\leq T_1$, $S_3=T_2$,
\item $S_1\circ_jS_3\leq T_1$, $S_2=T_2$.
\end{enumerate}
In each case, we define a map $B^c((S_1\circ_j S_2)\circ_k S_3)\to  B^c(T_1\circ_i T_2)$ using the composition of $\CK$. In the other cases (b),(c) of Lemma \ref{Lcofacialtree2}, we  define a map $B^c((S_1\circ_j S_2)\circ_k S_3)\to  B^c(T_1\circ_i T_2)$ similarly.\\
\indent A natural transformation
 $B^c\longrightarrow \CK(m_0,\dots,m_n)$
is defined using the composition of $\CK$. Here, $\CK(m_0,\dots,m_n)$ denotes the corresponding constant diagram. Thus we obtain a map
$\theta^c:\colim B^c\longrightarrow \CK(m_0,\dots,m_n).
$ The image of $\theta^c$ is $\partial\CK(m_0,\dots,m_n)$, the subcomplex spanned by all cofacial $n$-trees of codimension one.
A point of $\CK(m_0,\dots,m_n)$ can be presented as $t_0T_0+\cdots t_kT_k$ with $T_0<\cdots <T_k\in\CT(m_0,\dots,m_n)$ and $t_0,\dots, t_k\geq 0$, $t_0+\cdots +t_k=1$, . Using this presentation,
we define a map
\[
\tilde\theta^c: Cone(\colim B^c)\longrightarrow \CK(m_0,\dots,m_n)
\]
by $\tilde\theta^c(t\cdot u)= t\theta^c(u)+(1-t)T(m_0,\dots,m_n)$.
\begin{prop}\label{PCK}
With the above notations, the maps $\theta^c:\colim B^c\to\partial\CK(m_0,\dots,m_n)$ and  $\tilde\theta^c:Cone(\colim B^c)\to\CK(m_0,\dots,m_n)$ are   homeomorphisms.
\end{prop}
\begin{proof}
In view of Lemma.\ref{Lcofacialtree2}, the proof is completely analogous to the proof of Proposition \ref{PK}.
\end{proof}

\subsubsection{Monad $\MCK$}
Let $\CC$ be the category of cosimplicial objects over $\C$ (see subsection \ref{SSNT}). In this sub-subsection, we define a monad
\[
\MCK:\CC\longrightarrow \CC.
\]
We first define a functor $\MCK_n:(\CC)^{\times n}\longrightarrow \CC$.
Let $X_1,\dots, X_n$ be objects of $\CC$. The object of cosimplicial degree $l$ of  $\MCK_n(X_1,\dots,X_n)$ is defined as follows.
\[
\MCK_n(X_1,\dots,X_n)^l=\bigsqcup \CK(m_0,\dots,m_n)\hotimes X_1^{p_1}\otimes\cdots\otimes X_n^{p_n}/\sim.
\]
Here, the sequence of non-negative integers $m_0,\dots, m_n, p_1,\dots, p_n$ ranges over sequences satisfying  $m_{\leq n}+p_{\leq n}=l$, and  in the point-set expression the equivalence relation $\sim$ is generated by the following relations
\[
\begin{split}
(u\circ_i\df\,; x_1,\cdots, x_i,\cdots, x_n)&\sim (u\,; x_1,\cdots, d^0x_i,\cdots, x_n),\\
 (u\circ_i\dl\,; x_1,\cdots,  x_i,\cdots, x_n)&\sim (u\,; x_1,\cdots,  d^{p_i+1}x_i,\cdots, x_n)
 \end{split}
\]
where $u\in \CK(m_0,\dots, m_n), x_j\in X_j^{p_j}$, and $1\leq i\leq n$, and $\df$ and $\dl$ denote the elements $\df^1\dl^0$ and $\df^0\dl^1$ of $\CK(1)$ respectively. \\
\indent We define the coface and codegeneracy morphisms on $\{\MCK_n(X_1,\dots,X_n)^l\}_l$ by using the corresponding maps of $\CK$ and $X_i$ alternately. For example, for $0\leq j\leq m_0$, $d^j$ is defined using $d_0^j$ of $\CK$, and for $m_0+1\leq j\leq m_0+p_1$, $d^j$ is defined using $d^{j-m_0}$ of $X_1$, and for $m_0+p_1+1\leq j\leq m_0+p_1+m_1$, $d^j$ is defined using $d_1^{j-m_0-p_1}$ of $\CK$. Precisely speaking, we put 
\[
\begin{split}
d^j&(u\,; x_1,\cdots, x_n)\\
&=\left\{
\begin{array}{lll}
(d^0_0u\,; x_1,\dots, x_n) 
&
\vs{1mm} (j=0) 
&
\\
(d^{j_1}_{i}u\,; x_1,\dots, x_n)
&
 (m_{\leq i-1}+p_{\leq i}+1\leq j\leq m_{\leq i}+p_{\leq i},
&
 j_1=j-m_{\leq i-1}-p_{\leq i}) \vs{1mm}
\\
(u\,;x_1,\dots, d^{j_2}x_i,\dots, x_n)
&
 ( m_{\leq i-1}+p_{\leq i-1}+1\leq j\leq m_{\leq i-1}+p_{\leq i},
&
 j_2=j- m_{\leq i-1}-p_{\leq i-1}) \vs{1mm}
\\
(d^{m_n+1}_nu\,; x_1, \dots, x_n) 
&
 ( j=m_{\leq n}+p_{\leq n}+1),
&
\end{array}\right.\\
s^j&(u\,; x_1,\cdots, x_n)\\
&=\left\{
\begin{array}{lll}
(s^{j_1}_iu\,; x_1,\dots, x_n)
& 
(m_{\leq i-1}+p_{\leq i}\leq j\leq m_{\leq i}+p_{\leq i}-1,
& 
j_1=j-m_{\leq i-1}-p_{\leq i}) \vs{1mm}
\\
(u\,;x_1,\dots, s^{j_2}x_i,\dots, x_n)
& 
( m_{\leq i}+p_{\leq i}\leq j\leq m_{\leq i}+p_{\leq i+1}-1,
&
j_2=j- m_{\leq i}-p_{\leq i})
\\
\end{array}\right.
\end{split}
\]
where $p_i=|x_i|$, $u\in\CK(m_0,\dots,m_n)$, $0\leq i\leq n$ and $m_{\leq t}=m_0+\cdots +m_{t}$ and $p_{\leq t}$ is similar, and $m_{\leq -1}=p_{\leq -1}=p_{\leq 0}=0$. The following lemma easily follows from Lemma \ref{Lcosimplicialid} and the definition of the equivalence relation $\sim$.
\begin{lem}
The morphisms $d^j$ and $s^j$ defined above are well defined and satisfy the cosimplicial identity. \qed
\end{lem}
 Now, we put
\[
\MCK(X)=\bigsqcup_{n\geq 1}\MCK_n(X,\dots, X).
\]
A unit $\id\longrightarrow \MCK$ and a product $\MCK\circ\MCK\longrightarrow \MCK$ are defined by using the unit and the composition of $\CK$. By definition, these natural transformations are compatible with cofaces and codegeneracies and give a well-defined structure of a monad on the functor $\MCK$.
\begin{defi}\label{Dforgetmark}
For objects $X_1,\dots, X_n$ of $\CC$, A morphism $U_{n,X_1,\dots, X_n}:\MCK_n(X_1,\dots,X_n)\to \K(n)\hotimes X_1\, \square \, \cdots\, \square \, X_n$ is defined by
\[
U_n(u\,;x_1,\dots, x_n)=(U(u),\df^{m_0}x_1,\df^{m_1}x_2,\dots, \df^{m_{n-1}}\dl^{m_n}x_n).
\]
, see Definition \ref{DCK} and $\df$ and $\dl$ denote the first and the last cofaces, respectively. For example, we mean
\[
\df^{\,a}\dl^{\,b}x=d^{\,0}\cdots d^{\,0}d^{\,p+b}d^{\,p+b-1}\cdots d^{\,p+1}x\qquad (|x|=p),
\] where $d^0$ is used $a$ times. It is easy to see this morphism actually commutes with the cosimplicial operators. The collection $\{U_{n,X_1,\dots,X_n}\}_{X_1,\dots,X_n}$ defines a natural transformation $U_n:\MCK_n\Rightarrow\K(n)\hotimes(-)^{\, \square \, n}$.
\end{defi}
\begin{lem}\label{Lmorphismofmonad}
The natural transformation
\[
U=\sqcup_{n\geq 0}U_n:\MCK\Longrightarrow \MK:\CC\to \CC
\]
is a morphism of monads. (See sub-subsection \ref{SSSslightgenofms} for $\MK$.)
\end{lem}
\begin{proof}
 The commutativity of the morphism $U$ with the products follows from the fact that $U:\CK\to\K$ is a morphism of operads and the following cosimplicial identities:
$d^i\df ^mx=\df^{m+1}x,\  d^{|x|+i+1}\dl^m x=\dl^{m+1}x\ \   (m\geq 0,0\leq i\leq m)$. 
\end{proof}
\subsubsection{$A_\infty$-operad $\CB$}\label{SSSCB}					
In this sub-subsection, we define an $A_\infty$-operad $\CB$ which acts on $\ttot$ of an $\MCK$-algebra. We define a morphism of cosimplicial objects
\[
\bar U_n:\MCK_n(X_1,\dots, X_n)\longrightarrow X_1\, \square \,\cdots\, \square \, X_n
\]
  as the composition of $U_n$ and the obvious morphism $\K(n)\hotimes X_1\, \square \,\cdots\, \square \, X_n\to X_1\, \square \,\cdots\, \square \, X_n$.
\begin{lem}\label{Lhomotopytype}
 Suppose $\C=\CG$. Let  $X_1,\dots, X_n$ be Reedy cofibrant objects in $\CC$. Then the map $\bar U_n :\MCK_n(X_1,\dots, X_n)\longrightarrow X_1\, \square \,\cdots\, \square \, X_n$ is a weak homotopy equivalence.
\end{lem}
\begin{proof}
 We first define a poset $\sqcat_n^l$. An object of $\sqcat_n^l$ is a sequence of non-negative integers $(m_0,\dots, m_n,p_1,\dots, p_n)$ such that $m_{\leq n}+p_{\leq n}=l$. The partial order $\leq$ is generated by 
\[
\begin{split}
(m_0,\dots,m_n,p_1,\dots, p_n)&<(m_0,\dots, m_i-1,\dots, m_n,p_1,\dots, p_i+1,\dots, p_n)\quad (1\leq i\leq n)\\
(m_0,\dots,m_n,p_1,\dots, p_n)&<(m_0,\dots, m_i-1,\dots, m_n,p_1,\dots, p_{i+1}+1,\dots, p_n)\quad (0\leq i\leq n-1)
\end{split}
\] 
We define a functor
\[
F:\sqcat^l_n\longrightarrow \CG.
\]
For each object $(m_0,\dots,m_n,p_1,\dots, p_n)\in\sqcat^l_n$ we put $F(m_0,\dots,m_n,p_1,\dots, p_n)=X_1^{p_1}\times\cdots \times X_n^{p_n}$.  We associate the map $\id^{\times  i-1}\times\dl\times \id^{\times n-i-1}$to the inequality $(m_0,\dots,m_n,p_1,\dots, p_n)<(m_0,\dots, m_i-1,\dots, m_n, p_1\dots p_i+1\dots, p_n)\quad (1\leq i\leq n)$ and the map $\id^{\times  i}\times\df\times \id^{\times n-i-2}$ to the other generating inequality. 
It is easy to see there is a natural isomorphism $\colim\limits_{\sqcat^l_n}F\cong (X_1\, \square \,\cdots\, \square \, X_n)^l$, see Definition \ref{Dforgetmark} for the meaning of $\df,\ \dl$. \par
Fix a non-negative integer $l$ and let $l_1, l_2$ be two integers such that $0\leq l_1\leq l_2\leq l$. $\sqcat(l_1,l_2)$ denotes the subposet of $\sqcat^l_n$ consisting of objects $(m_0,\dots, m_n,p_1,\dots,p_n)$ such that $l_1\leq p_{\leq n}\leq l_2$. Let $l'$ be an integer with $l'\leq l$ and put $\sqcat(l')=\sqcat(l',l')$ for simplicity. The following diagram $(P_1)$ is a pushout diagram. 
\[
\xymatrix{
\bigsqcup\limits_{(m_0,\dots, p_n)\in\sqcat(l')}(\colim\limits_{\sqcat(l'-2,l'-1)/(m_0,\dots, p_n)}F)\ar[r]\ar[d]& \bigsqcup\limits_{(m_0,\dots, p_n)\in\sqcat(l')}F(m_0,\dots, p_n)\ar[d]\\
\colim\limits_{\sqcat(0,l'-1)}F\ar[r]& \colim\limits_{\sqcat(0,l')}F},
\]
where $\sqcat(l'-2,l'-1)/(m_0,\dots, p_n)$ denotes the subposet of elements smaller than $(m_0,\dots,m_n,p_1,\dots, p_n)$, and all the arrows are induced by inclusions of posets. We shall define a similar pushout diagram  for $\MCK_n(X_1,\dots, X_n)$. 
Put
\[
\MCK_n(0,l')=\bigsqcup\limits_{(m_0,\dots, p_n)\in\sqcat(0,l')}\CK(m_0,\dots, m_n)\times X_1^{p_1}\times\cdots\times X_n^{p_n}/\sim .
\]
($\sim$ is the equivalence relation used to define $\MCK_n(X_1,\dots, X_n)$.) 
Then, there exists a pushout diagram $(P_2)$ as follows:
\[
\xymatrix{\bigsqcup\limits_{(m_0,\dots, p_n)\in\sqcat(l')}(\CK'\times\colim\limits_{\sqcat(l'-2,l'-1)/(m_0,\dots, p_n)}F)\ar[r]\ar[d]& \bigsqcup\limits_{(m_0,\dots, p_n)\in\sqcat(l')}\CK'\times F(m_0,\dots, p_n)\ar[d]\\
\MCK_n(0,l'-1)\ar[r]& \MCK_n(0,l'),}
\]
where $\CK'$ denotes $\CK(m_0,\dots,m_n)$, and the left vertical map is induced by the inclusions
\[
\begin{split}
\CK(m_0,\dots, m_n)\times 
&
F(m_0,\dots, m_i+1,\dots ,p_i-1,\dots,p_n)
\\
&
\cong (\CK(m_0,\dots, m_n)\circ_{i}\dl )
\times F(m_0,\dots, m_i+1,\dots, p_i-1,\dots,p_n)
\\
&
\subset \CK(m_0,\dots, m_i+1,\dots, m_n) \times F(m_0,\dots, m_i+1,\dots, p_i-1,\dots,p_n)\quad (1\leq i\leq n),
\\
\CK(m_0,\dots, m_n)\times 
&
F(m_0,\dots, m_i+1,\dots ,p_{i+1}-1,\dots,p_n)
\\
&
\cong (\CK(m_0,\dots, m_n)\circ_{i+1}\df)
\times F(m_0,\dots, m_i+1,\dots, p_{i+1}-1,\dots,p_n)
\\
&
\subset \CK(m_0,\dots, m_i+1,\dots, m_n) \times F(m_0,\dots, m_i+1,\dots, p_{i+1}-1,\dots,p_n)\quad (0\leq i\leq n-1),
\end{split}
\]
 and the other arrows are defined similarly to the corresponding arrows of $(P_1)$. Here, $\CK(m_0,\dots, m_n)\circ_i\dl$ denotes the image by the partial composition with the vertex represented by $\dl\in \CT(0,1)$. $\CK(m_0,\dots, m_n)\circ_{i+1}\df$ is similarly understood. We define a map of diagram $P_1\to P_2$ similarly to $\bar U_n$.  The top horizontal arrow of $(P_1)$ $:\colim\limits_{\sqcat(l'-2,l'-1)/(m_0,\dots, p_n)}F\to F(m_0,\dots,p_n)$ is a cofibration because each $X_i$ is Reedy cofibrant and $\colim\limits_{\sqcat(l'-2,l'-1)/(m_0,\dots, p_n)}F$ is the latching object modeled by $\cup_i\Delta^{p_1}\times\cdots\times(\Delta^{p_i}_0\cup\Delta^{p_i}_{p_i})
\times\cdots\times \Delta^{p_n}$ ($\Delta^p_k$ denotes the $k$-th face of $\Delta^p$). This and the fact  that $\CK'$ is a cell complex imply  $(P_1)$ and $(P_2)$ are homotopy pushout diagrams. As $\CK'$ is contractible by  Lemma \ref{Lcofacialtree1},  the map $P_1\to P_2$ induces a weak equivalence between the resulting pushouts. So by induction on $l'$, we have proved the assertion.
\end{proof}
Now we define a topological operad $\CB$. Let $\tilde \Delta^{\bullet}$ be a  projective cofibrant replacement of the cosimplicial space $\Delta^\bullet$ (see subsection \ref{SSNT}). We put
\[
\CB(n)=\Map_{\CG^{\Delta}}(\tilde \Delta^{\bullet}, \MCK_n(\tilde\Delta^\bullet,\dots,\tilde\Delta^\bullet))
\]
From the monad structure of $\MCK$, we define an operad structure on $\CB(n)$ by the way exactly analogous to the definition of $\tB$ or $\tB'$ given in subsection \ref{SSms}. The morphism $U:\MCK\to\MK$ in Lemma \ref{Lmorphismofmonad} induces a morphism of operads
\[
U:\CB\longrightarrow \tB'.
\]
\begin{thm}\label{TCBttot}
\textup{(1)}  $U$ is weak  equivalence of topological operads. In particular, $\CB$ is an $A_\infty$-operad.\par
\textup{(2)} $\ttot$ induces a functor $\Alg_{\MCK}(\CC)\longrightarrow \Alg_{\CB}(\C)$ which satisfies the following condition. Let $X^\bullet$ be a $\MK$-algebra and $Y^\bullet$ be a $\MCK$-algebra and let $\alpha:Y^\bullet \to X^\bullet \in\CC$ be a morphism compatible with $U:\MCK\to\MK$. Then the induced morphism $\ttot\alpha:\ttot Y^\bullet\to \ttot X^\bullet$ is compatible with the morphism $U:\CB\to\tB'$. Here, $\ttot X^\bullet$ is considered as a $\tB'$-algebra by Proposition \ref{PtB'}.
\end{thm}
\begin{proof}
(1) follows from Lemma.\ref{Lhomotopytype} and the standard fact that the mapping space between  a projective cofibrant cosimplicial space and a termwise fibrant cosimplicial space is a weak homotopy invariant. (Note that a similar proof of a similar statement where $\tilde\Delta^\bullet$ is replaced with $\Delta^\bullet$ does not work as $\MCK_n(\Delta^\bullet,\dots,\Delta^\bullet)$ may not be Reedy fibrant). The proof of (2) is similar to Proposition \ref{PtB}.
\end{proof}

\subsection{Action of $\MCK$ on $\IM^\bullet$}\label{SSMCKalgebra}
In this subsection we prove the following theorem.
\begin{thm}\label{TMCKalgebra}
There exists an action of the monad $\MCK$ on $\IM^\bullet$ such that the both morphisms $p_0:\IM^\bullet\to \CL^\bullet$ and $\bar q_2:\IM^\bullet\to \thc^\bullet(A',B)$ defined in Definition \ref{Dmorphismscosimplicial} are compatible with the morphism $U:\MCK\to \MK$. Here, $\thc^\bullet(A',B)$ are regared as $\MK$-algebra as explained  below Lemma \ref{LmonoidTHC}.  Hence $p_0$ and $\bar q_2$ induce the morphisms 
\[
\ttot(\CL^\bullet)\xleftarrow{(p_0)_*} \ttot(\IM^\bullet)\xrightarrow{(\bar q_2)_*} \ttot(\thc^\bullet(A',B))
\] which are compatible with the morphism of $A_\infty$-operads $U:\CB\to \tB'$ by Theorem \ref{TCBttot}.
\end{thm}
In subsection \ref{SSproofTmain2}, we prove the two morphisms between $\ttot$ in the above theorem are stable equivalences.
\subsubsection{Sketch of the construction}\label{SSSroughoutlineMCK}
We shall sketch construction of the action of $\MCK$ in Theorem \ref{TMCKalgebra}. We restrict to the components $\CK(0,0,0)$ and $\CK(1,0,0)$. For the one-point set $\CK(0,0,0)$  corresponding to the unique tree of arity $2$ with no marks, the action is given by the product defined by the equation (\ref{EQfailedprod}) in subsubsection \ref{SSSmotivation}. As  is explained there, this product does not satisfy $d^0(xy)=(d^0x)y$. The  action of $\CK(1,0,0)$ gives a homotopy between these two. $\CK(1,0,0)$ is the left line segment in Figure \ref{FCT(1)}. Recall that the author could not define an action of $\MK$ on $\IM^\bullet$ since he could not make the equation
\begin{equation}\label{EQcauseoffail}
H_{\pi_\phi(v),s}=\tpsi_{1,s}\circ H_{\pi_{\phi_1}(v_1),s}
\end{equation}
hold, see \ref{SSSmotivation} for notations. This equation was necessary to make the equation $(d^0x)y=d^0(xy)$ hold. To construct an action  of $\CK(1,0,0)$ is substantially the same as to construct a homotopy between two sides of the equation (\ref{EQcauseoffail}). To do this, we need a parameter space $\tF(1,0,0)$ which is similar to $\D(2)$ used to define $\psi_i$, but includes $\CK(1,0,0)$ instead of $\K(2)$ as a component, and maps
\[
\omega:\tF(1,0,0)\times [0,2]\to \Map(M,M)\ \ \text{and}\ \ z : \tF(1,0,0)\to M
\]  
which satisfy the following conditions (among others). 
\begin{equation}\label{EQMCKsketch}
\omega_s(d^0(xy))=\id_M,\quad \omega_s((d^0x)y)=\tpsi_{1,s},\quad  z(d^0(xy))=\pi_{\phi}(v) ,\ \text{and}\ z((d^0x)y)=\pi_{\phi_1}(v_1)
\end{equation}
for $s\in [0,2]$. Here, the symbols $d^0(xy)$ and $(d^0x)y$ are regarded as elements of $\tF(1,0,0)$ as in Figure \ref{FCT(1)} and the  other components  $(\phi_i, \eps_i,v_i)$, the shaft of $h_i$,  are omitted. As is seen in this, $\omega$ is a homotopy between $\id_M$ and $\tpsi_1$ and $z$ is one between $\pi_{\phi}(v)$ and $\pi_{\phi_1}(v_1)$, so $\{\omega(u)\circ H_{z(u)}\}_{u\in \CK(1,0,0)}$ is a homotopy between two sides of (\ref{EQcauseoffail}). These maps will be constructed similarly to $\psi_i$ by the projection after  parallel transport in a tubular neighborhood.  
With these maps we define an action of $u\in \CK(1,0,0)$ by
\[
\begin{split}
u((\lambda_1,h_1),(\lambda_2,h_2))&=(d^0\lambda_1\cdot \lambda_2, \tilde h)\\
\text{where}\ \tilde h(f_1,\dots, f_{p+q+1})&=(\omega(u)\circ H_{z(u)})^*f_1\cdot h_1(\tpsi_1^*f_2,\dots \tpsi_1^*f_{p+1})\cdot h_2(\tpsi_2^*f_{p+2},\dots\tpsi_2^*f_{p+q+1}),
\end{split}
\]
where $d^0$ in the $\lambda$-part is the coface of $\CL^\bullet$ and we omit a permutation, and the subscript $s$. (Note that the product on $\CL^\bullet$ satisfies  $d^0(\lambda_1\cdot \lambda_2)=(d^0\lambda_1)\cdot \lambda_2$.) For $p=q=0$, using the equations (\ref{EQMCKsketch}), one can see the action of $u=d^0(xy)$ (resp. $(d^0x)y$) is the same as the left (resp. right) formula in (\ref{EQcofacefail}) in subsection \ref{SSSmotivation}, and for general $p,q$ similar claim holds. Thus, the action of $\CK(1,0,0)$ gives a compatiblity homotopy between the action of $\CK(0,0,0)$ and the first coface. For general components, we prepare $\omega$ and $z$ as many as marks and $\tpsi$ as arity. We pull back some of the functions by $\omega\circ H_z$ and others by $\tpsi$. We substitute  the latter functions to elements of $\IM^\bullet$, and take the product of the former functions and the elements substituted to.   The rest of this subsection is mostly devoted to the construction of $\tpsi, \omega, z$, and  $\omega, z$ described here is written as $\omega_0^1, z_0^1$ below. By a technical reason, we construct these maps on $s\in [0,1]$ not $[0,2]$, see next sub-subsection.
\subsubsection{Partition of the action}\label{SSSoutlineMCK}
We will define three actions of monads as follows:
\begin{enumerate}
\item an action of $\MCK$ on $\TIM^\bullet$, $\Omega:\MCK(\TIM^\bullet)\longrightarrow \TIM^\bullet$,
\item an action of $\MK$ on $\athc^\bullet$, $\hUpsilon:\MK(\athc^\bullet)\longrightarrow \athc^\bullet$, and
\item an action of $\MK$ on $\BIM^\bullet$, $\tUpsilon:\MK(\BIM^\bullet)\longrightarrow\BIM^\bullet$.
\end{enumerate}
We shall fix notations. An $\square$-object has the natural action of $\MK$ induced by the map $\K(n)\to *$.  $\athc^\bullet$  regarded as a $\MK$-algebra via the structure of $\square$- object defined in Lemma \ref{LmonoidTHC} is denoted by the same notation $\athc^\bullet$ while we  denote by $\bthc^\bullet$ the $\MK$-algebra $(\athc,\hUpsilon)$. \par
The above three actions satisfy the following conditions: 
\begin{enumerate}
\item $p_0:\TIM^\bullet\to\CL^\bullet$ and $p_1:\TIM^\bullet\to \bthc^\bullet$ are compatible with the morphism $\MCK\to \MK$, 
\item $q_1:\BIM^\bullet\to \bthc^\bullet$ and $q_2:\BIM^\bullet\to \athc^\bullet$ are compatible with the actions of $\MK$.
\end{enumerate}
Here, $p_s$ and $q_s$ denote the evaluations at $s\in [0,2]$. Using these three actions, the action of $\MCK$ on $\IM^\bullet$ is given by
\[
\MCK(\IM^\bullet)\to\MCK(\TIM^\bullet)\times_{\MCK(\bthc^\bullet)}\MCK(\BIM^\bullet)\to\MCK(\TIM^\bullet)\times_{\MK(\bthc^\bullet)}\MK(\BIM^\bullet)\to\TIM^\bullet\times_{\bthc^\bullet}\BIM^\bullet\cong \IM^\bullet ,
\]
where the first morphism is induced by the universal property of fiber products, the second  by the morphism $U:\MCK\to\MK$, and the third  by the three actions $\Omega$, $\hUpsilon$ and $\tUpsilon$. As the morphism $\zeta_*\eta^*$ defined in Definition \ref{Dmorphismscosimplicial} preserves $\square$-object structures, by the above two conditions, it is obvious that the  action of $\MCK$ on $\IM^\bullet$  satisfies the condition of Theorem \ref{TMCKalgebra}. Thus the proof of the theorem is reduced to the construction of $\Omega$, $\hUpsilon$, and $\tUpsilon$.\\
\indent The construction of $\Omega$ is lengthy and occupies most of this subsection (sub-subsections \ref{SSSformulaTIM}, \ref{SSSconditionOMEGA}, and \ref{SSSconstructiontpsi}) and the construction of $\hUpsilon$ and $\tUpsilon$ is very short (sub-subsection \ref{SSSactionBIM}).

\subsubsection{Formula of $\Omega$}\label{SSSformulaTIM}
\textbf{Notation}\ :\ In the rest of paper, we omit the subscript $s\, (\in [0,1])$ of $\tpsi,\,\tpsi_i,\, \omega,\, \omega^j_i,\,\homega$, $\homega^j_i$, and elements of $(\athc^\bullet)^{[0,1]}$, etc, defined below if unnecessary. In a single equation or formula, all the values of the omitted $s$'s are the same.\par
\phantom{a} \par
We put $\D^1_{k_1}=\K(1)\times \M_{k_1}$, and  put 
\[
\begin{split}
\tF_{k_1,\dots, k_n}(m_0,\dots,m_n)
&
(=\tF(m_0,\dots,m_n))
\\
&
:=(U\times \id)^{-1}\D^n_{k_1,\dots, k_n}\subset \CK(m_0,\dots,m_n)\times\M[k_1,\dots, k_n]
\end{split}
\], where $n\geq 1$, $m_0,\dots, m_n\geq 0$, and $U:\CK(m_0,\dots,m_n)\to \K(n)$ is the forgetful morphism and see sub-subsection \ref{SSSoutlineLMT} for $\D^n_{k_1,\dots,k_n}$ with $n\geq 2$. 
 Put $I=[0,1]$. To construct an action of $\MCK$ on $\IM^\bullet$, we will define the following three kind of data for $n\geq 1$: 
\begin{itemize}
\item maps $\tpsi_s=(\tpsi_{i, s}):\D^n\times I\longrightarrow \Map (M,M)^{ n}$ ($i=1,\dots, n$),
\item maps $\omega_s=(\omega^j_{i, s}):\tF(m_0,\dots,m_n)\times I\longrightarrow \Map(M,M)^{ m_{\leq n}}$ ($i=0,\dots, n$, $j=1,\dots ,m_i$)
\item maps $z=(z^j_i):\tF(m_0,\dots, m_n)\longrightarrow M^{ m_{\leq n}}$ ($i=0,\dots, n$, $j=1,\dots, m_i$)
\end{itemize}
Here, $m_{\leq n}=m_0+\cdots +m_n$, and if $m_{\leq n}=0$, the codomains of $\omega_s$ and $z$ are regarded as  one-point spaces.  These maps are also constructed for $n=1$ unlike $\teps$ and $\psi$ for convenience. Similarly to the construction of the $A_\infty$-action on $\LMT$ in subsection \ref{SSactiononLMT}, we first give the formula for the action $\Omega$ in this sub-subsection, then state conditions on these maps which ensure well-definedness of $\Omega$ in next sub-subsection, and finally construct them in sub-subsection \ref{SSSconstructiontpsi}.\par
Using $\tpsi$, $\omega$, and $z$, we shall define the structure map 
$\Omega: \MCK(\TIM^\bullet)\longrightarrow \TIM^\bullet
$. For $u\in \CK(m_0,\dots, m_n)$, $(\lambda_i,h_i)\in\IM^{p_i}_{k_i}$, we put
\[
(\tilde\lambda, \tilde h)=\Omega(u\,;(\lambda_1,h_1),\dots,(\lambda_n,h_n))\in\IM^{m_{\leq n}+p_{\leq n}}_{k_{\leq n}}
\] and will define $\tilde \lambda, \tilde h$ as follows.  We put
\[
\tilde \lambda= \Psi\circ U_n(u\, ;\lambda_1,\dots,\lambda_n)
\] (see subsection \ref{SScosimplicialLMT} for  $\Psi$ and Definition \ref{Dforgetmark} for $U_n$). If $h_i=*$ for some $i$, we set $\tilde h=*$. Suppose $h_1,\dots, h_n\not=*$. Let $(\phi_i,\eps_i,v_i)$ be the shaft of $h_i$ for $1\leq i\leq n$.  Put $d=(u\, ;(\phi_i,\eps_i,v_i)_{i=1}^n)$. If $d \not \in \tF(m_0,\dots, m_n)$, we set $\tilde h=*$. Suppose otherwise. Let $(f_1,\dots, f_{m_{\leq n}+p_{\leq n}})$
be a sequence of elements of $F(M)$. To write down $\tilde h(f_1,\dots,f_{m_{\leq n}+ +p_{\leq n}})$ we rename this sequence as 
\[
(g_0^1,\dots, g_0^{m_0}, f_1^1,\dots, f_1^{p_1},g^1_1,\dots, g_1^{m_1},f_2^1,\dots, f_2^{p_2},\dots,f_n^1,\dots,f_n^{p_n},g_n^1,\dots, g_n^{m_n}).
\]
  We abbreviate $\tpsi_{i}(U\times\id(d))$,  $\omega_{i}^j(d)$, and $z_i^j(d)$ as $\tpsi_{i}$, $\omega_{i}^j$, and $z_i^j$ respectively in the following formulae.
We put
\[
\begin{split}
\bar g_{i}&= (\omega^1_{i}\circ H_{z^1_{i}})^*g_i^1\cdots 
(\omega^{m_i}_{i}\circ H_{z^{m_i}_{i}})^*g_i^{m_i}\qquad \in F(M)\qquad (0\leq i\leq n),\\
\bar h_{i}&= h_{i}((\tpsi_{i})^*f_i^1,\dots, (\tpsi_{i})^*f_i^{p_i})\qquad \in \Gamma(M) \qquad (1\leq i\leq n).
\end{split}
\]
Here $H:M\times I\to M$ is the homotopy fixed in sub-subsection \ref{SSSdefIM}. Of course, these elements depend on $s\in I$ which is omitted.
Then, we put
\begin{equation}\label{ForMCK}
\tilde h(f_1,\dots, f_{m_{\leq n} +p_{\leq n}})=
\sigma^{-1}\cdot(
\bar g_{0}\cdot\bar h_{1}\cdot\bar g_1\cdot \bar h_2 \cdots \bar g_{n-1}\cdot \bar h_{n}\cdot \bar g_{n})^{\triangledown}.
\end{equation}
Here, we use the following notations.
\begin{itemize}
\item The dots denote either of the product on $\Gamma(M)$ or the action of $F(M)$ on $\Gamma(M)$.
\item If $n\geq 2$, $(\bar g_{0}\cdot\bar h_{1}\cdot\bar g_1\cdots \bar h_{n}\cdot \bar g_{n})^{\triangledown}$ is the same element  as $(
\bar g_{0}\cdot\bar h_{1}\cdot\bar g_1\cdots \bar h_{n}\cdot \bar g_{n})$ except for the $\eps$-part being replaced with $\teps(U\times \id(d))$. (By definition of the product on $\Gamma(M)$, the element without $\triangledown$ has $\min\{\eps_1,\dots \eps_n\}$ as its $\eps$-part.) If $n=1$, it is completely the same as the element without $\triangledown$.
\item $\sigma$ is the permutation corresponding to the transposition
\[
h_1,\dots, h_n,g_0,f_1,g_1,\dots f_n,g_n\longmapsto g_0,h_1,f_1,g_1,h_2,f_2,\dots h_n, f_n,g_n.
\]
For example, $g_i$ represents the sequence $g_i^1,\dots, g_i^{m_i}$. $f_i$ is similar abbreviation. More explicity, we set
\[
\sigma (i)=\left\{
\begin{array}{ll}
i+\sum_{l\leq k-1}\lev g_l+\lev  f_l 
&\text{ if $\sum_{l\leq k-1}\lev h_l+1\leq i\leq \sum_{l\leq k}\lev h_l$ for some $k\leq n$,  }\vs{1mm}\\  i-\sum_{k\leq l\leq n}\lev h_l & \text{  if $\sum_{l\leq n}\lev h_l+\sum_{l\leq k-1}(\lev g_l+\lev f_l)+1 \leq i$} \vs{1mm}\\
& \text{$\leq \sum_{l\leq n}\lev h_l+\sum_{l\leq k}(\lev g_l+\lev f_l)$ for some $k\leq n$.  } 
\end{array}\right.
\]
Here $\lev$ denotes the level (the level of a sequence is the sum of those of its elements).
\end{itemize}
 It is obvious that $\tilde h$ has a shaft, see the verification on the coface map in sub-subsection  \ref{SSSdefIM}.
Note that the formula (\ref{ForMCK}) is  somewhat analogous to the unsuccessful definition of an action of $\MK$ on $\IM^\bullet$ given in sub-subsection \ref{SSSmotivation}.
\subsubsection{Conditions which $\tpsi$, $z$, and $\omega$ satisfy}\label{SSSconditionOMEGA}

\textbf{Conditions on $\tpsi$. }
$\tpsi$ satisfies the following six conditions \tpsiuni\  - \tpsiconti\,  for any set of indexes and any element such that the involved notations make sense. Concretely speaking, the condition \tpsiuni\, is satisfied for any $s\in [0,1],\ k_1\geq 0, \ (1, \ve_1)\in \D^1_{k_1}$, and the condition \tpsicomp\, is satisfied for any $n,m \geq 1,\ i,\ s\in I,\  k_1,\dots, k_{n+m-1}\geq 0$ with $1\leq i\leq n$, and \tpsiconti\ is satisfied for any $n\geq 1, 1\leq i\leq n$.  The other conditions, where $\D^n$ is abbreviation of $\D^n_{k_1,\dots, k_n}$, are satisfied for $n\geq 2, s\in I, k_1,\dots, k_n\geq 0, (u\,;\ve_1,\dots,\ve_n)\in\D^n$.
\begin{enumerate}
\item[\tpsiuni] $\tpsi_s(1,\ve_1)=\id_M$ where $1\in \K(1)$ .
\item[\tpsiS] For  $t\in \R$ and $i$ with $1\leq i\leq n$, if $(u\,;\ve_1,\dots,t\times\ve_i,\dots,\ve_n)\in \D^n$,  then  $\tpsi_s(u\,;\dots,\ve_i,\dots)=\tpsi_s(u\,;\dots,t\times\ve_i,\dots)$. Here $t\times \ve_i$ denotes the element $(0\times \phi_i,\eps_i,(t,v_i))$.
\item[\tpsiSigma] For $\sigma_i\in \Sigma_{k_i}$,
$\tpsi_s(u\,;\sigma_1\cdot\ve_1,\dots,\sigma_n\cdot\ve_n)=\tpsi_s(u\,;\ve_1,\dots,\ve_n)$.
\item[\tpsiend] 
$\tpsi_{i,0}(u\,;\ve_1,\dots,\ve_n)=\psi_i(u\,;\ve_1,\dots,\ve_n)$ and $\tpsi_{i,1}(u\,;\ve_1,\dots,\ve_n)=\id_M$ for each $i=1,\dots , n$.
\item[\tpsicomp] The following diagram commutes:
\[
\xymatrix{
\beta_i^{-1}\D^{n+m-1}_{k_1,\dots, k_{n+m-1}}\times I\ar[rr]^{\tau\circ(\alpha_i\times \gamma_i\times\Delta)\qquad\qquad\qquad} \ar[dd]^{\beta_i} &&(\D^n_{k_1,\dots, k_{i-1}, k_{(i)}, k_{i+m},\dots k_{n+m-1}}\times I)\times(\D^m_{k_i,\dots, k_{i+m-1}}\times I)\ar[d]^{\tpsi\times\tpsi}\\
&&\Map(M,M)^n\times\Map(M,M)^m\ar[d]^{Comp_i}\\
 \D^{n+m-1}_{k_1,\dots, k_{n+m-1}}\times I\ar[rr]^{\tpsi} &&Map(M,M)^{n+m-1}.}
\]
Here, $k_{(i)}=k_i+\cdots +k_{i+m-1}$, and $\Delta: I\to I\times I$ denotes the diagonal and $\tau$ denotes the transposition $\D\times\D\times I^{\times 2}\cong \D\times I\times\D\times I$, and $\alpha_i, \beta_i,\gamma_i,$ and $Comp_i$ are  the maps defined in subsection \ref{SSSconditionsLMT}.
\item[\tpsiev] $|\tpsi_{i,s}(u\,;\ve_1,\dots,\ve_n)(y)-y|\leq
6n^2d((v_1,\dots,v_n),\phi(M))$ where $\phi=\phi_1\times\cdots\times\phi_n$, and the minus  and the norm in the left hand side are taken in $\R^{k_0}$, see subsecion \ref{SScohenjones}.
\item[\tpsiconti] Let $u\in \K(n)$ and $(\lambda_1,h_1),\dots, (\lambda_n, h_n)\in \TIM-\{*\}$. Let $\ve_l=(\phi_l,\eps_l,v_l)$ be the shaft of $(\lambda_l,h_l)$ for $1\leq l\leq n$.  When $(u\,;(\lambda_l,h_l)_l)$ runs through the range such that $(u\,;(\ve_l)_l)$ belongs to $\D^n$, the map $(u\,;(\lambda_l,h_l)_l)\mapsto\tpsi_i(u;(\ve_l)_l)\in \Map(M,M)$ is continuous (for the product topology of the domain). 
\end{enumerate}
 The conditions \tpsiS\ -\tpsiev\  are much analogous to those on $\psi$. \tpsiconti\  ensures that $\Omega$ is continuous.\par 
\phantom{a}\par
{\bf Notation}\ :\ In the rest of paper we sometimes write 
\[
\CK(m_0,\dots, m_n)\quad \text{as}\quad \CK(m_l)_{l=0}^n.
\]
To state conditions on $z$ and $\omega$, we need three maps
\[
\begin{split}
\alpha_i:
&
\CK (m_l)_{l=0}^{n_1}\times \CK (m'_l)_{l=0}^{n_2}\times \M[k_1,\dots,k_{n_1+n_2-1}]\longrightarrow \CK (m_l)_{l=0}^{n_1}\times \M[k_1,\dots, k_{(i)},\dots, k_{n_1+n_2-1}] 
\\
\beta_i:
&
\CK (m_l)_{l=0}^{n_1}\times \CK (m'_l)_{l=0}^{n_2}\times\M[k_1,\dots, k_{n_1+n_2-1}]\longrightarrow \CK(\tilde m_l)_{l=0}^{n_1+n_2-1}\times\M[k_1,\dots, k_{n+m-1}]
\\
\gamma_i :
&
\CK (m_l)_{l=0}^{n_1}\times \CK (m'_l)_{l=0}^{n_2}\times
\M[k_1,\dots, k_{n_1+n_2-1}]\longrightarrow \CK (m'_l)_{l=0}^{n_2}\times\M[k_i,\dots, k_{i+n_2-1}]
\end{split}
\]
 defined by the same formulae as $\alpha_i,\beta_i,$ and $\gamma_i$ in sub-subsection \ref{SSSconditionsLMT}. Here  we set $(\tilde m_l)_{l=0}^{n_1+n_2-1}=(m_0,\dots, m_{i-2},m_{i-1}+m'_0, m'_1,\dots, m'_{n_2-1},m_{i}+m'_{n_2}, m_{i+1},\dots ,m_{n_1})$, and $k_{(i)}=k_i+\cdots + k_{i+n_2-1}$. See (the proof of) Lemma \ref{Lcofacialtree2} for the meaning of $\tilde m_l$. We use the forgetful map $U$ to substitute an element of $\CK$ to $\teps$ in the definition of these maps. \par

\phantom{a}\par
\textbf{Conditions on $z$. } $z$ satisfies  the following conditions  for any set of indexes and any element such that the involved notations make sense, which should be understood similarly to the sentence in conditions on $\tpsi$.
\begin{itemize}
\item[\zS] For  $t\in\R$ and $i$ with $1\leq i\leq n$, if $(u\,;\dots,t\times\ve_i,\dots)\in\tF(m_0,\dots,m_n)$, \\
then
$z(u\,;\dots,t\times\ve_i ,\dots)= z(u\,;\ve_1,\dots,\ve_n)$.
\item[\zSigma] For  permutations $\sigma_1,\dots,\sigma_n$, $z(u\,;\sigma_1\ve_1,\dots,\sigma_n\ve_n)=z(u\,;\ve_1,\dots,\ve_n)$.
\item[\zcomp] The following diagram is commutative:
\[
\xymatrix{
\beta_i^{-1}\tF (\tilde m_l)_{l=0}^{n_1+n_2-1}
\ar[r]^{\alpha_i\times \gamma_i\qquad \quad } 
\ar[d]^{\beta_i} &
\tF(m_0,\dots,m_{n_1}) \times\tF(m'_0,\dots, m'_{n_2})
\ar[d]^{T\circ (z\times z)}\\
 \tF(\tilde m_l)_{l=0}^{n_1+n_2-1}
\ar[r]^{\qquad z} &
M^{m_{\leq n_1}+m'_{\leq n_2}}.}
\]
Here, $T$ is the transposition putting the last $m'_{\leq n_2}$ components between the $m_{\leq i-1}$-th and the $m_{\leq i-1}+1$-th components, explicitly given by $T(x_0^1,\dots,x_0^{m_0},\dots,x_{n_1}^{m_{n_1}},y_0^1,\dots, y_{n_2}^{m'_{n_2}})$ 
$=(x_0^1,\dots, x_{i-1}^{m_{i-1}},y_0^1,\dots,y_{n_2}^{m'_{n_2}},x_{i}^1,\dots,x_{n_1}^{m_{n_1}})$.
\item[\zcos] The following diagrams  are  commutative:
\[
\xymatrix{\tF(\dots ,m_i,\dots)
\ar[r]^{ d_i^j}
\ar[d]^z &
\tF(\cdots ,m_i+1,\cdots)
\ar[d]^z \\
M^{m_{\leq n}}
\ar[r]^{ \bar \delta^j_i} &
M^{m_{\leq n}+1}}
\qquad
\xymatrix{\tF(\dots, m_i,\dots)
\ar[r]^{s_i^j}
\ar[d]^z &
\tF(\dots, m_i-1,\dots)
\ar[d]^z \\
M^{m_{\leq n}}
\ar[r]^{ \bar \sigma^j_i} &
M^{m_{\leq n}-1}.}
\]
Here, $\bar \delta^j_i$ is the map repeating $m_{\leq i-1}+j$-th component if $(i,j)\not=(0,0), (n,m_n+1)$, and 
$\bar\delta^0_0$ (resp. $\bar\delta^{m_n+1}_n$) is the map putting $\pi_{\phi_1\times\cdots\times\phi_n}(v_1,\dots, v_n)$ at the first (resp. the last) component, and $\bar \sigma^j_i$ is the map skipping  the $m_{\leq i-1}+j+1$-th component ($m_{\leq -1}=0$).

\item[\zev] $|v-\phi(z_i^j(u\,;\ve_1,\dots, \ve_n))|\leq 4n|\phi|\cdot d(v,\phi(M))$, where $v=(v_1,\dots,v_n)$ and $\phi=\phi_1\times\cdots\times\phi_n$.
\item[\zconti]  Let $u\in \CK(m_0,\dots, m_n)$ and $(\lambda_1,h_1),\dots , (\lambda_n, h_n)\in \TIM-\{*\}$. Let $\ve_l=(\phi_l,\eps_l,v_l)$ be the shaft of $(\lambda_l,h_l)$ for $1\leq l\leq n$.  When $(u\,;(\lambda_l,h_l)_l)$ runs through the range such that $(u\,;(\ve_l)_l)$ belongs to $\tF(m_0,\dots,m_n)$, the map $(u\,;(\lambda_l,h_l)_l)\mapsto z_i^j(u\, ;(\ve_l)_l)\in M$ is continuous. 
\end{itemize}
 \zcomp (resp. \zcos ) is the condition to ensure compatibility of  $\Omega$ with the product (or composition) of the monad $\MCK\circ\MCK\to \MCK$ (resp. the cosimplicial operators). We shall verify that the equations for $z$ in (\ref{EQMCKsketch}) in sub-subsection \ref{SSSroughoutlineMCK} hold under these conditions. Since $d^0(xy)$ corresponds to the  image of the $2$-tree with no marks by $d^0_0$, by substituting $n=2,\ m_{\leq 2}=0,$ and $ i=j=0$ to the left square of \zcos, the first equation $z(d^0(xy))=\pi_{\phi}(v)$ clearly holds. By substituting $n=1,\ m_{\leq 1}=0,$ and $i=j=0$ to the same square, we have $z^1_1(\df; (\phi_1, \eps_1, v_1))=\pi_{\phi_1}(v_1)$ (where we regard $\df\in \CK(1,0)$). With this, by substituting $n_1=2,\ n_2=0,\  m_{\leq 2}=0,\  m'_0=1,\ m'_1=0,$ and $i=1$ to the square in \zcomp, we have
\[
z_1^1((d^0x)y; (\phi_1,\eps_1,v_1),(\phi_2,\eps_2,v_2))=\pi_{\phi_1}(v_1),
\]
which is the second equation for $z$ in (\ref{EQMCKsketch}), where $(d^0x)y$ is the $2$-tree with no marks composed at the first leaf with $\df$.\par 
\phantom{a}\par
\textbf{Conditions on $\omega$.} \  $\omega$ satisfies  the following conditions  for any set of indexes and any element such that the involved notations make sense.
\begin{enumerate}
\item[\omegaS] For  $t\in\R$, if $(u\,;\dots,t\times\ve_i,\dots;s)\in\tF(m_0,\dots,m_n)\times I$, then 
$\omega_s(u\,;\dots,t\times\ve_i,\dots)= \omega_s(u\,;\dots,\ve_i\dots)$.
\item[\omegaSigma] For permutations $\sigma_1,\dots,\sigma_n$, $\omega_s(u\,;\sigma_1\ve_1,\dots,\sigma_n\ve_n)=\omega_s(u\,;\ve_1,\dots,\ve_n)$.
\item[\omegacomp] The following diagram is commutative:
\[
\xymatrix{
\beta_i^{-1}\tF (\tilde m_l)_{l=0}^{n_1+n_2-1}\times  I
\ar[rr]^{\quad \tau\circ(\alpha_i\times \gamma_i\times \Delta)\qquad \qquad} 
\ar[dd]^{\beta_i\times \id_I} &&
(\tF(m_l)_{l=0}^{n_1}\times  I) \times (\tF(m'_l)_{l=0}^{n_2}\times  I)
\ar[d]^{\omega\times (\tpsi\circ\omega)}\\
&&\Map(M,M)^{m_{\leq n_1}}\times\Map(M,M)^{m'_{\leq n_2}}\ar[d]^T\\
 \tF(\tilde m_l)_{l=0}^{n_1+n_2-1}\times I
\ar[rr]^{\omega\qquad \qquad} &&
\Map(M,M)^{m_{\leq n_1}+m'_{\leq n_2}}.}
\]
Here,  for $x\in\tF(m_l)_{l=0}^{n_1}\times I$
and $y\in\tF(m'_l)_{l=0}^{n_2}\times  I$, 
\[
\omega\times (\tpsi\circ\omega)(x,y)=\omega(x)\times (\tpsi_i(x)\circ\omega_0^1(y),\dots,\tpsi_i(x)\circ\omega_{n_2}^{m'_{n_2}}(y)),
\]
where $\tpsi_i(x)$ means $\tpsi_i(U\times \id(x))$. $\tau$ and $T$  are the transpositions defined by the same formula as in \tpsicomp\  and \zcomp\  above.
\item[\omegacos] The following diagrams are commutative:
\[
\xymatrix{\tF(\dots, m_i,\dots)\times I
\ar[r]^{d_i^j}
\ar[d]^\omega &
\tF(\dots, m_i+1,\dots)\times I
\ar[d]^\omega \\
\Map (M,M)^{\times m_{\leq n}}
\ar[r]^{ \delta^j_i} &
\Map (M,M)^{\times m_{\leq n}+1}, }
\quad
\xymatrix{\tF(\dots, m_i,\dots)\times I
\ar[r]^{ s_i^j}
\ar[d]^\omega &
\tF(\dots, m_i+1,\dots)\times I
\ar[d]^\omega \\
\Map (M,M)^{\times m_{\leq n}}
\ar[r]^{ \sigma^j_i } &
\Map (M,M)^{\times m_{\leq n}+1}.}
\]
Here, $\delta^j_i$ is the map repeating the $m_{\leq i-1}+j$-th component  if $(i,j)\not=(0,0), (n,m_n+1)$, $\delta^0_0$ (resp. $\delta^{m_n+1}_n$) is the map putting $\id_M$ on the first component \vspace{1mm}(resp. the last component). $\sigma_i^j$ is the map deleting the $m_{\leq i-1}+j+1$-th component.

\item[\omegaend] The equations 
$\omega_{i,0}^j(u\,;\ve_1,\dots,\ve_n)(z_i^j(u\,;\ve_1,\dots,\ve_n))=\pi_{\phi_1\times\cdots\times\phi_n}(v_1,\dots,v_n)$, and \\
$\omega_{i,1}^j(u\,;\ve_1,\dots,\ve_n)=\id_M$ hold.

\item[\omegaconti]  Let $u\in \CK(m_0,\dots, m_n)$, $(\lambda_1,h_1),\dots , (\lambda_n, h_n)\in \TIM-\{*\}$, and $s\in [0,1]$. Let $\ve_l=(\phi_l,\eps_l,v_l)$ be the shaft of $(\lambda_l,h_l)$ for $1\leq l\leq n$.  When $(u\,;(\lambda_l,h_l)_l\,; s)$ runs through the range such that $(u\,;(\ve_l)_l\,;s)$ belongs to $\tF(m_0,\dots,m_n)\times I$, the map $(u\,;(\lambda_l,h_l)_l\,;s)\mapsto \omega_{i,s}^j(u\, ;(\ve_l)_l)\in \Map(M,M)$ is continuous. 

\end{enumerate}
Similarly to $z$, we can verify the equations for $\omega$ in (\ref{EQMCKsketch}) by using \omegacomp\  and \omegacos. We omit it as we will verify well-definedness of the action entirely later.\par
\phantom{a}\par
Assuming the above conditions, we shall show
\begin{prop}\label{Pwelldefinedomega}
Assuming the conditions \tpsiuni, \tpsiS, \tpsiSigma, \tpsiend, \tpsicomp,  \tpsiconti, \zS, \zSigma, \zcomp, \zcos, \zconti, \omegaS, \omegaSigma, \omegacomp, \omegacos, \omegaend, and \omegaconti, the definition of the action of $\MCK$ on $\TIM^\bullet$ given by the formula (\ref{ForMCK}) in sub-subsection \ref{SSSformulaTIM} is well-defined.
\end{prop}
\begin{proof}
We must show the map
\[
\Omega:\CK(m_0,\dots,m_n)\hotimes\TIM^{p_1}\otimes\cdots\otimes\TIM^{p_n}
\longrightarrow \TIM^{m_{\leq n}+p_{\leq n}}.
\]
is a well-defined map of symmetric spectra. Take elements 
$u\in \CK(m_0,\dots,m_n)$ and $(\lambda_i,h_i)\in\TIM^{p_i}_{k_i}$. 
Put $(\tilde\lambda,\tilde h)=\Omega(u\,;(\lambda_1,h_1),\dots,(\lambda_n,h_n))$. 
We begin by proving $(\tilde\lambda,\tilde h)$ belongs to the pullback $\CL\times_{\athc}\athc^I$.
Let $(\phi_i,\eps_i,v_i)$ be the shaft of $h_i$,  and  put $d=(u\, ;(\phi_1,\eps_1,v_1),\dots, (\phi_n,\eps_n,v_n))$. Let $y\in M$ be a point satisfying $|v-\phi(y)|\leq \teps(U(u)\,; (\phi_i,\eps_i)_{i=1}^n)$. An easy calculation shows that $|z_{i}^j(d)-y|\leq L_{e_0}/2$ holds. This inequality and the condition \omegaend\  imply 
$\omega_{i,0}^j(d)\circ H_{z_{i}^j(d),0}(y)=\pi_\phi(v)$. 
So for a sequence $(g_0^1,\dots, g_0^{m_0}, f_1^1,\dots f_1^{p_1},\dots)$ of elements of $F(M)$, we have
\[
p_0(\tilde h)(g_0,f_1,g_1,\dots, g_n)(y)=\sigma^{-1}\{\, g_0(\pi_\phi(v))\cdot h_{1,0}(\psi_1^*f_1)\cdots h_{n,0}(\psi_n^*f_n)\cdot g_n(\pi_\phi(v))\, \}^\triangledown (y).
\]
Here, $h_{i,0}(\psi_i^*f_l)$ denotes the element 
\[
h_{i,[s=0]}(\psi_i(d)^*f_i^1,\dots, \psi_i(d)^*f_i^{p_i}) \quad \text{with}\quad d=(U(u)\,;(\phi_l,\eps_l,v_l)_{l})),
\]
and  $g_i(\pi_\phi(v))$ denotes $g_i^1(\pi_\phi(v))\cdots g_i^{m_i}(\pi_\phi(v))\in \Sph$. For $y$ with $|v-\phi(y)|>\teps(U(u)\,; (\phi_i,\eps_i)_{i=1}^n)$, the above equation still holds as the both sides  are the basepoint by the  definition of $\TIM^\bullet$. We shall unwind the element $\tilde \rho(\tilde \lambda)$. We write $\lambda_i=(x_{i, 1},\dots, x_{i, p_i},\langle \phi'_i, \eps_i, v'_i\rangle )$. $\phi_i$ (resp. $v_i$) is the map (resp. vector) made from $\phi'_i$ (resp.  $v'_i$) removing the $0$-components. We have 
\[
\df^{m_{n-1}}\dl^{m_n} \lambda_n=(\pi_{\phi'_n}(v'_n),\dots, \pi_{\phi'_n}(v'_n), x_{n, 1},\dots, x_{n, p_n}, \pi_{\phi'_n}(v'_n),\dots, \pi_{\phi'_n}(v'_n))
\] Here, the left (resp. right) sequence of $\pi_{\phi'_n}(v'_n)$ has $m_{n-1}$ (resp. $m_n$) components. Since the $0$-components do not affect the position of the closest point, we have $\pi_{\phi'_n}(v_n')=\pi_{\phi_n}(v_n)$. Similar observation on $\df^{m_{i-1}} \lambda_i$ ($1\leq i\leq n-1$) and the condition \psibp\  imply the following equation
\[
\begin{split}
\tilde \lambda& =(\pi_{\phi}(v),\dots, \pi_{\phi}(v), \psi_1(x_{1, 1}),\dots, \psi_1(x_{1, p_1}), \pi_{\phi}(v),\dots \pi_{\phi}(v),\psi_2(x_{2,1}),\dots, \psi_2(x_{2, p_2}), \\
& \dots, \psi_n(x_{n,1}), \dots, \psi_n(x_{n, p_n}),\pi_{\phi}(v),\dots \pi_{\phi}(v), \langle \phi', \teps, v'\rangle)
\end{split}
\] Here, we write $\teps=\teps(U(u); (\phi'_i,\eps_i)_i)$,  $\phi'=\phi'_1\times \cdots \times \phi'_n$, and $v'=(v'_1,\dots, v'_n)$,  and $i+1$-th sequence of $\pi_\phi(v)$'s has $m_i$ components ($0\leq i\leq n$).  We see that this $\teps$ is the same as $\teps(U(u)\,; (\phi_i, \eps_i)_i)$ by the conditions \epsS\  and \epsSigma. Since $\rho(\langle \phi', \teps, v' \rangle)$ is the product of $\rho(\langle \phi'_l, \teps, v'_l\rangle)$' s, we have the following equations.
\[
\begin{split}
\tilde \rho (\tilde \lambda) &(g_0,f_1,g_1,\dots ,f_n, g_n)
=\rho (\langle \phi', \teps, v'\rangle )\cdot g_0(\pi_\phi(v))\cdot \prod_{i=1}^n \bigl\{ f_i(\psi_i(x_i))\cdot g_i(\pi_\phi(v)) \bigr\} \\
&= \sigma^{-1} \cdot \left [g_0(\pi_\phi(v))\cdot \prod_{i=1}^n\bigl\{ \rho (\langle \phi'_i,\teps, v'_i\rangle )\cdot f_i(\psi_i(x_i))\cdot g_i(\pi_\phi(v))\bigr\}\right ]\\
&=p_0(\tilde h)(g_0,f_1,g_1,\dots, g_n)
\end{split}
\]
Here, $\sigma$ is the permutation used in the formula (\ref{ForMCK}), and \vs{1mm} $f_i(\psi_i(x_i))=f_i^1(\psi_i(x_{i, 1}))\cdots f_i^{p_i}(\psi_i(x_{i, p_i}))$, and $\psi_i$ denotes $\psi_i(d')$ with $d'=(U(u)\,;(\phi'_l,\eps_l,v'_l)_l)$. We have $\psi_i(d')=\psi_i(d)$ by the conditions \psiS\  and \psiSigma, which is  used in the last equation.   We also use $p_0(h_i)=\tilde \rho(\lambda_i)$, which implies $h_{i,0}(f^1,\dots ,f^{p_i})=\rho(\langle \phi'_i,\eps_i,v'_i\rangle )f^1(x_{i, 1})\cdots f^{p_i}(x_{i, p_i})$, in the last equation. 
Thus, we have verified the assertion.  It is obvious that $(\tilde\lambda,\tilde h)$ actually belongs to the sub-spectrum $\TIM\subset\CL\times_{\athc}\athc^I$ (see the verification on the coface in sub-subsection \ref{SSSdefIM}).  The equivariance with spheres and symmetric groups follows from conditions \tpsiS, \zS, \omegaS, \tpsiSigma, \zSigma, and  \omegaSigma. \\
\indent We shall check the equation 
\begin{equation}\label{EQcofaceomega}
\Omega(u\,;\dots \df(\lambda_i,h_i),\dots)=\Omega(u\circ_i \df;
\dots,(\lambda_i,h_i),\dots).
\end{equation} Unwinding the formula of $\Omega$, we see it is enough to show 
\begin{equation}\label{EQunwind}
H_{\pi_{\phi_i}(v_i)}^*\tpsi_{i}(U(u))^*f= 
(\omega_{i-1}^{m_{i-1}+1}(u\circ_i\df)\circ H_{z_{i-1}^{m_{i-1}+1}(u\circ_i\df)})^*f.
\end{equation} Here, as before, we omit the components of an element of $\tF(m_0,\dots, m_n)$ other than $u$. (In view of the formula (\ref{ForMCK}),  $f$ is presented as $f_i^1$ in the left hand side of (\ref{EQcofaceomega}), and as $g_{i-1}^{m_{i-1}+1}$ in the right hand side.) By the conditions \omegacomp, \omegacos, \zcomp, and \zcos, we have
\[
\omega_{i-1}^{m_{i-1}+1}(u\circ_i\df)=\tpsi_{i}(U(u))\circ \id,\quad 
z_{i-1}^{m_{i-1}+1}(u\circ_i\df)=\pi_{\phi_i}(v_i).
\]
So the equation (\ref{EQunwind}) holds. Similarly we see the corresponding equation holds for $\dl$. Thus, we have checked that $\Omega$ induces a morphism of symmetric spectra 
$\MCK(\TIM)^p\longrightarrow \TIM^p.
$
 The commutativity between $\Omega$ and cofaces and codegeneracies follows from  the conditions \zcos\ and \omegacos\ obviously. For example, for the coface which uses the product of elements of $F(M)$ named $g_i^j,\ g_i^{j+1}$, the operation on $z_i^j$ which should corresponds to this coface  is to repeat it. This matches with the condition \zcos.\\
\indent Finally, we shall show commutativity with the monad product $\MCK\circ\MCK\longrightarrow \MCK$. It is enough to show 
\[
\begin{split}
\Omega(u_1\circ_iu_2&;(\lambda_1,h_1),\dots,(\lambda_{i-1},h_{i-1}),(\lambda'_1,h'_1),\dots,(\lambda'_{l},h'_l),(\lambda_{i},h_i),\dots,  (\lambda_{n-1},h_{n-1}))\\
&=\Omega(u_1;(\lambda_1,h_1),\dots, (\lambda_{i-1},h_{i-1}),
\Omega( u_2;(\lambda'_1,h'_1),\dots, (\lambda'_{l},h'_{l})),(\lambda_{i},h_i),\dots,(\lambda_{n-1},h_{n-1}))
\end{split}
\]
for numbers $n,l\geq 1$ and elements $u_1\in\CK(m_0,\dots,m_n)$, $u_2\in\CK(m'_0,\dots,m'_l)$, $(\lambda_1,h_1)$, $\dots$, $(\lambda_{n-1},h_{n-1}),(\lambda_1',h_1'),\dots,(\lambda'_l,h_l')\in\IM$.
Let 
\[
\begin{split}
g_0^1,\dots, g_0^{m_0}, f_1^1,\dots, f_1^{p_1},\dots, g_{i-1}^{m_{i-1}}, & \\
\tilde g_0^1,\dots, \tilde g_0^{m'_0},& \tilde f_1^1,\dots, \tilde f_1^{q_1},\dots \tilde g_l^{m'_l},\\
&g_i^1,\dots, g_i^{m_i},\dots, f_{n-1}^1,\dots, f_{n-1}^{p_{n-1}},g_n^1,\dots, g_n^{m_n}
\end{split}
\]
be $m_{\leq n}+m'_{\leq l}+ p_{\leq n-1}+q_{\leq l}$ elements of $F(M)$ ($p_k$ and  $q_k$ are cosimplicial degrees of  elements $(\lambda_{k},h_k)$ and $(\lambda'_k,h'_k)$ respectively).
Unwinding the formula of $\Omega$, we have
\vspace{-1.5\baselineskip}
\begin{spacing}{1.3}
\[
\begin{split}
\Omega(u_1;
\dots, 
\Omega( u_2;(\lambda_1'
&
,h'_1),\dots, (\lambda'_{l},h'_{l})),\dots
)(g_0^1,\dots, g_n^{m_n})\\
=\cdots h_{i-1}(\tpsi_{i-1}(u_1)^*f_{i-1})
&
\cdot (\omega_{i-1}(u_1)\circ H_{z_{i-1}(u_1)})^*g_{i-1}
\\
\cdot 
[\Omega( u_2;(\lambda'_1,h'_1)
&
,\dots, (\lambda'_{l},h'_{l}))(\tpsi_i(u_1)^*\tilde g_0^1,\dots, \tpsi_i(u_1)^*\tilde g_l^{m'_l})]
\\
&
\cdot (\omega_{i}(u_1)\circ H_{z_{i}(u_1)})^*g_{i}\cdot h_{i}(\tpsi_{i}(u_1)^*f_{i})\cdots
\\
=\cdots 
h_{i-1}(\tpsi_{i-1}(u_1)^*f_{i-1})
&
\cdot (\omega_{i-1}(u_1)\circ H_{z_{i-1}(u_1)})^*g_{i-1}
\\
\cdot 
[(\omega_0(u_2)\circ H
&
_{z_0(u_2)}
)^*\tpsi_i(u_1)^*
\tilde g_0\cdot h'_1(\tpsi_1(u_2)^*
\tpsi_i(u_1)^*\tilde f_1)\cdots (\omega_l(u_2)\circ H_{z_l(u_2)})^*\tpsi_i(u_1)^*\tilde g_l]
\\
&
\cdot (\omega_{i}(u_1)\circ H_{z_{i}(u_1)})^*g_{i}\cdot h_{i}(\tpsi_{i}(u_1)^*f_{i})\cdots
\end{split}
\]
\end{spacing}
Here, we use the usual abbreviation, i.e., we write
\[
\begin{split}
h_{j}(\tpsi_{j}(u_1)^*f_{j})
&
=h_{j}(\tpsi_{j}(u_1)^*f_{j}^1\dots \tpsi_{j}(u_1)^*f_{j}^{m_{j}})
\\
(\omega_{j}(u_1)\circ H_{z_{j}(u_1)})^*g_{j}
&
=(\omega^1_{j}(u_1)\circ H_{z^1_{j}(u_1)})^*g_{j}^1\cdots (\omega^{m_{j}}_{j}(u_1)\circ H_{z^{m_{j}}_{j}(u_1)})^*g_{j}^{m_{j}},
\end{split}
\]
and also omit $U$ for notational simplicity.  On the other hand, by definition, we have
\vspace{-1.5\baselineskip}
\begin{spacing}{1.3}
\[
\begin{split}
\Omega(u_1\circ_i u_2;(\lambda_1,h_1),
\dots, 
&
(\lambda_{n-1},h_{n-1})
)(g_0^1,\dots, g_n^{m_n})
\\
=\cdots 
h_{i-1}(\tpsi_{i-1}(u_1\circ_i u_2)^*f_{i-1})
&
\cdot (\omega_{i-1}(u_1\circ_i u_2)\circ H_{z_{i-1}(u_1\circ_i u_2)})^*g_{i-1}
\\
\cdot 
[(\omega_{i-1}(u_1\circ_iu_2)\circ H
&
_{z_{i-1}(u_1\circ_iu_2)}
)^*
\tilde g_0\cdot h'_1(\tpsi_i(u_1\circ_iu_2)^*
\tilde f_1)\cdots (\omega_{i+l-1}(u_1\circ u_2)\circ H_{z_{i+l-1}(u_1\circ_iu_2)})^*\tilde g_l]
\\
&
\cdot (\omega_{i+l-1}(u_1\circ_iu_2)\circ H_{z_{i+l-1}(u_1\circ_iu_2)})^*g_{i}\cdot h_{i}(\tpsi_{i+l}(u_1\circ_iu_2)^*f_{i})\cdots
\end{split}
\]
\end{spacing}
In view of these expansions, 
using  the conditions \tpsicomp, \omegacomp\  and \zcomp, we easily see the desired equation holds (actually, these conditions are deduced from this equation). Continuity of $\Omega$ follows from \tpsiconti, \omegaconti, and \zconti. (We do not need any condition to ensure continuity at the basepoint. It follows from the formula (\ref{ForMCK}) and the definitions of the product on $\Gamma(M)$ and the action of $F(M)$ on $\Gamma(M)$ whatever $\tpsi, z,$ and $\omega$ are. )
\end{proof}

\subsubsection{Construction of $\tpsi$, $z$, and $\omega$}\label{SSSconstructiontpsi}
Construction of $\tpsi$, $z$, and $\omega$ is much similar to that of $\psi$.\\

\textbf{Construction of $\tpsi$.} This is completely analogous to the construction of $\psi$ given in sub-subsection \ref{SSSconstructionLMT}, so we only give the formulae and omit verification of the conditions of $\tpsi$. When $n=1$, $\tpsi_s(pt\,;\ve)=\id_M$.\\
When $n=2$, we put
\[
\tpsi_{i,s}(\ve_1,\ve_2)(y)=\pi_\phi(
(1-s)(v-\phi(\pi_{\phi_i}(v_i)))+\phi(y)).
\]
For general $n$, suppose we have constructed for $\leq n-1$. Then, we define $\tpsi_{i,s}$ on $\delta\D^n$ in the manner completely analogous to $\psi_i$,  and put
\[
\tpsi_{i,s}(tu\,;\ve_1,\dots, \ve_n)(y)
=\pi_\phi \bigl ((1-t)\{\, (1-s)(v-\phi(\pi_{\phi_i}(v_j))+\phi(y)\,\}+t\{\,\tpsi_{i,s}(u)(y)\,\}\bigr).
\]
for $u\in \partial \K(n),\ t\in [0,1]$.\par
\phantom{a}\par
\textbf{Construction of $z$.} We give the set 
\[
\{(n,m_0,\dots, m_n)| n\geq 1, m_0,\dots, m_n\geq 0\}
\] the lexicographical order. Construction of $z$ proceeds on by induction on this ordered set. For $n=1$, we put $z_0^j=z_1^j=\pi_\phi(v)$.\\
Suppose we have constructed $z_i^j$ for $(l,m'_0,\dots,m'_l)$ such that $(l,m'_0,\dots,m'_l)<(n,m_0,\dots,m_n)$.\\
\indent We define a diagram
$B_{\tF}:\CT(m_0,\dots,m_n)_{1,2}\longrightarrow \CG$
as follows. For an element $T\in\CT(m_0,\dots,m_n)_{1,2}$ we denote by $\beta^T$ the following composition:
\[
\begin{split}
B^c(T)\times\M[k_1,\dots, k_n]&\xrightarrow{Comp \times \id}\CK(T)\times \M[k_1,\dots, k_n]\\
&\xrightarrow{U\times\id}\K(n)\times \M[k_1,\dots, k_n].
\end{split}
\]  
See sub-subsection \ref{SSSCK} for $\CT(m_0,\dots,m_n)_{1,2}$ and $B^c$. $Comp$ denotes the partial composition of $\CK$ indicated by $T$. We put $B_{\tF}(T)=(\beta^T)^{-1}\D^n$. The morphism $B_{\tF}(T)\to B_{\tF}(T')$ corresponding to $T\leq T'$ is given by the product of the corresponding morphism of $B^c$ and the identity. Then we easily see there is a natural identification:
\[
\colim_{\CT(m_0,\dots, m_n)_{1,2}}B_{\tF}\cong \delta\tF:=\{(u\,;\ve_1,\dots,\ve_n)\in \tF(m_l)_{l=0}^n| u\in \partial\CK(m_0,\dots,m_n)\}
\]
Using this identification, we define $z^j_i$ on $\delta\tF$. More precisely, for $T$ of codimension one, $z:B_{\tF}(T)\to M^{m_{\leq n}}$ by the composition through the top right corner of  the diagram of the condition \zcomp , and by exactly analogous way to the case of $\psi$, we can prove these maps fit together to define a natural transformation $B_{\tF}\Rightarrow M^{m_{\leq n}}$, which is omitted. Thus we get $z$ on $\delta\tF$. 
For a point $tu\in \CK(m_0,\dots, m_n)$ ($t\in  [0,1]$ and $u\in\CK(\partial T(m_0,\dots,m_n))$), we put 
\[
\begin{split}
z^j_i(tu\,;\ve_1,\dots,\ve_n )&=\pi_\phi(\hz^j_i(tu\,;\ve_1,\dots \ve_n ))\vs{1mm}\\
\hz^j_i(tu\,;\ve_1,\dots \ve_n )&=
\left\{
\begin{array}{ll}
(1-2t)(v_1,\dots, v_n)+2t\phi\circ z^j_i(u\,;\ve_1,\dots,\ve_n )
& (0\leq t\leq 1/2)\\
 \phi\circ z^j_i(u\,;\ve_1,\dots,\ve_n )
& (1/2\leq t\leq 1)
\end{array}\right.
\end{split}
\]
where $\phi=\phi_1\times\cdots\times \phi_n$. (The condition \zev\, is used to ensure that $\hz_i^j$ belongs to the domain of $\pi_\phi$, see below.) We will see this formula satisfies the conditions on $z$ below. \qed \\

\textbf{Construction of $\omega$} : Construction of $\omega$ is slightly more delicate. We actually construct a map
\[
\homega=(\homega_i^j):\tF_{k_1,\dots,k_n}(m_0,\dots,m_n)\times I\longrightarrow \Map(M,\R^{K})^{m_{\leq n}}\qquad (K=k_{\leq n})
\]
by induction and put 
\begin{equation}
\omega_i^j(u\,;\ve_1,\dots,\ve_n)=\pi_\phi\circ (\homega_i^j(u\,;\ve_1,\dots,\ve_n)) \label{Foromega}.
\end{equation}
This is because the value $|\omega(u\,;\ve_1,\dots,\ve_n)(y)-y|$, which we want to estimate would diverge if we try to construct $\omega$ by induction. We first list the conditions which $\homega$ satisfies. ( The condition \homegaev\, below ensures that the values of $\homega_i^j$ belong to a tubular neighborhood of $\phi(M)$ so we can compose $\pi_\phi$.) It is clear that the following conditions on $\homega$ implies the conditions \omegaS \ - \omegaconti\  on $\omega$ defined by the formula (\ref{Foromega}).\\
\phantom{a}\par
\textbf{Conditions on $\homega$.}    for any set of indexes and any element such that the involved notations make sense.
\begin{enumerate}
\item[\homegaS] For  $t\in\R$, if $(u\,;\dots,t\times \ve_i,\dots;s)\in\tF(m_0,\dots,m_n)\times I$, \\ 
then 
$\homega_s(u\,;\dots,t\times \ve_i,\dots)- \homega_s(u\,;\dots,\ve_i\dots)\in \R e_{k_{\leq i-1}+1}$. Here $e_{j}$ is the unit vector  whose unique non-zero components is the $j$-th component.
\item[\homegaSigma] For permutations $\sigma_1,\dots,\sigma_n$, $\homega_s(u\,;\sigma_1\ve_1,\dots,\sigma_n\ve_n)=(\sigma_1\times \cdots \times \sigma_n)\cdot\homega_s(u\,;\ve_1,\dots,\ve_n)$. Here, $\sigma_1\times \cdots \times \sigma_n$ is the component-wise permutation on $\R^K=\R^{k_1}\times\cdots\times \R^{k_n}$. 
\item[\homegacomp] If $n_1\geq 2$, the following diagram is commutative:
\[
\xymatrix{
\beta^{-1}\tF (\tilde m_l)_{l=0}^{n_1+n_2-1}\times  I
\ar[rr]^{\tau\circ(\alpha_i\times \gamma_i\times \Delta) \qquad} 
\ar[dd]^{\beta_i\times \id_I} &&
(\tF(m_l)_{l=0}^{n_1}\times  I) \times (\tF(m'_l)_{l=0}^{n_2}\times  I)
\ar[d]^{\homega\times (\tpsi'\circ\omega)}\\
&&\Map(M,\R^K)^{m_{\leq n_1}}\times\Map(M,\R^{K})^{m'_{\leq n_2}}\ar[d]^T\\
 \tF(\tilde m_l)_{l=0}^{n_1+n_2-1}\times I
\ar[rr]^{\homega\qquad \qquad} &&
\Map(M,\R^K)^{m_{\leq n_1}+m'_{\leq n_2}}}
\]
Here,
\begin{itemize}
\item $(\tilde m_l)_{l=0}^{n_1+n_2-1}$, $\tau$ and $T$  are the transpositions defined by the same formula in \zcomp \ above.

\item   For $x=(u_1;\ve_1',\dots,\ve_{n_1}';s_1)\in\tF(m_l)_{l=0}^{n_1}\times I$
and $y\in\tF(m'_l)_{l=0}^{n_2}\times  I$, we put
\[
\omega\times (\tpsi'\circ\omega)(x,y)=\omega(x)\times (\phi'\circ[\tpsi_i(x)]\circ\omega_0^1(y),\dots,\phi'\circ[\tpsi_i(x)]\circ\omega_{n_2}^{m'_{n_2}}(y)),
\]
 where $\phi'=\phi_1'\times\cdots\times \phi_{n_1}'$ with $\phi'_j$ being the $\phi$-part of $\ve'_j$, and we omit $U\times \id$ in the entry of $\tpsi_i$, and $\omega_i^j$ is given by the formula (\ref{Foromega}).
\item If  the subscripts of $\tF (\tilde m_l)_{l=0}^{n_1+n_2-1}$ in the diagram are $k_1,\dots,k_{n_1+n_2-1}$,  those of $\tF(m_l)_{l=0}^{n_1}$  (resp. $\tF(m'_l)_{l=0}^{n_2}$) are  $k_1,\dots, k_{i-1}, k_{(i)},k_{i+n_2},\dots, k_{n_1+n_2-1}$ (resp. $k_i,\dots, k_{i+n_2-1}$) by definition of $\alpha_i$ (resp. $\gamma_i$).  So all the $K$ in the diagram is understood to be the same $K=k_{\leq n_1+n_2-1}$.
\end{itemize}
If $n_1=1$, the following diagram is commutative:
\[
\xymatrix{
\beta^{-1}\tF (\tilde m_l)_{l=0}^{n_2}\times I
\ar[rr]^{\tau\circ(\alpha_i\times \gamma_i\times \Delta ) \qquad \quad  } 
\ar[dd]^{\beta_i\times \id_I} &&
(\tF(m_0,m_1)\times I) \times (\tF(m'_k)_{k=0}^{n_2}\times I)
\ar[d]^{\homega\times \homega}\\
&&\Map(M,\R^K)^{m_0+m_{1}}\times\Map(M,\R^K)^{m'_{\leq n_2}}\ar[d]^T\\
 \tF(\tilde m_k)_{k=0}^{n_1+n_2-1}\times I
\ar[rr]^{\homega\qquad \qquad} &&
\Map(M,\R^K)^{m_0+m_1+m'_{\leq n_2}}},
\]
where the maps are defined similarly to the case $n_1\geq 2$, for example,  $\homega\times \homega (x,y)=(\homega_i^j(x),\homega_{i'}^{j'}(y))$.
\item[\homegacos] The following diagrams are commutative:
\[
\xymatrix{\tF(\dots, m_i, \dots )\times I
\ar[r]^{ d_i^j\quad}
\ar[d]^\homega &
\tF(\dots, m_i+1,\dots)\times I
\ar[d]^\homega \\
\Map (M,\R^K)^{\times m_{\leq n}}
\ar[r]^{ \delta^j_i\quad} &
\Map (M,\R^K)^{\times m_{\leq n}+1}}
\quad
\xymatrix{\tF(\dots, m_i,\dots)\times I
\ar[r]^{s_i^j\quad}
\ar[d]^\homega &
\tF(\dots, m_i+1,\dots)\times I
\ar[d]^\homega \\
\Map (M,\R^K)^{\times m_{\leq n}}
\ar[r]^{ \sigma^j_i\quad} &
\Map (M,\R^K)^{\times m_{\leq n}+1}}
\]
Here, $\delta^j_i$ and $\sigma^j_i$ are the maps defined by the same formula in \omegacos.

\item[\homegaend] 
$\homega_{i,0}^j(u\,;\ve_1,\dots,\ve_n)(z_i^j(u\,;\ve_1,\dots,\ve_n))$ belongs to the fiber over $ \pi_{\phi}(v_1,\dots,v_n)$ of the projection $\pi_{\phi}:\nu_{\eps}(\phi)\to M$ for some $\eps <L_{\phi\circ e_0}$, and $\homega_{i,1}^j(u\,;\ve_1,\dots,\ve_n)=\phi(=\phi|_M)$. Here, $\phi=\phi_1\times \cdots \times \phi_n$.
\item[\homegaev] $|\homega(u\,;\ve_1,\dots,\ve_n)(y)-\phi(y)|\leq 10^n|\phi|\cdot d(v,\phi(M))$ for each $y\in M$, where $\phi=\phi_1\times \cdots \times \phi_n$.

\item[\homegaconti]  Let $u\in \CK(m_0,\dots, m_n)$, $(\lambda_1,h_1),\dots (\lambda_n, h_n)\in \TIM-\{*\}$, and $s\in [0,1]$. Let $\ve_l=(\phi_l,\eps_l,v_l)$ be the shaft of $(\lambda_l,h_l)$ for $1\leq l\leq n$ and put $\phi=\phi_1\times \cdots\times \phi_n$.  When $(u\,;(\lambda_l,h_l)_l\,; s)$ runs through the range such that $(u\,;(\ve_l)_l\,;s)$ belongs to $\tF(m_0,\dots,m_n)\times I$, the map $(u\,;(\lambda_l,h_l)_l\,;s)\mapsto \pi_{\phi}\circ (\homega_{i,s}^j(u\, ;(\ve_l)_l))\in \Map(M,M)$ is continuous. 

\end{enumerate}


\textbf{Construction of $\homega$.} The construction proceeds on by induction on the set $\{(n,m_0,\dots,m_n)\}$ with the lexicographical order. For $n=1$, we put $\homega^j_{0,s}(u\,;\ve_1)=\homega^j_{1,s}(u\,;\ve_1)= \phi_1$, where $\phi_1$ is the $\phi$-part of $\ve_1$.\\
Suppose we have constructed $\homega_i^j$ for $(l,m'_0,\dots,m'_l)<(n,m_0,\dots,m_n)$. We  define $\homega_i^j(u\,;\ve_1,\dots,\ve_n)$ for $(u\,;\ve_1,\dots,\ve_n)\in\delta\tF$  similarly to construction of $\psi$ or $z$, using the composition through the right top corner of the diagram in \homegacomp. Verification of well-definedness on $\delta\tF$ is completely analogous to verification of well-definedness of $\bpsi$ so we omit it. 
We put 
\[
\begin{split}
\homega&^j_{i,s}(tu)(y)=\\
&\left\{
\begin{array}{ll}
 (1-s)\{\, v+\phi (y-z^j_i(tu))\,\}+s\{\, (1-t)\phi(y)+t\,\homega_{i,s}^j(u)(y)\, \}
& (0\leq t\leq 1/2)\vs{1mm}\\
(1-s)\{\,(2-2t)(v+\phi (y-z^j_i(u)))+(2t-1)\,\homega_{i,s}^j(u)(y)\,\}
+s\{\,(1-t)\phi(y)+t\,\homega_{i,s}^j(u)(y)\, \}
& (1/2\leq t\leq 1)\\
\end{array}\right.
\end{split}
\]
for $u\in \delta\tF$ and $t\in[0,1]$. 
Here, $\phi=\phi_1\times\cdots\times\phi_n$, $v=(v_1,\dots,v_n)$ and $\homega(u)$ and $\homega(tu)$ denotes $\homega(u\,;\ve_1,\dots,\ve_n)$ and $\homega(tu\,;\ve_1,\dots,\ve_n)$, respectively. We will see this formula satisfies the conditions on $\homega$ below.　\qed \\

\textbf{Verification on construction of $z$ and $\homega$} : (On $z$) The condition \zcomp\  is satisfied by definition. Verification of the conditions \zS,  \zSigma\ is trivial routine work. \\
\indent  We shall verify the condition \zev\ and the claim that 
$\hz_i^j$ belongs to the domain of $\pi_\phi$. We shall show these conditions and  the claim by induction on the same poset as we used in construction of $z$. Let $u\in\partial\CK(m_0,\dots, m_n)$ be an element. By definition, there exist two elements $u_1$, $u_2$ which are not the  unit, such that  $u=u_1\circ_{i_0}u_2$ for some $i_0$. We first show $z_i^j(u)$ satisfies the inequality of the condition \zev.
By construction, $z_i^j(u)=z_{i_1}^{j_1}(u_1)$ or $z_i^j(u)=z_{i_2}^{j_2}(u_2)$  for some $(i_1,j_1)$ or $(i_2,j_2)$. In the former case, we observe
\[
|v-\phi (z_i^j(u))|=|v-\phi (z_i^j(u_1))|\leq 4n_1|\phi|\cdot d(v,\phi(M))\leq 4n|\phi|\cdot d(v,\phi(M))
\]
by inductive hypothesis. In the latter case, 
\[
|v-\phi (z_{i_2}^{j_2}(u_2))|\leq 
|v-\phi(\pi_\phi(v))|
+|\phi(\pi_\phi(v))-\phi(\pi_{\phi_{(i_0)}}(v_{(i_0)}))|
+|\phi(\pi_{\phi_{(i_0)}}(v_{(i_0)}))-\phi (z_{i_2}^{j_2}(u_2))|
\]
Here, $\phi_{(i_0)}=\phi_{i_0}\times\cdots\times \phi_{i_0+n_2-1}$, and 
$v_{(i_0)}=(v_{i_0},\dots, v_{i_0+n_2-1})$.
We have
\[
\begin{split}
|\phi(\pi_{\phi_{(i_0)}}(v_{(i_0)}))-\phi (z_{i_2}^{j_2}(u_2))|&
\leq \frac{|\phi |}{|\phi_{(i_0)}|}\cdot
(|\phi_{(i_0)}(\pi_{(i_0)}(v_{(i_0)}))-v_{(i_0)}|+|v_{(i_0)}-\phi_{(i_0)}(z_{i_2}^{j_2}(u_2))|) \\
&\leq \frac{|\phi |}{|\phi_{(i_0)}|}\cdot [d(v,\phi(M))+4 n_2|\phi_{(i_0)}|\cdot d(v,\phi(M))],\\
|\phi(\pi_\phi(v))-\phi(\pi_{\phi_{(i_0)}}(v_{(i_0)}))|&
\leq \frac{|\phi |}{|\phi_{(i_0)}|}\cdot 
(|\phi_{(i_0)}(\pi_{\phi_{(i_0)}}(v_{(i_0)}))-v_{(i_0)}|+|v_{(i_0)}-\phi_{(i_0)}(\pi_\phi(v)) |)\\
&\leq 2\frac{|\phi |}{|\phi_{(i_0)}|}\cdot d(v,\phi(M)).
\end{split}
\]
So we have 
\[
|v-\phi (z_{i_2}^{j_2}(u_2))|\leq \left( 1+3\frac{|\phi |}{|\phi_{(i_0)}|}+4n_2|\phi| \right) \cdot d(v,\phi (M)).
\]
As $|\phi_{(i_0)}|\geq 1$, if $n_2< n$, we see $|v-\phi (z_{i_2}^{j_2}(u_2))|\leq 4n|\phi|\cdot d(v,\phi (M))$. If $n_2=n$, as $v=v_{(i_0)}$ and $\phi=\phi_{(i_0)}$, clearly, $|v-\phi (z_{i_2}^{j_2}(u_2))|\leq 4n|\phi|\cdot d(v,\phi (M))$. Thus the inequality in the condition \zev\, is satisfied for $u\in \partial \CK(m_0,\dots,m_n)$.\\
Now we turn to prove that $\hz_i^j$ belongs to the domain of $\pi_\phi$. It is enough to prove  $d(\hz_i^j,\phi(M))<|\phi|L_{e_0}$.
By definition of $\hz^j_i$, 
we have 
$|v-\hz^j_i(tu) |\leq 2t|v-\phi(z^j_i(u))|$  for $t\in [0,1]$ and $u\in\partial \CK(m_0,\dots,m_n)$. So
\[
\begin{split}
d(\hz_i^j(tu),\phi(M))
&
\leq |\hz_i^j(tu)-\phi(\pi_\phi(v))|
\leq |\hz_i^j(tu)-v|+|v-\phi(\pi_\phi(v))|
\\
&
\leq  2t|v-\phi(z_i^j(u))|+d(v,\phi(M))
\leq  (8tn|\phi| +1)d(v,\phi(M))
\\
&
\leq (8tn|\phi|+1)\teps
\leq |\phi|L_{e_0}
\end{split}
\]
by the conditions \epsev\, and \zev\, for $u\in\partial\CK(m_0,\dots,m_n)$. We shall prove the condition \zev \  for general element $tu$. We have
\[
|\hz^j_i(tu)-\phi(z_i^j(tu))|\leq |\hz^j_i(tu)-\phi(z_i^j(u))|\leq (1-2t)|v-\phi(z^j_i(u))|.
\]
So
\[
\begin{split}
|v-\phi (z^j_i(tu))|&\leq |v-\hz^j_i(tu) |+|\hz^j_i(tu)-\phi(z_i^j(tu))|\\
&\leq |v-\phi(z^j_i(u))|\leq 4n|\phi|\cdot d(v,\phi (M)).
\end{split}
\]
\indent We shall prove the condition \zcos \ is satisfied. We consider coface maps $\{d_k^j\}$. We first consider the case $(k,j)\not =(0,0), (n,m_n)$. Let $T=T_1\circ_iT_2$ be a cofacial $n$-tree of codimension 1 with $T_1$, $T_2$ indecomposable. $d_k^j(T)$ is a composition of $d_{k'}^{j'}(T_1)$ and $T_2$ or $T_1$ and $d_{k''}^{j''}(T_2)$ with appropriate $(k',j')$ or $(k'',j'')$. Consider the case of $d_{k}^j(T)=d_{k'}^{j'}(T_1)\circ_i T_2$. Under the presentation: $\CK(m_0,\dots,m_n)\cong Cone(\colim B^c)$, $d_k^j$ on $B^c(T)$ is equal to the following composition.
\[
B^c(T_1\circ_iT_2)=\CK(T_1)\times\CK(T_2)\xrightarrow{d_{k'}^{j'}\times \id}
\CK(d_{k'}^{j'}(T_1))\times\CK(T_2)\subset \CK(T'_1)\times\CK(T_2),
\]
where $T'_1$ is an  indecomposable cofacial tree such that $d_{k'}^{j'}(T_1)\leq 
T'_1$.  Using this expression and inductive hypothesis, we see $z$ commutes 
with $d^j_k$ on $\delta\tF$. For general point $t\cdot u\in Cone(\colim B^c)$
, we have $d^j_k(t\cdot u)=td^j_k(u)$. By the construction, if 
$z_i^j(u)=z_i^{j+1}(u)$, then $z_i^j(tu)=z_i^{j+1}(tu)$ so the compatibility 
with $d_k^j$ holds on the whole $\tF$. The case of 
$d_k^j(T)=T_1\circ_id_{k''}^{j''}(T_2)$ is similarly verified. If 
$(k,j)=(0,0)$ or $(n,m_n+1)$, $d_k^j(T)=d\circ T$ where $d=\df$ or $\dl$. So 
the compatibility with $d_k^j$ follows from the condition \zcomp.   The condition \zconti\  is proved completely analogously to Lemma \ref{Lshaftconti}.\\

(On $\homega$) The condition \homegacomp\ is satisfied by definition. The verification of \homegacos\ is completely analogous to that of \zcos. \homegaend\ is easily verified using the formula of $z$ given in the construction. Verification of \homegaS\  and \homegaSigma\  is trivial.\\
\indent  We shall verify the condition \homegaev. We put
\[
k_n^{\omega}=10^n|\phi |,\qquad k_n^\psi=6n^2.
\]
Let $u=u_1\circ_{i_0}u_2\in \partial\CK(m_0,\dots,m_n)$ be an element and suppose $\ari(u_1)\geq 2$. For $0\leq t\leq 1/2$, by construction, we have
\[
|\homega_{i}^j(tu)(y)-\phi(y)|
\leq (1-s)|v-\phi(z_i^j(tu))|+ st|\homega_{i}^j(u)(y)-\phi(y)|.
\]
By the conditions \homegacomp, \tpsiev\ and  inductive hypothesis, we have
\[
\begin{split}
|\homega_{i}^j(u)(y)-\phi(y)|&= |\phi\circ \tpsi_{i}(u_1)\circ \omega_{i}^j(u_2)(y)-\phi(y)|
\\
&\leq |\phi|\cdot ( |\tpsi_{i}(u_1)(\omega_{i}^j(u_2)(y))-\omega_{i}^j(u_2)(y)|+|\omega_{i}^j(u_2)(y)-y|)\\
&
\leq |\phi|\cdot \left(k_{n_1}^{\psi}d(v,\phi(M)) 
+2k_{n_2}^{\omega}\frac{d(v_{(i_0)},\phi_{(i_0)}(M))}{|\phi_{(i_0)}|}\right)
\\
&
\quad \left(\because \ |\omega_{i}^j(u_2)(y)-y|\leq \frac{2}{|\phi_{(i_0)}|}|\homega_{i}^j(u_2)(y)-\phi_{(i_0)}(y)|\ \right)
\\
&\leq|\phi|\cdot \left( k_{n_1}^{\psi}+\frac{2k_{n_2}^{\omega}}{|\phi_{(i_0)}|}\right)d(v,\phi(M)).
\end{split}
\]
So by the condition \zev, 
\[
\begin{split}
|\homega_{i}^j(tu)(y)-\phi(y)|
&
\leq \left[(1-s)4n|\phi|+ st|\phi|\cdot \left(k_{n_1}^{\psi}+\frac{2k_{n_2}^{\omega}}{|\phi_{(i_0)}|}\right)\right]d(v,\phi(M))\\
&\leq \max\left\{4n|\phi|,\ |\phi|\cdot \left( k_{n_1}^{\psi}+2k_{n_2}^{\omega}\frac{1}{|\phi_{(i_0)}|}\right) \right\}d(v,\phi(M))
\end{split}
\]
So it is enough to prove
\[
|\phi|\cdot \left( k_{n_1}^{\psi}+2k_{n_2}^{\omega}\frac{1}{|\phi_{(i_0)}|}\right)\leq k^{\omega}_n\quad (n=n_1+n_2-1, n_1\geq 2).
\]
This inequality is equivalent to $10^n-2\cdot 10^{n_2}-6n_1^2\geq 0$. But this inequality is verified by easy calculation. \\
\indent The evaluation in the case $1/2\leq t\leq 1$ is much analogous: by construction, we have
\[
\begin{split}
|\homega_{i}^j(tu)(y)-\phi(y)|
&
\leq (1-s)(2-2t)|v-\phi(z_i^j(u)) |+((1-s)(2t-1)+st)|\homega^j_{i}(u)(y)-\phi(y)|
\\
&
\leq \max\left\{ 4n|\phi|, \  |\phi|\cdot \left( k_{n_1}^{\psi}+2k_{n_2}^{\omega}\frac{1}{|\phi_{(i_0)}|}\right) \right\}d(v,\phi(M))\\
&
\leq k_n^{\omega}d(v,\phi(M)).
\end{split}
\]
as above. The evaluation of the case $\ari(u_1)=1$ is much easier (so we omit it). The condition \homegaconti\ is proved completely analogously to Lemma \ref{Lshaftconti}. \qed
\subsubsection{Construction of $\hUpsilon$ and $\tUpsilon$ }\label{SSSactionBIM}
In this sub-subsection, we construct the actions $\hUpsilon$ and $\tUpsilon$, see sub-subsection \ref{SSSoutlineMCK}. This completes the proof of Theorem \ref{TMCKalgebra}. The only difference between the restriction of the action $\Omega$ to $s=1\in[0,1]$ and the action induced by the $\square$-object structure is the $\eps$-part. The former is a value of $\teps$ and the latter is $\min\{\eps_1,\dots, \eps_n\}$. $\tUpsilon$ gives a homotopy between these $\eps$-parts, and is stationary on the other parts. $\hUpsilon$ is defined as the restrction of $\tUpsilon$ to $s=1$, which is also factored through the restriction of $\Omega$ to $s=1$.   Put $J=[1,2]$. We first define a morphism of symmetric spectra
\[
\Upsilon_n :(\K(n)\times J)\hotimes \Gamma(M)^{\otimes n}\longrightarrow \Gamma(M)
\]
for each $n\geq 1$ (see sub-section \ref{SSatiyahduality} for the definition of $\Gamma(M)$). Let $u\in \K(n)$, $s\in J$, and $\langle \phi_i,\eps_i,s_i\rangle\in\Gamma(M)_{k_i}$ be elements ($i=1,\dots,n$).
We put  
\[
\begin{split}
\Upsilon_n(u,s\,;\langle \phi_1,\eps_1,\theta_1\rangle,\dots,\langle \phi_n,\eps_n,\theta_n\rangle)
&=\langle \phi_1\times\cdots\times\phi_n,\teps_s,\theta_1\wedge\cdots\wedge \theta_n\rangle,\\
\teps_s
=(2-s)\teps(u\,;(\phi_1,\eps_1)&,\dots,(\phi_n,\eps_n))+(s-1)\min\{\eps_1,\dots,\eps_n\}
\end{split}
\]
(see sub-subsection \ref{SSSconstructionLMT} for the definition of $\teps$). To define an action $\MK(\BIM^\bullet)\longrightarrow\BIM^\bullet$, we shall define a morphism
\[
\tUpsilon_{p_1,\dots,p_n}:\K(n)\hotimes \bigotimes_{i=1}^n\Inhom(F(M)^{\otimes p_i},\Gamma(M))^J\longrightarrow \Inhom(F(M)^{\otimes p_{\leq n}},\Gamma(M))^J
\]
for each $n\geq 1,p_1,\dots,p_n\geq 0$. By adjointness, this is equivalent to define a morphism
\[
(\K(n)\times J)\hotimes \bigotimes_{i=1}^n\Inhom(F(M)^{\otimes p_i},\Gamma(M))^J\otimes F(M)^{\otimes p_{\leq n}}\longrightarrow \Gamma(M). 
\]
We define this morphism as the following composition:
\[
\begin{split}
&(\K(n)\times J)\hotimes \bigotimes_{i=1}^n\Inhom(F(M)^{\otimes p_i},\Gamma(M))^J\otimes F(M)^{\otimes p_{\leq n}}
\\
\xrightarrow{\id\times \Delta_J\times \id}
&
(\K(n)\times J^{\times n+1})\hotimes \bigotimes_{i=1}^n\Inhom(F(M)^{\otimes p_i},\Gamma(M))^J\otimes F(M)^{\otimes p_{\leq n}}
\\
\xrightarrow{T}
&
(\K(n)\times J)\hotimes \bigotimes_{i=1}^n[J\hotimes\Inhom(F(M)^{\otimes p_i},\Gamma(M))^J]\otimes F(M)^{\otimes p_{\leq n}}
\\
\xrightarrow{\id \times ev_J^{\times n}}
&
(\K(n)\times J)\hotimes \bigotimes_{i=1}^n\Inhom(F(M)^{\otimes p_i},\Gamma(M))\otimes F(M)^{\otimes p_{\leq n}}
\\
\xrightarrow{T}
&
(\K(n)\times J)\hotimes \bigotimes[\Inhom(F(M)^{\otimes p_i},\Gamma(M))\otimes F(M)^{\otimes p_i}]
\\
\xrightarrow{\id\times ev_{F(M)}^{\times n}}
&
(\K(n)\times J)\hotimes \Gamma(M)^{\otimes n}
\xrightarrow{\Upsilon_n}
\Gamma(M)
\end{split}
\]
Here, $\Delta_J:J\to J^{\times n+1}$ is the diagonal, and $T$ is the appropriate transposition, and $ev_{-}$ is the evaluation. 
We also define a morphism 
\[
\hUpsilon_{p_1,\dots,p_n}:\K(n)\hotimes \bigotimes_{i=1}^n\Inhom(F(M)^{\otimes p_i},\Gamma(M))\longrightarrow \Inhom(F(M)^{\otimes p_{\leq n}},\Gamma(M))
\]
similarly to $\tUpsilon$ using the restriction of $\Upsilon_n$ to $(\K(n)\times \{1\})\hotimes \Gamma(M)^{\otimes n}$.
\begin{prop}\label{PtUpsilon}
The collection $\{\tUpsilon_{p_1,\dots,p_n}\}_{n,p_1,\dots,p_n}$ induces a well-defined action of $\MK$ on $\BIM^\bullet$, which we denote by $\tUpsilon$, and the collection $\{\hUpsilon_{p_1,\dots,p_n}\}_{n,p_1,\dots,p_n}$ induces a well-defined action of $\MK$ on $\athc^\bullet$, which we denote by $\hUpsilon$. These two actions $\tUpsilon$, $\hUpsilon$ and the action $\Omega$ constructed in sub-subsections \ref{SSSformulaTIM}, \ref{SSSconstructiontpsi} satisfy the compatibility conditions 1,2 stated in sub-subsection \ref{SSSoutlineMCK}.
\end{prop}
\begin{proof}
These are clear from the definition of $\Omega$ given in sub-subsections \ref{SSSformulaTIM}, \ref{SSSconstructiontpsi} and the conditions on $\teps$, and \tpsiend, \omegaend.
\end{proof}
If we replace $\teps_s$ with $\min\{\eps_1,\dots,\eps_n\}$ in the defintion of $\Upsilon_n$, the action of $\MK$ on $\BIM^\bullet$ becomes  the one naturally induced by the associative product on $\athc^\bullet$. \\
\indent By Proposition \ref{PtUpsilon}, we have completed the proof of Theorem \ref{TMCKalgebra}.
\subsection{Comparison of $A_\infty$-structures  on $\tot$ and  $\ttot$}\label{SScomparisontot}
In this subsection, we prove $\tot(X^\bullet)$ and $\ttot(X^\bullet)$ are weak equivalent as nu-$A_\infty$-ring spectra for a cs-spectrum $X^\bullet$ under some assumptions, see Corollary \ref{Ctot}. To do this we use the notion of colored operads. We mainly consider the case of non-symmetric operad as it is enough for the $A_\infty$ or little 1-cubes case, but to prove Theorem \ref{Tintro2} for general $n$, we need to consider the symmetric case, which is completely analogous to the non-symmetric case. In the last part of this subsection, we indicate the necessary changes for the symmetric case and prove Theorem \ref{Tintro2}. 
\begin{defi}\label{Dcoloredoperad}
\begin{itemize2}
\item[(1)] A \textit{(non-symmetric topological) colored operad} $\oper$ consists of 
\begin{enumerate}
\item a set $\OB(\oper) $ whose elements we call objects (or colors) of $\oper$, 
\item a family of spaces $\{\oper(c_1,\dots, c_n;d)\in\CG \mid n\geq 1, c_1,\dots, c_n, d\in\OB(\oper)\}$, 
\item a family of morphisms in $\CG$, 
\[
(-\circ -):\oper(c_1,\dots, c_n;d)\times\prod_{i=1}^n\oper(e^i_1,\dots, e^i_{k_i};c_i)\longrightarrow \oper(e^1_1,\dots, e^1_{k_1},\dots,e^n_1,\dots, e^n_{k_n};d),
\]
where $c_1,\dots,c_n,d,e^i_1,\dots e^i_{k_i}\in\OB(\oper)$, which  we call the composition of $\oper$ and 
\item an element $\id_c\in\oper(c;\,c)$ which we call the identity, for each $c\in \OB(\oper)$, 
\end{enumerate}
 which satisfy the associativity and unity laws which are exactly analogous to those of operad. For a colored operad $\oper$ and an object $c\in \OB(\oper)$, we define an operad $\oper_c$ by $\oper_c(n)=\oper(c^n;c)$ with the composition induced by the composition of $\oper$. Here $c^n$ denotes the $n$-tuple  $(c,c,\dots,c)$.
\item[(2)] Let $\oper$ be a colored operad. An $\oper$-\textit{algebra} $X$ is a collection of objects of $\SP$, $\{X_c\}_{c\in\OB(\oper)}$ equipped with a morphism 
$\oper(c_1,\dots,c_n;d)\,\hotimes \,X_{c_1}\otimes\cdots\otimes X_{c_n}\longrightarrow X_d$ for each integer $n\geq 1$ and each tuple  $(c_1,\dots,c_n,d)\in\OB(\oper)^{n+1}$ which satisfies compatibility conditions exactly analogous to those on algebras over an operad. For $c\in\OB(\oper)$, we always regard $X_c$ as an $\oper_c$-algebra by the obvious way.
\end{itemize2}
\end{defi}
\begin{thm}\label{Ttoper}
Let $\oper$ be a colored operad, and  $a, b$ two objects of $\oper$, and $\alpha\in\oper (a,b)$ an element. Let $X$ be an $\oper$-algebra. 
Then there exist 
\begin{itemize}
\item a topological operad $\toper$  with  two morphisms of operads $\zeta_a:\toper\rightarrow \oper_a$, $\zeta_b:\toper\to\oper_b$, and
\item an $\toper$-algebra $\tilde X$ with  two morphisms of symmetric spectra $\eta_a:\tilde X\to X_a$, $\eta_b:\tilde X\to X_b$
\end{itemize}
such that the following conditions are satisfied.
\begin{enumerate}
\item $\zeta_c$ is compatible with $\eta_c$ for $c=a, b$.
\item In  $\Ho(\CG)$ there exists an isomorphism $\toper(n)\cong \oper_a(n)\times^h_{\oper(a^n;\,b)}\oper_b(n)$ such that the following diagram commutes for $c=a,b$:
\[
\xymatrix{\toper(n)\ar[r]^{\cong\qquad\qquad}\ar[rd]^{\zeta_c}&\oper_a(n)\times^h_{\oper(a^n;\,b)}\oper_b(n)\ar[d]^{p_c}\\
&\oper_c(n).}
\]
Here $\oper_a(n)\times^h_{\oper(a^n;\,b)}\oper_b(n)$ denotes the homotopy fiber product of the diagram:\\
$\oper_a(n)\xrightarrow{\alpha\circ-}\oper(a^n;b)\xleftarrow{-\circ \alpha^{ n}}\oper_b(n)$,
and $p_c$ the natural projection.
\item $\eta_a$ is a level equivalence and 
the following diagram commutes in $\Ho(\SP)$:
\[
\xymatrix{\tilde X\ar[r]^{\eta_a}\ar[rd]^{\eta_b}&X_a\ar[d]^{\alpha_*}\\
&X_b.}
\] 
Here, $\alpha_*$ is the evaluation of the structure morphism $\oper(a;b)\, \hotimes\, X_a\longrightarrow X_b$ at $\alpha$.
\end{enumerate}

\end{thm}
\begin{cor}\label{Ctot}
Let $X^\bullet$ be a cs-spectrum. Suppose $X^\bullet$ is given a structure of $\square$-object or $\MK$-algebra. If the morphism $\tot(X^\bullet)\to\ttot(X^\bullet)$ induced by a weak equivalence $f:\tilde \Delta^\bullet\to \Delta^\bullet$ is a stable equivalence, then $\tot(X^\bullet)$ and $\ttot(X^\bullet)$ are equivalent as nu-$A_\infty$-ring spectra.
\end{cor}
\begin{proof}[Proof of Corollary \ref{Ctot} assuming Theorem \ref{Ttoper}]
We only prove the case of a $\square$-object. The case of $\MK$-algebra is completely analogous.\par
We define a colored operad $\oper$ as follows:
\[
\begin{split}
a=\Delta^\bullet, 
&
\quad b=\tilde\Delta^\bullet, 
\\
\oper(c_1,\dots,c_n;d)
&
=\Map_{\CCG}(d,c_1\, \square \,\cdots\, \square \, c_n),
\end{split}
\]
where the composition is naturally defined using the composition of $(\CCG)^{op}$. Clearly $\oper_a=\B$ and $\oper_b=\B'$. We put $\alpha=f$. We define an $\oper$-algebra $Y$ by $Y_c=\Map_{\CCG}(c,X^\bullet)$ with structure morphisms induced by the product on $X^\bullet$. Clearly $Y_a=\tot(X^\bullet)$ and $Y_b=\ttot(X^\bullet)$. We apply Theorem \ref{Ttoper} to $(\oper,Y,\alpha)$ and obtain an $\toper$-algebra $\tilde Y$ with morphisms $\eta_c:\tilde Y\to Y_c$ ($c=a,b$). $\toper$ is an $A_\infty$-operad by the second condition  of the theorem, and  the first and third conditions and the assumption of the corollary  imply $\eta_a$ and $\eta_b$ are weak equivalences of nu-$A_\infty$-ring spectra.
\end{proof}
The rest of this subsection is devoted to the proof of Theorem \ref{Ttoper}. The main task is  the construction of the operad $\toper$.
If the pullback $\oper_a(n)\times_{\oper(a^n;b)}\oper_b(n)$ has the `correct' homotopy type, we may define the $n$-th space of $\toper$ as this pullback and the composition as the component-wise composition. Our task is to construct a model of the homotopy pullback $\oper_a(n)\times^h_{\oper(a^n;b)}\oper_b(n)$ which carries a unital and associative composition. $\toper$ is defined in the part (5) of the following definition. It is something like 'multiple version' of Moore's loop space.  The reader who quickly wants to get some intuition  may jump to Example \ref{Etoper}.
\begin{defi}\label{Dtoper} Let $\oper$ be a colored operad and $a$, $b$ two objects of $\oper$, and $\alpha$  an element of $\oper(a\,;\, b)$. Let $\R_{\geq 0}$ 
denote the space of  non-negative 
real numbers with the usual  topology.
\begin{itemize2}
\item[(1)] We define an operad $\len$ as follows. Put  
$\len(n)=(\R_{\geq 0})^{\times n}$. For $l=(l_1,\dots,l_n)\in\len(n)$ and $l'=(l'_1,\dots,l'_m)\in\len(m)$, the composition $l\circ_il'\in\len(n+m-1)$ is given by
$l\circ_i l'=(l_1,\dots, l_{i-1},l_i+l'_1,\dots, l_i+l'_m,l_{i+1},\dots, l_n).$
The unit is $0\in\R_{\geq 0}$.
\item[(2)] We give the set $C=\{a,b\}$ an order $b< a$ and give the $n$ times product $C^n$ the  product partial order. We regard $C^n$ as a category in the usual manner. We define a functor 
$\FO: C^n\longrightarrow \CG$ as follows. We put $\FO(\col)=\oper(\col; b)$ for an element (object) $\col\in C^n$,  and for two elements  $\col=(c_1,\dots,c_n)$ and $\col'=(c'_1,\dots,c'_n)$ with $\col\leq \col'$, we put 
\[
\FO(\col\leq \col')(f)=f\circ (\beta_1,\dots,\beta_n),
\] 
where $\beta_i=\alpha$ if $c_i<c'_i$, and $\beta_i=\id$ if $c_i=c'_i$. 
\item[(3)] Let $l\in\len(n)$ and $\col=(c_1,\dots,c_n)\in C^n$. We put
\[
l_0=\max\{\,l_i\mid i=1,\dots, n\},\quad
l(\col )=\left\{
\begin{array}{ll}
0
&
\text{ if }\col=(a,a,\dots,a)
\\
\max\{\,l_i\mid c_i=b\}
&
 \text{ otherwise }
\end{array}\right. .
\]
We define a functor
$\GL:C^n\longrightarrow \CG$
by $\GL(\col)=[\,l(\col),l_0\,]$ and $\GL(\col\leq\col')$ being the usual inclusion. (Note that $l(\col')\leq l(\col)$ if $\col\leq\col'$.)
\item[(4)] For an element $l\in\len(n)$ let $\bar\oper(l)$ denote the set of all natural transformations $\GL\to \FO$. So an element of $\bar\oper(l)$ is a collection $\{p_\col\}_\col$ of paths $p_\col:[l(\col),l_0]\to \oper (\col ;b)$ which satisfy certain compatibility condition.  We define a space $\boper(n)\in\TOP$ as follows. As a set, we put 
\[
\bar \oper(n)=\coprod_{l\in\len(n)}\bar\oper(l).
\]
The topology of $\boper(n)$ is analogous to the compact-open topology of a mapping space. The open subbasis is defined as follows. Let $U$ be an open subset of $\len(n)$, $K$  a compact subset of $\R_{\geq 0}$ and $\{V_\col\}_{\col\,\in C^n}$ a collection of open subsets $V_\col$ of $\oper(\col;b)$. 
We put 
\[
B(U,K,\{V_\col\})=\{(l,p_\col)\in\bar\oper (n)|\ l\in U, \ p_\col(K\cap [l(\col),l_0 ])\subset V_\col\ {}^\forall \col\in C^n\}
\]
The open sub-basis of $\bar\oper(n)$ is $\{B(U,K,\{V_\col\})\}_{U,K,\{V_\col\}}$, where $U$ (resp. $K$, $V_\col$) runs through all open sets (resp. compact sets, open sets). We define a map $\lambda_n:\bar\oper(n)\to \oper(a^n,b)$ by $\lambda_n(l,\{p_\col\}_{\col\,\in C^n})=p_{a^n}(0)$. Clearly $\lambda_n$ is continuous.
 \item[(5)] Recall the notion of k-ification, which is a functor takes  a topological space to a compactly generated spaces, see \cite{hovey}. We define an operad $\toper$ as follows. $\toper(n)$ is the k-ification of the fiber product (in $\TOP$) of the following diagram:
\[
\oper_a(n)\xrightarrow{\alpha\circ-}\oper(a^n,b)\xleftarrow{\lambda_n}\bar\oper(n).
\] 
 Let $(f',l',\{p_{\col'}\}_{\col'\in C^n})\in\toper(n)$, and $(f'',l'',\{q_{\col''}\}_{\col''\in C^m})\in\toper(m)$. The $i$-th composition 
\[
(f,l,\{r_{\col}\})=(f',l',\{p_{\col'}\})\circ_i(f'',l'',\{q_{\col''}\})
\] is defined as follows. We put $f=f'\circ_i f''$ and $l=l'\circ_i l''$. For an element $\col=(c_1,\dots, c_{n+m-1})\in C^{n+m-1}$, we define three elements $\ocol_a$, $\ocol_b\in C^n$, and $\ucol\in C^m$ by $\ocola=(c_1,\dots, c_{i-1},a, c_{i+m},\dots, c_{n+m-1})$, $\ocolb=(c_1,\dots, c_{i-1},b, c_{i+m},\dots, c_{n+m-1})$, and $\ucol=(c_i,\dots,c_{i+m-1})$. If $l(\col)\leq l'_i$, we put
\[
r_{\col}(t)=\left\{
\begin{array}{ll}
p_{\ocola}(t)\circ _i f'' & (l(\col )\leq t \leq l'_i)\\
p_{\ocolb}(t)\circ _iq_{a^m}(t-l'_i)& (l'_i< t\leq l_0).
\end{array}\right.
\]
If $l(\col)>l'_i$, we put
\[
r_{\col}(t)=p_{\ocolb}(t)\circ _iq_{\ucol}(t-l'_i).
\]
Here, we regard $p_{\ocola}$, $p_{\ocolb}$ etc. as continuous maps from $\R$ by extending the domain using  constant maps if necessary.
\item[(6)] We define a morphism of operads $\zeta_a:\toper\to\oper_a$ as the projection to the first component, and another morphism $\zeta_b:\toper\to \oper_b$ by
$(f,l,\{p_\col\})\longmapsto p_{b^n}(l_0)$.
\end{itemize2}
\end{defi}

\begin{rem}\label{Rintuitivedef}
A more intuitive (but less convenient) description of $\boper(l)$. Let $L_1<\dots <L_N$ be the different values of $l_1,\dots, l_n$. Let $\col_j$ be the minimum of $\col$ such that $l(\col)=L_j$ (so $\col_N=b^n$). Then there is a natural bijection:
\[
\boper (l)\cong \oper (a^n;b)^{[0,L_1]}\times_{\oper (a^n;b)} \oper(\col_1;b)^{[L_1,L_2]}\times_{\oper (\col_1;b)}\cdots 
\times_{\oper(\col_{N-1};b)} \oper (\col_N;b).
\]
\end{rem}

\begin{exa}[$\toper(1)$, $\toper(2)$]\label{Etoper}
(1) An element of $\toper(1)$ is a 4-tuple $(f,l,p,q)$ consisting of 
\[
f\in \oper(a\,;a),\quad
l\geq 0,\quad p\in\Map(\,[0,l\,]\,, \oper(a\,;b)), \quad q\in\oper(b\,;b)
\] which satisfies $\alpha\circ f=p(0)$ and $p(l)=q\circ \alpha$. For two elements $(f,l,p,q),\ (f',l',p',q')\in\toper (1)$, the composition 
\[
(\bar f,\bar l,\bar p, \bar q)=(f,l,p,q)\circ (f',l',p',q')
\] is defined as follows. We put $\bar f=f\circ f'$ and $\bar l=l+l'$, and define $\bar p$ as in Figure (\ref{Fcompositiontoper(1)}). $\bar p$ is actually continuous because $p(l)\circ f'=q\circ \alpha\circ f'=q\circ p'(0)$. Note that the definition of composition is similar to that of the product of Moore's loop space. This composition is clearly associative and the unit is $(\id_a,0,\alpha,\id_b)$.  
 \\
\indent (2) An element of $\toper(2)$ is (essentially) a 6-tuple $(f,l_1,l_2,p,q,r)$ consisting of   
\[
\begin{split}
f\in\oper(a^2;a),\quad l_1\geq 0,\quad l_2\geq 0,\quad 
&
p\in\Map
([0,L_1]\, ,\oper(a^2;b)),
\\
&
 q\in\Map([L_1,L_2 ]\, , \oper(c_1,c_2\, ;b)),\quad r\in\oper(b^2;b),
\end{split}
\] where $L_1$ is the smaller of $l_1$, $l_2$ and $L_2$ is the larger, and  $(c_1,c_2)=(b,a)$ if $l_1\leq l_2$, and $(c_1,c_2)=(a,b)$ otherwise (see Figure \ref{Fbaroper(2)}). This 6-tuple is imposed the following conditions: 
\[
\begin{split}
\alpha\circ f=p(0),\quad p(l_1)=q(l_1)\circ_1\alpha,\quad  q(l_2)=r\circ_2\alpha
& 
\qquad \text{ if }l_1\leq l_2
\\
\alpha\circ f=p(0),\quad p(l_2)=q(l_2)\circ_2\alpha,\quad   q(l_2)=r\circ_1\alpha
&
\qquad \text{otherwise}.
\end{split}
\]  
 For two elements $(f,l_1,l_2,p,q,r)\in\toper(2)$,  and$(f',l',p',q')\in\toper(1)$, the composition 

\[
(\bar f,\bar l_1,\bar l_2,\bar p,\bar q,\bar r)=(f,l_1,l_2,p,q,r)\circ_1(f';l',p',q')
\]
 is defined as follows. We put $\bar f=f\circ f'$, $\bar l_1=l_1+l'$, $\bar l_2=l_2$, and $\bar r=r\circ_1 q'$ and define $\bar p$ and $\bar q$ as in Figure (\ref{Fcompositiontoper(2)}). 
\end{exa}

Now we shall define the $\toper$-algebra $\tilde X$ in Theorem \ref{Ttoper}. 
\begin{defi}\label{DtildeX} Let $\oper$ be a colored operad, and $a, b$ two objects of $\oper$, and $\alpha\in\oper(a\,;b)$. Let $X$ be an $\oper$-algebra.
\begin{itemize2}
\item[(1)] A symmetric spectrum $\tilde X$ is defined as follows. We first define a space $\bar X'_k$ for each $k\geq 0$. An element of $\bar X'_k$ is a pair $(L, h)$ of non-negative number $L\in \R_{\geq 0}$ and a path $h:[0,L]\to X_{b,k}$. $\bar X'_k$ is topologized  analogously to  $\boper(n)$ in Definition \ref{Dtoper}. We define $\bar X_k$ as the k-ification of the quotient $\bar X'_k/\{(L,*_L)\mid L\geq 0\}$ where $*_L$ denotes the constant path at the basepoint, and regard the point represented by $\{(L,*_L)\mid L\geq 0\}$ as the basepoint of $\bar X_k$. The sequence $\bar X=\{\bar X_k\}_k$ is  regarded as a symmetric spectrum in the obvious manner. Then, we put
\[
\tilde X:=X_a\times_{\alpha,X_b, ev_0}\bar X,
\]
where $ev_0:\bar X\to X_b$ is the evaluation at $0\in [0,L]$.
We shall define an action of $\toper$ on $\tilde X$. Let $(f,l,\{p_{\col}\})\in\toper (n)$ and $(x_i,\langle L_i,h_i\rangle)\in \tilde X_{k_i}$   ($i=1,\dots, n$).  The value of the structure morphism $\toper(n)\,\hotimes\, \tilde X^{\otimes n}\to  \tilde X$, 
\[
(x,\langle L, h\rangle)=(f,l,\{p_{\col}\})((x_1,\langle L_1,h_1\rangle),\dots, (x_n,\langle L_n,h_n\rangle))
\] is defined as follows:
\[
\begin{split}
x&=f(x_1,\dots, x_n),\\
L&=\max \{ l_1+L_1,\dots, l_n+L_n\},\\
h(t)&=p_{\col}(t)(h'_1(t),\dots, h'_n(t)).
\end{split}
\]
In the last formula, we write $\col=\min\{\,\col'\mid l(\col')\leq t\,\}$ and 
\[
h'_i(t)=
\left\{
\begin{array}{ll}
x_i & t\leq l_i\\ 
h_i(t-l_i) & \text{otherwise,}
\end{array}\right.
\]
and  extend the domains of $p_\col$ and $h'_i$ using  constant maps if necessary.\par
\item[(2)] We define a morphism of symmetric spectra $\eta_a: \tilde X\to X_a$ as the projection to the first component, and another morphism  $\eta_b:\tilde\to X_b$ by $(x,\langle L,h\rangle)\longmapsto h(L)$. 
\end{itemize2}
\end{defi}
We shall show the definitions given above are well-defined.
\begin{lem}\label{Lk-ification}
For each $n\geq 0$, the space $\toper(n)$ is compactly generated.
\end{lem}
\begin{proof}
It is enough to see the k-ification of $\bar\oper(n)$ is weak Hausdorff. We shall define a continous monomorphism $\bar\oper(n)\to \len (n)\hat\times\hat\prod_{\col}\Map(\R_{\geq 0},\oper(\col,b))$ by 
$(l,\{p_\col \})\longmapsto (l,\{\tilde p_\col\})$.
Here, $\tilde p_\col$ is the extension of $p_\col$ by a constant map and $\hat\times$ and $\hat\prod$ denote the products in $\TOP$. 
Clearly, the k-ification of $\len (n)\hat\times\hat\prod_{\col}\Map(\R_{\geq 0},\oper(\col,b))$ is weak Hausdorff and so is the k-ification of $\bar\oper(n)$.
\end{proof}
\begin{lem}\label{Lwelldefinedtoper}
The topological operad $\toper$ given in Definition \ref{Dtoper}, (5) is well-defined. 
\end{lem}

\begin{proof}
We use the notations in Definiton \ref{Dtoper}, (5). By  naturality, 
\[
p_{\ocola}(l'_i)\circ _i f''=(p_{\ocolb}(l'_i)\circ_i \alpha)\circ_i f''=p_{\ocolb}\circ_iq_{a^m}(0),
\] which implies $r_{\col}:[l(\col ),l_0]\to\oper(\col ;b)$ is continuous. It is  a trivial routine work  to verify $\{r_\col\}$ forms a natural transformation and  the composition of $\toper$ is associative. The unit is $(\id_a,l_1=0,\alpha,\id_b)$.\par
 We shall show the composition is continuous. Let $\toper'(n)$ denote  the fiber product of the diagram 
\[
\oper_a(n)\xrightarrow{\alpha\circ-}\oper(a^n,b)\xleftarrow{\lambda_n}\bar\oper(n),
\]
in $\TOP$.  As the k-ification preserves  products (see \cite{hovey}), it is enough to show the map
\[
(-\circ_i-):\toper'(n)\times\toper'(n)\longrightarrow \toper'(n+m-1)
\]
defined by the same formula as the composition of $\toper$ is continuous.
 Take elements $(f',l',\{p_{\col'}\})\in\toper'(n)$ and $(f'',l'',\{q_{\col''}\})\in\toper'(m)$. Let $R\times B(U,K,\{V_\col\})$ be an open neighborhood of $(f,l,\{r_{\col}\}):=(f',l',\{p_{\col'}\})\circ_i(f'',l'',\{q_{\col''}\})$. Let $L_1<\dots <L_N$ be all the different values of $l_1,\dots,l_{n+m-1}$. We put
\[
A_j=\{\col\in C^{n+m-1}\mid l(\col) =L_j\}.
\]
Let $j_0$ be the number such that $L_{j_0}\leq l'_i< L_{j_0+1}$. Let $j$ be a number with $1\leq j\leq j_0$. By an elementary argument, we see there exist 
\begin{enumerate}
\item a number $\epsilon_j>0$,
\item an integer $M_j\geq 1$,
\item a compact subset $K^{j,k}\subset\R_{\geq 0}$ for each $k=1,\dots, M_j$,
\item an open subset $W^{j,k}(\ocolb)\subset\oper (\ocol_b;b)$ for each $k=1,\dots, M_j$ and $\col\in A_j$,
\item an open subset $X^{j,k}(\ucol)\subset\oper (\ucol\,; b)$ for each $k=1,\dots, M_j$ and $\col\in A_j$, and
\item an open subset $Y^j(\ocola)\subset\oper (\ocola; b)$  for each  $\col\in A_j$, and 
\item an open subset $Z^j\subset\oper_a(n)$
\end{enumerate}
which satisfy the following conditions:
\begin{enumerate}
\item $K^{j,1}\cup\cdots\cup K^{j ,M_j}=K\cap [l'_i\, ,l_0]$,
\item  $p_{\ocolb}(K^{j,k}_{\pm\eps_j})\subset W^{j,k}(\ocolb)$, \ $q_{\ucol}(K^{j,k}_{\pm\eps_j}-l'_i)\subset X^{j,k}(\ucol)$,
\item $p_{\ocola}(([L_j,l'_i]\cap K)_{\pm \eps_j})\subset Y^j(\ocola)$ and $f''\in Z^j$,
\item $(-\circ_i-)(W^{j,k}(\ocolb)\times X^{j,k}(\ucol))\subset V_\col$ and $(-\circ_i-)(Y^j(\ocola)\times Z^j)\subset V_\col$.
\end{enumerate}
Here, for a subset $S\subset \R_{\geq 0}$ and a number $\delta>0$, we use the following notations:
\[ 
S_{\pm \delta}=\{t\in \R_{\geq 0} \mid [t-\delta,t+\delta]\cap S\not=\emptyset \},\quad S-\delta=\{t\mid t+\delta\in S\}.
\]
\indent Similarly, for a number $j$ with  $j_0<j<N$, we can take 
\begin{enumerate}
\item a number $\epsilon_j>0$,
\item an integer $M_j\geq 1$,
\item a compact subset $K^{j,k}\subset\R_{\geq 0}$ for each $k=1,\dots, M_j$,
\item an open subset $W^{j,k}(\ocolb)\subset\oper (\ocol_b;b)$ for each $k=1,\dots, M_j$ and $\col\in A_j$,
\item an open subset $X^{j,k}(\ucol)\subset\oper (\ucol; b)$ for each $k=1,\dots, M_j$ and $\col\in A_j$,
\end{enumerate}
satisfying conditions similar to the above first and second conditions (but in the first condition, $l_i'$ is replaced by $L_j$).\par
We put $\eps_0=\min\{|L_j-L_{j'}| \mid j\not= j'\}$ and put $\eps=\frac{1}{2}\min\{\eps_0,\eps_1,\dots, \eps_N\}$. Let $U_1\times U_2\subset\len(n)\times\len(m)$ be an open neighborhood of $(l',l'')$ such that $(-\circ_i-)(U_1\times U_2)\subset U$, and for each $(\bar l',\bar l'')\in U_1\times U_2$, $\max\{|\bar l'_i-l'_i|, |\bar l''_i-l''_i|\}<\eps$. Let $R_1\times R_2\subset \oper_a(n)\times \oper_a(m)$ be an open neighborhood of $(f',f'')$ such that $(-\circ_i-)(R_1\times R_2)\subset R$. We put
\[
\begin{split}
D&^{j,k}=K^{j,k}_{\pm \eps_j},\quad E^j= ([L_j,l'_i]\cap K)_{\pm\eps_j},\\
S_1&=R_1\times\Big(\bigcap_{j=1}^{N}\bigcap_{k=1}^{M_j}B\big( U_1,D^{j,k},\{ W^{j,k}(\ocolb )|\col\in A_j\}\big)\cap \bigcap_{j\leq j_0}B\big (U_1,E^j, \{Y^j(\ocola)|\col\in A_j\}\big)\Big)\ \subset \toper'(n)\\
S_2&=\big( R_2\cap\bigcap_{j\leq j_0}Z_j\big) \times \bigcap_{j=1}^{N}\bigcap_{k=1}^{M_j}B\big( U_2,D^{j,k},\{ X^{j,k}(\ucol )|\col\in A_j\}\big)\ \subset \toper'(m)
\end{split}
\] 
Here precisely speaking, $\{ W^{j,k}(\ocolb )|\col\in A_j\}$ denotes the collection $\{ W^{j,k}_{\col'}\}_{\col'\in C^n}$ defined by \[
\bar W^{j,k}_{\col'}=\left\{
\begin{array}{ll}
W^{j,k}(\ocolb)
&
 \text{ if } \col'=\ocolb \text{ for some }\col\in A_j
\\
\oper(\col' ;b)
&
\text{ otherwise}
\end{array}\right. ,
\]
and $\{Y^j(\ocola)|\col\in A_j\}$ and $\{ X^{j,k}(\ucol )|\col\in A_j\}$ are similarly understood.\par
By definition, $(f',l',\{p_{\col'}\})\in S_1$ and $(f'', l'',\{q_{\col''}\})\in S_2$. We shall show $(-\circ_i -)(S_1\times S_2)\subset R\times B(U,K,\{V_\col\})$. Take elements $(\bar f',\bar l',\{\bar p'_{\col'}\})\in S_1$ and $(\bar f'',\bar l'',\{\bar q_{\col''}\})\in S_2$ and put $(\bar f,\bar l,\{\bar r_{\col}\})=(\bar f',\bar l',\{\bar p_{\col'}\})\circ_i (\bar f'',\bar l'',\{\bar q_{\col''}\})$. Take $\col\in C^{n+m-1}$ and let $l(\col)=L_j$. Then $\bar l(\col)\in(L_j-2\eps,L_j+2\eps)$. Take $t\in [\bar l(\col),\bar l_0]\cap K$.  If $t\leq \bar l'_i$, as $\bar l'_i<l'_i+\eps$, we have $t\in E^j$, which implies  $\bar r_{\col}(t)=\bar p_{\ocola}(t)\circ_i \bar f''\in V_\col$. On the other hands, if $t> \bar l'_i$, as $\bar l'_i\geq l'_i-\eps$, we have $t\in D^{j,k}$ and $t-l'_i-\eps<t-\bar l'_i<t-l'_i+\eps$, which imply $\bar p_{\ocolb}(t)\circ_i \bar q_{\ucol}(t-\bar l'_i)\in V_\col$. Thus, we have $(\bar f,\bar l,\{\bar r_{\col}\})\in R\times B(U,K,\{V_\col\})$.
\end{proof}

Proof of the following lemma is similar to that of Lemma \ref{Lwelldefinedtoper} and much easier so omitted.
\begin{lem}\label{LwelldefinedtildeX}
The $\toper$-algebra $\tilde X$ given in Definition \ref{DtildeX}, (1) is well-defined.\qed
\end{lem}
To prove Theorem \ref{Ttoper}, we need some lemmas. We shall define two continuous maps
$F:\boper (n)\longrightarrow \oper_ b (n)$, $F':\oper _b(n)\longrightarrow \boper (n)$
by $F(l,\{p_\col\})=p_{b^n}$, and $F'(x)=(1^n,\{x_{\col}\})$, where $1^n$ denotes $(1,\dots,1)\in \len(n)$ and $x_{\col}$ is defined by $x_{\col}(t)=F_{\oper}(b^n\leq\col)(x)$.

\begin{lem}\label{Lhomotopytypetoper1}
$F$ and $F'$ are homotopy equivalences which are homotopy inverse to each other.
\end{lem}
\begin{proof}
Clearly, $F\circ F'$ is the identity, so it is enough to give a homotopy $H:\boper (n)\times [0,2]\longrightarrow \boper (n)$ such that $H_0=id$ and $H_2=F'\circ F$.  $(l^s,\{p^s_{\col}\})=H(l,\{p_\col\},s)$ is defined as follows. When $0\leq s\leq 1$, we put $l^s(i)=(1-s)l(i)+sl(0)$, and $p^s_{\col}(t)=p_\col(t)$. (Note that $l^s(i)\geq l(i)$ so this is possible.) When $1\leq s\leq 2$, we put $l^s(i)=(2-s)l(0)+s-1$ and $p^s_{\col}(t)=p_{\col}(l(0))$. 
\end{proof}
\begin{lem}\label{Lhomotopytypetoper2}
The map $\lambda_n:\boper (n)\longrightarrow \oper_a(n)$ is a Serre fibration.
\end{lem}
\begin{proof}
Consider a commutative diagram
\[
\xymatrix{D^{k-1}\times\{1\}\ar[r]^f\ar[d]^{i_1}& \boper (n)\ar[d]^{\lambda_n}\\
D^{k-1}\times[0,1]\ar[r]^g&\oper( a^n,b)}
\]
We define a lift $h:D^{k-1}\times [0,1]\to \boper (n)$, $D^{k-1}\times [0,1]\ni (u,s)\mapsto (l^{u,s},p^{u,s}_\col)$ by 
\[
\begin{split}
(l^{u,1},\{p^{u,1}_\col\} )&=f(u),\quad
l^{u,s}(i)=l^{u,1}(i)+1-s,\\
p^{u,s}_{\col}(t)&=\left\{
\begin{array}{ll}
g(u,t) &(0\leq t\leq 1-s)\\
p_\col(t+s-1) &(1-s\leq t\leq l^{u,s}(i)).
\end{array}\right.
\end{split}
\]
\end{proof}

\begin{lem}\label{Leta}
$\eta_a$ defined in Definition \ref{DtildeX},(2) is a level equivalence.
\end{lem}
\begin{proof}
In fact, for each $k\geq 0$, $\eta_{a,k}:\tilde X_k\to X_{a,k}$ has a homotopy inverse: $\eta'_k:X_{a,k}\to \tilde X_k$, $x\mapsto (0,\alpha(x))$. A homotopy between $\eta'_k\circ \eta_{a,k}$ and $id$ is given by $((x,\langle L,h\rangle),s)\mapsto (x, \langle sL,h|_{[0,sL]}\rangle)$.
\end{proof}

\begin{proof}[Proof of Theorem \ref{Ttoper}]
The topological operad $\toper$ defined in Definition \ref{Dtoper},(5) is well-defined by Lemmas \ref{Lk-ification} and \ref{Lwelldefinedtoper}. Verification of the first condition of the theorem (compatibility of $\zeta_c$ with $\eta_c$) is trivial. The second condition easily follows from Lemmas \ref{Lhomotopytypetoper1}, \ref{Lhomotopytypetoper2} and the  equation $\lambda_n\circ F'=\alpha_{b^n,a^n}$. The third condition also easily follows from Lemma \ref{Leta}. 
\end{proof}
\begin{rem}\label{Rmorphismoperad}
It is straightforward to construct certain homotopy invariant version of colored operad of morphisms  by extending  $\toper$. Here, the colored operad of morphisms of a colored operad $\oper$ is defined by saying that objects are 1-array morphisms of $\oper$ and $n$-array morphisms are pairs of $n$-array morphisms of $\oper$ compatible with 1-array morphisms at the  sources and targets.  In fact, $\toper$ will fit in it as  the endomorphism operad of $\alpha$. This construction should be equivalent to a special case of internal hom-objects of dendoroidal sets introduced by Moerdijk and Weiss \cite{mw07}. Our construction is  more intuitive and efficient for our purpose i.e., we can easily determine the homotopy type of $\toper(n)$ and  construct the accompanying algebra $\tilde X$.
\end{rem}
We shall prove Theorem \ref{Tintro2}. We first state it more precisely. In \cite{ms} two operads $\Doper_n$ and $\tDoper_n$ both of which are weak equivalent to the little $n$-cubes operads are introduced (see sections 9  and 15 of \cite{ms}). In our notation, $\Doper_1=\B$ and $\tDoper_1=\tB$. Let $X^\bullet$ be a cosimplicial space or symmetric spectrum. It is proved that a  $\Xi^n$-algebra structure on $X^\bullet$ (see Definition 4.5 and 8.4 of \cite{ms}) induces an action of $\Doper_n$ (resp. $\tDoper_n$) on $\tot(X^\bullet)$ (resp. $\ttot(X^\bullet)$) in a functorial way. 
\begin{thm}[precise statement of Theorem \ref{Tintro2}]\label{Ttot}
Let $X^\bullet$ be a $\Xi^n$-algebra. Suppose the  morphism $\tot(X^\bullet)\to \ttot(X^\bullet)$ induced from a weak equivalence $\tilde\Delta^\bullet\to\Delta^\bullet$ is a weak equivalence. Then, there exist a symmetric operad $\toper$, $\toper$-algebra $\tilde Y$, weak equivalences $\zeta_1:\toper\to \Doper_n$, $\zeta_2:\toper\to\tDoper_n$ of symmetric operads, and weak equivalences $\eta_1:\tilde Y\to \tot(X^\bullet)$,  $\eta_2:\tilde Y \to \ttot(X^\bullet)$ such that $\eta_i$ is compatible with $\zeta_i$ for $i=1,2$. Here a weak equivalence of spaces or symmetric spectra means any of a weak homotopy equivalence, level equivalence,  or stable equivalence.
\end{thm}
\begin{proof}
A symmetric colored operad is a non-symmetric colored operad equipped with an action of the $k$-th symmetric group $\Sigma_k$ on each morphism space $\oper(c_1,\dots,c_k;d)$ for each $k\geq 1$. Of course, this action and composition satisfy the compatibility analogous to (uncolored) symmetric operads. If $\oper$ is a symmetric colored  operad, the space $\toper(k)$ has a natural action of $\Sigma_k$ induced from that of $\oper$ and permutations of components of $\len(k)=\R^k_{\geq 0}$. Theorem \ref{Ttoper} is still valid if we replace all morphisms of operads in the statement by morphisms of symmetric operads. Now the proof of Theorem \ref{Ttot} is completely analougous to Corollary \ref{Ctot}. The only necessary change is to replace $c_1\, \square \,\dots \, \square \, c_k$ by $\Xi_k^n(c_1,\dots,c_k)$ (see Definition 8.4 of \cite{ms}). 
\end{proof}

\begin{figure}[H]
\begin{center}
\begin{picture}(190,70)(-109,-24)
\put(-80,5){\line(1,0){160}}
\put(-80,10){\line(0,-1){10}}
\put(80,10){\line(0,-1){10}}
\put(0,10){\line(0,-1){10}}
\multiput(0,10)(0,-4){11}
    {\line(0,-1){2}}
\multiput(-80,10)(0,-4){11}
    {\line(0,-1){2}}
\multiput(80,10)(0,-4){11}
    {\line(0,-1){2}}
\multiput(-100,5)(4,0){5}
{\line(1,0){2}}
\put(-95,15){$t$}
\put(-98,-20){$\bar p(t)$}
\put(-70,-20){$p(t)\circ f$}
\put(10,-20){$q\circ p'(t-l)$}
\put(-80,15){$0$}
\put(0,15){$l$}
\put(70,15){$l+l'$} 

\end{picture}
\caption{Composition of two elements $(f,l,p,q)$, $(f',l',p',q')\in\toper(1)$}\label{Fcompositiontoper(1)}
\end{center}
\end{figure}
\begin{figure}[H]
\begin{center}
\begin{picture}(400,70)(-199,-34)
\put(-200,-34){
\begin{picture}(180,70)(-99,-24)
\put(-80,5){\line(1,0){160}}
\put(-80,10){\line(0,-1){10}}
\put(80,10){\line(0,-1){10}}
\put(0,10){\line(0,-1){10}}
\multiput(0,10)(0,-4){11}
    {\line(0,-1){2}}
\multiput(-80,10)(0,-4){11}
    {\line(0,-1){2}}
\multiput(80,10)(0,-4){11}
    {\line(0,-1){2}}
\put(-70,-20){$p(t)\in\oper(a^2;b)$}
\put(10,-20){$q(t)\in\oper(b,a;b)$}
\put(-80,15){$0$}
\put(0,15){$l_1$}
\put(80,15){$l_2$}
\put(-10,30){$(l_1\leq l_2)$}   
\end{picture}}
\put(0,-35){
\begin{picture}(180,70)(-99,-24)
\put(-80,5){\line(1,0){160}}
\put(-80,10){\line(0,-1){10}}
\put(80,10){\line(0,-1){10}}
\put(0,10){\line(0,-1){10}}
\multiput(0,10)(0,-4){11}
    {\line(0,-1){2}}
\multiput(-80,10)(0,-4){11}
    {\line(0,-1){2}}
\multiput(80,10)(0,-4){11}
    {\line(0,-1){2}}
\put(-70,-20){$p(t)\in\oper(a^2;b)$}
\put(10,-20){$q(t)\in\oper(a,b;b)$}
\put(-80,15){$0$}
\put(0,15){$l_2$}
\put(80,15){$l_1$}
\put(-10,30){$(l_1> l_2)$}   
\end{picture}}
\end{picture}
\caption{An element of $\toper(2)$ ($f\in\oper_a(2)$ and $r\in\oper_b(2)$ are omitted)}\label{Fbaroper(2)}
\end{center}
\end{figure}
\begin{figure}[H]
\begin{center}
\begin{picture}(300,150)(-149,-99)
\put(-35,25){(a)\quad  $l_1+l'\leq l_2$}
\put(-135,0){\line(1,0){270}}
\put(-135,5){\line(0,-1){10}}
\put(-45,5){\line(0,-1){10}}
\put(45,5){\line(0,-1){10}}
\put(135,5){\line(0,-1){10}}
\multiput(-135,5)(0,-4){20}
    {\line(0,-1){2}}
\multiput(-45,5)(0,-4){20}
    {\line(0,-1){2}}
\multiput(45,5)(0,-4){20}
    {\line(0,-1){2}}
\multiput(135,5)(0,-4){20}
    {\line(0,-1){2}}
\multiput(-160,0)(4,0){74}
    {\line(1,0){2}}
\multiput(-160,-40)(4,0){74}
    {\line(1,0){2}}
\multiput(-160,-75)(4,0){74}
    {\line(1,0){2}} 
\put(-155,10){$t$}
\put(-135,10){$0$}   
\put(-45,10){$l_1$}
\put(45,10){$l_1+l'$}
\put(135,10){$l_2$}
\put(-155,-25){$\bar p(t)$}
\put(-120,-25){$p(t)\circ_1f'$}
\put(-35,-25){$q(t)\circ_1p'(t-l_1)$}
\put(-155,-60){$\bar q(t)$}
\put(60,-60){$q(t)\circ_1q'$}
\end{picture}\\
\begin{picture}(300,150)(-149,-99)
\put(-35,25){(b)\quad $l_1\leq l_2< l'+ l_1$}
\put(-135,0){\line(1,0){270}}
\put(-135,5){\line(0,-1){10}}
\put(-45,5){\line(0,-1){10}}
\put(45,5){\line(0,-1){10}}
\put(135,5){\line(0,-1){10}}
\multiput(-135,5)(0,-4){20}
    {\line(0,-1){2}}
\multiput(-45,5)(0,-4){20}
    {\line(0,-1){2}}
\multiput(45,5)(0,-4){20}
    {\line(0,-1){2}}
\multiput(135,5)(0,-4){20}
    {\line(0,-1){2}}
\multiput(-160,0)(4,0){74}
    {\line(1,0){2}}
\multiput(-160,-40)(4,0){74}
    {\line(1,0){2}}
\multiput(-160,-75)(4,0){74}
    {\line(1,0){2}} 
    \put(-155,10){$t$}
\put(-135,10){$0$}   
\put(-45,10){$l_1$}
\put(45,10){$l_2$}
\put(135,10){$l_1+l'$}
\put(-155,-25){$\bar p(t)$}
\put(-120,-25){$p(t)\circ_1f'$}
\put(-35,-25){$q(t)\circ_1p'(t-l_1)$}
\put(-155,-60){$\bar q(t)$}
\put(60,-60){$r\circ_1p'(t-l_1)$}
\end{picture}\\
\begin{picture}(300,150)(-149,-99)
\put(-35,25){(c)\quad $l_2< l_1$}
\put(-135,0){\line(1,0){270}}
\put(-135,5){\line(0,-1){10}}
\put(-45,5){\line(0,-1){10}}
\put(45,5){\line(0,-1){10}}
\put(135,5){\line(0,-1){10}}
\multiput(-135,5)(0,-4){20}
    {\line(0,-1){2}}
\multiput(-45,5)(0,-4){20}
    {\line(0,-1){2}}
\multiput(45,5)(0,-4){20}
    {\line(0,-1){2}}
\multiput(135,5)(0,-4){20}
    {\line(0,-1){2}}
\multiput(-160,0)(4,0){74}
    {\line(1,0){2}}
\multiput(-160,-40)(4,0){74}
    {\line(1,0){2}}
\multiput(-160,-75)(4,0){74}
    {\line(1,0){2}} 
    \put(-155,10){$t$}
\put(-135,10){$0$}   
\put(-45,10){$l_2$}
\put(45,10){$l_1$}
\put(135,10){$l_1+l'$}
\put(-155,-25){$\bar p(t)$}
\put(-120,-25){$p(t)\circ_1f'$}
\put(-155,-60){$\bar q(t)$}
\put(-30,-60){$q(t)\circ_1f'$}
\put(60,-60){$r\circ_1p'(t-l_1)$}
\end{picture}
\caption{Composition of $(f;l_1,l_2;p,q,r)\in\toper(2)$ and 
$(f';l';p',q')\in\toper(1)$}\label{Fcompositiontoper(2)}
\end{center}
\end{figure}


\subsection{Proof of Theorem \ref{mainthm1} (2)}\label{SSproofTmain2}
The following lemma easily follows from compatibility of totalization and homotopy pushout in \cite{bousfield}
\begin{lem}\label{LtotCL}
The morphism $\tot(\CL^\bullet)\to \ttot(\CL^\bullet)$ induced by a weak equivalence $\tilde\Delta^\bullet\to \Delta^\bullet$ is a level equivalence.\qed
\end{lem}
Let $F:\CL^\bullet\to X^\bullet$ be a term-wise stable fibrant replacement. We must prove the induced morphism $\ttot(\CL^\bullet)\to \ttot(X^\bullet)$ is a stable equivalence.
\begin{defi}
\begin{itemize2}
\item[(1)] We say a symmetric spectrum $X$ is a \textit{strongly semi-stable object} (
\textit{sss-object} for short) if there exists a number $\alpha>1$ such that for sufficiently large $k$,
 the canonical map $\pi_i(X_k)\to \pi_{i+1}(X_{k+1})$ is an isomorphism for $0\leq i\leq \alpha k$.
\item[(2)]  $H\Z\in \SP$ denotes a fixed cofibrant model of  the Eilenberg-MacLane spectrum of $\Z$ and $-\otimes^LH\Z$ denotes the derived tensor product. Explicitly, let $q_X:QX\to X$ be the cofibrant replacement in $\SP$, and we set $(X\otimes^LH\Z)_k$ to be the free abelian topological group generated by the points of $X_k$ other than the basepoint. The action of $S^1$ and permutations are defined in the obvious way. 
\end{itemize2}
\end{defi}
We can take the morphism $q_X$ as a trivial fibration by the axioms of a model category, and trivial fibrations in $\SP$ are level equivalences. So If $X$ and $Y$ are level equivalent, i.e., connected by a zigzag of level equivalences, $X\otimes^LH\Z$ and $Y\otimes^LH\Z$ are also level equivalent, and we can consider $X\otimes^LH\Z$ up to level equivalences. It makes sense to say $X\otimes^LH\Z$ is an sss-object for an object $X\in \SP$. \par 
The first part of the following lemma trivially follows from the fact that stable fibrant objects are $\Omega$-spectra (in the sense of \cite{hss}) and the second part  is proved in section 5.6 of\cite{hss}.
\begin{lem}[\cite{hss}]\label{Lhss}
\begin{itemize2}
\item[(1)] Any stable fibrant object is an sss-object. 
\item[(2)] Any stable equivalence between sss-objects is a $\pi_*$-isomorphism.\qed
\end{itemize2}
\end{lem}
\begin{lem}\label{Liroiro}
\begin{itemize2}
\item[(1)] A finite homotopy limit of sss-objects in the level model structure  is also  a sss-object and is (weak equivalent to)  the homotopy limit    in the stable model structure. Here, a finite homotopy limit is a homotopy limit over a category whose sets of morphisms and objects are finite.
\item[(2)] If $X$ is a sss-object, $X\otimes^L H\Z$ is also a sss-object. 
\item[(3)] $(-\otimes^L H\Z)$ preserves finite homotopy limits in the stable model structure.
\item[(4)] For each $p\geq 0$, $\CL^p$ is a sss-object.
\end{itemize2}
\end{lem}
\begin{proof}
(1) follows from  Lemma \ref{Lhss}. (2) and (4) are trivial. (3) follows from the fact that homotopy pullback squares and homotopy pushout squares coincide in the stable model structure of $\SP$. 
\end{proof}
\begin{prop}\label{PinvarianceCL}
The morphism $\ttot(\CL^\bullet)\to \ttot(X^\bullet)$ induced by $F$ is a stable equivalence. 
\end{prop}
\begin{proof}
Let $G:X^\bullet\otimes^{L} H\Z\to Y^\bullet$ be a termwise  stable fibrant replacement. \par
We first prove the composition
\[ 
\ttot(X^\bullet)\otimes^{L}H\Z\longrightarrow \ttot(X^\bullet\otimes^LH\Z)
\xrightarrow{G_*} \ttot(Y^\bullet)
\] is a stable equivalence.\par
By Thom isomorphism, we have $ \tilde H_t(\CL^s_k)\cong H_{t-k+d}(M^{\times n}\times (M\times \TV_k))$ ($d=dim M$). So $N^s\tilde H_t(\CL^\bullet_k)=0$ if $t-k+d\leq 2s$. So  $N^s\tilde H_t^S(\CL^\bullet)=0$ if $t\leq 2s-d$. Here $N^s$ is the part of degree $s$ of the usual normalization and $\tilde H_t^S(X)=\colim_k H_{t+k}(X_k) $, where maps in the sequence is give by the cross product with a fundamental class of $S^1$. So we have
\[
\begin{split}
H^S_t(N^s(X^\bullet))\cong \pi_t^S((N^sX^\bullet)\otimes^L H\Z)
&
\cong \pi_t^S(N^s(X^\bullet\otimes^L H\Z))
\\
&
\cong N^s\pi_t^S(X^\bullet\otimes^L H\Z)\cong N^s\tilde H^S_t(\CL^\bullet)=0 
\end{split}
\]
if $t\leq 2s-d$. (See \cite{gj} for $N^s$ of a cosimplicial space and we apply it in the levelwise manner.) Here, the first isomorphism is trivial, and the second one follows from Lemmas \ref{Lhss} and \ref{Liroiro}, (1),(2),(3) as $N^s$ is a finite homotopy limit in the level model structure, and the third is proved in \cite{gj}, and the forth follows from Lemmas \ref{Lhss} and \ref{Liroiro},(4). (For the definition of $N^s$ for cosimplicical spaces, see \cite{gj} and we apply it in a levelwise manner.)
This and the Hurewicz theorem imply $\pi_t^S(N^s(X^\bullet))=0$ for $t\leq 2s-d$. This and the homotopy fiber sequence
\[
\Omega^sN^sX^\bullet\to \ttot^s(X^\bullet)\to \ttot^{s-1}(X^\bullet)
\]
(see \cite{gj}) imply $\ttot^s(X^\bullet)\to \ttot^{s-1}(X^\bullet)$ is $s-d$-connected. Here $\ttot^l$ is the homotopy limit of the restriction to the full subcategory of $\Delta$ consisting of $[k]$ with $k\leq l$. 
Consider the following diagram:
\[
\xymatrix{\ttot(\CL^\bullet)\otimes^L H\Z\ar[r]^{F_*}\ar[rd]&\ttot(X^\bullet)\otimes^LH\Z\ar[r]\ar[d]^{G_*}
&
\ttot^s(X^\bullet)\otimes^LH\Z\ar[d]^{G_*}\\
&\ttot(Y^\bullet)
\ar[r]&\ttot^s(Y^\bullet).
}
\]
By the preceding argument and Hurewicz theorem,  the top right arrow   is  $s-d$-connected. Similarly, the right bottom arrow  is also $s-d$-connected. As $\ttot^s$ is a finite homotopy limit, The right vertical arrow is a stable equivalence by Lemma \ref{Liroiro} (3).  Thus, the middle vertical arrow is $s-d$-connected for all $s\geq 0$ hence a  stable equivalence. 
We easily see the slanting arrow is also a stable equivalence using the homology spectral sequence of cosimplicial space. Thus the top left horizontal arrow is also a stable equivalence. As both the objects $\ttot(\CL^\bullet)$ and $\ttot(X^\bullet)$ are connected, i.e., $\pi_k^S=0$ for all sufficiently small $k\in \Z$, and sss-objects by Lemma \ref{LtotCL}, we see the morphism $F_*:\ttot(\CL^\bullet)\to \ttot(X^\bullet)$ is a stable equivalence by the Hurewicz theorem.
\end{proof}
We shall complete the proof of Theorem \ref{mainthm1} (2). In fact, gathering results and construction we obtained by now,  we have the following zig-zag of weak equivalences of nu-$A_\infty$-ring spectra.
\[
\begin{split}
\LMT\stackrel{\text{(A)}}{\cong}\tot(\CL^\bullet)&
\stackrel{\text{(B)}}{\simeq}
\ttot(\CL^\bullet)\stackrel{\text{(C)}}{\simeq}\ttot(\IM^\bullet)\\
&\stackrel{\text{(D)}}{\simeq}
\ttot(\thc^\bullet(A',B))\stackrel{\text{(E)}}{\simeq}
\tot(\thc^\bullet(A',B))=
\thc(A',B)\stackrel{\text{(F)}}{\simeq}
\thc(Q,Q)
\end{split}
\]
The isomorphism (A) is given in subsection \ref{SScosimplicialLMT}, and the equivalences (B) and (E) follow from Corollary \ref{Ctot}, Lemma \ref{LtotCL} and the fact that $\thc^\bullet(A,B')$ is a Reedy fibrant object ,which implies the morphism $\tot(\thc^\bullet(A,B'))\to\ttot(\thc^\bullet(A,B'))$ is a (level) equivalence. \\
\indent By Lemma \ref{LweqIM}, the morphism $p_0:\IM^\bullet\to \CL^\bullet$ is a termwise level equivalence so induces a level equivalence  $(p_0)_*:\ttot(\IM^\bullet)\to \ttot(\CL^\bullet)$.  By the same lemma and the Atiyah duality (Theorem \ref{Tatiyah} ), the morphism $\bar q_2$ in Definition \ref{Dmorphismscosimplicial} is a termwise stable equivalence. This fact, Proposition \ref{PinvarianceCL}, and the fact that $\thc^\bullet(A',B)$ is termwise stable fibrant imply the induced morphism  
$(\bar q_2)_*:\ttot(\IM^\bullet) \to \ttot(\thc^\bullet(A',B))$ is a stable equivalence. By these facts and  Theorem \ref{TMCKalgebra},
 the  equivalences (C) and (D) hold. The equivalence (F) is proved in Proposition \ref{propinvarianceofthc}. Thus, we have proved $\LMT$ and $\thc(Q,Q)$ are equivalent as nu-$A_\infty$-ring spectra.\qed


\begin{thebibliography}{99}

\bibitem{adams} J. F. Adams, {\em Infinite loop spaces,}  Annals of Mathematics Studies 90 (1978).

\bibitem{atiyah} MF. Atiyah, {\em Thom complexes,} Proceedings of  London Mathematical Society, 11 (1961), 291-310.

\bibitem{bousfield} A. K. Bousfield,  {\em On the homology spectral sequence of a cosimplicial space,} American Journal of Mathematics 109 (1987), 361-394.

\bibitem{ekmm} A. D.  Elmendorf, I. Kriz, M. A. Mandell, J. P. May, {\em Rings, Modules, and Algebras
in Stable Homotopy Theory,} Mathematical Surveys and Monographs,  47 (1997) Amer.
Math. Soc. 

\bibitem{cs} M. Chas,  D. Sullivan. {\em String topology}, arXiv preprint math/9911159 (1999).

\bibitem{cm11} 
Denis-Charles Cisinski,  I. Moerdijk, {\em Dendroidal sets as models for homotopy operads,} Journal of Topology 4 (2011), no. 2, 257-299.
\bibitem{cooke} G. Cooke, {\em Replacing homotopy actions by topological actions}, Trans. Amer. Math.
Soc. 237 (1978), 391–406.
\bibitem{muro} F. Muro {\em Homotopy theory of nonsymmetric operads,} Algebraic \& Geometric Topology 11 (2011), 1541-1599.
\bibitem{cj} R. L. Cohen, J. D. S. Jones, {\em A homotopy theoretic realization of string topology,} Mathematische Annalen 324  (2002),  no.4, 773-798.

\bibitem{cohen} R. L. Cohen, {\em Multiplicative properties of Atiyah duality,} Homology, Homotopy \& Applications 6 (2004), no.1, 269-281.
\bibitem{dugger} D. Dugger, {\em A primer on homotopy colimits}, preprint, available at https://pages.uoregon.edu/ddugger/hocolim.pdf
\bibitem{ekmm} A. D. Elmendorf, I. Kriz, M. A. Mandell, and J. P. May, {\em Rings, modules, and algebras in stable homotopy theory ,}  Mathematical Surveys and Monographs 47 (1997).

\bibitem{ft2} 
Y. Felix, Jean-Claude Thomas, {\em Rational BV-algebra in string topology,} Bull. Soc. math. France 136  (2008), no. 1, 311-327.

\bibitem{ft} 
Y. Felix, Jean-Claude Thomas, {\em String topology on Gorenstein spaces,} Mathematische Annalen 345 (2009), no.2, 417-452.
\bibitem{FGV} 
Z. Fiedorowicz, Steven Gubkin, and R. M. Vogt. {\em Associahedra and weak monoidal structures on categories,}  arXiv:1005.3979 (2010).
\bibitem{fukaya} K. Fukaya, {\em Application of Floer homology of Langrangian submanifolds to symplectic topology,} Morse theoretic methods in nonlinear analysis and in symplectic topology, Springer Netherlands, (2006), 231-276.



\bibitem{gerstenhaber} M. Gerstenhaber, {\em The cohomology structure of an associative ring,} Annals of Mathematics 79 (1963), no.2, 267-288.

\bibitem{gj} 
P. G. Goerss, J. F. Jardine, {\em Simplicial Homotopy Theory,} Progress in Mathematics 174 (2009).

\bibitem{gh} 
M. Goresky, N. Hingston, {\em Loop products and closed geodesics,} Duke Mathematical Journal 150 (2009), no. 1, 117-209.
\bibitem{gh} 
M. Goresky, N. Hingston, {\em Loop products and closed geodesics,} Duke Mathematical Journal 150 (2009), no. 1, 117-209.


\bibitem{hirschhorn} P. S. Hirschhorn {\em Model categories and their localizations} Mathematical Survey \& Monographs, No. 99, American Mathematical Soc., (2009).


\bibitem{hovey} M. Hovey, {\em Model Categories,} Mathematical Surveys and Monographs 63  (1999), AMS, Providence, RI .

\bibitem{hss} M. Hovey, B. Shipley, and J. Smith, {\em Symmetric spectra,} Journal of the American Mathematical Society 13 (2000), no. 1, 149-208.
\bibitem{irie} K. Irie, {\em A chain level Batalin-Vilkovisky structure in string topology via de Rham chains,} Int. Math. Res. Notices, (2018) no. 15, 4602-4674. 
\bibitem{jones} J. D. S. Jones, {\em Cyclic homology and equivariant homology} Inventiones mathematicae 87 (1987), no.2,  403-423.

\bibitem{klein}
J. Klein, {\em Fiber products, Poincare duality and $A_\infty$-ring spectra,} Proceedings of the American Mathematical Society 134 (2006), no.6, 1825-1833.
\bibitem{kuribayashi} 
K. Kuribayashi,  {\em The Hochschild cohomology ring of the singular cochain algebra of a space,}  arXiv:1006.0884 (2010).
\bibitem{ls} 

\bibitem{kmn12} K. Kuribayashi, L. Menichi and T. Naito, {\em Derived string topology and the Eilenberg-Moore spectral sequence,} arXiv:1211.6833 (2012).
\bibitem{kmn13} K. Kuribayashi, L. Menichi and T Naito, {\em Behavior of the Eilenberg-Moore spectral sequence in derived string topology,}  Topology and its Applications 164 (2014), 24-44. 
\bibitem{KNWY} K. Kuribayashi, T. Naito, S. Wakatsuki, T. Yamaguchi, (2021). {\em A reduction of the string bracket to the loop product,}  preprint, arXiv:2109.10536 (2021).
\bibitem{km} I. Kriz,  J. P. May, {\em Operads, algebras, modules and motives,} Societe mathematique de France (1995).

\bibitem{ls} P. Lambrechts,  D. Stanley, {\em Poincare Duality and Commutative Differential Graded Algebras,} Annales Scientifiques de l'Ecole Normale Superieure 41 (2008), No. 4, 1029-1052.




\bibitem{maclane} 
S. Mac Lane, {\em Categories for the working mathematician,} Graduate Texts in Mathematics 5 (1998).
\bibitem{malm} E. J. Malm,  {\em String topology and the based loop space,} arXiv preprint arXiv:1103.6198 (2011).

\bibitem{mms} M. A. Mandell, J. P. May, S. Schwede, and B. Shipley, {\em Model categories of diagram spectra,} Proceedings of the London Mathematical Society 82 (2001), no. 2, 441-512.
\bibitem{menichi} 
L. Menichi,  {\em String topology for spheres,} Commentarii Mathematici Helvetici 84 (2009), no. 1, 135-157.

\bibitem{ms2} J. E. McClure, J. H. Smith. {\em Multivariable cochain operations and little $n$-cubes,} Journal of the American Mathematical Society 16 (2003), no.3, 681-704.

\bibitem{ms} J. E. McClure, J. H. Smith, {\em Cosimplical objects and little n-cubes I,} American Journal of Mathematics 126 (2004), no.5 1109-1153.
\bibitem{moriya} S. Moriya,  {\em Multiplicative formality of operads and Sinha’s spectral sequence for long knots,} Kyoto Journal of Mathematics 55.1 (2015) 17-27.
\bibitem{moriya1} S. Moriya, {\em Models for knot spaces and Atiyah duality}, preprint, available at arXiv. 
\bibitem{mw07} 
I. Moerdijk and I. Weiss,{\em Dendroidal sets}, Algebraic \& Geometric Topology 7 (2007),
1441-1470.

\bibitem{naito} T. Naito, {\em String operations on rational Gorenstein spaces,} arXiv:1301.1785 (2013).


\bibitem{pr} 
K. Poirier, N. Rounds, {\em Compactifying string topology,}  arXiv:1111.3635 (2011).
\bibitem{rs} B. Richter, B. Shipley,  {\em An algebraic model for commutative $H\mathbb{Z}$-algebras,} Algebraic \& Geometric Topology, 17(4), (2017) 2013-2038.
\bibitem{sinha2}
D. Sinha,  {\em Operads and knot spaces,} Journal of the American Mathematical Society 19 (2006), no.2, 461-486.
\bibitem{sinha} D. P. Sinha, {\em The topology of spaces of knots: cosimplicial models,} American Journal of Mathematics 131 (2009), no. 4,  945-980.
\bibitem{stasheff} 
J. D. Stasheff, {\em Homotopy associativity of H-spaces I,} Transactions of the American Mathematical Society (1963), 275-292.

\bibitem{tsopmene} P. A. Songhafouo Tsopm\'en\'e,  {\em Formality of Sinha’s cosimplicial model for long knots spaces and the Gerstenhaber algebra structure of homology,} Algebraic \& Geometric Topology 13 (2013), 2193-2205.
 \bibitem{ungheretti} M. Ungheretti, {\em Free loop space and the cyclic bar construction,}  Bulletin of the London Mathematical Society, 49(1), (2017) 95-101.


















\end{thebibliography}
\end{document}